\documentclass[12pt,a4paper]{article}

\usepackage{latexsym}
\usepackage{blindtext}

\PassOptionsToPackage{force}{filehook}

\usepackage[utf8]{inputenc}
\usepackage[english]{babel}
\usepackage{arabtex}
\usepackage{utf8}
\usepackage{graphicx}
\usepackage{graphics,subfigure}
\usepackage{fancyhdr}
\usepackage{lmodern}
\usepackage{color}
\usepackage{transparent}
\usepackage{float}
\usepackage{pythontex}
\usepackage{tikz-cd}
\usepackage{listings}
\usepackage[linktocpage=true,citebordercolor=green, urlcolor=black]{hyperref}
\usepackage{tabularx}
\usepackage[title,titletoc]{appendix}
\usepackage{xcolor}
\usepackage[mathscr]{eucal}
\usepackage{calrsfs}

\usetikzlibrary{shapes.misc}

\definecolor{codegreen}{rgb}{0,0.6,0}
\definecolor{codegray}{rgb}{0.5,0.5,0.5}
\definecolor{codepurple}{rgb}{0.58,0,0.82}
\definecolor{backcolour}{rgb}{0.95,0.95,0.92}

\usepackage{pgf,tikz}
\usepackage{caption}
\usetikzlibrary{decorations.markings}
\usetikzlibrary{arrows}
\usetikzlibrary{patterns}
\usepackage{tikz-3dplot}
\usetikzlibrary{calc}
\usepackage{amsfonts}
\usepackage{amsmath,amssymb,amsthm}
\usepackage{mathtools}
\usepackage{relsize}
\usepackage{titlesec}

\usepackage{stmaryrd}
\usepackage{trimclip}

\usepackage{enumitem}

\usepackage{wasysym}

\definecolor{grgr}{RGB}{33,189,63}

\makeatletter
\DeclareRobustCommand{\shortto}{%
	\mathrel{\mathpalette\short@to\relax}%
}

\newcommand{\short@to}[2]{%
	\mkern2mu
	\clipbox{{.5\width} 0 0 0}{$\m@th#1\vphantom{+}{\shortrightarrow}$}%
}
\makeatother

\newtheorem{theorem}{Theorem}[section]

\theoremstyle{definition}
\newtheorem{definition}{Definition}[section]
\newtheorem{example}[definition]{Example}
\newtheorem{remark}[definition]{Remark}
\newtheorem{corollary1}[definition]{Corollary}
\newtheorem{lemma1}[definition]{Lemma}
\newtheorem{theorem1}[definition]{Theorem}
\newtheorem{Proposition}[definition]{Proposition}
\newtheorem*{proposition*}{Proposition}

\titleformat{\paragraph}
{\normalfont\normalsize\bfseries}{\theparagraph}{1em}{}
\titlespacing*{\paragraph}
{0pt}{3.25ex plus 1ex minus .2ex}{1.5ex plus .2ex}

\usepackage[margin=2.5cm]{geometry}

\usepackage{mathptmx}

\begin{document}
	
\setcounter{page}{1}
\pagenumbering{arabic}	

\title{From Babylonian lunar observations to Floquet multipliers and Conley--Zehnder Indices}

\author{Cengiz Aydin}




\maketitle

\begin{abstract}
	The lunar periods of our moon - the companion of the Earth - which date back to the Babylonians until around 600 BCE, are 29.53 days for the synodic, 27.55 days for the anomalistic and 27.21 days for the draconitic month.\ In this paper we define and compute these periods in terms of Floquet multipliers and Conley--Zehnder indices for planar periodic orbits in the spatial Hill lunar problem, which is a limit case of the spatial circular restricted three body problem.\ For very low energies, we are able to prove analytically the existence of the families of planar direct (family $g$) and retrograde periodic orbits (family $f$) and to determine their Conley--Zehnder index.\ For higher energies, by numerical approximations to the linearized flow, we also study other families of planar and spatial periodic orbits bifurcating from the families $g$ and $f$.\ Moreover, our framework provide an organized structure for the families, especially to see how they are connected to each other.\ Since the solutions we analyze are of practical interest, our work connects three topics:\ Babylonian lunar periods, symplectic geometry, and space mission design.\	
\end{abstract}

\makeatletter{\renewcommand*{\@makefnmark}{}

\tableofcontents

\section{Introduction}

\subsection{Main results}

The moon is a complex dynamical system.\ Indeed, it is attracted not only by the earth but also with a comparable force by the sun.\ To this complexity George William Hill (1838--1914) made a good first approximation in 1877 and 1878 (see \cite{hill_det} and \cite{hill}).\ In particular, in his equation the true trajectory of the moon is close to a periodic orbit centered at the earth, which is known as ``variational orbit''.\ In other words, the true orbit is almost periodic.\ The variational orbit was found by Hill in \cite[p.\ 259]{hill} numerically.\ It is a planar symmetric direct periodic orbit
of Hill's system.\ Based on these works by Hill, Hénon \cite{henon}, \cite{henon_0} numerically explored and studied the stability of the families of planar direct and retrograde periodic orbits, which are referred to as family $g$ and $f$, respectively.\ These two families are the fundamental families of symmetric planar periodic orbits in the Hill lunar system.\

\textbf{On the geometry of planar periodic orbits in the spatial problem.}\ The Hamiltonian describing the motion of the moon
\begin{align*}
T^* \big( \mathbb{R}^3 \setminus \{ (0,0,0) \} \big) \to \mathbb{R},\quad (q,p) \mapsto \underbrace{\frac{1}{2}|p|^2 - \frac{1}{|q|} + p_1q_2 - p_2q_1 - q_1^2 + \frac{1}{2}q_2^2 + \frac{1}{2}q_3^2}_{\textstyle \frac{1}{2}\big( (p_1 + q_2)^2 + (p_2 - q_1)^2 + p_3^2 \big) - \frac{1}{|q|} - \frac{3}{2}q_1^2 + \frac{1}{2}q_3^2}
\end{align*}
is invariant under the symplectic involution
$$ \sigma \colon T^* \mathbb{R}^3 \to T^* \mathbb{R}^3,\quad (q_1,q_2,q_3,p_1,p_2,p_3) \mapsto (q_1,q_2,-q_3,p_1,p_2-p_3) $$
which arises from the reflection at the ecliptic $\{q_3=0\}$.\ Moreover, the planar problem can be viewed as the restriction of this system to the fixed point set $ \text{Fix}(\sigma) = \{ (q_1,q_2,0,p_1,p_2,0) \}$.\ For a planar periodic orbit $q=(q,p) \in \text{Fix}(\sigma)$ with $q_0 = \big(q(0),p(0)\big)$ and first return time $T_q$, we consider the time $T_q$ map of the linearized Hamiltonian flow, which is a linear symplectomorphism and is called the monodromy.\ Its restriction to the 5 dimensional energy hypersurface $\Sigma$ induces on the quotient by the line bundle ker$\omega|_{\Sigma} = \langle X_H | _{\Sigma} \rangle \subset T\Sigma$ the reduced monodromy, which is the linear symplectic map
$$ \overline{d \varphi _H^{T_q} | _{\Sigma} (q_0) } \colon T_{q_0} \Sigma / \text{ker} \omega_{q_0} \to T_{q_0} \Sigma / \text{ker} \omega_{q_0}.$$
Since the decomposition $ T_{q_0} \Sigma = T_{q_0}\text{Fix}(\sigma | _{\Sigma}) \oplus E_{-1} \big( d \sigma ( q_0 ) \big)$ into the 3-dimensional planar and 2-dimensional spatial components, the induced 4-dimensional symplectic vector space splits into two 2-dimensional symplectic vector spaces
$$ T_{q_0} \Sigma / \text{ker} \omega_{q_0} = \big( T_{q_0} \text{Fix}(\sigma|_{\Sigma}) / \text{ker} \omega_{q_0} \big) \oplus \Big( E_{-1} \big( d \sigma ( q_0 ) \big)  \Big). $$
Hence the reduced monodromy splits into
$$\overline{d \varphi _H^{T_q} | _{\Sigma} (q_0) }  = \begin{pmatrix}
\overline{A}_p & 0 \\
0 & A_s
\end{pmatrix},$$
where $\overline{A}_p$ is the reduced monodromy of $q$ viewed in the planar problem and $A_s$ is a $2 \times 2$ symplectic matrix which arises by linearization only along the spatial components.\ With respect to this symplectic splitting, we can study the linearization in the planar and spatial direction independently from each other.\ In particular, we obtain the following two important properties.\
\begin{itemize}[noitemsep]
	\item[i)] The Floquet multipliers, which are the eigenvalues of the reduced monodromy, are determined by the Floquet multipliers of $\overline{A}_p$ and $A_s$, which are real or lie on the unit circle.\ Consequently, it is not possible that the Floquet multipliers are given by four different complex numbers of the form $\lambda, 1/\lambda, \overline{\lambda}$ and $1/\overline{\lambda}$.\
	\item[ii)] The transversal Conley--Zehnder index of $q$ splits additively 
				$$ \mu_{CZ} = \mu_{CZ}^p + \mu_{CZ}^s, $$
				where $\mu_{CZ}^p$ and $\mu_{CZ}^s$ are the Conley--Zehnder indices of the path of symplectic matrices generated by the planar and spatial part of the linearized Hamiltonian flow, respectively.\
\end{itemize}
Furthermore, if the planar orbit $q$ is planar and spatial elliptic, i.e.\ the Floquet multipliers are on the unit circle, then $\overline{A}_p$ and $A_s$ are conjugate to rotations in $\mathbb{R}^2$, respectively.\ In particular, $\mu_{CZ}^p$ and $\mu_{CZ}^s$ measure the number of complete rotations of neighbouring orbit during $T_q$, respectively.\ In this paper, for elliptic cases, we will define the \textbf{synodic month} of $q$ by $T_q$ expressed in days, the \textbf{anomalistic} and \textbf{draconitic period} as the mean time (in days) for a complete rotation of planar and spatial neighbouring orbits during $T_q$, respectively.\ These periods of $q$ are explicitely determined in terms of their Floquet mutlipliers (rotation angles) and Conley--Zehnder indices.\

\textbf{For very low energies.}\ To determine the Conley--Zehnder indices for the families $g$ and $f$, we go down to very low energies.\ In particular, after regularization, for very small energies the Hill lunar problem approaches the regularized Kepler problem, and the Kepler flow is just the geodesic flow.\ For all sufficiently small energies up until an undetermined $\varepsilon_0>0$, we are able to prove analytically the following two theorems.\
\begin{theorem}[Planar problem] \label{theorem_a}
	From the circular direct and retrograde periodic orbit in the Kepler problem one family of periodic orbits bifurcates in each case, which are referred to as direct periodic orbits (family $g$) and retrograde periodic orbits (family $f$), respectively.\ These two orbits exist for all sufficiently low energies $\varepsilon \in (0,\varepsilon_0]$, and
	\begin{align} \label{conley_zehnder_index_planar}
	\mu_{CZ}^p = \begin{cases}
	3 & \text{ for family $g$}\\
	1 & \text{ for family $f$.}
	\end{cases}
	\end{align}
\end{theorem}
\begin{theorem}[Spatial problem] \label{theorem_b}
	The planar families $g$ and $f$, and two families of spatial collision periodic orbits bifurcate from the Kepler problem.\ These four orbits exist for all sufficiently low energies $\varepsilon \in (0,\varepsilon_0]$, and
	\begin{align} \label{conley_zehnder_index_spatial}
	\mu_{CZ} = \begin{cases}
	6 & \text{ for family $g$ (planar)}\\
	4 & \text{ for the one family of collision orbits bouncing back (spatial)}\\
	4 & \text{ for the other family of collision orbits bouncing back (spatial)}\\
	2 & \text{ for family $f$ (planar).}
	\end{cases}
	\end{align}
	Moreover, in view of $\mu_{CZ} = \mu_{CZ}^p + \mu_{CZ}^s$, by (\ref{conley_zehnder_index_planar}) and (\ref{conley_zehnder_index_spatial}),
	\begin{center}
		\begin{tabular}{c|c|c}
			 & family $g$ (planar) & family $f$ (planar)\\
			\hline $\mu_{CZ}$ / $\mu_{CZ}^p$ / $\mu_{CZ}^s$ & 6 / 3 / 3 & 2 / 1 / 1
		\end{tabular}
	\end{center}
\end{theorem}
\noindent
Geometrically, during $T_q$ the planar and spatial neighbouring orbits of the planar direct periodic orbit make a complete rotation and additionally rotate by an angle, hence the anomalistic and draconitic periods are shorter than the synodic one.\ The planar and spatial neighbouring orbits of the planar retrograde periodic orbits only rotate by their respective angles during $T_q$, thus in this case the anomalistic and draconitic periods are longer than the synodic one.\ Note that the families $g$ and $f$ are in $ \text{Fix}(\sigma) = \{ (q_1,q_2,0,p_1,p_2,0) \} $ and that the two families of spatial collision periodic orbits are in $ \text{Fix}(-\sigma) = \{ (0,0,q_3,0,0,p_3) \}$.\ Note that $-\sigma$ corresponds to a rotation around the $q_3$-axis by $\pi$.\
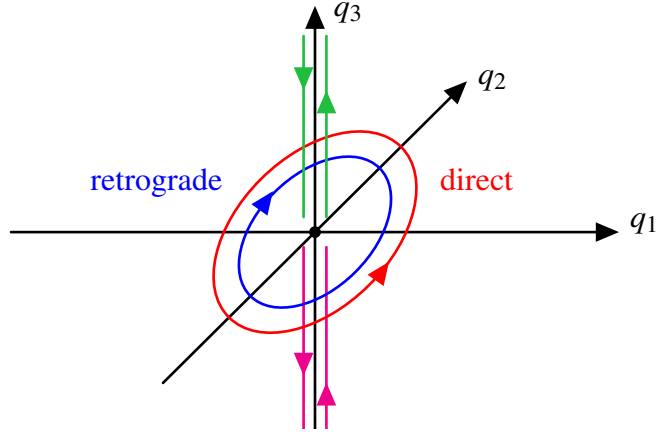
\begin{figure}[H]
	\centering
	\definecolor{grgr}{RGB}{33,189,63}
	\begin{tikzpicture}[line cap=round,line join=round,>=triangle 45,x=1cm,y=1cm]
	\clip(-5,-2.6) rectangle (5,3.2);
	

	\draw [->,line width=1pt] (-2,-2)-- (2,2);
	\draw [->,line width=1pt] (-4,0)-- (4,0);
	\draw [->,line width=1pt] (0,-2.8)-- (0,3);
	
	
	\draw [decoration={markings, mark=at position 0.5 with {\arrow{<}}},postaction={decorate},,color=magenta,line width=1pt] (-0.15,-3) -- (-0.15,-0.2);
	\draw [decoration={markings, mark=at position 0.75 with {\arrow{<}}},postaction={decorate},,color=magenta,line width=1pt] (0.15,-0.2) -- (0.15,-3);
	
	
	\draw [decoration={markings, mark=at position 0.8 with {\arrow{>}}},postaction={decorate},line width=1pt,rotate=45,color=red] (0,0) ellipse (1.6cm and 1cm);
	\draw [decoration={markings, mark=at position 0.3 with {\arrow{<}}},postaction={decorate},line width=1pt,rotate=45,color=blue] (0,0) ellipse (1.2cm and 0.75cm);
	
	
	\draw [decoration={markings, mark=at position 0.3 with {\arrow{>}}},postaction={decorate},,color=grgr,line width=1pt] (-0.15,2.6) -- (-0.15,0.2);
	\draw [decoration={markings, mark=at position 0.7 with {\arrow{>}}},postaction={decorate},,color=grgr,line width=1pt] (0.15,0.2) -- (0.15,2.6);
	
	\draw (4,0.4) node[anchor=north west] {$q_1$};
	\draw (2,2.3) node[anchor=north west] {$q_2$};
	\draw (0.1,3.2) node[anchor=north west] {$q_3$};
	\draw[color=red] (1.5,1) node[anchor=north west] {$\text{direct}$};
	\draw[color=blue] (-3.1,1) node[anchor=north west] {$\text{retrograde}$};
	\begin{scriptsize}
	\draw [fill=black] (0,0) circle (2pt);
	\end{scriptsize}
	\end{tikzpicture}
	\caption{Periodic orbits for very low energies}
	\label{figure_spatial_collision_orbits}
\end{figure}
\textbf{Symmetries.}\ All the periodic orbits we consider are symmetric with respect to an anti-symplectic involution.\ In particular, the spatial problem is endowed with linear symmetries, which are classified in this paper.\ By a symmetry we mean an anti-symplectic or symplectic involution which leaves the Hamiltonian invariant.\ The only linear symplectic symmetries are $ \pm\sigma$ and $\pm \text{id}$, and the anti-symplectic ones are:\
\begin{table}[H] \centering
	\begin{tabular}{c|c}
		notation & underlying geometry\\
		\hline $\rho_1(q,p)=(q_1,-q_2,q_3,-p_1,p_2,-p_3)$ & reflection at the $q_1q_3$-plane \\
		$\rho_2(q,p)=(-q_1,q_2,q_3,p_1,-p_2,-p_3)$ & reflection at the $q_2q_3$-plane \\
		$\overline{\rho_1}(q,p)=(q_1,-q_2,-q_3,-p_1,p_2,p_3)$ & rotation around the $q_1$-axis by $\pi$ \\
		$\overline{\rho_2}(q,p)=(-q_1,q_2,-q_3,p_1,-p_2,p_3)$ & rotation around the $q_2$-axis by $\pi$
	\end{tabular}
	\caption{The linear anti-symplectic symmetries}
	\label{linear_anti_s_sym}
\end{table}
\noindent
In view of
$$ \rho_i \circ \sigma = \sigma \circ \rho_i = \overline{\rho_i},\quad \text{for }i \in \{1,2\},$$
and
$$ \rho_i \circ \rho_j = \overline{\rho_i} \circ \overline{\rho_j} = - \sigma,\quad \rho_i \circ \overline{\rho_j} = \overline{\rho_j} \circ \rho_i = -\text{id},\quad \text{for }i,j \in \{1,2\},\quad i\neq j, $$
these eight symmetries form the group $\mathbb{Z}_2 \times \mathbb{Z}_2 \times \mathbb{Z}_2$.\ If we restrict the system to the planar case, then two linear anti-symplectic symmetries are given by the reflection at the $q_1$- and $q_2$-axis, and two symplectic ones by $\pm \text{id}$.\ These four form a Klein four-group $\mathbb{Z}_2 \times \mathbb{Z}_2$.\

\textbf{Numerical results for higher energies.}\ We follow and study numerically the families $g$ and $f$ and also other families which bifurcate from them and have been found numerically as well, in the following way.\ If one of the indices jumps, then there bifurcates a new family of planar or spatial periodic orbits, respectively.\ These bifurcations occur when the eigenvalue 1 is crossed.\ In particular, if the rotation angle is a $\tilde{k}$-th root of unity, then the $\tilde{k}$-th cover moves through the eigenvalue 1.\ To determine the index of this kind of new families, the Floquet multipliers alone do not provide enough information.\ We will show that the monodromy and reduced monodromy of symmetric periodic orbits satisfy special symmetries, which allows for a general construction to specify the Floquet multipliers, in particular the rotation angles, and thereby the index jump.\ In addition, by using the invariance of the local Floer homology before and after bifurcation, and the stability, we will determine the Conley--Zehnder index of the families in the Table \ref{families_in_this_paper}.\ Note that the Euler characteristic of the local Floer homology groups stays invariant as well.\ 
\begin{table}[H] \centering
	\begin{tabular}{c|c|c}
		from the articles by & family & remark\\
		\hline Hénon \cite{henon} (1969) & $g$, $f$ & planar\\
		 & $g'$ & planar (from $\mu_{CZ}^p$ jump from $g$) \\
		Hénon \cite{henon_1} (1970), \cite{henon_0} (2003) & $g_3$ & planar (from the 3rd cover of $f$)\\
		Batkhin--Batkhina \cite{batkhin} (2009) & $g_{2v}$ & spatial (from $\mu_{CZ}^s$ jump from $g$)\\
		& $g_{1v}^{YOZ}$ & spatial (from the 2nd cover of $g$) \\
		Michalodimitrakis \cite{michalodimitrakis} (1980) & $g1v$ & spatial (from the 2nd cover of $g$ \& $g'$)\\
		Kalantonis \cite{kalantonis} (2020) & $f_g^{(2,3)}$, $f_g^{(2cut,3)}$ & spatial (from the 3rd cover of $g$) \\
		& $f_{g'}^{(2,3)}$, $f_{g'}^{(2cut,3)}$ & spatial (from the 3rd cover of $g'$) \\
		& $f_g^{(1,4)}$, $f_g^{(1cut,4)}$ & spatial (from the 4th cover of $g$) \\
		& $f_{g'}^{(1,4)}$, $f_{g'}^{(1cut,4)}$ & spatial (from the 4th cover of $g'$)
	\end{tabular}
	\caption{The families of planar and spatial periodic orbits in this paper}
	\label{families_in_this_paper}
\end{table}
\noindent
These data and invariants provide an organized structure for the practical work in the context of space mission design, which in particular shows how these families are connected to each other.\ In order to give such overviews we illustrate the resp.\ scenarios in form of bifurcation graphs (see for instance Figure \ref{overview_conclusion}).\ Furthermore, note that in this paper we use the traditional Jacobi integral $\Gamma=-2c$, where $c$ is the energy value given by the Hamiltonian.\ Starting with the indices for very low energies, our results are:\
\begin{table}[H] \centering
	\begin{tabular}{c|c|c|c|c|c}
		energy values $\Gamma$ & planar & spatial & $\mu_{CZ}^p$ & $\mu_{CZ}^s$ & $\mu_{CZ}$\\
		\hline $(+\infty,4.49999)$ & elliptic & elliptic & 3 & 3 & 6\\
		$(4.49999,1.3829)$ & pos. hyperbolic & elliptic & 2 & 3 & 5\\
		$(1.3829,- \infty)$ & pos. hyperbolic & pos. hyperbolic & 2 & 4 & 6
	\end{tabular}
	\caption{The family $g$}
	\label{overview_g}
\end{table}
\noindent
Therefore for our moon's variational orbit by Hill, which is at $\Gamma = 6.5088$, the indices do not change and the orbit is planar as well as spatial elliptic.\ Moreover, our geometrical approach and numerical computation give for the periods
$$ T_s = 29.528396,\quad T_a = 27.553954,\quad T_d = 27.212712, $$
which is a remarkably good approximation to the experimentally measured data.\ Note that the orbits of the family $g$ are doubly-symmetric with respect to the reflection at the $q_1$- and $q_2$-axis.\ 

The orbits from the family $f$ are doubly-symmetric as well, and they are planar and spatial elliptic for all times, thus their indices do not change.\ Furthermore, in the limit case, in which the Jacobi integral goes to $- \infty$, we analytically show that these orbits converge to a degenerate planar retrograde periodic orbit, where all three months coincide with the period of the earth around the sun, namely 365.25 days.\

In Table \ref{overview_g}, at the planar transition from elliptic to positive hyperbolic, $\mu_{CZ}^p$ jumps from 3 to 2 and there bifurcates the family $g'$, whose data are collected in Table \ref{overview_g'}.\ The orbits of the family $g'$ are simply-symmetric with respect to the reflection at the $q_1$-axis, and by using the reflection at the $q_2$-axis one obtains its symmetric family, hence the family $g'$ appears twice.\
\begin{table}[H] \centering
	\begin{tabular}{c|c|c|c|c|c}
		energy values $\Gamma$ & planar & spatial & $\mu_{CZ}^p$ & $\mu_{CZ}^s$ & $\mu_{CZ}$\\
		\hline $(4.49999,4.2851)$ & elliptic & elliptic & 3 & 3 & 6\\
		$(4.2851,4.2806)$ & elliptic & neg. hyperbolic & 3 & 3 & 6\\
		$(4.2806,4.2714)$ & elliptic & elliptic & 3 & 3 & 6\\
		$(4.2714,3.3901)$ & neg. hyperbolic & elliptic & 3 & 3 & 6\\
		$(3.3901,0.4771)$ & neg. hyperbolic & pos. hyperbolic & 3 & 4 & 7\\
		$(0.4771,-0.2195)$ & neg. hyperbolic & elliptic  & 3 & 5 & 8\\
		$(-0.2195,-4.6921)$ & neg. hyperbolic & neg. hyperbolic & 3 & 5 & 8\\
		$(-4.6921,-4.7047)$ & elliptic & neg. hyperbolic & 3 & 5 & 8\\
		$(-4.7047,- \infty)$ & pos. hyperbolic & neg. hyperbolic & 4 & 5 & 9
	\end{tabular}
	\caption{The family $g'$}
	\label{overview_g'}
\end{table}
\noindent
It is easy to check that the Euler characteristics before and after the bifurcation of $g'$ coincide, namely they are in the planar problem
$$ (-1)^3 = -1,\quad \text{resp.}\quad (-1)^2 + 2\cdot(-1)^3 = -1, $$
and in the spatial problem (obtained by a shift of $\mu_{CZ}^s = 3$)
$$ (-1)^6 = 1,\quad \text{resp.}\quad (-1)^5 + 2\cdot(-1)^6 = 1. $$

At the spatial transition of the family $g$ from elliptic to positive hyperbolic, where $\mu_{CZ}^s$ jumps from 3 to 4, the family $g_{2v}$ bifurcates.\ Its orbits are doubly-symmetric with respect to $\rho_1$ and $\rho_2$, and by using $\sigma$ one obtains its symmetric family which is doubly-symmetric with respect to $\overline{\rho_1}$ and $\overline{\rho_2}$.\ They have $\mu_{CZ} = 5$.\

The other families of spatial orbits from Table \ref{families_in_this_paper} bifurcate from the respective iteration of the underlying planar periodic orbits.\ Moreover, we want to emphasize the following special bifurcation result, which is illustrated in Figure \ref{overview_conclusion}.\ We interpret this figure as a graph and call such a graph a ``\textbf{bifurcation graph}".\ It is constructed as follows:\

We draw from bottom to top in the direction of increasing energy.\ Each vertex corresponds to a degenerate periodic orbit and each edge to a family of periodic orbits with their (constant) Conley--Zehnder index.\ We distinguish two kinds of edges.\ The first kind corresponds to the underlying family of planar periodic orbits where the index jumps and the bifurcations happen.\ We draw these edges in black, vertically and shortly before and after the bifurcations.\ The second kind corresponds to a new family branching out from the index jump.\ We draw such edges coloured, and every new family gets his own colour.\ If there is a symmetric family, then these edges are drawn in the same colour, but dashed.\  The cross stands for collision and the term ``b-d" for a periodic orbit of birth-death type.\ In general, a periodic orbit of birth-death type is a degenerate orbit from which two families bifurcate with an index difference of 1 and into the same energy direction.\ Its local Floer homology and its Euler characteristic are therefore zero.\

\begin{figure}[H]
	\centering
	\definecolor{grgr}{RGB}{33,189,63}
	\begin{tikzpicture}[line cap=round,line join=round,>=triangle 45,x=1.0cm,y=1.0cm]
	\clip(-9,-3) rectangle (7,11.5);
	
	\draw [->, line width=1pt] (-3-4.5,-2.5) -- (-3-4.5,11);
	\draw (-3-4.5,11) node[anchor=south] {$\Gamma$};
	\draw (-4-4.5,11) node[anchor=south] {$-\infty$};
	\draw (-4-4.5,-2.5) node[anchor=north] {$+\infty$};
	\draw[fill] (-3-4.5,-2) circle (1.5pt);
	\draw (-3-4.5,-2) node[anchor=east] {4.347};
	\draw[fill] (-3-4.5,0.5) circle (1.5pt);
	\draw (-3-4.5,0.5) node[anchor=east] {3.876};
	\draw[fill] (-3-4.5,4) circle (1.5pt);
	\draw (-3-4.5,4) node[anchor=east] {3.274};
	\draw[fill] (-3-4.5,10) circle (1.5pt);
	\draw (-3-4.5,10) node[anchor=east] {0.755};
	
	\draw[fill] (-3-4.5,4-0.5) circle (1.5pt);
	\draw (-3-4.5,4-0.5) node[anchor=east] {3.280};
	
	\draw[fill] (-3-4.5,4-1.1) circle (1.5pt);
	\draw (-3-4.5,4-1.1) node[anchor=east] {3.362};
	
	\draw[fill] (-3-4.5,4+1.1) circle (1.5pt);
	\draw (-3-4.5,4+1.1) node[anchor=east] {3.136};
	
	\draw[fill] (-3-4.5,4+1.6) circle (1.5pt);
	\draw (-3-4.5,4+1.6) node[anchor=east] {3.101};

	\draw [line width=1pt] (-2,-2.5) -- (-2,-1.5);
	
	\draw (-2,-2.5) node[anchor=north] {$14$};
	\draw (-2,-1.5) node[anchor=south] {$16$};

	\draw [line width=1pt] (0,0) -- (0,1);
	
	\draw (0,0) node[anchor=north] {$13$};
	\draw (0,1) node[anchor=south] {$15$};

	\draw [line width=1pt] (0,9.5) -- (0,10.5);
	
	\draw (0,9.5) node[anchor=north] {$16$};
	\draw (0,10.5) node[anchor=south] {$14$};
	
	\draw [line width=1pt,color=red] (-2,-2) .. controls (0,0.6) and (1,-2.2) .. (2.5,4);
	
	\draw [dashed,line width=1pt,color=red] (-2,-2) .. controls (-3.1,1) and (-3,3) .. (-2.5,4);
	
	\draw [color=red] (1.8,-0.65) node[anchor=west] {$15$};
	\draw [color=red] (-1.8,-0.65) node[anchor=east] {$15$};
	
	\draw [color=red] (1.05,-0.95) node[anchor=west] {$15$};
	\draw [color=red] (-1.05,-0.95) node[anchor=east] {$15$};
	
	\draw [color=magenta] (2.4,-1.1) node[anchor=west] {$14$};
	\draw [color=magenta] (-2.4,-1.1) node[anchor=east] {$14$};
	
	\draw [color=magenta] (0.8,-1.45) node[anchor=west] {$14$};
	\draw [color=magenta] (-0.8,-1.45) node[anchor=east] {$14$};
	
	\draw [color=grgr] (0.7,2) node[anchor=west] {$13$};
	\draw [color=grgr] (-0.7,2) node[anchor=east] {$13$};
	
	\draw [color=grgr] (3.45,4.5) node[anchor=west] {$14$};
	\draw [color=grgr] (-3.45,4.5) node[anchor=east] {$14$};
	
	\draw [color=grgr] (4.7,4.8) node[anchor=west] {$15$};
	\draw [color=grgr] (-4.7,4.8) node[anchor=east] {$15$};
	
	\draw [color=grgr] (5.6,3.3) node[anchor=west] {$14$};
	\draw [color=grgr] (-5.6,3.3) node[anchor=east] {$14$};
	
	\draw [line width=1pt,color=blue] (0,0.5) .. controls (1,2.5) and (1.5,3.5) .. (2.5,4);
	
	\draw [dashed,line width=1pt,color=blue] (0,0.5) .. controls (-1,2.5) and (-1.5,3.5) .. (-2.5,4);
	
	\draw [color=blue] (0.5,2) node[anchor=south] {$14$};
	\draw [color=blue] (-0.5,2) node[anchor=south] {$14$};
	
	\draw [line width=1pt,color=blue] (2.5,4) .. controls (1.5,5) and (1.2,7.5) .. (0,10);
	
	\draw [dashed,line width=1pt,color=blue] (-2.5,4) .. controls (-1.5,5) and (-1.2,7.5) .. (0,10);
	
	\draw [color=blue] (1.3,4.5) node[anchor=west] {$15$};
	\draw [color=blue] (-1.3,4.5) node[anchor=east] {$15$};

	\draw [line width=1pt] (2,-2.5) -- (2,-1.5);
	
	\draw (2,-2.5) node[anchor=north] {$14$};
	\draw (2,-1.5) node[anchor=south] {$16$};
	
	\draw [dashed,line width=1pt,color=red] (2,-2) .. controls (0,0.6) and (-1,-2.2) .. (-2.5,4);
	
	\draw [line width=1pt,color=red] (2,-2) .. controls (3.1,1) and (3,3) .. (2.5,4);

	\draw [line width=1pt,color=grgr] (0,0.5) .. controls (1.5,2.25) and (2,2.5) .. (3.5,3.5);
	
	\draw [dashed,line width=1pt,color=grgr] (0,0.5) .. controls (-1.5,2.25) and (-2,2.5) .. (-3.5,3.5);

	\draw [line width=1pt,color=grgr] (3.5,3.5) .. controls (3.9,3.9) and (4.3,4.6) .. (4.5,5.1);
	
	\draw [dashed,line width=1pt,color=grgr] (-3.5,3.5) .. controls (-3.9,3.9) and (-4.3,4.6) .. (-4.5,5.1);

	\draw [line width=1pt,color=grgr] (4.5,5.1) .. controls (4.9,4.6) and (5.3,3.9) .. (5.5,2.9);
	
	\draw [dashed,line width=1pt,color=grgr] (-4.5,5.1) .. controls (-4.9,4.6) and (-5.3,3.9) .. (-5.5,2.9);

	\draw [line width=1pt,color=grgr] (5.5,2.9) .. controls (5.9,3.9) and (6.1,4.6) .. (6.5,5.6);
	
	\draw [dashed,line width=1pt,color=grgr] (-5.5,2.9) .. controls (-5.9,3.9) and (-6.1,4.6) .. (-6.5,5.6);
	
	\draw [line width=1pt,color=magenta] (2,-2) .. controls (3.5,1) and (3.5,2) .. (3.5,3.5);
	
	\draw [line width=1pt,color=magenta] (-2,-2) .. controls (3,-0.2) and (3,1.5) .. (3.5,3.5);
	
	\draw [dashed,line width=1pt,color=magenta] (-2,-2) .. controls (-3.5,1) and (-3.5,2) .. (-3.5,3.5);
	
	\draw [dashed,line width=1pt,color=magenta] (2,-2) .. controls (-3,-0.2) and (-3,1.5) .. (-3.5,3.5);
	
	\draw[fill] (5.5,2.9) circle (2pt);
	\draw (5.5,2.9) node[anchor=north] {b-d};
	\draw[fill] (-5.5,2.9) circle (2pt);
	\draw (-5.5,2.9) node[anchor=north] {b-d};
	
	\draw[fill] (4.5,5.1) circle (2pt);
	\draw (4.5,5.1) node[anchor=south] {b-d};
	\draw[fill] (-4.5,5.1) circle (2pt);
	\draw (-4.5,5.1) node[anchor=south] {b-d};
	
	\draw[fill] (3.5,3.5) circle (2pt);
	\draw[fill] (-3.5,3.5) circle (2pt);

	\node[solid, cross out, draw=black, thick] at (6.5,5.6) {};
	\node[solid, cross out, draw=black, thick] at (-6.5,5.6) {};

	\draw[fill] (2,-2) circle (2pt);
	\draw (2.1,-2) node[anchor=west] {$g'^3$};
	
	\draw[fill] (-2.5,4) circle (2pt);
	
	\draw[fill] (2.5,4) circle (2pt);
	
	\draw[fill] (0,10) circle (2pt);
	\draw (0.1,10) node[anchor=west] {$f^5$};
	
	\draw[fill] (-2,-2) circle (2pt);
	\draw (-2.2,-2) node[anchor=east] {$g'^3$};
	
	\draw[fill] (0,0.5) circle (2pt);
	\draw (0.1,0.5) node[anchor=west] {$g^3$};
	
	
	\end{tikzpicture}
	\caption{The bifurcation graph between the 3rd cover of $g$, the 3rd cover of $g'$ and the 5th cover of $f$ with the families \textcolor{blue}{$f_g^{(2,3)}$}, \textcolor{grgr}{$f_g^{(2cut,3)}$}, \textcolor{red}{$f_{g'}^{(2,3)}$} and \textcolor{magenta}{$f_{g'}^{(2cut,3)}$}}
	\label{overview_conclusion}
\end{figure}
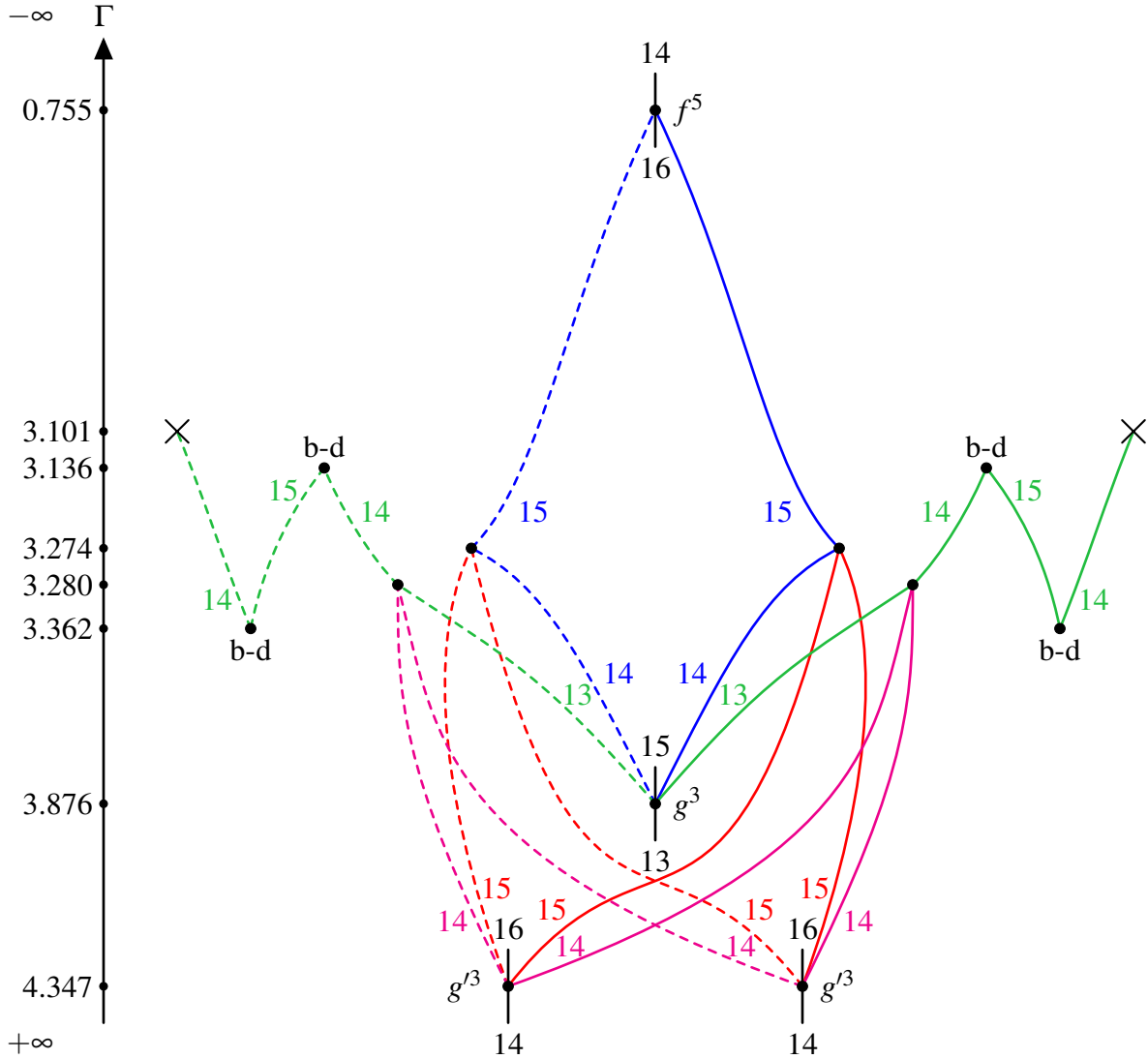
The branches (not dashed) correspond to the families \textcolor{blue}{$f_g^{(2,3)}$}, \textcolor{grgr}{$f_g^{(2cut,3)}$}, \textcolor{red}{$f_{g'}^{(2,3)}$} and \textcolor{magenta}{$f_{g'}^{(2cut,3)}$} with their resp.\ colours.\ In view of the symmetry $\sigma$, these families give rise to a second bifurcation branch (dashed).\

At $\Gamma = 3.876$ the index of the 3rd cover of the family $g$ jumps from 13 to 15.\ At this transition the two families \textcolor{blue}{$f_g^{(2,3)}$} and \textcolor{grgr}{$f_g^{(2cut,3)}$} bifurcate.\ Their indices are \textcolor{blue}{14} and \textcolor{grgr}{13}, respectively.\ The orbits of the first family are doubly-symmetric with respect to $\overline{\rho_1}$ and $\overline{\rho_2}$, hence they are invariant under $-\sigma$.\ They end at the value $\Gamma = 0.755$ at the 5th cover of $f$, and inbetween there is an index jump from \textcolor{blue}{14} to \textcolor{blue}{15}.\ In the second family, the orbits are doubly-symmetric with respect to $\rho_1$ and $\rho_2$, thus they are invariant under $-\sigma$ as well.\ At the value $\Gamma = 3.280$ the index jumps from \textcolor{grgr}{13} to \textcolor{grgr}{14}, and the orbits eventually undergo collision.\ Until then, this family consists of branches that bifurcate from a periodic orbit of birth-death type.\ The dashed branches are obtained by using $\sigma$.\ Note that the orbits of the symmetric family (dashed) of \textcolor{blue}{$f_g^{(2,3)}$} are doubly-symmetric with respect to $\rho_1$ and $\rho_2$, and the orbits of the symmetric family (dashed) of \textcolor{grgr}{$f_g^{(2cut,3)}$} are doubly-symmetric with respect to $\overline{\rho_1}$ and $\overline{\rho_2}$.\

Consider the family $g'$, where at $\Gamma = 4.347$ the index of the 3rd cover jumps from 14 to 16.\ At this value of $\Gamma$ the two families \textcolor{red}{$f_{g'}^{(2,3)}$} and \textcolor{magenta}{$f_{g'}^{(2cut,3)}$} bifurcate, with indices \textcolor{red}{15} resp.\ \textcolor{magenta}{14}.\ According to Kalantonis \cite[p.\ 11]{kalantonis}, the two families \textcolor{red}{$f_{g'}^{(2,3)}$} and \textcolor{magenta}{$f_{g'}^{(2cut,3)}$} terminate at the 3rd cover of the resp.\ planar orbit of $g'$ which is symmetric to the planar orbit of $g'$, from which they have bifurcated.\ These are the two not-dashed branches, respectively.\ 

We have a deeper insight:\
\begin{itemize}[noitemsep]
	\item The orbits of the family \textcolor{red}{$f_{g'}^{(2,3)}$} are simply-symmetric with respect to $\overline{\rho_1}$, and its two not-dashed branches are symmetric by $\overline{\rho_2}$.\ Recall that the orbits of \textcolor{blue}{$f_g^{(2,3)}$} are doubly-symmetric with respect to $\overline{\rho_1}$ and $\overline{\rho_2}$.\ Therefore, very close to the value $\Gamma = 3.274$, by comparing the initial data, and by using these symmetries and especially the indices, we conclude that the family \textcolor{red}{$f_{g'}^{(2,3)}$} ends at the first index jump of the family \textcolor{blue}{$f_g^{(2,3)}$}.\ This explains why they come together at the value $\Gamma = 3.274$.\
	\item The same happens for the family \textcolor{magenta}{$f_{g'}^{(2cut,3)}$}, namely its orbits are simply-symmetric with respect to $\rho_1$ and its two not-dashed branches are symmetric by $\rho_2$.\ Recall that the orbits of \textcolor{grgr}{$f_g^{(2cut,3)}$} are doubly-symmetric with respect to $\rho_1$ and $\rho_2$.\ Hence very close to the value $\Gamma = 3.280$, by comparing the initial data, and by using these symmetries and especially the indices, we obtain that \textcolor{magenta}{$f_{g'}^{(2cut,3)}$} ends at the first index jump of the family \textcolor{grgr}{$f_g^{(2cut,3)}$} (at the value $\Gamma = 3.280$).\
	\item Each family yields the bifurcation of a second branch (dashed) by using $\sigma$.\ Furthermore, note that the symmetry properties of the orbits of the spatial families \textcolor{red}{$f_{g'}^{(2,3)}$} compared to \textcolor{blue}{$f_g^{(2,3)}$}, and of \textcolor{magenta}{$f_{g'}^{(2cut,3)}$} compared to \textcolor{grgr}{$f_g^{(2cut,3)}$}, are similar to the symmetry properties of the orbits of the planar family $g'$ compared to $g$.\
\end{itemize}
\noindent
In particular, the bifurcation graph organizes the local bifurcations and thereby helps to check the Euler characteristic of the local Floer homology groups.\ For instance, at the value $\Gamma = 3.876$ at $g^3$ the Euler characteristics before and after the bifurcation are
$$ (-1)^{13} = -1,\quad\text{resp.}\quad 2\cdot(-1)^{13} + 2\cdot(-1)^{14} + (-1)^{15} = -1.  $$
Note that at the value $\Gamma = 0.755$ at $f^5$ the Euler characteristics before and after the bifurcation are
$$ 2\cdot(-1)^{15} + (-1)^{16} = -1,\quad\text{resp.}\quad (-1)^{14} = 1.  $$
Therefore, at this transition there are still undiscovered families branching out from $f^5$.\

\subsection{Organization of the paper}

Section \ref{sec:periodic_orbits} is written for readers not familiar with this kind of flow language and the reduced monodromy in terms of symplectic geometry.\ Furthermore, in the case of $\text{Sp}(1) = \text{SL}(2,\mathbb{R})$, we discuss the stability, the Floquet multipliers and the transversal Conley--Zehnder index of periodic orbits.\ This section is based on the books of Hofer--Zehnder \cite{hofer} and Frauenfelder--van Koert \cite{frauenfelder}, and on the articles by Hofer--Wysocki--Zehnder \cite[Appendix]{hofer_w_z_1}, \cite[Section 3]{hofer_w_z}.\

In order to discuss the symplectic decomposition in a more general and more conceptual way, we introduce in Section \ref{sec:hamiltonian} the concept of Hamiltonian manifold, which generalizes energy hypersurfaces, based on \cite{frauenfelder} and fruitful discussions with Urs Frauenfelder.\ The general statement on the symplectic splitting is formulated in the Symplectic Splitting Theorem \ref{theorem_splitting}.\

In Section \ref{sec:4} we introduce symmetries of Hamiltonian systems and show that the monodromy and reduced monodromy of symmetric periodic orbits satisfies special symmetries.\ The content of this section gives a theoretical framework to specify the Floquet multipliers, in particular the rotation angles, and thereby the index jump.\ This section is based on \cite{frauenfelder_moreno} and helpful discussions with Urs Frauenfelder.\

The subject of Section \ref{sec:spatial_Hills_lunar} is to discuss the spatial Hill lunar problem.\ Firstly, we give a short astronomical lunar overview and a description of Hill's concept.\ Then we derive its Hamiltonian and equation of motion as a limit case of the circular restricted three body problem, following \cite{frauenfelder} where the planar case is studied.\ In addition, we determine the group of linear symmetries and discuss planar as well as spatial symmetric periodic orbits.\ By applying the theory developed in Section \ref{sec:4}, we show in the Subsections \ref{sec:6.5.1} and \ref{sec:6.5.2} how to calculate its reduced monodromy, and for planar ones how its three periods are explicitely defined.\

The goal of Section \ref{sec:6} is to prove Theorems \ref{theorem_a} and \ref{theorem_b}.\ Frauenfelder and van Koert showed in \cite[Chapter 8]{frauenfelder} the bifurcation scenario for the planar problem.\ We use and extend their technique to the spatial problem and prove the existence of two additional (spatial collision) periodic orbits.\ Moreover, by an explicit calculation of Morse--Bott indices, we determine their Conley--Zehnder indices, and the anomalistic and draconitic periods for the families $g$ and $f$ for very small energies.\

In Section \ref{sec:local_rabinowitz} we briefly sketch the local equivariant Rabinowitz-Floer homology and its Euler characteristic associated to good and bad orbits.\

The data for our numerical results for all the families from Table \ref{families_in_this_paper} are presented in Section \ref{sec:8}.\ They consist of all relevant data such as the inital conditions, the Floquet multipliers, the indices and the periods.\ We also plot these periodic orbits in the configuration space and give overviews in form of bifurcation graphs such as in the Figure \ref{overview_conclusion}.\

For our numerical approximation of the linearized flow and thereby for the computation of relevant data, we have written several Python codes which are collected in the Appendix.\

\textbf{Outlook.}\ In this paper, we interpret symmetric periodic orbits as periodic orbits.\ In view of the fixed point set of an anti-symplectic involution, which is a Lagrangian submanifold, they can also be interpreted as Lagrangian intersection points.\ Instead of the Conley--Zehnder index, we can therefore assign to them the Lagrangian Maslov index $\mu_L$ and hence consider the local Lagrangian Floer homology.\ However, the difference of $\mu_{CZ}$ and $\mu_L$ is the Hörmander index.\ By using the formula from Theorem 2.3 in \cite{frauenfelder_van}, these indices for all families in this paper will be computed as well.\ Furthermore, the local equivariant Rabinowitz Floer homology in Section \ref{sec:8} will be worked out in much more detail.\ Moreover, our general construction can be applied to every Hamiltonian system with the relevant symmetries.\

\textbf{Acknowledgement.}\ The author would like to thank Urs Frauenfelder, who was the supervisor at the Universität Augsburg of the author's Master's thesis, on which this research of his PhD thesis is based, for valuable discussions and support.\ He is deeply grateful to its supervisor Felix Schlenk for reading carefully the text, giving improvements and also helpful inputs.\ This work is supported by the SNF under grant No.\ 200021-181980/2.\ Moreover, he is grateful to Vassilis S.\ Kalantonis and Alexander Batkhin for providing the initial data for the orbits they have found.\ Furthermore, he is also thankful to his family for the support during the difficult time by the sudden death of his father Ibrahim Aydin on 8 December 2019.\ This paper is in memory of him.\ May God be merciful to him.\

\section{Periodic orbits of Hamiltonian systems} \label{sec:periodic_orbits}

\subsection{Periodic orbits, monodromy and reduced monodromy}
\label{sec:2.1}

\noindent
Let $(M,\omega)$ be a $2n$ dimensional symplectic manifold, i.e.\ $\omega \in \Omega^2(M)$ is a 2-form, called symplectic form on $M$, which is closed, i.e.\ $d \omega =0$, and non-degenerate.\

\begin{example}
The archetypical example is the cotangent bundle $T^*Q$ of a $n$ dimensional smooth manifold $Q$.\ In physics $T^*Q$ corresponds to phase space and $Q$ to configuration space.\ In canonical coordinates $(q,p)=(q_1,...,q_n,p_1,...,p_n) \in T^*Q$, where $q$ is a point in configuration space and $p \in T^*_qQ$ its momentum in the fiber, the cotangent bundle is endowed with the canonical symplectic form $\omega = \sum\limits_{i=1}^{n}dq_i \wedge dp_i$.\
\end{example}

\begin{remark}
The first analytical condition of a symplectic form forces that all symplectic manifolds locally look like the Euclidean space $\mathbb{R}^{2n}$ equipped with the standard symplectic form (Darboux's theorem, see for instance \cite[pp. 10--11]{hofer} for details).\\
For the second algebraic condition:\ Since $\omega \in \Omega^2(M)$, we get for each point $p \in M$ an alternating multilinear map
\begin{align} \label{omega_1}
\omega_p \colon T_p M \times T_p M \to \mathbb{R}.
\end{align}
Non-degeneracy means that for all $p \in M$ and for all $0 \neq v \in T_pM$ there exists $0 \neq w \in T_pM$ such that $\omega_p(v,w) \neq 0$.\ Equivalently, if $\omega_p(v,w)=0$ for all $v \in T_pM$, then $w=0$.\ Moreover, this condition implies that the top exterior power $ \omega^{\wedge n} = \omega \wedge ... \wedge \omega \neq 0 $ is a volume form, hence symplectic manifolds are necessarily even-dimensional and orientable.\ Furthermore, (\ref{omega_1}) defines a map
\begin{align} \label{omega_2}
T_p M \to T_p^* M,\quad v \mapsto \omega_p(v,\cdot).
\end{align}
Being non-degenerate is equivalent to (\ref{omega_2}) being an isomorphism.\
\end{remark}

Let $H \in C^{\infty}(M,\mathbb{R})$ be an autonomous Hamiltonian function and $X_H$ the Hamiltonian vector field, which is uniquely defined by
\begin{align} \label{ham_vector_field}
dH (\cdot) = \omega(X_H,\cdot).
\end{align}
\begin{definition}
	A \textbf{periodic orbit} $x \in C^{\infty}(\mathbb{R},M)$ of $H$ is a solution to the first order ODE
	\begin{align} \label{periodic_orbit_hamiltonian}
	\dot{x}(t) = X_H\big(x(t)\big),\quad t \in \mathbb{R}
	\end{align}
	such that there exists a period $T > 0$ with $x(t+T) = x(t)$ for all $t \in \mathbb{R}$.\
	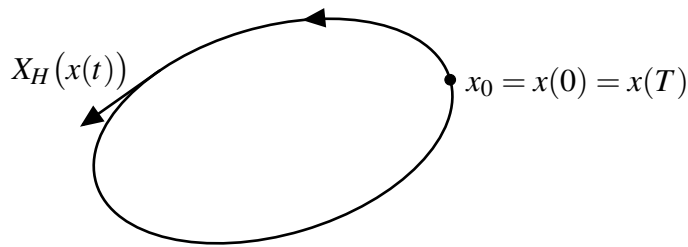
\begin{figure}[H]\centering
		\begin{tikzpicture}[line cap=round,line join=round,>=triangle 45,x=1.0cm,y=1.0cm]
		\clip(-6,-1.7) rectangle (6,1.9);
		
		\draw [rotate around={-163.3773418816643:(0.040509448327610796,0.04313429971099568)}] (0.040509448327610796,0.04313429971099568) [decoration={markings, mark=at position 0.685 with {\arrow{>}}}, postaction={decorate},line width=1pt] ellipse (2.428036135165941cm and 1.374491119408365cm);
		\draw (2.44,1.0400000000000003) node[anchor=north west] {$x_0 = x(0)=x(T)$};
		\draw (-3.54,1.2400000000000002) node[anchor=north west] {$X_H\big(x(t)\big)$};
		\draw [->,line width=1pt] (-1.4500880042828483,0.8455453250780841) -- (-2.5,0.09999999999999998);
		\begin{scriptsize}
		\draw [fill=black] (2.3712772886663096,0.7231285200436788) circle (2pt);
		\end{scriptsize}
		\end{tikzpicture}
		\caption{Periodic orbit}
		\label{figure3}
	\end{figure}
\end{definition}
\noindent
A special case of periodic orbits are critical points of $H$.\ For example the Lagrange points of the circular restricted three body problem (see Section \ref{subsec:discussion}) are such trivial periodic orbits.\ For a non-trivial periodic orbit $x$ we define the first return time by
$$T_x := \min \{ T \geq 0 \mid x(T) = x(0) \},$$
hence for every period $T$ we obtain $T = n \cdot T_x$ for some $n \in \mathbb{Z}$.\ Since $H$ is autonomous, there are no self-intersections, i.e.\ for $\tau \in (0,T_x)$ it holds that $x(0) \neq x(\tau).$\

The solutions of (\ref{periodic_orbit_hamiltonian}) generate the family of Hamiltonian flows $\varphi_H^t \colon M \to M$ of $X_H$ via
$$\varphi_H^0 = \text{id}_M,\quad \frac{d}{dt}\varphi_H^t(z) = X_H \big(\varphi_H^t(z)\big),$$
which are symplectomorphisms, meaning that the symplectic form $\omega$ is preserved under $\varphi_H^t$, i.e.\ $(\varphi_H^t)^* \omega = \omega$.\ Since the pullback commutes with the wedge product, the flow is volume preserving as well, meaning that $ (\varphi_H^t)^* \omega^{\wedge n} = \omega^{\wedge n}. $
\begin{remark}
	Take $Q=\mathbb{R}^n$ as configuration space and as phase space its trivial cotangent bundle $T^*\mathbb{R}^n = \mathbb{R}^n \times \mathbb{R}^n$ with the canonical symplectic form.\ We have
	$$dH = \sum\limits_{i=1}^{n} \bigg(\frac{\partial H}{\partial q_i} dq_i  + \frac{\partial H}{\partial p_i} dp_i\bigg),\quad X_H = \sum\limits_{i=1}^n \bigg( \frac{\partial H}{\partial p_i} \frac{\partial }{\partial q_i} - \frac{\partial H}{\partial q_i} \frac{\partial }{\partial p_i} \bigg),$$
	since the equation (\ref{ham_vector_field}) is satisfied by $X_H$,
	\begin{align*}
	\sum\limits_{i=1}^n dq_i \wedge dp_i (X_H,\cdot) &= \sum\limits_{i=1}^n \Big( dq_i(X_H)dp_i(\cdot) - dq_i(\cdot)dp_i(X_H) \Big)\\
	&= \sum\limits_{i=1}^n \Bigg( \frac{\partial H}{\partial p_i}dp_i(\cdot) - dq_i(\cdot)\bigg(-\frac{\partial H}{\partial q_i}\bigg) \Bigg)\\
	&= dH(\cdot).
	\end{align*}
	Therefore (\ref{periodic_orbit_hamiltonian}) is equivalent to the Hamiltonian equation of motion
	$$ \frac{dq_i}{dt} = \frac{\partial H}{\partial p_i},\quad \frac{d p_i}{dt} = - \frac{\partial H}{\partial q_i}. $$
\end{remark}
Now we consider the linearised flow along a non-trivial periodic orbit $x$ with $x(0)=x_0 \in M$.\ Since $(\varphi_H^t)^*\omega=\omega,$ we obtain the linear symplectomorphism
\begin{align} \label{symplectic_autom_1}
d\varphi_H^T(x_0) \colon (T_{x_0} M, \omega_{x_0}) \to (T_{x_0}M,\omega_{x_0})
\end{align}
which is called the \textbf{monodromy}.\ We compute
\begin{align} \label{ker_invariant}
X_H(x_0) = \frac{d}{dt}\varphi_H^t(x_0) = \frac{d}{dt}\varphi_H^{t+T}(x_0) = \frac{d}{dt} \varphi_H^T\big(\varphi_H^t(x_0)\big) = d \varphi_H^T(x_0)X_H(x_0).
\end{align}
Since the periodic orbit $x$ is not trivial, $X_H(x_0)$ does not vanish.\ Hence $X_H(x_0)$ is an eigenvector of the monodromy (\ref{symplectic_autom_1}) with eigenvalue 1.\

The time independence of the Hamiltonian $H$ implies that the energy is constant along its flow.\ Thus we can assign the energy value $c = H(x)$ to the periodic orbit.\ For a regular value $c$ of $H$ we know that $dH(x) \neq 0$ which by (\ref{ham_vector_field}) is equivalent to $X_H(x) \neq 0$, for all $x \in H^{-1}(c)$.\ Therefore the level set
\begin{align} \label{energy_hypersurface}
\Sigma := \Sigma_c := H^{-1}(c) \subset M
\end{align}
is a codimension one energy hypersurface that is invariant under the Hamiltonian flow $\varphi_H^t$.\ Its tangent space at a point $x \in \Sigma \subset M$ is given by
$$T_x\Sigma = \{ \xi \in T_xM: dH(x)\xi=0 \} = \text{ker}\big( dH(x_0) \big)$$
and we obtain the following decomposition of the $2n$ dimensional symplectic vector space
\begin{align} \label{orientable}
T_xM =  T_x \Sigma \oplus \langle \nabla H(x) \rangle.
\end{align}
Recall that a symplectic manifold is orientable via the volume form $\omega^{\wedge n}$.\ Since the line bundle $ \langle \nabla H(x) \rangle $ is also orientable, by (\ref{orientable}) the energy hypersurface $\Sigma$ is orientable.\
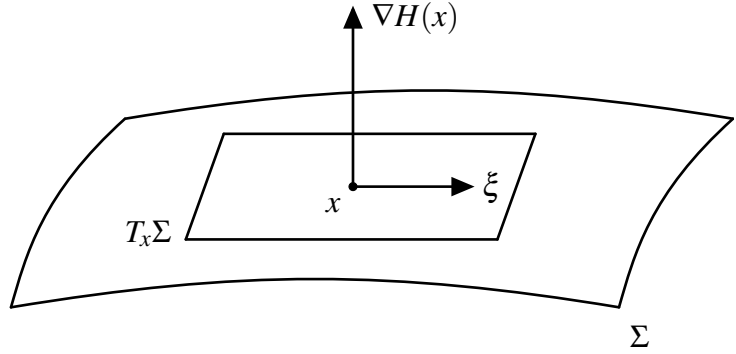
\begin{figure}[H]
	\centering
	\begin{tikzpicture}[line cap=round,line join=round,>=triangle 45,x=1.0cm,y=1.0cm]
	\clip(-5,-1.6) rectangle (5.5,3.1);
	
	\draw [line width=1pt] (-4,-1) .. controls (-1,-0.5) and (1,-0.5) .. (4,-1);
	
	\draw [line width=1pt] (-2.5,1.5) .. controls (0.5,2) and (2.5,2) .. (5.5,1.5);
	
	\draw [line width=1pt] (-4,-1) .. controls (-3.8,-0.3) and (-3.6,0.5) .. (-2.5,1.5);
	
	\draw [line width=1pt] (4,-1) .. controls (4.2,-0.3) and (4.4,0.5) .. (5.5,1.5);
	
	\draw (4,-1.1) node[anchor=north west] {$\Sigma$};
	
	\draw[fill] (0.5,0.6) circle (1.5pt);
	\draw (0,0.6) node[anchor=north west] {$x$};
	
	\draw [line width=1pt] (-1.7,-0.1) -- (2.4,-0.1);
	\draw [line width=1pt] (-1.2,1.3) -- (2.9,1.3);
	\draw [line width=1pt] (-1.7,-0.1) -- (-1.2,1.3);
	\draw [line width=1pt] (2.4,-0.1) -- (2.9,1.3);
	\draw (-2.65,0.3) node[anchor=north west] {$T_x \Sigma$};
	
	\draw [->,line width=1pt] (0.5,0.6) -- (0.5,3);
	\draw (0.6,3.2) node[anchor=north west] {$\nabla H(x)$};
	
	\draw [->, line width=1pt] (0.5,0.6) -- (2.1,0.6);
	\draw (2.05,0.95) node[anchor=north west] {$\xi$};
	
	\end{tikzpicture}
	\caption{The tangent space of the energy hypersurface $\Sigma$ and the gradient of $H$}
	\label{figure4}
\end{figure}
\noindent
By Definition (\ref{ham_vector_field}) of $X_H$ and by the anti-symmetry of $\omega$, we have
$$dH(x)X_H(x) = \omega \big(X_H(x), X_H(x)\big) = 0,\quad x \in \Sigma.$$
Hence the Hamiltonian vector field $X_H$ is tangent to the level sets (\ref{energy_hypersurface}) of $H$.\ In other words, $X_H$ defines a non-vanishing vector field on $\Sigma$, i.e.\
$$X_H(x) \in T_x \Sigma \setminus \{0\}.$$
For $x \in \Sigma$ we consider the subspace
$$\text{ker}\omega_x = \{ v \in T_x\Sigma: \omega_x(v,w) = 0, \forall w \in T_x\Sigma \} \subset T_x\Sigma$$
and compute for $\xi \in T_x \Sigma$ that
$$\omega_x \big(X_H(x),\xi\big) = dH(x)\xi = 0,$$
meaning that
$$X_H(x) \in \text{ker}\omega_x.$$
By the non-degeneracy of $\omega$,
\begin{align} \label{line_bundle}
\text{ker} \omega_x = \langle X_{H} (x) \rangle \subset T_x\Sigma.
\end{align}
Thus $\text{ker}\omega \vert _{\Sigma} \subset T\Sigma$ is a one-dimensional distribution, i.e.\ ($\text{ker}\omega \vert _{\Sigma},\pi_{\text{ker}\omega \vert _{\Sigma}},\Sigma$) is a line subbundle of the tangent bundle ($T\Sigma,\pi,\Sigma$).\ So the line bundle $\text{ker}\omega_{|\Sigma}$ is spanned by $X_H \vert _{\Sigma}$, which is a non-vanishing section of the line bundle $\text{ker}\omega \vert _{\Sigma}$, i.e.\ a smooth map $X_{H} \vert _{\Sigma}: \Sigma \to \text{ker}\omega \vert _{\Sigma}$ such that $\pi_{\text{ker}\omega \vert _{\Sigma}} \circ X_{H} \vert _{ \Sigma} = \text{id}_\Sigma$.\ The pair $(\Sigma,\omega)$ has the further good property that the restriction of the symplectic form $\omega$ to $\Sigma$ is closed, i.e.\ $d\omega \vert _{\Sigma} = 0.$\ These two properties motivate the general formulation in terms of a Hamiltonian manifold in Section \ref{sec:involutive}.\

Furthermore we have a foliation on $\Sigma$ (see Figure \ref{foliation}) where a leaf $L \subset \Sigma$ of the foliation is a one-dimensional submanifold such that $T_xL = \text{ker}\omega_x,$ for all $x \in L$.\ Indeed, a leaf trough $x$ corresponds to a trajectory of the Hamiltonian flow $\varphi_H^t$, i.e.\ $L_x = \{ \varphi_H^t(x): t \in \mathbb{R} \}$.\ Compact leaves are periodic orbits, which are diffeomorphic to $S^1$, and non-compact leaves are diffeomorphic to $\mathbb{R}$, hence non-periodic orbits.\ Describing the leaves instead of the flow is somewhat easier, since one does not care about time.\ Therefore, to understand the foliation on $\Sigma$ means to understand the dynamics of $X_{H} | _{\Sigma}$ up to time-parametrization.\
\begin{figure}[H]
	\centering
	\begin{tikzpicture}[line cap=round,line join=round,>=triangle 45,x=1.0cm,y=1.0cm]
	\clip(-5,-1.5) rectangle (5.5,2);
	
	\draw [line width=1pt] (-4,-1) .. controls (-1,-0.5) and (1,-0.5) .. (4,-1);
	
	\draw [line width=1pt] (-2.5,1.5) .. controls (0.5,2) and (2.5,2) .. (5.5,1.5);
	
	\draw [line width=1pt] (-4,-1) .. controls (-3.8,-0.3) and (-3.6,0.5) .. (-2.5,1.5);
	
	\draw [line width=1pt] (4,-1) .. controls (4.2,-0.3) and (4.4,0.5) .. (5.5,1.5);
	
	\draw (4,-1.1) node[anchor=north west] {$\Sigma$};
	
	\draw[fill] (0.5,0.6) circle (1.5pt);
	\draw (0,0.6) node[anchor=north west] {$x$};
	
	\draw [line width=1pt] (-1.7,-0.1) -- (2.4,-0.1);
	\draw [line width=1pt] (-1.2,1.3) -- (2.9,1.3);
	\draw [line width=1pt] (-1.7,-0.1) -- (-1.2,1.3);
	\draw [line width=1pt] (2.4,-0.1) -- (2.9,1.3);
	\draw (-2.65,0.3) node[anchor=north west] {$T_x \Sigma$};
	
	\draw [line width=1pt] (-0.6,0.6) -- (1.8,0.6);
	\draw (0.7,1.25) node[anchor=north west] {$\text{ker}\omega_x$};
	
	\end{tikzpicture}
	\caption{The characteristic foliation on $\Sigma$}
	\label{foliation}
\end{figure}
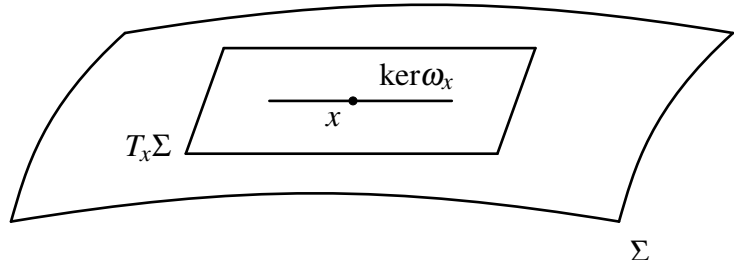
\noindent
Moreover, the symplectic form $\omega$ induces a symplectic form on the quotient space
$$T_x\Sigma / ( \text{ker}\omega_x \vert _{T_x\Sigma} )$$
which is in particular orientable.\ We obtain the quotient bundle $T\Sigma / ( \text{ker}\omega \vert _{ \Sigma} )$ over $\Sigma$, which is a symplectic vector bundle of rank $2n-2$.\ Recall that the tangent bundle $T \Sigma$ is orientable.\ The line bundle $\text{ker}\omega|_{\Sigma}$ is also orientable, since it has the non-vanishing section $X_H | _{\Sigma}$.\

Back to the monodromy (\ref{symplectic_autom_1}).\ We restrict it to the energy hypersurface $\Sigma$, so let $x$ be a periodic orbit on $\Sigma$ with $x(0) = x_0 \in \Sigma$.\ We obtain the linear diffeomorphism
$$d \varphi_H^T | _{\Sigma}(x_0) \colon (T_{x_0} \Sigma, \omega_{x_0} \vert _{T_{x_0}\Sigma}) \to (T_{x_0} \Sigma, \omega_{x_0} \vert _{T_{x_0}\Sigma}),$$
which leaves the one-dimensional distribution $\text{ker}\omega \vert_\Sigma \subset T\Sigma$ invariant by (\ref{ker_invariant}).\ This induces a symplectic bundle map, which is characterized by the commutative diagram
$$
\begin{tikzcd}
T\Sigma \arrow[r, "d\varphi_{H}^T | _{\Sigma}"] \arrow[swap,d, "\bar{\pi}"] & T\Sigma  \arrow[d, "\bar{\pi}"]\\
T\Sigma / (\text{ker}\omega\vert_{ \Sigma}) \arrow[r] & T\Sigma / (\text{ker}\omega\vert_{ \Sigma})
\end{tikzcd}
$$
\begin{definition}
The induced map
\begin{align} \label{symplectic_autom_2}
A:= \overline{d\varphi _H ^{T_x} | _{\Sigma} (x_0)} \colon T_{x_0}\Sigma / ( \text{ker}\omega_{x_0} \vert _{T_{x_0}\Sigma} ) \to T_{x_0}\Sigma / ( \text{ker}\omega_{x_0} \vert _{T_{x_0}\Sigma} )
\end{align}
is called \textbf{reduced monodromy} which is a symplectomorphism of the $2n-2$ dimensional symplectic vector space $T_{x_0}\Sigma / ( \text{ker}\omega_{x_0} \vert _{T_{x_0}\Sigma} )$.\ Moreover, the \textbf{Floquet multipliers} are defined as the eigenvalues of (\ref{symplectic_autom_2}) and we call the periodic orbit $x$ \textbf{non-degenerate} if 1 is not an eigenvalue of (\ref{symplectic_autom_2}), which is equivalent to $\text{ker} (A - \text{id} ) = \{0\}.$\
\end{definition}
\noindent
Note that if $\lambda$ is a Floquet multiplier, then so are $1/\lambda, \overline{\lambda}$ and $ 1/ \overline{\lambda}$ (see Figure \ref{figure_7} and \cite[pp.\ 124--125]{frauenfelder} for details) and the eigenvalues of the monodromy and the reduced monodromy differ by the double eigenvalue 1 from (\ref{ker_invariant}).\
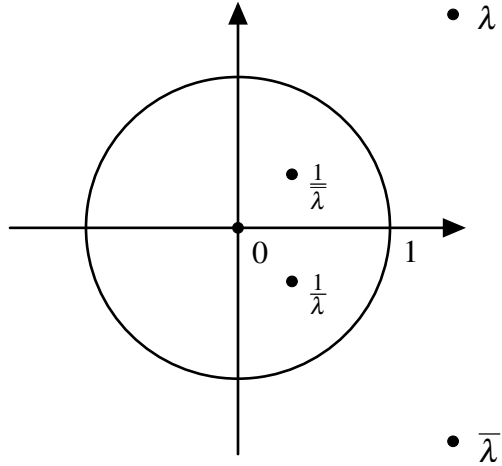
\begin{figure}[H]
	\centering
	\begin{tikzpicture}[line cap=round,line join=round,>=triangle 45,x=1.0cm,y=1.0cm]
	\clip(-3.5,-3.1) rectangle (3.5,3);
	
	\draw [line width=1pt] (0.0,-0.0) circle (2.0cm);
	\draw [->,line width=1pt] (-3.0,0.0) -- (3.0,0.0);
	\draw [->,line width=1pt] (0.0,-3.0) -- (0.0,3.0);
	\draw (3.0,3.1000000000000005) node[anchor=north west] {$\lambda$};
	\draw (3.02,-2.5600000000000005) node[anchor=north west] {$\overline{\lambda}$};
	\draw (0.04,-0.04000000000000001) node[anchor=north west] {$0$};
	\draw (2.02,-0.020000000000000004) node[anchor=north west] {$1$};
	\draw (0.76,1.0000000000000002) node[anchor=north west] {$\frac{1}{\overline{\lambda}}$};
	\draw (0.76,-0.4600000000000001) node[anchor=north west] {$\frac{1}{\lambda}$};
	\begin{scriptsize}
	\draw [fill=black] (0.0,-0.0) circle (2pt);
	\draw [fill=black] (2.83,2.83) circle (2pt);
	\draw [fill=black] (2.83,-2.83) circle (2pt);
	\draw [fill=black] (0.71,0.71) circle (2pt);
	\draw [fill=black] (0.71,-0.71) circle (2pt);
	\end{scriptsize}
	\end{tikzpicture}
	\caption{Complex eigenvalues not on the unit circle}
	\label{figure_7}
\end{figure}

\subsection{The case of $\text{Sp}(1)=\text{SL}(2,\mathbb{R})$}

We consider the case that the reduced monodromy (\ref{symplectic_autom_2}) is a symplectomorphism in
$$\text{Sp}(1)=\text{SL}(2,\mathbb{R}) = \{ \Psi \colon \mathbb{R}^2 \to \mathbb{R}^2 \text{ linear}  \mid \det \Psi = 1\},$$
where linear symplectomorphisms are exactly the linear orientation area-preserving transformations of $\mathbb{R}^2$.\

\subsubsection{Stability and Floquet multipliers}
\label{sec:stability_floquet}

The Floquet multipliers are the zeros of the polynomial $\lambda^2 - \lambda \text{tr}(A) + 1$.\ If $|\text{tr}A|<2$, then the eigenvalues are on the unit circle, so of the form $e^{\pm \text{i} \theta}$.\ If $|\text{tr}A|>2$, then they are real and of the form $\lambda, 1 / \lambda$.\ In particular, they are given respectively by
\begin{align*}
\frac{1}{2} \text{tr}(A) \pm \text{i} \frac{1}{2} \sqrt{4 - \big(\text{tr}(A)\big)^2}, \quad \frac{1}{2} \text{tr}(A) \pm \frac{1}{2} \sqrt{\big(\text{tr}(A)\big)^2 - 4}.
\end{align*}
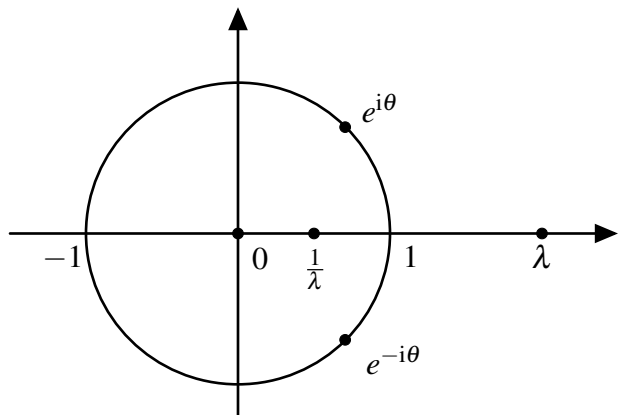
\begin{figure}[H]
	\centering
	\begin{tikzpicture}[line cap=round,line join=round,>=triangle 45,x=1.0cm,y=1.0cm]
	\clip(-5,-2.4) rectangle (5,3);
	
	\draw [line width=1pt] (0.0,0.0) circle (2.0cm);
	\draw [->,line width=1pt] (-3.0,0.0) -- (5,0.0);
	\draw [->,line width=1pt] (0.0,-3.0) -- (0.0,3.0);
	\draw (1.48,2.0000000000000004) node[anchor=north west] {$e^{\text{i} \theta}$};
	\draw (1.54,-1.3400000000000003) node[anchor=north west] {$e^{- \text{i} \theta}$};
	\draw (0.04,-0.04000000000000001) node[anchor=north west] {$0$};
	\draw (2.02,-0.020000000000000004) node[anchor=north west] {$1$};
	\draw (-2.72,-0.020000000000000004) node[anchor=north west] {$-1$};
	\draw (0.7199999999999999,-0.06000000000000001) node[anchor=north west] {$\frac{1}{\lambda}$};
	\draw (3.7200000000000015,0.0) node[anchor=north west] {$\lambda$};
	\begin{scriptsize}
	\draw [fill=black] (0.0,-0.0) circle (2pt);
	\draw [fill=black] (1.41,1.41) circle (2pt);
	\draw [fill=black] (1.41,-1.41) circle (2pt);
	\draw [fill=black] (1,0) circle (2pt);
	\draw [fill=black] (4,0) circle (2pt);
	\end{scriptsize}
	\end{tikzpicture}
	\caption{Eigenvalues of a $2 \times 2$ symplectic matrix}
	\label{1_figure_eigenvalues}
\end{figure}
\noindent
Furthermore, $x$ is called \textbf{elliptic} if $|\text{tr}(A)|<2$, \textbf{positive hyperbolic} if $\text{tr}(A)>2$ and \textbf{negative hyperbolic} if $\text{tr}(A)< - 2$.\ Geometrically, in the elliptic case the reduced monodromy is conjugate to a rotation, i.e.\
$$ A = \overline{d\varphi _H ^{T_x} | _{\Sigma} (x_0)} \sim \begin{pmatrix}
\cos \theta & - \sin \theta\\
\sin \theta & \cos \theta
\end{pmatrix}. $$
Hence elliptic periodic orbits have the property that orbits starting sufficiently close to $x$, i.e.\ neighbouring orbits of the same energy, remain near $x$ for a long time, while for hyperbolic orbits they may fly away.\ In the elliptic cases, to measure the number of complete rotations of the neighbouring orbits during $T_x$ the Conley--Zehnder index is helpful.\

\subsubsection{The Conley--Zehnder index}
\label{sec:conley_zehnder_index}

In 1984, Charles Conley and Eduard Zehnder \cite{conley_zehnder} defined an index theory, which generalizes the usual Morse index for closed geodesics on a Riemannian manifold.\ We refer the curious reader for details on Morse Theory to the books of Milnor \cite{milnor}, Banyaga--Hurtubise \cite{banyaga} and for Morse--Bott Theory, which is a generalization, to the articles of Bott \cite{bott}, Banyaga--Hurtubise \cite{banyaga} and Frauenfelder \cite[Appendix]{frauenfelder_3}.\

One can roughly describe the Conley--Zehnder index as a mean winding number for the linearized Hamiltonian flow along the orbit $x$ or the number of times that an eigenvalue crosses 1.\ It roughly measures how often neighbouring orbits of the same energy wind round the orbit $x$.\ Since we treat the reduced monodromy (\ref{symplectic_autom_2}) in Sp(1), we consider the transversal Conley--Zehnder index with standard normalization or counter-clockwise normalization for non-degenerate paths, as defined by Hofer--Wysocki--Zehnder \cite[Appendix]{hofer_w_z_1}, \cite[Section 3]{hofer_w_z} where the details can be seen.\ Explicitly, it is given in the following way.\

Let $b_1(0), b_2(0)$ and $X_H|_{\Sigma}(x_0)$ be a basis for the 3 dimensional vector space $T_{x_0}\Sigma$ such that $\omega_{x_0}\big(b_1(0),b_2(0)\big) = 1$.\ Note that $\omega_{x_0} \big( b_i(0), X_H|_{\Sigma}(x_0) \big) = 0$, for $i=1,2$.\ The first two basis vectors induce a symplectic basis for the 2 dimensional symplectic vector space $T_{x_0}\Sigma / ( \text{ker}\omega_{x_0} \vert _{T_{x_0}\Sigma} )$ and we denote by $P:= \langle b_1(0), b_2(0) \rangle_{\mathbb{R}}$ the 2-plane.\

We choose a smooth disc map $\overline{x} \in C^{\infty} ( \mathbb{D}, \Sigma )$, where $\mathbb{D} = \{ z \in \mathbb{C} \mid |z| \leq 1 \}$, such that on the boundary it satisfies $\overline{x} (e^{2\pi \text{i}t/T_x}) = x(t)$.\ Furthermore, we fix a symplectic trivialization for the pullback bundle $\tau \colon \mathbb{D} \times \mathbb{R}^2 \to \overline{x}^*T\Sigma / \text{ker}\omega|_{\Sigma} $.\ For details about such trivializations we refer to \cite[Section 2.6]{mcduff_salamon}.\ With respect to these choices, the linearized flow along $x$ generates a path $\Phi_x \colon [0,T_x] \to \text{Sp}(1)$ of symplectic matrices in $\mathbb{R}^2$ defined by
$$
\begin{tikzcd}[column sep=4.5em]
\mathbb{R}^2 \arrow[r,dashed, "\Phi_x(t)"] \arrow[swap,d, "\big(\tau(1)\big)^{-1}"] & \mathbb{R}^2 \\
T_{x_0}\Sigma / ( \text{ker}\omega_{x_0} \vert _{T_{x_0}\Sigma}) \arrow[r,"\overline{d\varphi _H ^t | _{\Sigma} (x_0)}"] & T_{x(t)}\Sigma / ( \text{ker}\omega_{x(t)} \vert _{T_{x(t)}\Sigma}) \arrow[swap,u, "\tau(e^{2 \pi \text{i}t/T_x})"]
\end{tikzcd}
$$
This path starts at $\Phi_x(0)=\text{id}$ and has a well-defined Conley--Zehnder index, and the transversal Conley--Zehnder index of $x$ is the Conley--Zehnder index of this path, which we denote by $\mu_{CZ}$.\

If $x$ is elliptic, then the 2-plane $P$ was rotated by an angle.\ Consider the rotation function $\theta(t)$ which gives the rotation angle of the 2-plane $P$ at each time $t \in [0,T_x]$.\ Note that $\theta(0)=0$ and $\theta(t)$ is continuous in $t$.\ Then
\begin{align} \label{index_1}
\mu_{CZ} = 2 \lfloor \theta(T_x)/(2\pi) \rfloor + 1
\end{align}
and the number of complete rotations of the neighbouring orbits during $T_x$ is given by
\begin{align} \label{index_2}
\text{rot}(x) := \lfloor \theta(T_x)/(2\pi) \rfloor = \frac{1}{2}\left(\mu_{CZ} - 1\right) \in \mathbb{Z}.
\end{align}
Therefore for every complete rotation the index jumps by 2 and is odd.\

If $x$ is hyperbolic, then the 2-plane $P$ was rotated by $m\pi$ for an integer $m$ and
\begin{align*}
\mu_{CZ} = m \in \begin{cases}
2 \mathbb{Z} + 1 & \text{if $x$ is negative hyperbolic}\\
2 \mathbb{Z} & \text{if $x$ is positive hyperbolic.}
\end{cases}
\end{align*}

By the implicit function theorem, non-degenerate periodic orbits always come in a smooth family of periodic orbits and hence form a smooth orbit cylinder (see Figure \ref{orbitcylinder} and \cite[p.\ 202]{meyer} for details).\ Moreover, all periodic orbits on the orbit cylinder have the same index, since they are connected by a path of non-degenerate orbits.\
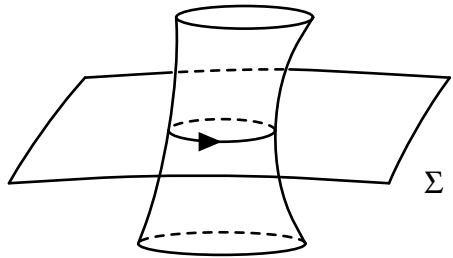
\begin{figure}[H]
	\centering
	\begin{tikzpicture}[line cap=round,line join=round,>=triangle 45,x=1.0cm,y=1.0cm]
	\clip(-4,-1.9) rectangle (3.5,1.8);
	
	\draw [dashed,line width=1pt] (0.4,-1.5) arc (0:180:1.1 and 0.15);
	
	\draw [line width=1pt] (-1.8,-1.5) arc (180:360:1.1 and 0.15);
	
	\draw [dashed,line width=1pt] (0,0) arc (0:180:0.7 and 0.15);
	
	\draw [decoration={markings, mark=at position 0.5 with {\arrow{>}}}, postaction={decorate},line width=1pt] (-1.4,0) arc (180:360:0.7 and 0.15);
	
	\draw [line width=1pt] (0.5,1.5) arc (0:360:0.9 and 0.15);
	
	\draw[rounded corners=40pt,line width=1pt](0.4,-1.5)--(-0.4,0)--(0.5,1.5);
	
	\draw[rounded corners=40pt,line width=1pt](-1.8,-1.5)--(-1.23,0)--(-1.3,1.5);
	
	
	\draw[rounded corners=40pt,line width=1pt](-3.5,-0.7)--(-1,-0.55)--(1.5,-0.7);
	
	
	\draw[rounded corners=20pt,line width=1pt] (1.5,-0.7)--(1.8,0)--(2.3,0.7);
	
	
	\draw[rounded corners=20pt,line width=1pt] (-3.5,-0.7)--(-3.1,0)--(-2.5,0.7);
	
	

	\draw[dash pattern=on 2pt off 2pt,opacity=0] (-2.5,0.7) .. controls (-0.1,0.85) ..
	coordinate[pos=0.21] (A)
	coordinate[pos=0.57] (B)
	coordinate[pos=0.58] (G)
	coordinate[pos=0.1] (C)
	coordinate[pos=0.4] (D)
	coordinate[pos=0.7] (E)
	coordinate[pos=0.19] (F) (2.3,0.7);
	
	
	\draw [line width=1pt] (-2.5,0.7) .. controls (C) .. (F);
	
	\draw [dashed,line width=1pt] (A) .. controls (D) .. (B);
	
	\draw [line width=1pt] (G) .. controls (E) .. (2.3,0.7);

	
	
	\draw (1.82,-0.4) node[anchor=north west] {$\Sigma$};
	
	\end{tikzpicture}
	\caption{Orbit cylinder}
	\label{orbitcylinder}
\end{figure}
Let $x^n$ be the $n$ times iteration of $x$ with first return time $nT_x$, $n \geq 1$.\ We call $x^n$ the $n$-th cover of $x$.\ Assume that $x^n$ is non-degenerate for all $n \geq 1$, then the index iteration (we refer to \cite[p.\ 249]{hofer_w_z_1} for details) is given by
\begin{table}[H]
	\centering
	\begin{tabular}{c|c}
		$x$ & Conley--Zehnder index of $x^n$\\
		\hline pos. / neg. hyperbolic & $n \mu_{CZ}$\\
		elliptic & $2 \lfloor n \theta(T_x)/(2\pi) \rfloor + 1$
	\end{tabular}
	\vspace{0.1cm}
	\caption{Index iteration}
	\label{table_index_iteration}
\end{table}
\begin{remark} \label{remark_2_7}
	Suppose that $x$ is elliptic.\ If $x$ becomes positive hyperbolic, i.e.\ the reduced monodromy moves through the eigenvalue 1, then $\mu_{CZ}$ jumps by $\pm 1$.\ If the rotation angle moves through 0 to positive hyperbolic, then $\mu_{CZ}$ jumps by $-1$ and if it goes through $2 \pi$ to positive hyperbolic, then $\mu_{CZ}$ jumps by $+1$.\ For the other way, the change of index is exactly backwards, see the Figure \ref{scenario_1}.\
	
	For the cases from elliptic to negative hyperbolic or from negative hyperbolic to elliptic, the Conley--Zehnder index of $x$ does not change, but the reduced monodromy of its double cover crosses the eigenvalue 1, hence its Conley--Zehnder index jumps by $\pm 1$.\ Note that this gives rise to a bad orbit (see Example \ref{example_7_2} in Subsection \ref{sec:7.1}) and furthermore, the double cover of a negative hyperbolic periodic orbit is positive hyperbolic.\
	
	In the elliptic case, if the roation angle is a $\tilde{k}$-th root of unity, for $\tilde{k} \geq 3$ the $\tilde{k}$-th cover is still elliptic and its reduced monodromy goes through the eigenvalue 1, thus  its Conley--Zehnder index jumps by $\pm 2$.\
	\begin{figure}[H]
		\centering
		\begin{tikzpicture}[line cap=round,line join=round,>=triangle 45,x=1.0cm,y=1.0cm]
		\clip(-3,-2.4) rectangle (14,3.7);
		
		\draw [decoration={markings, mark=at position 0.07 with {\arrow{<}}}, postaction={decorate},line width=1pt] [decoration={markings, mark=at position 0.94 with {\arrow{>}}}, postaction={decorate},line width=1pt] [line width=1pt] (0.0,0.0) circle (2.0cm);
		\draw [decoration={markings, mark=at position 0.8 with {\arrow{>}}}, postaction={decorate},line width=1pt] [line width=1pt] [->,line width=1pt] (-2.7,0.0) -- (4.5,0.0);
		\draw [->,line width=1pt] (0.0,-3.0) -- (0.0,3.0);
		\draw (0.04,-0.04000000000000001) node[anchor=north west] {$0$};
		\draw (2.02,-0.020000000000000004) node[anchor=north west] {$1$};
		\draw (-2.72,-0.020000000000000004) node[anchor=north west] {$-1$};
		\draw (2.2,1) node[anchor=north west] {$\mu_{CZ} - 1$};
		\draw (2.2,-0.5) node[anchor=north west] {$\mu_{CZ} + 1$};
		\draw (-2.7,3.7) node[anchor=north west] {from elliptic to positive hyperbolic};
		\begin{scriptsize}
		\draw [fill=black] (0.0,0.0) circle (2pt);
		\draw [fill=black] (2,0.0) circle (2pt);
		\draw [fill=black] (-2,0.0) circle (2pt);
		\draw [fill=black] (1.41,1.41) circle (2pt);
		\draw [fill=black] (1.41,-1.41) circle (2pt);
		\draw [fill=black] (3.5,0) circle (2pt);
		\end{scriptsize}
		
		\draw [decoration={markings, mark=at position 0.07 with {\arrow{>}}}, postaction={decorate},line width=1pt] [decoration={markings, mark=at position 0.94 with {\arrow{<}}}, postaction={decorate},line width=1pt] [line width=1pt] (9,0.0) circle (2.0cm);
		\draw [decoration={markings, mark=at position 0.8 with {\arrow{<}}}, postaction={decorate},line width=1pt] [line width=1pt] [->,line width=1pt] (-2.7+9,0.0) -- (4.5+9,0.0);
		\draw [->,line width=1pt] (9,-3.0) -- (9,3.0);
		\draw (0.04+9,-0.04000000000000001) node[anchor=north west] {$0$};
		\draw (2.02+9,-0.020000000000000004) node[anchor=north west] {$1$};
		\draw (-2.72+9,-0.020000000000000004) node[anchor=north west] {$-1$};
		\draw (2.2+9,1) node[anchor=north west] {$\mu_{CZ} + 1$};
		\draw (2.2+9,-0.5) node[anchor=north west] {$\mu_{CZ} - 1$};
		\draw (-2.7+9,3.7) node[anchor=north west] {from positive hyperbolic to elliptic};
		\begin{scriptsize}
		\draw [fill=black] (9,0.0) circle (2pt);
		\draw [fill=black] (2+9,0) circle (2pt);
		\draw [fill=black] (-2+9,0) circle (2pt);
		\draw [fill=black] (1.41+9,1.41) circle (2pt);
		\draw [fill=black] (1.41+9,-1.41) circle (2pt);
		\draw [fill=black] (3.5+9,0) circle (2pt);
		\end{scriptsize}
		
		\end{tikzpicture}
		\caption{The index jump by $\pm 1$ in the elliptic and positive hyperbolic transitions}
		\label{scenario_1}
	\end{figure}
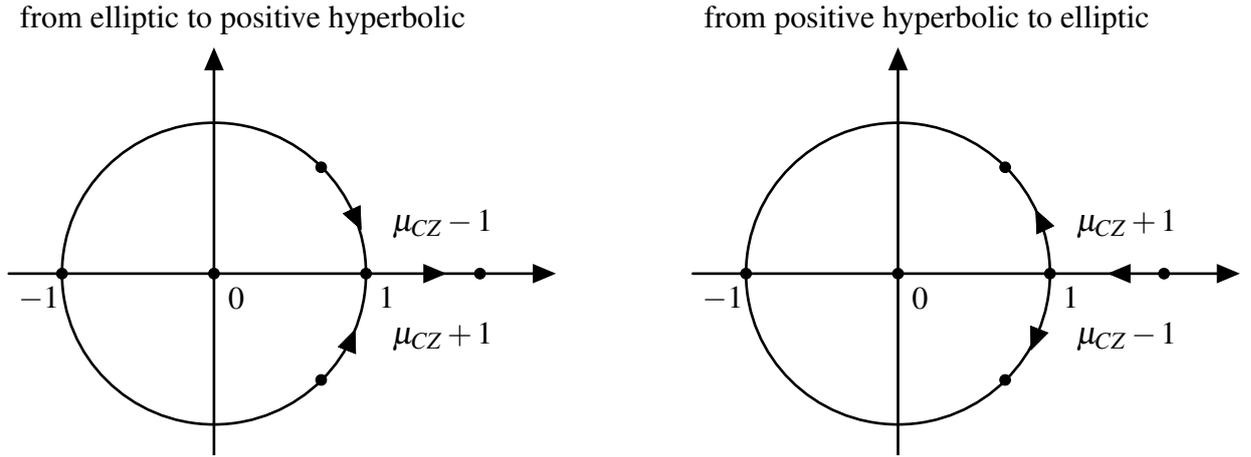
	\noindent
	The index jump gives rise to the bifurcation of new families of periodic orbits (see discussion in the Subsection \ref{sec8.2}).\
\end{remark}

\section{Hamiltonian manifolds}
\label{sec:hamiltonian}

\begin{center}
	\textit{``A Hamiltonian manifold is the odd-dimensional analog of a symplectic manifold."}
\end{center}
\begin{flushright}
	- U. Frauenfelder and O. van Koert in \cite[p.\ 20]{frauenfelder}
\end{flushright}

\subsection{Involutive Hamiltonian manifolds and Symplectic Splitting}
\label{sec:involutive}

The archetypical examples of Hamiltonian manifolds are energy hypersurfaces, see (\ref{energy_hypersurface}) above.
\begin{definition}
A \textbf{Hamiltonian manifold} is a pair $(\Sigma,\omega)$ where $\Sigma$ is a $2n-1$ dimensional manifold and $\omega \in \Omega^2(\Sigma)$ a closed $2$-form such that $\text{ker}\omega \subset T\Sigma$ is a one-dimensional distribution.\ The $2$-form $\omega$ is called a \textbf{Hamiltonian structure} on $\Sigma$.
\end{definition}
\begin{definition}
A \textbf{Hamiltonian vector field} $X$ on a Hamiltonian manifold is a non-vanishing section of the line bundle $\text{ker}\omega$. 
\end{definition}
\begin{remark}
Since the Lie derivative vanishes by Cartan's identity, $\mathcal{L}_X \omega = d \iota_X \omega + \iota_x d \omega = 0$, the Hamiltonian structure is preserved under the flow of $X$ for all times $t \in \mathbb{R}$, i.e. $(\varphi_X^t)^* \omega = \omega.$
\end{remark}
\begin{remark}
As seen before in the case of energy hypersurfaces (\ref{energy_hypersurface}), a Hamiltonian manifold comes with a foliation on $\Sigma$ (see Figure \ref{foliation}), whose leaves $L \subset \Sigma$ are the one-dimensional submanifolds such that $T_xL = \text{ker}\omega_x$, for all $x \in L$.\ Compact leaves are diffeomorphic to $S^1$ and non compact ones are diffeomorphic to $\mathbb{R}$.
\end{remark}
\begin{definition}
A \textbf{symplectic involution} $\sigma$ on a Hamiltonian manifold $(\Sigma,\omega)$ is a diffeomorphism $\sigma: \Sigma \to \Sigma$ such that $\sigma^2=\text{id}_{\Sigma}$ and $\sigma^* \omega = \omega$.
\end{definition}\
\noindent
We next study the properties of the linear isomorphism
\begin{align} \label{linear_isomorphism}
d \sigma (x) : T_x \Sigma \to T_{\sigma(x)}\Sigma,\quad x \in \Sigma.
\end{align}
\begin{lemma1} \label{lemma_3_6}
	The one-dimensional distribution is invariant under $d\sigma(x)$, i.e.\ if $\xi \in \text{ker}\omega_x$, then
	$$d\sigma(x)\xi \in \text{ker}\omega_{\sigma(x)}.$$
\end{lemma1}
\begin{proof}
	Let $\xi \in \text{ker}\omega_x$, i.e.
	$$\omega_x(\xi,\eta)=0,\quad \text{for all }\eta \in T_x \Sigma.$$
	By (\ref{linear_isomorphism}) we know that for all $\eta' \in T_{\sigma(x)}\Sigma$ there exists a unique element $\eta \in T_x\Sigma$ such that $d\sigma(x)\eta=\eta'$. Using $\sigma^*\omega = \omega$ we compute for all $\eta' \in T_{\sigma(x)}\Sigma$ that
	$$\omega_{\sigma(x)}\big(d\sigma(x)\xi,\eta' \big) = \omega_{\sigma(x)}\big(d\sigma(x)\xi, d\sigma(x)\eta \big) = \omega_x(\xi,\eta) = 0.$$
\end{proof}\
\noindent
Denote by $F$ the fixed point set of $\sigma$, i.e.\ $F=\text{Fix}(\sigma)=\{ x \in \Sigma \mid \sigma (x) = x \}$.
\begin{remark}
	Let $x \in F$.\ By choosing a Riemannian metric on $\Sigma$ which is $\sigma$-invariant, we can parametrize $F$, locally around $x$, by the restriction of the exponential map to $E_1 \big( d\sigma(x) \big)$.\ Thus $F$ is a submanifold of $\Sigma$ and
	\begin{align*}
	T_x F = E_1 \big( d\sigma(x) \big) = \text{ker}\big( d \sigma(x) - \text{id} \big).
	\end{align*}
\begin{lemma1}
For $x \in F$, the eigenvalues of $d\sigma(x)$ are $\pm 1$.
\end{lemma1}
\begin{proof}
Let $x \in F$ and $\lambda$ be an eigenvalue of $d\sigma(x)$ and $\xi \in T_x\Sigma$ an eigenvector to $\lambda$.\ The calculation
$$\xi = d\sigma^2(x)\xi = d\sigma(x)\big(d\sigma(x)\xi\big) = d\sigma(x)(\lambda\xi) = \lambda d \sigma(x)\xi = \lambda^2 \xi$$
implies that $\lambda=\pm1$.
\end{proof}\
\noindent
For $x \in F=\text{Fix}(\sigma)$ we obtain the two eigenspaces
$$ T_x F = E_1\big(d\sigma(x)\big) = \{ \xi \in T_x\Sigma: d\sigma(x)\xi=\xi \} = \text{ker}\big(d\sigma(x) - \text{id}\big) \subset T_x\Sigma$$
and
$$E_{-1}\big(d\sigma(x)\big) = \{ \eta \in T_x\Sigma: d\sigma(x)\eta=-\eta \} = \text{ker}\big(d\sigma(x) + \text{id}\big) \subset T_x\Sigma.$$
\begin{lemma1} \label{lemma_1}
	The tangent space splits into the direct sum of the eigenspaces, i.e.
	$$T_x\Sigma = E_1\big(d\sigma(x)\big) \oplus E_{-1}\big(d\sigma(x)\big),\quad x \in F.$$
\end{lemma1}
\begin{proof}
	Obviously,
	$$E_1\big(d\sigma(x)\big) \cap E_{-1}\big(d\sigma(x)\big) = \{0\}.$$
	Moreover, with the projection maps
	$$
	\begin{tikzcd}
	& T_x\Sigma \arrow[dl,"\pi_1"'] \arrow{rd}{\pi_2} & \\
	E_1\big(d\sigma(x)\big) &  & E_{-1}\big(d\sigma(x)\big)
	\end{tikzcd}$$
	$$\pi_1 = \frac{1}{2} \big( \text{id} + d\sigma(x) \big) \text{ and } \pi_2 = \frac{1}{2} \big( \text{id} - d\sigma(x) \big)$$
	we can write any $\xi \in T_x\Sigma$ as
	$$\xi = \pi_1(\xi) + \pi_2(\xi) = \frac{1}{2}\big(\xi + d\sigma(x)\xi\big) + \frac{1}{2}\big(\xi - d\sigma(x)\xi\big).$$
\end{proof}
\noindent
\begin{remark} \label{remark_linear_symplectic}
	We recall from linear symplectic geometry:\ Let $(V,\omega)$ be a symplectic vector space and $W \subset V$ a linear subspace.\ The symplectic complement $W^{\omega}$ of $W$ is defined as the subspace
	$$W^{\omega} := \{ v \in V \mid \omega(v,w)=0, \forall w \in W \}.$$
	Furthermore $W$ is called symplectic if $\omega | _{W}$ is symplectic. Equivalently,
	$$W \cap W^{\omega} = \{ 0 \},$$
	i.e. $W$ and $W^{\omega}$ are symplectically orthogonal. In addition, for any subspace $W$ it holds that
	$$\text{dim}V = \text{dim}W + \text{dim}W^{\omega},\quad (W^{\omega})^{\omega} = W.$$
	The non-degeneracy of $\omega$ implies that for a symplectic subspace $W \subset V$ we obtain
	$$V = W \oplus W^{\omega}$$
	and the symplectic complement $W^{\omega} \subset V$ is also symplectic.\ Finally, if $V_1$ and $V_2$ are sub-vector spaces of $V$ that are symplectically orthogonal, then $V_1$ and $V_2$ are symplectic.\
\end{remark}
\begin{definition}
	Let $S \subset T_x\Sigma$ be a linear subspace. The \textbf{symplectic complement} of $S$ is defined as the linear subspace $\{ \xi \in T_x\Sigma \mid \omega_x(\xi,\eta) = 0, \forall \eta \in S \}$.\ We call two subspaces $S_1, S_2 \subset T_x\Sigma$ \textbf{symplectically orthogonal} if
	$$\omega_x(\xi,\eta) = 0,\quad \text{for all } \xi \in S_1, \text{ }\eta \in S_2.$$
\end{definition}
\begin{lemma1} \label{lemma_2}
	The eigenspaces $E_1\big(d\sigma(x)\big)$ and $E_{-1}\big(d\sigma(x)\big)$ are symplectically orthogonal.
\end{lemma1}
\begin{proof}
	Since $\sigma^* \omega = \omega$ we obtain for all $\xi \in E_1 \big( d\sigma(x) \big)$ that
	$$ \omega_x(\xi,\eta) = \omega_{\sigma(x)}\big( d\sigma(x)\xi, d\sigma(x)\eta \big) = \omega_x(\xi,-\eta) = - \omega_x(\xi,\eta)$$
	for all $\eta \in E_{-1}\big(d\sigma(x)\big)$.
\end{proof}
\begin{lemma1} \label{lemma_flow_invariant}
	Let $X$ be a non-vanishing section of the line bundle ker$\omega$ such that $X$ is invariant under the symplectic involution $\sigma$.\ In other words, $\sigma$ and the flow of $X$ commute.\ Then the linearized flow map
	$$ d \varphi_X^t(x) \colon T_x \Sigma \to T_{\varphi_X^t(x)} \Sigma,\quad x \in F $$
	leaves the two eigenspaces invariant.
\end{lemma1}
\begin{proof}
	For $\xi \in E_{\pm 1} \big( d \sigma(x) \big)$
	$$ d \varphi_X^t(x) \xi = \pm d \varphi_X^t(x) \big( d \sigma(x) \xi \big) = \pm d \sigma \big( \varphi_X^t(x) \big) \big( d \varphi_X^t(x) \xi \big)  $$
	and hence
	$$ d \sigma \big( \varphi_X^t(x) \big) \big( d \varphi_X^t(x) \xi \big) = \pm d \sigma^2 \big( \varphi_X^t(x) \big) \big( d \varphi_X^t(x) \xi \big) = \pm d \varphi_X^t(x) \xi, $$
	which means that $d \varphi_X^t(x) \xi \in E_{\pm 1} \Big( d \sigma \big( \varphi_X^t(x) \big) \Big)$.
\end{proof}
Moreover, we obtain an induced $2n-2$ dimensional symplectic vector space
$$T_x\Sigma / \text{ker}\omega_x$$
and a quotient bundle $T\Sigma / \text{ker}\omega$ over $\Sigma$ which is a symplectic vector bundle of rank $2n-2$.\ Therefore, if the Hamiltonian manifold $(\Sigma,\omega)$ is orientable, then the line bundle $\text{ker}\omega$ is also orientable.\ This motivates the next definition and lemma.
\end{remark}
\begin{definition} \label{definition_inv_ham}
An \textbf{involutive Hamiltonian manifold} is a triple $(\Sigma,\omega,\sigma)$, where $(\Sigma,\omega)$ is an orientable Hamiltonian manifold and $\sigma$ a symplectic involution which preserves the orientation of the line bundle $\text{ker}\omega$.
\end{definition}
\begin{corollary1} \label{corollary_1}
Let $(\Sigma,\omega,\sigma)$ be an involutive Hamiltonian manifold and $x \in F = \text{Fix}(\sigma)$.\ Then the line bundle $\text{ker}\omega | _{F}$ is contained in the eigenspace to the eigenvalue $1$, i.e.
$$ \text{ker}\omega_x \subset E_1 \big( d\sigma(x) \big) = T_x F.$$
\end{corollary1}
\begin{proof}
Since $x$ is a fixed point of $\sigma$, the linear map (\ref{linear_isomorphism})
$$d\sigma(x): T_x\Sigma \to T_x\Sigma$$
is an automorphism.\ Together with Lemma \ref{lemma_3_6} we obtain for  $\xi \in \text{ker}\omega_x$ that $d\sigma(x)\xi = \pm \xi$, and in fact $+\xi$ since $\sigma$ is orientation preserving.
\end{proof}
\begin{theorem1}[\textbf{Symplectic Splitting}] \label{theorem_splitting} \textit{
	Let $(\Sigma,\omega,\sigma)$ be an involutive Hamiltonian manifold and $x \in F = \textnormal{Fix}(\sigma)$.\ Then:\
	\begin{itemize}
		\item[1)] The induced $2n-2$ dimensional symplectic vector space at $x$ splits into two symplectic vector spaces,
		\begin{align*}
		T_x\Sigma / \textnormal{ker} \omega_x &= \Big( E_1\big( d\sigma(x) \big) / \textnormal{ker} \omega_x \Big) \oplus \Big( E_{-1}\big( d\sigma(x) \big) \Big)\\
		&= \big( T_x\textnormal{Fix}(\sigma) / \textnormal{ker} \omega_x \big) \oplus \Big( E_{-1}\big( d\sigma(x) \big) \Big).
		\end{align*}
		\item[2)] Let $X$ be a non-vanishing section of the line bundle \textnormal{ker}$\omega$ such that $X$ is invariant under the symplectic involution $\sigma$.\ Then the induced linearized flow map
		$$ \overline{d\varphi _X ^t (x)} \colon T_{x}\Sigma / \textnormal{ker}\omega_{x} \to T_{\varphi _X ^t (x)}\Sigma / \textnormal{ker}\omega_{\varphi _X ^t (x)} $$
		is a $(2n-2) \times (2n-2)$ symplectic matrix that splits into two symplectic maps.\ More precisely,
		\begin{align} \label{decomposition}
		\overline{d\varphi _X ^t (x)} = \begin{pmatrix}
		A_1 & 0\\
		0 & A_2
		\end{pmatrix},
		\end{align}
		where
		$$ A_1 \colon T_{x} \textnormal{Fix}(\sigma) / \textnormal{ker}\omega_{x} \to T_{\varphi _X ^t (x)} \textnormal{Fix}(\sigma) / \textnormal{ker}\omega_{\varphi _X ^t (x)} $$
		$$ A_2 \colon E_{-1}\big( d\sigma(x) \big) \to E_{-1}\big( d\sigma(\varphi _X ^t (x)) \big) $$
		are symplectic matrices.\
	\end{itemize}}
\end{theorem1}
\begin{proof}
		Lemma \ref{lemma_1} gives the splitting
		$$T_x\Sigma = E_1\big(d\sigma(x)\big) \oplus E_{-1}\big(d\sigma(x)\big),$$
		where the two eigenspaces are symplectically orthogonal by Lemma \ref{lemma_2}.\ By Corollary \ref{corollary_1} we know that $ \text{ker}\omega_x \subset E_1 \big( d\sigma(x) \big) = T_x\text{Fix}(\sigma)$, which implies the first statement.\ The matrix form of the induced linearized flow consists of four block matrices,
		$$\overline{d\varphi _X ^t (x)} = \begin{pmatrix}
		A_1 & B_1\\
		B_2 & A_2
		\end{pmatrix},$$
		where
		$$ A_1 \colon T_{x} \text{Fix}(\sigma) / \text{ker}\omega_{x} \to T_{\varphi _X ^t (x)} \text{Fix}(\sigma) / \text{ker}\omega_{\varphi _X ^t (x)} $$
		$$ A_2 \colon E_{-1}\big( d\sigma(x) \big) \to E_{-1}\big( d\sigma(\varphi _X ^t (x)) \big) $$
		and where
		$$ B_1 \colon E_{-1}\big( d\sigma(x) \big) \to T_{\varphi _X ^t (x)} \text{Fix}(\sigma) / \text{ker}\omega_{\varphi _X ^t (x)} $$
		$$ B_2 \colon T_{x} \text{Fix}(\sigma) / \text{ker}\omega_{x} \to E_{-1}\big( d\sigma(\varphi _X ^t (x)) \big) $$
		are zero maps since, by Lemma \ref{lemma_flow_invariant}, the linearized flow leaves the two eigenspaces invariant.\ It now follows that $A_1$ and $A_2$ are symplectic.
\end{proof}
\begin{remark}
The dimensions of $\big( T_x\text{Fix}(\sigma) / \text{ker} \omega_x \big)$ and $\Big( E_{-1}\big( d\sigma(x) \big) \Big)$ are even, since they are symplectic vector spaces.\ Moreover, they are determined by the fixed point set $F= \text{Fix}(\sigma)$, which is odd-dimensional.\ In particular,
\begin{center}
	\begin{tabular}{c|c}
		 & dimension\\
		 \hline $T_x\text{Fix}(\sigma) / \text{ker} \omega_x$ & $\text{dim}F - 1$\\
		 $E_{-1}\big( d\sigma(x) \big)$ & $2n-1-\text{dim}F$
	\end{tabular}
\end{center}
\end{remark}
\begin{remark}
	An important property arises for the cases if $A_1$ and $A_2$ are both symplectic matrices in $\text{Sp}(1)=\text{SL}(2,\mathbb{R})$.\ Assume that the Floquet multipliers of $\overline{d\varphi _X ^t (x)}$ are neither real nor on the unit circle, then they are four different complex numbers of the form $\lambda, 1/\lambda, \overline{\lambda}$ and $ 1/ \overline{\lambda}$ (see Figure \ref{figure_7}).\ The decomposition (\ref{decomposition}) shows that this kind of eigenvalues are not possible.\ This is a phenomenon for the reduced monodromy of planar periodic orbits in the spatial Hill lunar problem as well as in the spatial CR3BP.\
\end{remark}

\subsection{Every contact manifold is a Hamiltonian manifold}
\label{subsec:contact}

Let $(\Sigma,\lambda)$ be a contact manifold, i.e. $\Sigma$ is a $2n-1$ dimensional manifold and $\lambda$ a $1$-form on $\Sigma$, called contact form on $\Sigma$, such that $\lambda \wedge (d\lambda)^{\wedge(n-1)}$ is a volume form on $\Sigma$.\ The Reeb vector field $R$ on $\Sigma$ is defined by the conditions $\lambda(R)=1$ and $\iota_R d \lambda = 0$.\ Then $\omega:=d\lambda$ is a Hamiltonian structure on $\Sigma$, i.e. the tuple $(\Sigma, \omega= d \lambda)$ is a Hamiltonian manifold.\ Namely, for $x \in \Sigma$ the Reeb vector field is a non-vanishing section of the line bundle ker$d \lambda$ = ker$\omega$ $\subset T\Sigma$ and
\begin{align} \label{ker_reeb}
\text{ker}\omega_x = \langle R(x) \rangle.
\end{align}
The hyperplane distribution is defined as $ \xi := \text{ker}\lambda \subset T\Sigma $, it is also called contact structure on $\Sigma$.\ This leads to the decomposition
\begin{align}\label{decomposition_1}
T \Sigma = \xi \oplus \langle R \rangle.
\end{align}
The restriction of $\omega = d \lambda$ to $\xi$ makes $\xi$ a symplectic vector bundle of rank $2n-2$ over $\Sigma$, i.e.\ for $x \in \Sigma$ the space $( \xi_x , \omega |_{\xi_x} = d\lambda | _{\xi_x} )$ is a symplectic vector space.\ The contact structure $\xi = \text{ker}\lambda$ is determined by the contact form $\lambda$, but the converse is not true:\ For every positive smooth function $f$, the 1-form $f\lambda$ is also a contact form that gives the same contact structure, i.e.
$$ \xi = \text{ker} \lambda = \text{ker} f \lambda, $$
but in general, the Reeb vector fields of $\lambda$ and $f \lambda$ are not parallel.\ Therefore we cannot obtain the Hamiltonian structure $\omega = d \lambda$ from the contact structure $\xi = \text{ker} \lambda$.\ Since the Hamiltonian structure $\omega = d \lambda$ determines the dynamics up to time reparametrization, our major attention is on $\omega=d\lambda$.
\begin{remark}
If contact manifolds arise as energy hypersurfaces (\ref{energy_hypersurface}), then by (\ref{line_bundle}) and (\ref{ker_reeb}), i.e.\ by
$$ \langle X_{H} (x) \rangle = \text{ker}\omega_x = \langle R(x) \rangle, $$
the restriction of the Hamiltonian vector field $X_H |_{\Sigma}$ and the Reeb vector field are parallel, i.e.\ up to time reparametrization their flows coincide.\ In view of (\ref{decomposition_1}), for a periodic orbit $x$ with first return time $T_x$ the reduced monodromy (\ref{symplectic_autom_2}) corresponds to the transverse linearized Reeb flow
$$  d \varphi_R^{T_x} (x_0)|_{\xi_{x_0}} \colon \xi_{x_0} \to \xi_{x_0}.$$
\end{remark}
\begin{definition}
	A \textbf{contact form} on a Hamiltonian manifold $(\Sigma,\omega)$ is a $1$-form $\lambda$ on $\Sigma$ such that $d\lambda=\omega$ and $\lambda \wedge \omega^{\wedge(n-1)} > 0.$
\end{definition}
\begin{remark}
	Every Hamiltonian manifold $(\Sigma,\omega)$ with a contact form $\lambda$ is orientable via the volume form $\lambda \wedge \omega^{\wedge(n-1)} >0$.\ Not every Hamiltonian manifold has a contact form.\ The next lemma gives a condition.
\end{remark}
\begin{lemma1}
	Let $(\Sigma,\omega)$ be a Hamiltonian manifold which is simply connected. If there exists a closed leaf $L \subset \Sigma$ with filling disk $D \subset \Sigma$ (i.e.\ $\partial D= L$) such that $\int \limits_D \omega \leq 0$, then $(\Sigma,\omega)$ has no contact form.
\end{lemma1}
\begin{proof}
	Assume that $(\Sigma,\omega)$ has a contact form $\lambda$ and there is a closed leaf $L \subset \Sigma$ with filling disk $D\subset \Sigma$, which is a periodic orbit of the Reeb vector field $R$.\ Then there exists $T>0$ and a smooth map $x : [0,T] \to \Sigma$, which is injective on $[0,T)$ and such that $\dot{x}(t) = R\big(x(t)\big), \text{ } x(T)=x(0)$ and im$(x)=L$ (see Figure \ref{figure7}).\
	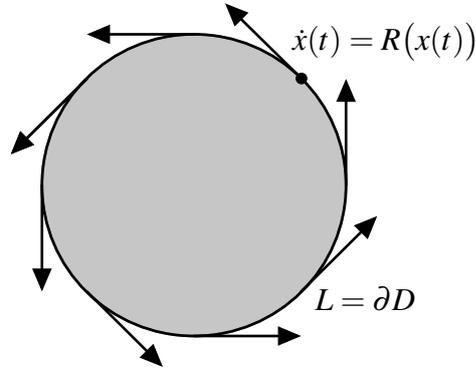
\begin{figure}[H]
		\centering
		\begin{tikzpicture}
		[line cap=round,line join=round,>=triangle 45,x=1cm,y=1cm]
		\clip(-4.2,-2.8) rectangle (4.2,2.8);
		
		\draw [line width=1pt,fill=gray,opacity=0.45] (0,0) circle (2cm);
		\draw [line width=1pt] (0,0) circle (2cm);
		\draw (1.45,-1.2) node[anchor=north west] {$L= \partial D$};
		\draw [->,line width=1pt] (1.4142,1.4142) -- (0.41,2.41);
		\draw (1.15,2.3) node[anchor=north west] {$\dot{x}(t) = R \big( x(t) \big)$};
		\draw [->,line width=1pt] (0,2) -- (-1.41,2);
		\draw [->,line width=1pt] (-1.41,1.41) -- (-2.41,0.41);
		\draw [->,line width=1pt] (-2,0) -- (-2,-1.41);
		\draw [->,line width=1pt] (-1.41,-1.41) -- (-0.41,-2.41);
		\draw [->,line width=1pt] (0,-2) -- (1.41,-2);
		\draw [->,line width=1pt] (1.41,-1.41) -- (2.41,-0.41);
		\draw [->,line width=1pt] (2,0) -- (2,1.41);
		\draw [->,line width=1pt] (-1.41,1.41) -- (-2.41,0.41);
		\begin{scriptsize}
		\draw [fill=black] (1.4142,1.4142) circle (2pt);
		\end{scriptsize}
		\end{tikzpicture}
		\caption{Periodic orbit of Reeb vector field with filling disk}
		\label{figure7}
	\end{figure}
	\noindent
	By Stokes' theorem and the definition of $R$ we obtain
	$$ \int \limits_D \omega = \int \limits_D d \lambda = \int \limits_{\partial D = L} \lambda = \int \limits_0^T \lambda \big( x(t) \big) \underbrace{\dot{x}(t)}_{\mathlarger{R\big(x(t)\big)}} dt = \int \limits_0^T 1 dt = T > 0.$$
\end{proof}
In the following let $(\Sigma, \omega = d \lambda, \sigma)$ be an involutive Hamiltonian manifold, i.e.\ $\lambda$ is a contact form on $(\Sigma,\omega)$ and $\sigma$ is a symplectic involution which preserves the orientation of the line bundle ker$\omega$ = ker$d\lambda$.\
\begin{remark} \label{remark_orientation}
For $x \in \Sigma$, as in (\ref{ker_reeb}) and (\ref{decomposition_1}) consider the decomposition
$$ T_x \Sigma = \text{ker}\omega_x + \xi_x. $$
Since $\sigma$ preserves the orientation of the line bundle, the orientation of $T \Sigma$ is $\sigma$-invariant.\
\end{remark}
\begin{remark}
For $x \in \text{Fix}(\sigma)$ we obtain by Corollary \ref{corollary_1} that
$$\langle R(x) \rangle \subset E_1\big( \sigma(x) \big) = T_x\text{Fix}(\sigma).$$
Moreover, by the first part of the Symplectic Splitting Theorem \ref{theorem_splitting} the $2n-2$ dimensional symplectic vector space $\xi_x$ splits into two symplectic spaces, namely
\begin{align} \label{splitting_contact}
\xi_x = \big( T_x\text{Fix}(\sigma) / \langle R(x) \rangle \big) \oplus \Big( E_{-1}\big( d\sigma(x) \big) \Big) = T_x \Sigma / \langle R(x) \rangle.
\end{align}
\end{remark}
\begin{lemma1} \label{lemma:contact}
	The restriction of $\lambda$ to the fixed point set $F = \text{Fix}(\sigma)$ is preserved under $\sigma$ and a contact form on $F$.
\end{lemma1}
\begin{proof}
	Since $\sigma^* \omega = \omega$ and the pullback commutes with the exterior derivate, we have
	$$d \lambda = \omega = \sigma^* \omega = \sigma^* d \lambda = d \sigma^* \lambda.$$
	Hence there is a closed 1-form $\lambda'$ on $\Sigma$ such that
	$$ \sigma^* \lambda = \lambda + \lambda'. $$
	Therefore for all $\eta \in T_x F$ we obtain
	$$ \lambda_x(\eta) = \lambda_{x} \big( d \sigma(x)\eta \big) = \lambda_x(\eta) + \lambda'_x(\eta), $$
	i.e.\ $\lambda'$ vanishes on $T_x F$.\ Thus $\lambda_F := \lambda | _{TF}$ is preserved under $\sigma$.\ We next show that $\lambda_F$ is a volume form.\ For $x \in F$ consider the symplectic splitting
	$$ \text{ker}\lambda_x = \xi_x = \big( T_x F / \langle R(x) \rangle \big) \oplus \Big( E_{-1}\big( d\sigma(x) \big) \Big) = T_x \Sigma / \langle R(x) \rangle $$
	from (\ref{splitting_contact}).\ Since $E_1':= T_x F / \langle R(x) \rangle$ is symplectic, for a basis $v_1,...,v_{2k}$ of $E_1'$ we have
	$$ (\omega_F)_x (v_1,...,v_{2k}) \neq 0. $$
	Then on the basis $R,v_1,...,v_{2k}$ of $T_x F$ we obtain
	$$ \lambda_F (x) \wedge \big( \omega_F (x) \big)^{\wedge k} (R,v_1,...,v_{2k}) = \lambda_x (R) \cdot (\omega_x)^{\wedge k} (v_1,...,v_{2k}) \neq 0 .$$
\end{proof}
\begin{remark}
	As in the spatial Hill lunar problem, the planar CR3BP arises in the spatial CR3BP as a fixed point set of a symplectic involution.\ In \cite{albers} it was proved that the regularized planar CR3BP is of contact type for energies below and also slightly above the first critical value.\ In \cite{cho} it was shown that the regularized spatial CR3BP has also the contact property.\ Now Lemma \ref{lemma:contact} together with the result from \cite{cho} implies the result from \cite{albers}.\
\end{remark}
In the previous lemma it was not necessary that $\lambda$ is preserved by $\sigma$.\ In the next lemma and corollary we show that the average $\frac{1}{2} (\lambda + \sigma^* \lambda)$ is always a contact form which is preserved under $\sigma$.\
\begin{lemma1}
	$\sigma^* \lambda$ is a contact form on $(\Sigma,\omega)$.
\end{lemma1}
\begin{proof}
	It holds that
	$$ d \sigma^*\lambda = \sigma^* d \lambda = \sigma^* \omega = \omega. $$
	By using $ \lambda \wedge \omega^{\wedge (n-1)} > 0$ and Remark \ref{remark_orientation} we have that
	\begin{align} \label{contact_volume}
	\sigma^* \lambda \wedge \omega^{\wedge (n-1)} = \sigma^* \lambda \wedge (\sigma^*\omega)^{\wedge (n-1)} = \sigma^*( \lambda \wedge \omega^{\wedge (n-1)} ) > 0.
	\end{align}
\end{proof}
\begin{corollary1}
	If $\sigma^* \lambda \neq \lambda$, then $\frac{1}{2} (\lambda + \sigma^* \lambda)$ is a contact form on ($\Sigma,\omega$) which is preserved by $\sigma$.
\end{corollary1}
\begin{proof}
	We define the average as $ \tilde{\lambda} := \frac{1}{2}( \lambda + \sigma^* \lambda ) $.\ It is easy to see that $d \tilde{\lambda} = \omega$.\ By taking the sum of $ \lambda \wedge \omega^{\wedge (n-1)} > 0$ and (\ref{contact_volume}) 
	we obtain that $ \tilde{\lambda} \wedge \omega^{\wedge (n-1)} >0$, hence $\tilde{\lambda}$ is a contact form on ($\Sigma,\omega$).\ Since $\sigma$ is an involution, it is obvious that $\sigma^* \tilde{\lambda} = \tilde{\lambda}$.\
\end{proof}

\section{On monodromy with respect to symmetries}
\label{sec:4}

\subsection{Symmetries of Hamiltonian systems}

Let $(M,\omega)$ be a $2n$ dimensional symplectic manifold and $H \in C^{\infty}(M,\mathbb{R})$ be an autonomous Hamiltonian function.\ A \textbf{symmetry} is a symplectic or anti-symplectic involution on $(M,\omega)$ which leaves $H$ invariant, meaning that $\sigma \colon M \to M$ is a diffeomorphism satisfying $\sigma^2 = \text{id}_M$, $\sigma^* \omega = \pm \omega$ and $H \circ \sigma = H$.\ If it is symplectic, then we call it a symplectic symmetry and otherwise an anti-symplectic symmetry.\

Let $\sigma$ be a symplectic symmetry.\ Then in view of definition (\ref{ham_vector_field}) of the Hamiltonian vector field $X_H$ we obtain
$$ \omega (X_H,\cdot) = \omega (X_{H \circ \sigma},\cdot) = d (H \circ \sigma) (\cdot) = \sigma^* \big( dH(\cdot) \big) = \sigma^* \big( \omega (X_H,\cdot) \big) = \underbrace{\sigma^* \omega}_{\omega} (\sigma^*X_H,\cdot). $$
Since $\omega$ is non-degenerate, $X_H$ is therefore invariant under $\sigma$, i.e.\
$$ \sigma ^* X_H = X_H. $$
In other words, its flow and $\sigma$ commute, i.e.\
\begin{align} \label{flow_sigma}
\varphi_H^t \circ \sigma = \sigma \circ \varphi_H^t.
\end{align}

Let $\rho$ be an anti-symplectic symmetry, then we have
$$ \rho ^* X_H  = - X_H,$$
which means that $X_H$ is anti-invariant under $\rho$.\ Equivalently,
\begin{align} \label{flow_rho}
\varphi_H^t \circ \rho = \rho \circ \varphi_H^{-t}.
\end{align}
A periodic orbit $x$ with first return time $T_x$ is called \textbf{symmetric with respect to $\rho$} if
$$ x(t) = \rho \big( x(-t) \big),\quad t \in [0,T_x]. $$
Note that $x(0) = x_0,x(T_x/2) \in \text{Fix}(\rho)$.\
\begin{remark}
	Consider the standard symplectic vector space $(\mathbb{R}^{2n},\omega_0)$ with
	$$ \omega_0 (v,w) : = \langle Jv,w \rangle = v^T J^T w = \langle v,J^Tw \rangle,\quad \text{ for all }v,w \in \mathbb{R}^{2n}, $$
	where
	$$ J = \begin{pmatrix}
	0 & I_n\\
	-I_n & 0
	\end{pmatrix} $$
	with respect to the splitting $\mathbb{R}^{2n} = \mathbb{R}^n \times \mathbb{R}^n$.\ Note that $J^2 = -I_{2n}$ and $J^T=J^{-1}=-J$.\ A linear isomorphism $\Psi$ of $(\mathbb{R}^{2n},\omega_0)$ is called symplectic if $ \omega_0 (\Psi v, \Psi w) = \omega_0 (v,w), \text{ for all }v,w \in \mathbb{R}^{2n},$ which is equivalent to $\Psi^T J \Psi = J$.\ The set of symplectic matrices in $\mathbb{R}^{2n}$ is denoted by
	$$ \text{Sp}(n) = \{ \Psi \colon (\mathbb{R}^{2n},\omega_0) \to (\mathbb{R}^{2n},\omega_0) \text{ linear isomorphism } | \text{ } \Psi^T J \Psi = J  \}. $$
	It is easy to show that if $\Psi,\Phi \in \text{Sp}(n)$, then $\Psi \Phi, \Psi^{-1}$, $\Psi^T \in \text{Sp}(n)$ and also $J \in \text{Sp}(n)$.\ In particular, $\text{Sp}(n)$ is a group under matrix multiplication.\ Moreover, a $2n \times 2n$ matrix which is written as
	\begin{align} \label{matrix_block}
	\begin{pmatrix}
	A & B\\
	C & D
	\end{pmatrix}
	\end{align}
	with respect to the splitting $\mathbb{R}^{2n} = \mathbb{R}^n \times \mathbb{R}^n$, is symplectic if and only if
	\begin{align} \label{matrix_symplectic}
	A^T C , B^T D \text{ are symmetric and } A^T D - C^T B = I_n.
	\end{align}
	Its inverse is given by
	\begin{align*}
	\begin{pmatrix}
	D^T & -B^T\\
	-C^T & A^T
	\end{pmatrix}.
	\end{align*}
	
	The set of anti-symplectic matrices in $\mathbb{R}^{2n}$ we denote by
	$$ \text{Sp}^-(n) = \{ \Psi \colon (\mathbb{R}^{2n},\omega_0) \to (\mathbb{R}^{2n},\omega_0) \text{ linear isomorphism } | \text{ } \Psi^T J \Psi = - J  \}, $$
	which is not a group, since for $\Psi,\Phi \in \text{Sp}^-(n)$ the multiplication $\Psi \Phi$ is symplectic.\ Nevertheless $\Psi^{-1},\Psi^T \in \text{Sp}^-(n)$ and $-J \in \text{Sp}^-(n)$.\ A $2n \times 2n$ matrix given in the block form (\ref{matrix_block}) is anti-symplectic if and only if
	\begin{align} \label{matrix_anti_symplectic}
	A^T C , B^T D \text{ are symmetric and } A^T D - C^T B = -I_n.
	\end{align}
	The inverse matrix is given by
	$$\begin{pmatrix}
	-D^T & B^T\\
	C^T & -A^T
	\end{pmatrix}.$$
\end{remark}

\subsection{Monodromy with respect to symplectic symmetries}
\label{sec:2.2}

Let $\sigma$ be a symplectic symmetry and $x$ be a periodic orbit with $x_0 \in \text{Fix}(\sigma)$ and first return time $T_x$.\ In particular, the symplectic vector space splits into two symplectic vector spaces
$$ T_{x_0} M = E_1\big( d \sigma (x_0) \big) \oplus E_{-1}\big( d \sigma (x_0) \big) = T_{x_0}\text{Fix}(\sigma) \oplus E_{-1}\big( d \sigma (x_0) \big), $$
since they are symplectically orthogonal.\ Moreover, because the Hamiltonian flow and $\sigma$ commute (see (\ref{flow_sigma})), the monodromy
$$ d \varphi_H^{T_x} (x_0) \colon T_{x_0} M \to T_{x_0}M $$
leaves the two eigenspaces invariant, i.e.\ if $\xi \in E_{\pm 1} \big( d \sigma(x_0) \big)$ then $d \varphi_H^{T_x}(x_0) \xi \in E_{\pm 1} \big( d \sigma(x_0) \big)$.\ Therefore the monodromy is of the form
$$ d\varphi _H ^{T_x} (x_0) = \begin{pmatrix}
A_1 & 0\\
0 & A_2
\end{pmatrix}, $$
where
$$ A_1 \colon T_{x_0} \text{Fix}(\sigma) \to T_{x_0} \text{Fix}(\sigma) ,\quad A_2 \colon E_{-1}\big( d\sigma(x_0) \big) \to E_{-1}\big( d\sigma(x_0) \big) $$
are symplectic matrices.\ Recall from (\ref{ker_invariant}) that $X_H(x_0)$ is an eigenvector of the monodromy to the eigenvalue 1, hence
$$ X_H(x_0) \in T_{x_0}\text{Fix}(\sigma) .$$
On energy level sets $\Sigma$, which are orientable by (\ref{orientable}), $\sigma|_{\Sigma}$ preserves the orientation of the line bundle ker$\omega | _{\Sigma} = \langle X_H | _{\Sigma} \rangle$.\ Note that the triple $(\Sigma, \omega | _{\Sigma}, \sigma | _{\Sigma})$ is an involutive Hamiltonian manifold and
$$ \text{dim}\big( T_{x_0}\text{Fix}(\sigma) \cap T_{x_0} \Sigma \big) = \text{dim}\big(\text{Fix}(\sigma)\big) - 1,\quad E_{-1}\big( d\sigma | _{\Sigma}(x_0) \big) = E_{-1}\big( d\sigma(x_0) \big). $$
Therefore by the Symplectic Splitting Theorem \ref{theorem_splitting}, the $2n-2$ dimensional symplectic vector space $T_x\Sigma / (\text{ker}\omega_{x_0} | _{T_{x_0}\Sigma})$ splits into two symplectic vector spaces, namely
$$ T_{x_0}\Sigma / (\text{ker}\omega_{x_0} | _{T_{x_0}\Sigma}) = \big( T_{x_0}\text{Fix}(\sigma | _{\Sigma}) / (\text{ker}\omega_{x_0} | _{T_{x_0}\Sigma}) \big) \oplus \Big( E_{-1}\big( d\sigma(x_0) \big) \Big) $$
and the reduced monodromy is of the form
$$ \overline{d\varphi _H ^{T_x} | _{\Sigma} (x_0)} = \begin{pmatrix}
\overline{A}_1 & 0\\
0 & A_2
\end{pmatrix}, $$
where $\overline{A}_1$ is a linear symplectic map on $T_{x_0} \text{Fix}(\sigma|_{\Sigma}) / (\text{ker}\omega_{x_0} | _{T_{x_0}\Sigma})$.\
\begin{remark}
	Note that a periodic orbit starting at the fixed point set of $\sigma$ stays there forever and the Floquet multipliers are given by those of $\overline{A}_1$ and $A_2$.\
\end{remark}

\subsection{Monodromy with respect to anti-symplectic symmetries}
\label{sec:4.3}
\subsubsection{Monodromy}
\label{sec:4.3.1}
Let $\rho$ be an anti-symplectic symmetry and $x$ be a periodic orbit with first return time $T_x$ which is symmetric with respect to $\rho$.\ Then the differential
$$d \rho (x_0) \colon T_{x_0}M \to T_{x_0}M $$
is a linear anti-symplectic involution on the symplectic vector space $(T_{x_0}M,\omega_{x_0})$.\
\begin{lemma1}
	The decomposition
	\begin{align} \label{lagrangian_splitting}
	T_{x_0}M = T_{x_0}\text{Fix}(\rho) \oplus E_{-1}\big( d \rho(x_0) \big)
	\end{align}
	is a Lagrangian splitting, meaning that the two eigenspaces are Lagrangian submanifolds, i.e.\ their respective dimensions are $\frac{1}{2} \text{dim} (M) = n$ and $\omega$ vanishes on the respective tangent bundles.\
\end{lemma1}
\begin{proof}
Their dimensions are $n$ in view of the isomorphism
$$ E_1 \big( d\rho(x_0) \big) = T_{x_0}\text{Fix}(\rho) \to E_{-1} \big( d\rho(x_0) \big),\quad \xi + d \rho (x_0) \xi \mapsto \xi - d \rho (x_0) \xi $$
and the decomposition
$$T_{x_0}M \to T_{x_0}\text{Fix}(\rho) \oplus E_{-1}\big( d \rho(x_0) \big),\quad \xi \mapsto \frac{1}{2} \big( \xi + d \rho (x_0) \xi \big) + \frac{1}{2} \big( \xi - d \rho (x_0) \xi \big). $$
Moreover, for all $\xi, \eta \in T_{x_0}\text{Fix}(\rho)$ we have
$$ \omega_{x_0} (\xi,\eta) = - (\rho^* \omega)_{x_0} (\xi,\eta) = - \omega_{\rho(x_0)} \big( d\rho(x_0) \xi, d\rho(x_0) \eta \big) = - \omega_{x_0} (\xi,\eta), $$
hence $\omega_{x_0} (\xi,\eta)=0$.\ The same holds for all $\xi, \eta \in E_{-1}\big( d \rho(x_0) \big)$.\
\end{proof}
The existence of a symplectic (or canonical) basis for a symplectic vector space is given by a skew-symmetric version of the Gram--Schmidt process (see for instance \cite[pp.\ 3--4]{hofer}).\ Similarly, in view of the Lagrangian splitting (\ref{lagrangian_splitting}), there exist bases
$$(v_1,...,v_n) \text{ of } T_{x_0}\text{Fix}(\rho),\quad (w_1,...,w_n) \text{ of } E_{-1}\big( d \rho(x_0) \big)$$
such that
$$ \omega_{x_0}(v_i,w_j) = \delta_{ij},\quad i,j=1,...,n $$
and
\begin{align} \label{basis_symplectic}
v_1,...,v_n,w_1,...,w_n
\end{align}
is a symplectic basis of $T_{x_0}M$ (see for instance \cite[pp.\ 532--533]{albers_frauenfelder}).\ We refer to this kind of basis as \textbf{Lagrangian basis}.\ With respect to this basis the differential $d \rho (x_0)$ is represented by the standard anti-symplectic involution
$$ \rho_0 = \begin{pmatrix}
I_n & 0\\
0 & -I_n
\end{pmatrix}.$$
\begin{Proposition}\label{prop_block}
	The monodromy written in block form (\ref{matrix_block}) has the form
	\begin{align} \label{matrix_block_2}
	d \varphi_H^{T_x} (x_0) = \begin{pmatrix}
	A & B\\
	C & A^T
	\end{pmatrix},
	\end{align}
	where
	\begin{align} \label{matrix_block_3}
	B, C, CA, AB \text{ are symmetric and } A^2 - BC = I_n.
	\end{align}
\end{Proposition}
\begin{proof}
In view of (\ref{flow_rho}) we get
\begin{align} \label{equation_monodromy}
d \rho (x_0) \circ d \varphi_H^{T_x} (x_0) \circ d \rho (x_0) = \big( d \varphi_H^{T_x} (x_0) \big)^{-1},
\end{align}
which means that the monodromy is conjugated to its inverse by the linear anti-symplectic involution.\ With respect to the choice of a Lagrangian basis (\ref{basis_symplectic}) and the monodromy written in block form (\ref{matrix_block}), the equation (\ref{equation_monodromy}) becomes
$$ \begin{pmatrix}
I_n & 0\\
0 & -I_n
\end{pmatrix} \begin{pmatrix}
A & B\\
C & D
\end{pmatrix} \begin{pmatrix}
I_n & 0\\
0 & -I_n
\end{pmatrix} = \begin{pmatrix}
D^T & -B^T\\
-C^T & A^T
\end{pmatrix}, $$
which is equivalent to
$$ \begin{pmatrix}
A & -B\\
-C & D
\end{pmatrix} = \begin{pmatrix}
D^T & -B^T\\
-C^T & A^T
\end{pmatrix} $$
and the statement follows by (\ref{matrix_symplectic}).\
\end{proof}
\begin{lemma1}
	All linear anti-symplectic involutions are symplectically conjugated to each other.\
\end{lemma1}
\begin{proof}
	Let $\rho_1, \rho_2 \in Sp^{-}(n)$ and $\{v_1,...,v_n,w_1,...w_n\}$ and $\{\tilde{v}_1,...,\tilde{v}_n,\tilde{w}_1,...\tilde{w}_n\}$ be Lagrangian bases, respectively.\ Then the basis change is given by the linear symplectic map
	$$ v_i \mapsto \tilde{v}_i,\quad w_i \mapsto \tilde{w}_i,\quad i=1,...,n. $$
\end{proof}
\begin{corollary1} \label{corollary_involution}
	Every $\rho \in Sp^{-}(n)$ is symplectically conjugated to the standard anti-symplectic involution $\rho_0$.\
\end{corollary1}
We denote the set of symplectic matrices of the form (\ref{matrix_block_2}) satisfying (\ref{matrix_block_3}) by
$$ \text{Sp}^{\rho_0}(n) = \left\{ \begin{pmatrix}
A & B\\
C & A^T
\end{pmatrix} \colon B, C, CA, AB \text{ are symmetric and } A^2 - BC = I_n \right\}. $$
The next lemma generalizes Proposition \ref{prop_block}.\
\begin{lemma1}
	Every symplectic matrix $\Psi \in \text{Sp}(n)$ is symplectically conjugated to a symplectic matrix from $\text{Sp}^{\rho_0}(n)$.\
\end{lemma1}
\begin{proof}
	By a thereom of Wonenburger \cite{wonenburger}, every $\Psi \in \text{Sp}(n)$ is the product of two linear anti-symplectic involutions, i.e.\
	$$ \Psi = \rho_1 \rho_2, $$
	where $\rho_1,\rho_2 \in \text{Sp}^-(n)$.\ Hence
	$$ \Psi^{-1} = \rho_2 \rho_1 = \rho_1 \Psi \rho_1, $$
	which is a general form of the equation (\ref{equation_monodromy}).\ By Corollary \ref{corollary_involution}, there exists $\Psi_1 \in \text{Sp}(n)$ such that
	$$ \Psi_1^{-1} \rho_1 \Psi_1 = \rho_0. $$
	This implies
	$$ \Psi_1^{-1} \Psi \Psi_1 = \Psi_1^{-1} (\rho_1 \rho_2 ) \Psi_1 = \Psi_1^{-1} (\Psi_1 \rho_0 \Psi_1^{-1} ) \rho_2 \Psi_1 = \rho_0 \Psi_1^{-1} \rho_2 \Psi_1. $$
	One readily checks that $\Psi_1^{-1} \rho_2 \Psi_1$ belongs to $\text{Sp}^-(n)$, thus $\Psi$ is symplectically conjugated to the product of two linear anti-symplectic involutions as well.\ In addition,
	$$ (\Psi_1^{-1} \Psi \Psi_1)^{-1} = \Psi_1^{-1} \rho_2 \Psi_1 \rho_0 = \rho_0 ( \Psi_1^{-1} \Psi \Psi_1 ) \rho_0, $$
	i.e.\ the symplectic conjugacy $\Psi_1^{-1} \Psi \Psi_1$ is conjugated to its inverse by the standard anti-symplectic involution.\ Therefore by the same steps as in the proof of Proposition \ref{prop_block}, the symplectic conjugacy $\Psi_1^{-1} \Psi \Psi_1$ is an element from $\text{Sp}^{\rho_0}(n)$.\
\end{proof}
\begin{remark} \label{remark4.8}
	Since $X_H$ is anti-invariant under $\rho$, we have
	$$ X_H(x_0) \in E_{-1} \big( d \rho(x_0) \big). $$ 
	On energy level sets $\Sigma$, the restriction $ T_{x_0} \text{Fix} (\rho | _{\Sigma}) $ is $n-1$ dimensional, and on the quotient by the line bundle ker$\omega | _{\Sigma} = \langle X_H | _{\Sigma} \rangle$ we obtain the splitting
	$$ T_{x_0}\Sigma / (\text{ker}\omega_{x_0} | _{T_{x_0}\Sigma}) = \big( T_{x_0} \text{Fix}(\rho | _{\Sigma}) \big) \oplus \Big( E_{-1} \big( d\rho|_{\Sigma} (x_0) \big) / ( \text{ker}\omega_{x_0} | _{T_{x_0}\Sigma} )  \Big). $$
	By using the Lagrangian basis $v_1,...,v_n,w_1,...,w_n$ (see (\ref{basis_symplectic})), $n-1$ basis vectors of the first vector space are determined by $v_1,...,v_n$ and the energy condition.\ We denote them by $\tilde{v}_1,...,\tilde{v}_{n-1}$.\ By the Steinitz exchange lemma on $w_1,...,w_n$ and $X_H|_{\Sigma}(x_0)$ we choose $\tilde{w}_1,...,\tilde{w}_{n-1}$ such that
	$$ \omega_{x_0}(\tilde{v}_i,\tilde{w}_j) = \delta_{ij},\quad i,j=1,...,n-1 $$
	and
	$$ T_{x_0} \Sigma = \langle \tilde{v}_1,...,\tilde{v}_{n-1} \rangle_{\mathbb{R}} \oplus \langle \tilde{w}_1,...,\tilde{w}_{n-1}, X_H|_{\Sigma}(x_0) \rangle_{\mathbb{R}}. $$
	With respect to this basis, the reduced monodromy is an element from $\text{Sp}^{\rho_0}(n-1)$.\
\end{remark}

\subsubsection{The signatures of a symmetric periodic orbit}

Let the monodromy be a symplectic matrix from $\text{Sp}^{\rho_0}(n)$, then the following lemma shows that its spectrum is determined by the spectrum of $A$.\
\begin{lemma1}
	The characteristic polynomial of the monodromy equals
	$$ \lambda^n \det \big( -2A - (- \lambda - \frac{1}{\lambda})I_n \big), $$
	i.e.\ $\lambda^n \chi_{-2A} (- \lambda - \frac{1}{\lambda})$.
\end{lemma1}
\begin{proof}
	The decomposition
	$$ \begin{pmatrix}
	A - \lambda I_n & B\\
	C & A^T - \lambda I_n
	\end{pmatrix} \begin{pmatrix}
	A - \lambda I_n & 0\\
	-C & I_n
	\end{pmatrix} = \begin{pmatrix}
	\lambda^2 I_n - 2 \lambda A + I_n & B\\
	0 & A^T - \lambda I_n
	\end{pmatrix} $$
	implies that the characteristic polynomial of the monodromy is given by
	$$ \det ( \lambda^2 I_n - 2 \lambda A + I_n ), $$
	which is equivalent to $\lambda^n \det \big( -2A - (- \lambda - \frac{1}{\lambda})I_n \big)$.
\end{proof}
\begin{remark}
	In the case $n=1$, i.e.\
	$$ d \varphi_H^{T_x} (x_0) =  \begin{pmatrix}
	a & b\\
	c & a
	\end{pmatrix},\quad a^2 - bc = 1, $$
	where $a,b,c \in \mathbb{R}$, we have for its characteristic polynomial,
	$$ \lambda^2 - 2a \lambda + 1. $$
	Note that $a$ is the half of its trace.\ For the case $n=2$ we obtain
	$$ \lambda^4 - 2 \text{tr}A \lambda^3 + (2 + 4 \det A)\lambda^2 - 2 \text{tr}A \lambda + 1. $$
\end{remark}
\begin{definition}
	Let $\lambda$ be an eigenvalue of $A$ of multiplicity $1$, $v$ an eigenvector of $A$ and $\tilde{v}$ an eigenvector of $A^T$ to the eigenvalue $\lambda$.\ The \textbf{signature with respect to $C$} of a symmetric periodic orbit is defined as the signature of $v^TCv$ and the \textbf{signature with respect to $B$} as the signature of $\tilde{v}^T B \tilde{v}$.\ We denote them by
	$$\text{sign}_C (\lambda) = \text{sign} (v^TCv), \quad \text{sign}_B (\lambda) = \text{sign}(\tilde{v}^T B \tilde{v}),$$
	respectively.\
\end{definition}
\begin{remark}
	Neither of the signatures depends on the resp.\ eigenvectors, since for a constant $k \in \mathbb{R}^*$,
	$$ \text{sign}\big( (kv)^T C (kv) \big) = k^2 \text{sign}(v^T C v) = \text{sign} (v^T C v) = \text{sign}_C(\lambda). $$
\end{remark}
\begin{remark} \label{remark4.13}
	If a Lagrangian basis (\ref{basis_symplectic}) is given, then $R \in \text{GL}(n,\mathbb{R})$ acts $\text{Sp}^{\rho_0}(n)$ by conjugation
	$$ \begin{pmatrix}
	R & 0\\
	0 & (R^{-1})^T
	\end{pmatrix} \begin{pmatrix}
	A & B\\
	C & A^T
	\end{pmatrix} \begin{pmatrix}
	R^{-1} & 0\\
	0 & R^T
	\end{pmatrix} = \begin{pmatrix}
	RAR^{-1} & RBR^T\\
	(R^{-1})^TCR^{-1} & (R^{-1})^TA^TR^T
	\end{pmatrix}, $$
	where $(R^{-1})^TA^TR^T = (RAR^{-1})^T$.\ Moreover, this basis change is symplectic, since
	$$ \begin{pmatrix}
	R^T & 0\\
	0 & R^{-1}
	\end{pmatrix} \begin{pmatrix}
	0 & I_n\\
	-I_n & 0
	\end{pmatrix} \begin{pmatrix}
	R & 0\\
	0 & (R^{-1})^T
	\end{pmatrix} = \begin{pmatrix}
	0 & I_n\\
	-I_n & 0
	\end{pmatrix}. $$
\end{remark}
\begin{Proposition} \label{prop_signature}
	Both signatures of a symmetric periodic orbit are invariant under the Lagrangian basis change, i.e.\ they are independent of the choice of the Lagrangian basis.\
\end{Proposition}
\begin{proof}
	For the invariance of $\text{sign}_C(\lambda)$ we consider the identity
	$$ AR^{-1}R v = Av = \lambda v, $$
	thus
	$$ RAR^{-1}Rv = \lambda Rv, $$
	meaning that $Rv$ is an eigenvector of $RAR^{-1}$ to the eigenvalue $\lambda$.\ In view of the transformations
	$$ C \mapsto (R^{-1})^TCR^{-1} ,\quad A \mapsto RAR^{-1},\quad v \mapsto Rv, $$
	we have
	$$ \text{sign} \big( (Rv)^T (R^{-1})^T C R^{-1} (Rv) \big) = \text{sign} (v^T C v) = \text{sign}_C(\lambda). $$
	For $\text{sign}_B(\lambda)$ we consider
	$$ A^T R^T (R^{-1})^T \tilde{v} = A^T \tilde{v} =  \lambda \tilde{v} $$
	and therefore
	$$ (R^{-1})^T A^T R^T (R^{-1})^T \tilde{v} =  \lambda (R^{-1})^T \tilde{v} ,$$
	which means that $(R^{-1})^T \tilde{v}$ is an eigenvector of $(R^{-1})^TA^TR^T$ to the eigenvalue $\lambda$.\ In view of
	$$ B \mapsto RBR^T,\quad  A^T \mapsto  (R^T)^{-1}A^TR^T,\quad \tilde{v} \mapsto (R^{-1})^T \tilde{v},$$
	we obtain
	$$ \text{sign} \Big( \big( (R^{-1})^T \tilde{v} \big)^T RBR^T (R^{-1})^T \tilde{v} \Big) = \text{sign} (\tilde{v}^T B \tilde{v}) = \text{sign}_B(\lambda).$$
\end{proof}
\begin{remark} \label{remark4.15}
	In the case $n=1$, scaling on the Lagrangian yields
	$$ \begin{pmatrix}
	k & 0\\
	0 & \frac{1}{k}
	\end{pmatrix} \begin{pmatrix}
	a & b\\
	c & a
	\end{pmatrix} \begin{pmatrix}
	\frac{1}{k} & 0\\
	0 & k
	\end{pmatrix} = \begin{pmatrix}
	a & k^2 b\\
	\frac{1}{k^2}c & a
	\end{pmatrix},\quad k \in \mathbb{R}^*, $$
	hence the trace is invariant as well under conjugation.\
\end{remark}

\subsection{Monodromy if the symplectic \& anti-symplectic symmetries commute}
\label{sec:4.4}

Let $\sigma$ be a symplectic and $\rho$ be an anti-symplectic symmetry such that
$$ \sigma \circ \rho = \rho \circ \sigma. $$
Moreover, let $x$ be a periodic orbit with $x_0 \in \text{Fix}(\sigma)$ and first return time $T_x$, which is symmetric with respect to $\rho$, hence $x_0 \in \text{Fix}(\rho)$.\ Consider the symplectic decomposition
$$ T_{x_0} M = T_{x_0}\text{Fix}(\sigma) \oplus E_{-1} \big( d \sigma(x_0)\big) $$
and let
$$ \text{dim} \big(  T_{x_0}\text{Fix}(\sigma) \big) = 2k,\quad \text{dim} \Big( E_{-1} \big( d \sigma(x_0)\big) \Big) = 2 \tilde{k},\quad 2k + 2\tilde{k} = 2n. $$
\begin{lemma1}
	The monodromy is of the form
	$$ d\varphi _H ^{T_x} (x_0) = \begin{pmatrix}
	A_1 & 0\\
	0 & A_2
	\end{pmatrix}, $$
	where
	$$ A_1 \colon T_{x_0} \text{Fix}(\sigma) \to T_{x_0} \text{Fix}(\sigma),\quad A_2 \colon E_{-1}\big( d\sigma(x_0) \big) \to E_{-1}\big( d\sigma(x_0) \big) $$
	and
	$$ A_1 \in \text{Sp}^{\rho_0}(k),\quad A_2 \in \text{Sp}^{\rho_0}(\tilde{k}) .$$
\end{lemma1}
\begin{proof}
Since $\sigma$ and $\rho$ commute, the linear anti-symplectic involution
$$ d \rho (x_0) \colon T_{x_0} M \to T_{x_0} M $$
leaves the symplectic decomposition
$$ T_{x_0} M = T_{x_0}\text{Fix}(\sigma) \oplus E_{-1} \big( d \sigma(x_0)\big) $$
invariant, meaning that for $\xi \in E_{\pm 1} \big( d \sigma(x_0) \big)$ we have $d \rho(x_0) \xi \in E_{\pm 1} \big( d \sigma(x_0) \big)$.\ If we denote
$$ E_1^{d\sigma}:= T_{x_0}\text{Fix}(\sigma),\quad E_{-1}^{d\sigma}:= E_{-1}\big( d\sigma(x_0)\big), $$
this invariance implies that the restrictions
$$ d\rho|_{E_1^{d\sigma}} (x_0) \colon E_1^{d\sigma} \to E_1^{d\sigma},\quad d\rho|_{E_{-1}^{d\sigma}} (x_0) \colon E_{-1}^{d\sigma} \to E_{-1}^{d\sigma} $$
are linear anti-symplectic involutions.\ Hence the symplectic decomposition splits into two Lagrangian splittings,
\begin{align*}
T_{x_0}M &= E_1^{d\sigma} \oplus E_{-1}^{d\sigma}\\
&= \Big(E_1 \big( d\rho|_{E_1^{d\sigma}} (x_0) \big) \oplus E_{-1} \big( d\rho|_{E_1^{d\sigma}} (x_0)\big)\Big) \oplus \Big(E_1 \big( d\rho|_{E_{-1}^{d\sigma}} (x_0) \big) \oplus E_{-1} \big( d\rho|_{E_{-1}^{d\sigma}} (x_0)\big)\Big),
\end{align*}
with Lagrangian bases
$$ v_1,...,v_k,w_1,...,w_k,\quad \tilde{v}_1,...,\tilde{v}_{\tilde{k}},\tilde{w}_1,...,\tilde{w}_{\tilde{k}}, $$
respectively.\ In view of Section \ref{sec:4.3.1}, this proves the lemma.\
\end{proof}
\begin{remark}
The two Lagrangian bases from the proof give a Lagrangian basis
$$ v_1,...,v_k,\tilde{v}_1,...,\tilde{v}_{\tilde{k}},w_1,...,w_k,\tilde{w}_1,...,\tilde{w}_{\tilde{k}} $$
on $T_{x_0}M$ with respect to the Lagrangian splitting and decomposition
\begin{align*}
T_{x_0}M &= T_{x_0}\text{Fix}(\rho) \oplus E_{-1}\big(d\rho(x_0)\big)\\
&= \Big(E_1 \big( d\rho|_{E_1^{d\sigma}} (x_0) \big) \oplus E_{1} \big( d\rho|_{E_{-1}^{d\sigma}} (x_0)\big)\Big) \oplus \Big(E_{-1} \big( d\rho|_{E_1^{d\sigma}} (x_0) \big) \oplus E_{-1} \big( d\rho|_{E_{-1}^{d\sigma}} (x_0)\big)\Big).
\end{align*}
\end{remark}
\begin{remark}
In view of Remark \ref{remark4.8}, the reduced monodromy is of the form
$$ \overline{d\varphi _H ^{T_x} | _{\Sigma} (x_0)} = \begin{pmatrix}
\overline{A}_1 & 0\\
0 & A_2
\end{pmatrix}, $$
where
$$\overline{A}_1 \colon T_{x_0} \text{Fix}(\sigma|_{\Sigma}) / (\text{ker}\omega_{x_0} | _{T_{x_0}\Sigma}) \to T_{x_0} \text{Fix}(\sigma|_{\Sigma}) / (\text{ker}\omega_{x_0} | _{T_{x_0}\Sigma}),\quad \overline{A}_1 \in \text{Sp}^{\rho_0}(k-1). $$
\end{remark}

\subsection{GIT quotient in dimension two}
\label{sec:4.5}

Let $G$ be a Lie group, i.e.\ $G$ is a group and a smooth manifold such that
$$ G \to G,\quad g \mapsto g^{-1},\quad \quad G \times G \to G,\quad (g,h) \mapsto gh $$
are smooth.\ Let $G$ act on a manifold $M$, meaning that there is a group homomorphism $ \psi \colon G \to \text{Diff}(M)$ such that
$$ G \times M \to M,\quad (g,m) \mapsto \psi(g)(m) =: g_* m $$
is smooth.\ For any $m \in M$ the orbit through $m$ is the set $Gm = \{g_* m \mid g \in G\}$.\ If $m$ and $n$ lie in the same orbit, then $Gm=Gn$.\ Moreover, $M$ can be written as the disjoint union of orbits and the space of orbits is the quotient space $M/G$ which is in general not a Hausdorff space.\ To ensure the Hausdorff property we consider the orbit closure relation on $M$ which is defined by
$$ m \sim n \quad : \Leftrightarrow \quad \overline{Gm} \cap \overline{Gn} \neq \emptyset, $$
meaning that $m$ is related to $n$ if the closure of the orbits through $m$ and $n$ intersect.\ It is clear that this relation is reflexive and symmetric, but it is not necessarily transitive.\ If it is an equivalence relation, then the \textbf{GIT quotient} (geometric invariant theory quotient) is defined as
$$ M /\!/ G := M / \sim. $$
\begin{example}
	Let $\mathbb{R}_{>0}$ acting on $\mathbb{R}$ by multiplication, then there are exactly the three orbits $  \mathbb{R}^-, \{0\}$ and $\mathbb{R}^+$.\ The open sets in the orbit space $\mathbb{R} / \mathbb{R}_{>0}$ are $ \{ \mathbb{R}^- \}, \{\mathbb{R}^+\}, \{ \mathbb{R}^-, \{ 0\}, \mathbb{R}^+\}$ and $\{\emptyset\},$
	hence it is not Hausdorff.\ By the closure of the orbits $ \overline{\{0\}} = \{0\}, \overline{\mathbb{R}^-} = (-\infty,0]$ and $\overline{\mathbb{R}^+} = [0,\infty)$ we see that the GIT quotient is
	$$ \mathbb{R} /\!/ \mathbb{R}_{>0} = \{\text{pt}.\} .$$
\end{example}
\begin{example}
	Let $\text{GL}(n,\mathbb{R})$ act on $\text{Mat}(n,\mathbb{R})$ by conjugation.\ The GIT quotient avoids Jordan factors, therefore two matrices $A,B$ are equivalent if and only if their characteristic polynomials are the same (see \cite[Appendix A]{frauenfelder_moreno} for details).\ For a matrix $A$ let $\chi_A(\lambda) = \lambda^n + a_{n-1}\lambda^{n-1} + ... + a_0$ be its characteristic polynomial.\ Then a homeomorphism is given by
	$$ \text{Mat}(n,\mathbb{R}) /\!/  \text{GL}(n,\mathbb{R}) \to \mathbb{R}^n,\quad [A] \mapsto (a_{n-1},...,a_0). $$
\end{example}
\begin{example}
For this paper the relevant GIT quotient is
$$ \text{Sp}^{\rho_0} (1) /\!/ \text{GL}(1,\mathbb{R}), $$
which is well studied in \cite[pp.\ 25--28]{zhou}.\ In particular, this space is important for the study of periodic orbits whose reduced monodromy is an element from $\text{Sp}^{\rho_0}(1)$.\ Let $A \in  \text{Sp}^{\rho_0} (1)$, and recall from Subsection \ref{sec:stability_floquet} that the eigenvalues in the elliptic and hyperbolic case are resp.\ given by
$$ a \pm \text{i} \sqrt{(1-a^2)},\quad a \pm \sqrt{(a^2 - 1)}. $$ Each of the positive and negative hyperbolic cases consists of two subcases, namely
\begin{center}
\begin{tabular}{cccc}\centering
	pos. hyperb. I & pos. hyperb. II & neg. hyperb. I & neg. hyperb. II\\
	$\begin{pmatrix}
	a > 1 & b < 0\\
	c < 0 & a > 1
	\end{pmatrix}$ & $\begin{pmatrix}
	a > 1 & b > 0\\
	c > 0 & a > 1
	\end{pmatrix}$ & $\begin{pmatrix}
	a < -1 & b > 0\\
	c > 0 & a < -1
	\end{pmatrix}$ & $\begin{pmatrix}
	a < -1 & b < 0\\
	c < 0 & a < -1
	\end{pmatrix}$
\end{tabular}
\end{center}
Furthermore, recall from Remark \ref{remark4.15} that the action is given by
$$ k_* A :=  \begin{pmatrix}
a & k^2 b\\
\frac{1}{k^2}c & a
\end{pmatrix},\quad k \in \mathbb{R}^*. $$
Note that $A_1, A_2 \in \text{Sp}^{\rho_0} (1)$ are equivalent in the GIT quotient if and only if $ \overline{k_* A_1} \cap \overline{k_* A_2} \neq \emptyset $.\

If 1 is not an eigenvalue and $b\neq0,c\neq0$, then one can always choose $k$ such that $k^2 b = \pm \frac{1}{k^2}c$.\

If $A$ is elliptic, then $A$ is equivalent to a rotation of $\mathbb{R}^2$.\

If $A$ is hyperbolic, then $A$ is equivalent to
$$ \begin{pmatrix}
a & \pm \sqrt{(a^2 - 1)}\\
\pm \sqrt{(a^2 - 1)} & a
\end{pmatrix},\quad a>1 \text{ or } a<-1 .$$

Obviously, the identity matrix lies in the closure of the orbits of
$$ \begin{pmatrix}
1 & \pm b\\
0 & 1
\end{pmatrix},\quad \begin{pmatrix}
1 & 0\\
\pm b & 1
\end{pmatrix},\quad b > 0, $$
hence these three matrices are equivalent in the GIT quotient and identified to the single point $\{ +1 \}$.\ The same holds for the matrices replacing 1 in the diagonals by $-1$, which we identify to the single point $\{ -1 \}$.\

Topologically, the GIT quotient $\text{Sp}^{\rho_0} (1) /\!/ \text{GL}(1,\mathbb{R})$ is isomorphic to a circle with four spikes (see Figure \ref{circle_four_spikes}).\ Geometrically, the unit circle $\{ z \in \ \mathbb{C} \mid |z| = 1 \} \setminus \{ \pm 1 \}$ corresponds to equivalence classes of elliptic matrices, and each spike minus $\{ \pm 1 \}$ represents a hyperbolic subcase.\
\begin{figure}[H]
	\centering
	\begin{tikzpicture}[line cap=round,line join=round,>=triangle 45,x=1.0cm,y=1.0cm]
	\clip(-8,-4) rectangle (8,4.5);
	
	\draw [line width=1.5pt,color=blue] (0,0) circle (1.5cm);
	\draw [->,line width=1pt] (-4,0) -- (4.5,0);
	\draw [->,line width=1pt] (0,-4) -- (0,4.5);
	
	\draw [fill=blue,color=blue] (1.5,0) circle (2pt);
	\draw [fill=blue,color=blue] (-1.5,0) circle (2pt);
	
	\draw (-1.7,3.2) node[anchor=north west,color=blue] {$\begin{pmatrix}
		\cos \theta & - \sin \theta\\
		\sin \theta & \cos \theta
		\end{pmatrix}$};
	
	\draw (-1.7,-1.6) node[anchor=north west,color=blue] {$\begin{pmatrix}
		\cos \theta & \sin \theta\\
		- \sin \theta & \cos \theta
		\end{pmatrix}$};
	
	\draw (3,3.5) node[anchor=north west,color=blue] {$\begin{pmatrix}
		\cosh x & \sinh x\\
		\sinh x & \cosh x
		\end{pmatrix}$};
	
	\draw (3.4,1.8) node[anchor=north west,color=blue] {$\text{pos. hyperb. II}$};
	
	\draw (-7.1,3.5) node[anchor=north west,color=blue] {$\begin{pmatrix}
		- \cosh x & \sinh x\\
		\sinh x & - \cosh x
		\end{pmatrix}$};
	
	\draw (3.4,-1.2) node[anchor=north west,color=blue] {$\text{pos. hyperb. I}$};
	
	\draw (-7.1,-1.9) node[anchor=north west,color=blue] {$\begin{pmatrix}
		- \cosh x & - \sinh x\\
		- \sinh x & - \cosh x
		\end{pmatrix}$};
	
	\draw (-6.5,1.8) node[anchor=north west,color=blue] {$\text{neg. hyperb. I}$};
	
	\draw (3,-1.9) node[anchor=north west,color=blue] {$\begin{pmatrix}
		\cosh x & - \sinh x\\
		- \sinh x & \cosh x
		\end{pmatrix}$};
	
	\draw (-6.5,-1.2) node[anchor=north west,color=blue] {$\text{neg. hyperb. II}$};
	
	\draw [line width=1.5pt,color=blue] (3,2.8284271247) .. controls (1,0) .. (3,-2.8284271247);
	\draw [line width=1.5pt,color=blue] (-3,2.8284271247) .. controls (-1,0) .. (-3,-2.8284271247);
	
	\end{tikzpicture}
	\caption{Topology of $\text{Sp}^{\rho_0} (1) /\!/ \text{GL}(1,\mathbb{R})$}
	\label{circle_four_spikes}
\end{figure}
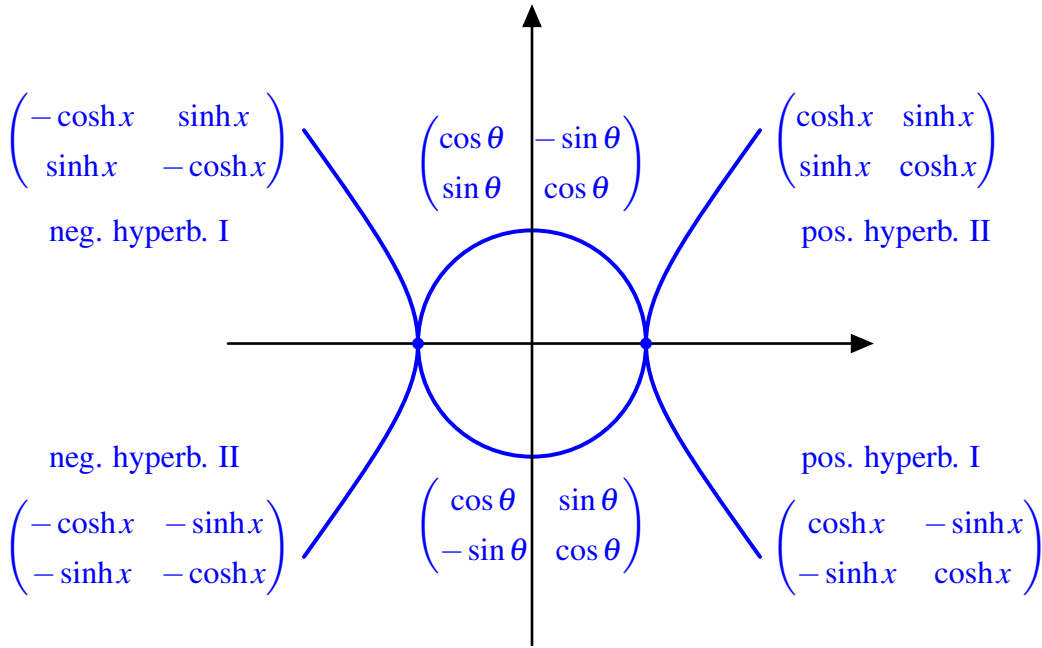
\noindent
Note that the eigenvalues of the hyperbolic matrices are $e^{\pm x}$ for the pos.\ hyperb.\ and $-e^{\pm x}$ for the neg.\ hyperb.\ cases, which equal the Floquet multipliers $\lambda$ and $1/\lambda$.\ Furthermore, in the elliptic case, if $b<0$, then the rotation is by $\theta \in (0,\pi)$ and if $b>0$, then it is by $-\theta$, so the rotation angle equals $2\pi - \theta \in (\pi,2\pi)$.\
\end{example}
\begin{example}
	Topologically, the GIT quotient $\text{Sp}(1) /\!/ \text{Sp}(1)$, where $\text{Sp}(1) = \text{SL}(2,\mathbb{R})$ acts on itself by conjugation, is isomorphic to a circle with two spikes (see Figure \ref{circle_two_spikes} and \cite[Section 10.5]{frauenfelder} for details).\ Geometrically, the unit circle $\{ z \in \ \mathbb{C} \mid |z| = 1 \} \setminus \{ \pm 1 \}$ corresponds to equivalence classes of elliptic matrices, the spike $\{ r \in \mathbb{R} \mid r > 1 \}$ represents the positive hyperbolic case, the spike $\{ r \in \mathbb{R} \mid r < -1 \}$ represents the negative hyperbolic case.\ Moreover, since
	$$ \begin{pmatrix}
	1 & 0\\
	0 & 1
	\end{pmatrix}, \begin{pmatrix}
	1 & 1\\
	0 & 1
	\end{pmatrix}, \begin{pmatrix}
	1 & -1\\
	0 & 1
	\end{pmatrix},\quad\quad \begin{pmatrix}
	-1 & 0\\
	0 & -1
	\end{pmatrix}, \begin{pmatrix}
	-1 & 1\\
	0 & -1
	\end{pmatrix}, \begin{pmatrix}
	-1 & -1\\
	0 & -1
	\end{pmatrix}, $$
	are resp.\ equivalent, the single point $\{ +1 \}$ corresponds to the first three Jordan forms and the single point $\{ -1 \}$ corresponds to the second three Jordan forms.\
	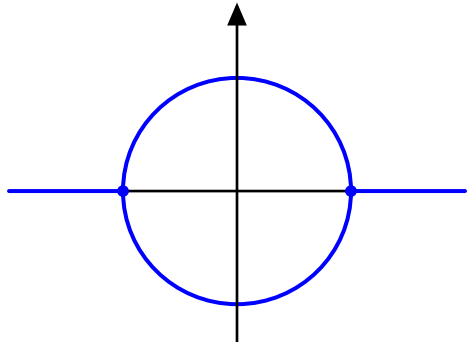
\begin{figure}[H]
		\centering
		\begin{tikzpicture}[line cap=round,line join=round,>=triangle 45,x=1.0cm,y=1.0cm]
		\clip(-8,-2) rectangle (8,2.5);
		
		\draw [line width=1.5pt,color=blue] (0,0) circle (1.5cm);
		\draw [line width=1.5pt,color=blue] (-3,0) -- (-1.5,0);
		\draw [line width=1pt] (-1.5,0) -- (1.5,0);
		\draw [line width=1.5pt,color=blue] (1.5,0) -- (3,0);
		\draw [->,line width=1pt] (0,-2) -- (0,2.5);
		\draw [fill=blue,color=blue] (1.5,0) circle (2pt);
		\draw [fill=blue,color=blue] (-1.5,0) circle (2pt);
		
		\end{tikzpicture}
		\caption{Topology of $\text{Sp}(1) /\!/ \text{Sp}(1)$}
		\label{circle_two_spikes}
	\end{figure}
\end{example}

\section{Planar and spatial Hill lunar problem} \label{sec:spatial_Hills_lunar}

\begin{center}
	\textit{``The pioneer in the search for periodic orbits in the restricted three-body problem was Hill}\\
	\textit{with his discovery of the retrograde and direct periodic orbit in Hill's lunar problem...}\\
	\textit{His motivation was to describe the motion of the moon."}
\end{center}
\begin{flushright}
	- Urs Frauenfelder and Otto van Koert \cite[p.\ 94]{frauenfelder}
\end{flushright}

\subsection{Short astronomical lunar overview and Hill's concept}

In astronomy, one distinguishes two kinds of periodic orbits, retrograde and direct ones.\ The sun rotates about its own axis.\ The planets circle around the sun in the same direction, and all of them rotate about their own axis in this direction, except Venus and Uranus, which rotate about their own axis in the other direction.\ Usually, moons, which are the companions of the planet, move around the planet in the way the planet circles around the sun.\ Such a periodic orbit of the moon is called a direct periodic orbit, and in the other case is a retrograde one, see Figure \ref{figure1}.\ For instance, the moon ``Triton" of Neptun is retrograde.\ Our moon is direct.
\begin{figure}[H]
	\centering
	\definecolor{qqqqff}{rgb}{0.0,0.0,1.0}
	\definecolor{uuuuuu}{rgb}{0.26666666666666666,0.26666666666666666,0.26666666666666666}
	\begin{tikzpicture}[line cap=round,line join=round,>=triangle 45,x=1.0cm,y=1.0cm]
	\clip(-1.5,-1.4) rectangle (10.0,1.4);
	
	\draw [fill=black,fill opacity=0.3, line width=1pt] (-0.4,0.0) circle (1.0cm);
	
	\draw [fill=black,fill opacity=0.3,line width=1pt] (4.0,0.0) circle (0.3cm);
	
	\draw [shift={(0.0-0.4,0.0)},line width=1pt] [decoration={markings, mark=at position 0.8 with {\arrow{>}}}, postaction={decorate}] plot[domain=-0.6186443509247468:0.7687330396835076,variable=\t]({1.0*1.4484474446799924*cos(\t r)+-0.0*1.4484474446799924*sin(\t r)},{0.0*1.4484474446799924*cos(\t r)+1.0*1.4484474446799924*sin(\t r)});
	
	\draw [shift={(0,0.0)},line width=1pt] [decoration={markings, mark=at position 0.18 with {\arrow{>}}}, postaction={decorate}] plot[domain=0.07500383243652106:0.6998928697192437,variable=\t]({1.0*3.9761263820253814*cos(\t r)+-0.0*3.9761263820253814*sin(\t r)},{0.0*3.9761263820253814*cos(\t r)+1.0*3.9761263820253814*sin(\t r)});
	
	\draw [shift={(3.95,0.0)},line width=1pt] [decoration={markings, mark=at position 0.65 with {\arrow{>}}}, postaction={decorate}] plot[domain=-0.588002603547574:0.8615476803894101,variable=\t]({1.0*0.5812585692566169*cos(\t r)+-0.0*0.5812585692566169*sin(\t r)},{0.0*0.5812585692566169*cos(\t r)+1.0*0.5812585692566169*sin(\t r)});
	
	\draw (-0.3923106060605868-0.4,-1.05) node[anchor=north west] {$sun$};
	\draw (3.37018939393941,-0.2249368686868733) node[anchor=north west] {$planet$};
	
	\draw [dashed,line width=1pt] (4.0,0.0) [decoration={markings, mark=at position 0.224 with {\arrow{>}}}, postaction={decorate}] circle (0.9105960150386894cm);
	
	\draw (4.746325757575772,1.0984722222222156) node[anchor=north west] {$moon$};
	\begin{scriptsize}
	\draw [fill=black] (4.673598484848501,0.6127398989898932) circle (2pt);
	\end{scriptsize}
	
	
	\draw [fill=black,fill opacity=0.3,line width=1pt] (4.0+4,0.0) circle (0.3cm);
	
	
	\draw [shift={(4.0,-0.0)},line width=1pt] [decoration={markings, mark=at position 0.18 with {\arrow{>}}}, postaction={decorate}] plot[domain=0.07500383243652106:0.6998928697192437,variable=\t]({1.0*3.9761263820253814*cos(\t r)+-0.0*3.9761263820253814*sin(\t r)},{0.0*3.9761263820253814*cos(\t r)+1.0*3.9761263820253814*sin(\t r)});
	
	\draw [shift={(3.95+4,0.0)},line width=1pt] [decoration={markings, mark=at position 0.65 with {\arrow{>}}}, postaction={decorate}] plot[domain=-0.588002603547574:0.8615476803894101,variable=\t]({1.0*0.5812585692566169*cos(\t r)+-0.0*0.5812585692566169*sin(\t r)},{0.0*0.5812585692566169*cos(\t r)+1.0*0.5812585692566169*sin(\t r)});
	
	\draw (3.37018939393941+4,-0.2249368686868733) node[anchor=north west] {$planet$};
	
	\draw [dashed,line width=1pt] (4.0+4,0.0) [decoration={markings, mark=at position 0.04 with {\arrow{<}}}, postaction={decorate}] circle (0.9105960150386894cm);
	\draw (4.746325757575772+4,1.0984722222222156) node[anchor=north west] {$moon$};
	\begin{scriptsize}
	\draw [fill=black] (4.673598484848501+4,0.6127398989898932) circle (2pt);
	\end{scriptsize}

	\end{tikzpicture}
	\caption{Direct and retrograde periodic orbit}
	\label{figure1}
\end{figure}
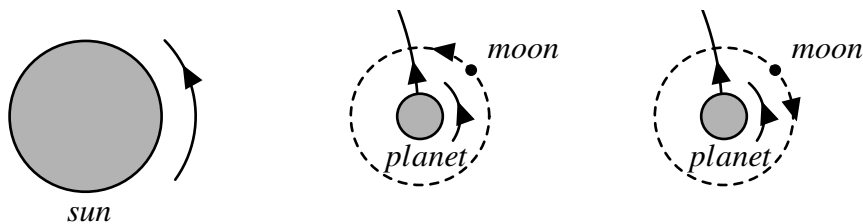
\noindent
The synodic month, or lunation, is the period of the Moon's full moon phase, i.e.,\ the average period of the Moon's orbit with respect to the line joining the sun and the earth.\ Of course one can look at new moons instead.\ If the moon is closer to the earth, then its motion is faster, i.e.\ the speed of the moon varies.\ The anomalistic month is the time the moon takes from the closest point (``perigee") back to the closest point, or equivalently to return to the same speed.\ Considering the farthest point (``apogee") gives the same period.\ The orbit of the moon is inclined to the ecliptic by about $5^{\circ}$, and the draconitic month is the period from one intersection point with the ecliptic, called node, back to itself.\ The term ``draconitic" or ``the nodes dragon" has the following background.\ When the moon is near to one of its nodes, then a solar eclipse appears at full moon and a moon eclipse at new moon.\ In the past, in case of these eclipses people said that there is a dragon eating the sun or the moon.\ However, the mean values of these three periods, named in the abstract, were already known to the Babylonians around 600 BCE.\ For a description about Babylonian astronomy we refer to Neugebauer \cite{neugebauer} and to a more current work by Brack-Bernsen \cite{brack_bernsen}, and about the history of lunar theory to Linton \cite{linton}.\ These periods are also mentioned in Ptolemy's (c.90--c.160) Almagest (see \cite{toomer}) which was one of the most influential works for scientists, in particular for those mentioned in the following quote.\
\begin{center}
	\textit{``Over the
		centuries, through the work of men such as Ptolemy, Ibn ash-Sh\={a}\d{t}ir,}\\
	\textit{Copernicus, Tycho Brahe, Kepler, and Newton, models of the heavens}\\
	\textit{came to reproduce the results of observations with greater and greater accuracy.}"
\end{center}
\begin{flushright}
	- Christopher M. Linton \cite[p.\ xi]{linton}
\end{flushright}
According to planetological data (see \cite[pp.\ 52--53, 62]{schultz}) we have Table \ref{planetological_data}.\
\begin{table}[H]
	\centering
	\begin{tabular}{c|c|c|c}
		& mass & mean distance to sun & mean distance to earth\\
		\hline sun & $1.989 \times 10^{30}$ kg & & \\
		earth & $5.98 \times 10^{24}$ kg & $149.6 \times 10^6$ km & \\
		moon & $7.35 \times 10^{22}$ kg & & $384000$ km
	\end{tabular}
	\caption{Planetological data}
	\label{planetological_data}
\end{table}
\noindent
We see that the mass of the earth is about 0.0003\% compared to the sun and the one of the moon is about 1.234\% compared to the earth.\ Moreover, the mean distance of the moon to the earth compared to the one between earth and sun is about 0.25668\%.\ Given these extreme proportions, Hill's idea was to study a limit case in which the earth is at the origin, the moon is very close to it, but the huge sun is infinitely far away.\ One zooms into a region around the earth (see Figure \ref{figure:hill}) and this is a simplified model of the circular restricted three body problem.\
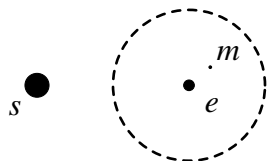
\begin{figure}[H]
	\centering
	\begin{tikzpicture}[line cap=round,line join=round,>=triangle 45,x=1.0cm,y=1.0cm]
	\clip(-2.5,-1) rectangle (2.5,1.1);
	
	\draw (-2.52,-0.05999999999999944) node[anchor=north west] {$s$};
	\draw (0.06,-0.019999999999999435) node[anchor=north west] {$e$};
	\draw (0.22,0.6600000000000007) node[anchor=north west] {$m$};
	\draw [dashed,color=black,line width=1pt] (0.0,-0.0) circle (1.0cm);
	\begin{scriptsize}
	\draw [fill=black] (0.0,-0.0) circle (2.0pt);
	\draw [fill=black] (-2.0,0.0) circle (4.5pt);
	\draw [fill=black] (0.28,0.24) circle (0.5pt);
	\end{scriptsize}
	\end{tikzpicture}
	\caption{Hill lunar problem}
	\label{figure:hill}
\end{figure}
Hill's clever concept was that in his equation the true trajectory of the moon must be close to a periodic orbit centered at the earth, called ``variational orbit" or ``Hill's intermediate orbit", i.e.\ the true orbit is almost periodic.\ For the ``motion of the lunar perigee" (i.e.\ the analysis of anomalistic period) he arrived, by transforming his equations to an infinite set of homogeneous linear equations, at an infinite determinant, which corresponds to these equations written in power series.\ He was the first who attacked such a problem.\ In 1877 John Couch Adams (1819--1892) used exactly the same method in \cite{adams} to study the motion of the lunar node (i.e.\ the analysis of draconitic period).\ Since the determinant is infinite, it was not obvious that it converges.\ In 1881, according to \cite[p.\ 116]{wilson}, Poincaré, who was greatly influenced by Hill's approach, proved the relevant theorem, namely that an infinite determinant converges if and only if the non-diagonal elements have finite sum and the product of the elements on the diagonal is finite.\ We refer to \cite{gutzwiller_2} and \cite{wilson} for more details with a lot of history.\

Nowadays nobody makes use of their huge computations (see \cite{adams}, \cite{hill_det} and \cite{hill}).\ In contrast to this computational approach by Hill and Adams, our geometrical way leads us to a much less computational approximation of their periods in terms of the Floquet multipliers of the linearized spatial Hill equation and their Conley--Zehnder indices.\

\subsection{Derivation of the Hamiltonians}
\label{subsec:discussion}

Since our periodic orbits are in the plane, we start to introduce the planar circular restricted three body problem (PCR3BP from now on).\ In the restricted three body problem we consider two masses, which we call sun and earth, and a massless moon which does not influence the two masses and is attracted by them according to Newton's law of gravitation.\ We denote them respectively by $s$, $e$ and $m$.\ This assumption is a good approximation of the actual system in view of the relations of their masses (see Table \ref{planetological_data} from previous subsection).\ The goal is to understand the dynamics of the massless body, which moves in the same plane as the sun and the earth.\ Moreover, we normalize the total mass to unity, i.e.\ the mass of the earth is $\mu \in [0,1]$ and that of the sun is $1-\mu$.\ If $\mu$ is bigger than $1-\mu$, then we can just interchange their roles.\ Furthermore, the two masses move on circles with common center of mass, with coordinates $ s(t)= - \mu (\cos t, \sin t)$ and $e(t) = (1 - \mu)(\cos t, \sin t)$ (see Figure \ref{figure8}).\
\begin{figure}[H]
	\centering
	\begin{tikzpicture}[line cap=round,line join=round,>=triangle 45,x=1.0cm,y=1.0cm]
	\clip(-4.2,-2.3) rectangle (4.2,2.8);
	
	\draw(0.0,0.0) [decoration={markings, mark=at position 0.63 with {\arrow{>}}}, postaction={decorate},line width=1pt] circle (1.0cm);
	\draw(0.0,0.0) [decoration={markings, mark=at position 0.065 with {\arrow{>}}}, postaction={decorate},line width=1pt] circle (2.0cm);
	\draw [->,line width=1pt] (0.0,-3.0) -- (0.0,2.8);
	\draw (-1.826666666666667,0.6999999999999997) node[anchor=north west] {$s(t)$};
	\draw (1.9533333333333345,0.01999999999999944) node[anchor=north west] {$e(t)$};
	\draw (3.533333333333354,-0.040000000000000584) node[anchor=north west] {$q_1$};
	\draw (0.19333333333333377,2.9400000000000006) node[anchor=north west] {$q_2$};
	\draw (2.5533333333333344,1.6799999999999998) node[anchor=north west] {$m$};
	\draw [->,line width=1pt] (-3.0,0.0) -- (3.6,0.0);
	\begin{scriptsize}
	\draw [fill=black] (0.0,0.0) circle (2pt);
	\draw [fill=black] (-1.0,1.2246467991473532E-16) circle (2pt);
	\draw [fill=black] (2.0,0.0) circle (2pt);
	\draw [fill=black] (2.54,1.28) circle (1pt);
	\end{scriptsize}
	\end{tikzpicture}
	\caption{PCR3BP}
	\label{figure8}
\end{figure}
\noindent
We exclude collisions of the moon with one of the masses, such that the configuration space is $\mathbb{R}^2 \setminus\{s(t),e(t)\}$ and the phase space the trivial cotangent bundle $T^*\big(\mathbb{R}^2\setminus\{s(t),e(t)\}\big) = \big(\mathbb{R}^2\setminus\{s(t),e(t)\}\big) \times \mathbb{R}^2$ with the canonical symplectic form $\omega = dq_1 \wedge dp_1 + dq_2 \wedge dp_2$.\ Let $q=(q_1,q_2)$ denote the position of $m$ and $p=(p_1,p_2)$ its momentum in the fiber.\

By the time dependence of the Hamiltonian in the inertial frame of the moon, which is given by the kinetic energy and Newton's potential, the energy is not preserved by the Hamiltonian flow.\ Hence we consider the angular momentum
\begin{align} \label{angular_momentum_original}
L \colon T^* \mathbb{R}^2 \to \mathbb{R},\quad (q,p) \mapsto p_1q_2 - p_2q_1,
\end{align}
which generates a uniform counterclockwise rotation of the coordinate system such that the sun and the earth are fixed at $s = (- \mu,0)$ and $e = (1 - \mu,0)$ on the $q_1$-axis in the rotating frame.\ 

In this new coordinate system, the Hamiltonian describing the motion of the moon is now autonomous,
$$H \colon T^* \big(\mathbb{R}^2\setminus\{s,e\}\big) \to \mathbb{R},\quad (q,p) \mapsto \frac{1}{2}|p|^2 - \frac{1-\mu}{|q-s|} - \frac{\mu}{|q-e|} + p_1q_2 - p_2q_1.$$
This integral was discovered by Jacobi, so it is called the Jacobi integral.\ The traditional one is $-2H$.\ To understand the physics, we complete the squares and obtain
\begin{align} \label{hamiltonian_1}
H(q,p) = \frac{1}{2}\big( (p_1 + q_2)^2 + (p_2 - q_1)^2 \big) - \frac{1-\mu}{|q-s|} - \frac{\mu}{|q-e|} - \frac{1}{2}(q_1^2 + q_2^2).
\end{align}
The last three terms only depend on the position $q$, so one defines the effective potential by
\begin{align} \label{effective_1}
U \colon \mathbb{R}^2\setminus\{s,e\} \to \mathbb{R},\quad q \mapsto - \frac{1-\mu}{|q-s|} - \frac{\mu}{|q-e|} - \frac{1}{2}(q_1^2 + q_2^2).
\end{align}
Note that $U$ consists of the Newtonian potential for the gravitational force between sun and earth and of the term $-\frac{1}{2}(q_1^2 + q_2^2)$, which appears in rotating coordinates - it is the centrifugal force.\ In the kinetic part is a twist corresponding to an extra force, namely the Coriolis force, which depends on the velocity.\ The dynamics of the moon is therefore very complicated.\

The Lagrange points are the critical points of the Hamiltonian $H$.\ The set of critical points of $H$ and of the effective potential $U$ are in bijection under the restriction of the footpoint projection
$$\pi\vert_{\text{crit}(H)} \colon \text{crit}(H) \to \text{crit}(U),\quad (q,p) \mapsto q,$$
and its inverse given by $(q_1,q_2) \mapsto (q_1,q_2,-q_2,q_1)$.\

For the mass $\mu \in (0,1)$, i.e.\ if neither the sun nor the earth have zero mass, there are five Lagrange points (see \cite[pp.\ 63--68]{frauenfelder} for details).\ $L_1$, $L_2$ and $L_3$ are saddle points and $L_4$ and $L_5$ are global maxima of $U$.\
\begin{figure}[H]
	\centering
	\begin{tikzpicture}[line cap=round,line join=round,>=triangle 45,x=1.0cm,y=1.0cm]
	\clip(-4.5,-2.2) rectangle (4.5,2.2);
	
	\draw (-1.4066666666666668,0.01999999999999887) node[anchor=north west] {$s$};
	\draw (1.8833333333333345,0.01999999999999887) node[anchor=north west] {$e$};
	\draw (-0.22666666666666638,0.6069999999999991) node[anchor=north west] {$0$};
	\draw [line width=1pt] (-3.0,0.0)-- (4.0,0.0);
	\draw (1.1333333333333342,0.01999999999999887) node[anchor=north west] {$\textcolor{red}{L_1}$};
	\draw (-2.446666666666667,0.01999999999999887) node[anchor=north west] {$\textcolor{red}{L_3}$};
	\draw (1.1133333333333342,2.32) node[anchor=north west] {$\textcolor{red}{L_4}$};
	\draw (1.013333333333334,-1.7200000000000017) node[anchor=north west] {$\textcolor{red}{L_5}$};
	\draw (2.3333333333333344,0.01999999999999887) node[anchor=north west] {$\textcolor{red}{L_2}$};
	\begin{scriptsize}
	\draw [fill=black] (0.0,0.0) circle (2pt);
	\draw [fill=black] (-1.0,0.0) circle (2pt);
	\draw [fill=black] (2.0,0.0) circle (2pt);
	\draw [fill=red,color=red] (-2.0,0.0) circle (2pt);
	\draw [fill=red,color=red] (1.7,0.0) circle (2pt);
	\draw [fill=red,color=red] (2.3,0.0) circle (2pt);
	\draw [fill=red,color=red] (1.0,1.73) circle (2pt);
	\draw [fill=red,color=red] (1.0,-1.73) circle (2pt);
	\end{scriptsize}
	\end{tikzpicture}
	\caption{The five Lagrange points}
	\label{figure9}
\end{figure}
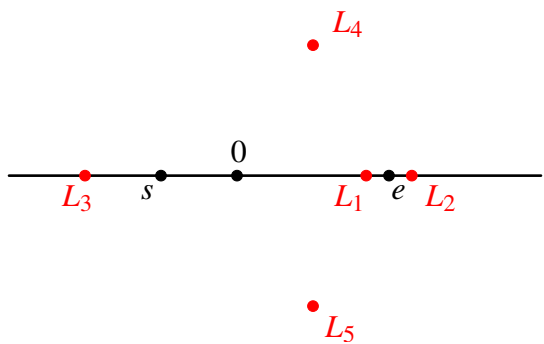
\noindent
One limit case of PCR3BP is the rotating Kepler problem, a two body problem, for which $\mu = 0$, i.e.\ the earth is massless and can ben neglected.\ Hence the moon is just attracted by the sun with mass 1.\ For a detailed discussion we refer to \cite[pp.\ 71--72]{frauenfelder}.\ Another limit case is Hill lunar problem, whose idea and motivation we described in the introduction.\ In contrast to the Kepler problem, the motion here is chaotic, but it simplifies the Hamiltonian (\ref{hamiltonian_1}).\

Recall from the introduction that the orbit of the moon is inclined to the ecliptic, thus from now on the moon moves spatially in three dimensional Euclidean space $\mathbb{R}^3$, i.e.\ its position is given by $q=(q_1,q_2,q_3)$ and its momentum by $p=(p_1,p_2,p_3)$.\ The fixed sun and earth in rotating coordinates are at the positions
$$s = (-\mu, 0,0),\quad e =(1-\mu,0,0).$$
We assume that the sun has a much bigger mass than the earth and the earth a much bigger mass than the moon.\ In addition, the moon moves very close to the earth.\ To goal is to shift the earth into the origin, blow up the coordinates around it and let its mass $\mu$ tend to zero.\ Hence we zoom in an area around the earth (see Figure \ref{figure:hill}).\

Based on the arguments for the planar case in \cite[pp.\ 77--78]{frauenfelder}, we treat the spatial problem in an analogous way.\ The symplectic change of coordinates and momenta
$$ T^* \mathbb{R}^3 \to T^*\mathbb{R}^3, \quad (q,p) \mapsto (q_1 - 1 + \mu, q_2,q_3, p_1, p_2 - 1 + \mu,p_3) $$
puts the earth to the origin and the sun to $(-1,0,0)$.\ By adding the constant $\frac{(1-\mu)^2}{2}$, which does not change the Hamiltonian vector field, we obtain the new Hamiltonian $\widetilde{H}$ on $T^*\big( \mathbb{R}^3 \setminus \{ (-1,0,0),(0,0,0) \} \big)$,
$$\widetilde{H}(q,p) = \frac{1}{2}|p|^2 - \frac{\mu}{|q|} - (1 - \mu) \Bigg( \frac{1}{\sqrt{(q_1 + 1)^2 + q_2^2 + q_3^2}} + q_1 \Bigg) + p_1q_2 - p_2q_1. $$
Now consider the conformally symplectic scaling (the blow up)
$$ \phi_{\mu} \colon T^*\mathbb{R}^3 \to T^*\mathbb{R}^3,\quad (q,p) \mapsto (\mu^{\frac{1}{3}}q,\mu^{\frac{1}{3}}p) $$
by the constant conformal factor $\mu^{\frac{2}{3}}$, i.e.\ $\phi^*_{\mu} \omega = \mu^{\frac{2}{3}}\omega$.\ We define the family of Hamiltonians
$$H^{\mu} \colon T^* \big( \mathbb{R}^3 \setminus \{ (-\mu^{-\frac{1}{3}},0,0),(0,0,0) \} \big) \to \mathbb{R},\quad (q,p) \mapsto \mu^{-\frac{2}{3}} \big( ( \tilde{H} \circ \phi_\mu )(q,p) + 1 - \mu \big). $$
Then
$$ \phi_{\mu}^* X_{\widetilde{H}} = X_{H^{\mu}} $$
and, in explicit form,
$$ H^{\mu}(q,p) = \frac{1}{2}|p|^2 - \frac{1}{|q|} + p_1 q_2 - p_2 q_1 - \frac{1 - \mu}{\mu^{\frac{2}{3}}} \bigg( \frac{1}{\sqrt{ 1 + 2 \mu^{\frac{1}{3}} q_1 + \mu^{\frac{2}{3}}|q|^2  } } + \mu^{\frac{1}{3}} q_1 - 1 \bigg). $$
We simplify $H^{\mu}$ by using the second order Taylor expansion of the function
$$\frac{1}{\sqrt{1+x}} = 1 - \frac{x}{2} + \frac{3x^2}{8} + \mathcal{O}(x^3)$$
for $|x|<1$.\ By setting $x = 2\mu^{\frac{1}{3}}q_1 + \mu^{\frac{2}{3}}|q|^2$, we obtain
$$ H^{\mu}(q,p) = \frac{1}{2}|p|^2 - \frac{1}{|q|} + p_1 q_2 - p_2 q_1 - \bigg( - \frac{1}{2}|q|^2 + \frac{3}{2}q_1^2 + \mathcal{O}(\mu)\bigg). $$
For $\mu \to 0$, $H^{\mu}$ converges uniformly in the $C^{\infty}$-topology on each compact subset to the following Hamiltonian $H$, which we call \textbf{Hamiltonian of the spatial Hill lunar problem}.\ It is given by
\begin{align} \label{hamiltonian_hill}
H \colon T^* \big( \mathbb{R}^3 \setminus \{ (0,0,0) \} \big) \to \mathbb{R},\quad (q,p) \mapsto \frac{1}{2}|p|^2 - \frac{1}{|q|} + p_1q_2 - p_2q_1 - q_1^2 + \frac{1}{2}q_2^2 + \frac{1}{2}q_3^2.
\end{align}
Completing the squares gives us
\begin{align} \label{hamiltonian_2}
H \colon T^* \big( \mathbb{R}^3 \setminus \{ (0,0,0) \} \big) \to \mathbb{R},\quad (q,p) \mapsto \frac{1}{2}\big( (p_1 + q_2)^2 + (p_2 - q_1)^2 + p_3^2 \big) +V(q),
\end{align}
where the effective potential is defined as
\begin{align*}
V \colon \mathbb{R}^3 \setminus \{ (0,0,0) \} \to \mathbb{R},\quad q \mapsto - \frac{1}{|q|} - \frac{3}{2}q_1^2 + \frac{1}{2}q_3^2.
\end{align*}
In the planar case there are no $q_3$ and $p_3$ terms. Compared to PCR3BP, (\ref{hamiltonian_1}) and (\ref{effective_1}), we see that there is again the velocity dependent Coriolis force in the kinetic part.\ This time the effective potential includes just the Newtonian potential for the earth, given by the gravitational force.\ In addition, the therm $-\frac{3}{2}q_1^2$ appers, which is a combination of the gravitational force and the centrifugal force by the huge sun infinitely far away.\ They cancel each other up which is a kind of tidal force and repulsive in the $q_1$-direction.\ In the spatial case, since the sun is in the plane, there is a strong attraction back to the ecliptic, hence the term $\frac{1}{2}q_3^2$ is a kind of left over of the gravitational force of the sun.\

Moreover, in contrast to the PCR3BP the planar Hill lunar problem has only two critical points (see \cite[pp.\ 80--81]{frauenfelder} for details).\ They are saddle points at $(\pm 3 ^{-\frac{1}{3}},0)$, which are the limits of $L_1$ and $L_2$ (from Figure \ref{figure9}), by blowing up the coordinates around the earth.\ In view of symmetries (see (\ref{involution_1} in next Section \ref{sec:involutions_shooting}), they have the same critical value $- \frac{1}{2} 3^{4/3}$, while the traditional one is $3^{4/3}$.\

\subsection{Hill equation and linearization}

For finding periodic orbits, we use the Hamiltonian equation of motion
$$ \frac{dq_i}{dt} = \frac{\partial H}{\partial p_i},\quad \frac{d p_i}{dt} = - \frac{\partial H}{\partial q_i} $$
and compute for the Hamiltonian (\ref{hamiltonian_2}) the \textbf{spatial Hill equation in $(q,p)$-coordinates}:
\begin{align} \label{equation_of_motion_q_p}
\begin{cases}
\dot{q}_1 = p_1 + q_2,\quad\quad\quad & \mathlarger{\dot{p}_1 = p_2 - q_1  - \frac{\partial V}{\partial q_1}},\\[0.8em]
\dot{q}_2 = p_2 - q_1, & \mathlarger{\dot{p}_2 = - p_1 - q_2 - \frac{\partial V}{\partial q_2}},\\
\dot{q}_3 = p_3, & \mathlarger{\dot{p}_3 = - \frac{\partial V}{\partial q_3}}.
\end{cases}
\end{align}
For the \textbf{spatial Hill equation in $(q,\dot{q})$-coordinates} we transform the first order ODE to a second order ODE
$$ \ddot{q}_1 = \dot{p}_1 + \dot{q}_2,\quad \ddot{q}_2 = \dot{p}_2 - \dot{q}_1,\quad \ddot{q}_3 = \dot{p}_3 $$
and obtain
\begin{align} \label{equation_of_motion_0}
\begin{pmatrix}
\ddot{q}_1\\
\ddot{q}_2\\
\ddot{q}_3
\end{pmatrix} + 2 \begin{pmatrix}
0 & -1 & 0\\
1 & 0 & 0\\
0 & 0 & 0
\end{pmatrix} \cdot \begin{pmatrix}
\dot{q}_1\\
\dot{q}_2\\
\dot{q}_3
\end{pmatrix} + \nabla V (q) = 0,
\end{align}
which reads
\begin{equation*}
\left\{  \begin{array}{l}
\mathlarger{\ddot{q}_1 = 2 \dot{q}_2 + 3q_1 - \frac{q_1}{|q|^3}} \\[0.7em]
\mathlarger{\ddot{q}_2 = -2\dot{q}_1 - \frac{q_2}{|q|^3}} \\[0.7em]
\mathlarger{\ddot{q}_3 = -q_3\bigg(\frac{1}{|q|^3} + 1 \bigg)}.
\end{array}  \right.
\end{equation*}
Note that the first two equations are for the planar case and the submatrix
$$ \begin{pmatrix}
0 & -1\\
1 & 0
\end{pmatrix}, $$
which is a rotation by $\frac{\pi}{2}$, arises by the coriolis force and is inherited from the planar problem.\

To linearize the Hill equation along a periodic orbit we expand $(q + \Delta q, \dot{q} + \Delta \dot{q})$ near $(q,\dot{q})$.\ Therefore, in view of (\ref{equation_of_motion_0}) the \textbf{linearized equation in $(q,\dot{q})$-coordinates} is given by
$$ \begin{pmatrix}
\Delta \ddot{q}_1\\
\Delta \ddot{q}_2\\
\Delta \ddot{q}_3
\end{pmatrix} + 2 \begin{pmatrix}
0 & -1 & 0\\
1 & 0 & 0\\
0 & 0 & 0
\end{pmatrix} \cdot \begin{pmatrix}
\Delta \dot{q}_1\\
\Delta \dot{q}_2\\
\Delta \dot{q}_3
\end{pmatrix} + \text{H}_V (q) \cdot \begin{pmatrix}
\Delta q_1\\
\Delta q_2\\
\Delta q_3
\end{pmatrix} = 0, $$
which is
\begin{equation*}
\left\{  \begin{array}{l}
\mathlarger{\Delta \ddot{q}_1 - 2 \Delta \dot{q}_2 = \bigg(\frac{2q_1^2 - q_2^2 - q_3^2}{|q|^5} + 3\bigg) \Delta q_1 + \frac{3q_1q_2}{|q|^5} \Delta q_2 + \frac{3q_1q_3}{|q|^5} \Delta q_3}\\[1em]
\mathlarger{\Delta \ddot{q}_2 + 2 \Delta \dot{q}_1 = \frac{3q_1q_2}{|q|^5} \Delta q_1 + \frac{2q_2^2 - q_1^2 - q_3^2}{|q|^5} \Delta q_2 + \frac{3q_2q_3}{|q|^5} \Delta q_3} \\[1em]
\mathlarger{\Delta \ddot{q}_3 = \frac{3q_1q_3}{|q|^5} \Delta q_1 + \frac{3q_2q_3}{|q|^5} \Delta q_2 + \bigg(\frac{2q_3^2 - q_1^2 - q_2^2}{|q|^5} - 1\bigg) \Delta q_3 }.
\end{array}  \right.
\end{equation*}
For a planar periodic orbit which is considered in the spatial system, by the restriction of the Hessian to $\{q_3=0\}$, the linearized equations become
\begin{equation*}
\left\{  \begin{array}{l}
\mathlarger{\Delta \ddot{q}_1 - 2 \Delta \dot{q}_2 = \bigg(\frac{2q_1^2 - q_2^2}{|q|^5} + 3\bigg) \Delta q_1 + \frac{3q_1q_2}{|q|^5} \Delta q_2} \\[1em]
\mathlarger{\Delta \ddot{q}_2 + 2 \Delta \dot{q}_1 = \frac{3q_1q_2}{|q|^5} \Delta q_1 + \frac{2q_2^2 - q_1^2}{|q|^5} \Delta q_2} \\[1em]
\mathlarger{\Delta \ddot{q}_3 = - \bigg(\frac{1}{|q|^3} + 1\bigg) \Delta q_3, }
\end{array}  \right.
\end{equation*}\\
which are similar to the equations in Gutzwiller \cite[p.\ 70]{gutzwiller} and for the planar case in Wintner \cite[p.\ 403]{wintner}.\ The \textbf{linearized equation in $(q,p)$-coordinates} with respect to (\ref{equation_of_motion_q_p}) yields
\begin{align*}
\begin{cases}
\Delta \dot{q}_1 = \Delta p_1 + \Delta q_2,\quad\quad\quad & \Delta \dot{p}_1 = \Delta p_2 - \Delta q_1  - \left(\text{H}_V (q) \cdot \Delta q \right)_1,\\
\Delta \dot{q}_2 = \Delta p_2 - \Delta q_1, & \Delta \dot{p}_2 = - \Delta p_1 - \Delta q_2 - \left(\text{H}_V (q) \cdot \Delta q \right)_2,\\
\Delta \dot{q}_3 = \Delta p_3, & \Delta \dot{p}_3 = - \left(\text{H}_V (q) \cdot \Delta q \right)_3,
\end{cases}
\end{align*}
which reads for planar periodic orbits
\begin{align} \label{linearized_q_p_fix}
\begin{cases}
\Delta \dot{q}_1 = \Delta p_1 + \Delta q_2,\quad\quad\quad & \Delta \dot{p}_1 = \Delta p_2 - \Delta q_1  + \mathlarger{\bigg(\frac{2q_1^2 - q_2^2}{|q|^5} + 3\bigg) \Delta q_1 + \frac{3q_1q_2}{|q|^5} \Delta q_2},\\[1em]
\Delta \dot{q}_2 = \Delta p_2 - \Delta q_1, & \Delta \dot{p}_2 = - \Delta p_1 - \Delta q_2 + \mathlarger{\frac{3q_1q_2}{|q|^5} \Delta q_1 + \frac{2q_2^2 - q_1^2}{|q|^5} \Delta q_2}, \\[1em]
\Delta \dot{q}_3 = \Delta p_3, & \Delta \dot{p}_3 = \mathlarger{ - \bigg(\frac{1}{|q|^3} + 1\bigg) \Delta q_3 }.
\end{cases}
\end{align}

\subsection{The group of linear symmetries}
\label{sec:involutions_shooting}

\textbf{Planar case:}\ The planar part of the Hamiltonian (\ref{hamiltonian_hill})
\begin{align} \label{hamiltonian_linear_planar}
H \colon T^* \big( \mathbb{R}^2 \setminus \{ (0,0) \} \big) \to \mathbb{R},\quad (q,p) \mapsto \frac{1}{2}|p|^2 - \frac{1}{|q|} + p_1q_2 - p_2q_1 - q_1^2 + \frac{1}{2}q_2^2
\end{align}
is invariant under the double-symmetry given by the two commuting linear anti-symplectic involutions
\begin{align} \label{involution_1}
&\rho_1 \colon T^* \mathbb{R}^2 \to T^*\mathbb{R}^2,\quad (q,p) \mapsto (q_1,-q_2,-p_1,p_2), \nonumber \\
&\rho_2 \colon T^* \mathbb{R}^2 \to T^*\mathbb{R}^2,\quad (q,p) \mapsto (-q_1,q_2,p_1,-p_2).
\end{align}
Their product $\rho_1 \circ \rho_2 = \rho_2 \circ \rho_1 = -\text{id}$ is symplectic.\ Geometrically, in the configuration space, the Hamiltonian $H$ in (\ref{hamiltonian_linear_planar}) is invariant under the reflections about the $q_1$- and $q_2$-axes, i.e.\ it is not possible to say whether we are going to the sun or away from it.\ The symplectic involutions $\pm$id correspond to a rotation by $0$ and $\pi$, respectively.\
\begin{remark}
	Algebraically, $\rho_1, \rho_2$ and $\pm$ id form a Klein four-group, i.e.
	$$ \text{Sym}^{\text{P}}_{\text{L}}(H) := \langle \rho_1, \rho_2 \mid \rho_1^2 = \rho_2^2 = (\rho_1 \circ \rho_2)^2 = \text{id} \rangle \cong \mathbb{Z}_2 \times \mathbb{Z}_2.$$	
\end{remark}
\noindent
Now we show that these four linear symmetries are the only ones in the planar system.\
\begin{Proposition} \label{prop_planar}
The set
\begin{align} \label{set_1}
\{ \rho \colon T^* \mathbb{R}^2 \to T^* \mathbb{R}^2 \text{ linear} \mid \rho^2 = \text{ id, } H \circ \rho = H \text{ and } \rho ^* \omega = \pm \omega \}
\end{align}
is $\text{Sym}^{\text{P}}_{\text{L}}(H)$.
\end{Proposition}
\begin{proof}
	Let $\rho$ be an element from the set (\ref{set_1}).\ We prove the proposition in three steps where the first one is obvious.\\
	\textbf{Step 1.}\ \textit{The Hamiltonian (\ref{hamiltonian_linear_planar}) is the sum of
	$$ H_2(q,p) = \frac{1}{2}|p|^2 + p_1q_2 - p_2q_1 - q_1^2 + \frac{1}{2}q_2^2 \quad \text{ and } \quad H_{-1}(q,p) = - \frac{1}{|q|}, $$
	where $H_2$ is homogeneous of degree 2 and $H_{-1}$ is homogeneous of degree $-1$.\ Hence $H_2 \circ \rho = H_2$ and $H_{-1} \circ \rho = H_{-1}.$}\\
	\textbf{Step 2.}\ \textit{The matrix form of $\rho$ with respect to the splitting $\mathbb{R}^4 = \mathbb{R}^2 \times \mathbb{R}^2$ and to the coordinates $(q_1,q_2,p_1,p_2)$ is
	\begin{align*}
	\begin{cases}
	\begin{pmatrix}
	A & 0\\
	C & A
	\end{pmatrix},\quad A \in O(2),\quad A=A^T,\quad C=-C^T,\quad AC=-CA, & \text{ if $\rho$ is symplectic}\\
	\\
	\begin{pmatrix}
	A & 0\\
	C & -A
	\end{pmatrix},\quad A \in O(2),\quad A=A^T,\quad C=C^T,\quad AC=CA, & \text{ if $\rho$ is anti-symplectic.}
	\end{cases}
	\end{align*}}
	To see that, we write $\rho$ in matrix form
	$$ \begin{pmatrix}
	A & B\\
	C & D
	\end{pmatrix}, $$
	with respect to the coordinates $(q_1,q_2,p_1,p_2)$, where $A,B,C,D \in \text{Mat}(2,\mathbb{R}).$\ The $\rho$-invariance of $H_{-1}$ yields
	$$ |Aq + Bp| = |q|,\quad \forall q,p. $$
	For fixed $p$ we take $q$ with $|q|$ very small and find $Bp=0$.\ Hence
	$$B = 0,\quad A \in O(2).$$
	Next, the $\rho$-invariance of $H_2$ yields for $q=0$
	$$ |Dp| = |p|,\quad \forall p, $$
	whence also $D \in O(2).$\ Since $\rho$ is an involution, we obtain
	$$ \rho \circ \rho = \begin{pmatrix}
	A^2 & 0\\
	CA + DC & D^2
	\end{pmatrix} = \begin{pmatrix}
	I_2 & 0\\
	0 & I_2
	\end{pmatrix}, $$
	and with $AA^T = DD^T = I_2$, we obtain
	\begin{align} \label{relation_1}
	A = A^T,\quad D = D^T,\quad CA + DC = 0.
	\end{align}
	If $\rho$ is symplectic, then (\ref{relation_1}) and the linear symplectic relations (\ref{matrix_symplectic}) imply
	$$AC=A^TC=C^TA,\quad A^TD=AD=I_2.$$
	With $ A^2=I_2 $ we have
	$$D=A,$$
	and therefore
	$$ 0 = CA + DC = CA + AC = CA + C^TA = (C+C^T)A.$$
	Since det$(A)=\pm1$, the matrix $C$ is skew-symmetric and this proves the first assertion of the second step.\\
	If $\rho$ is anti-symplectic, then by (\ref{relation_1}) and the linear anti-symplectic conditions (\ref{matrix_anti_symplectic}) we obtain
	$$AC=A^TC=C^TA,\quad A^T D = AD = -I_2,$$
	hence
	$$D=-A,$$
	and
	$$ 0 = CA + DC = CA - AC = CA - C^TA = (C-C^T)A.$$
	Therefore the matrix $C$ is symmetric and the second assertion follows.\\
	\textbf{Step 3.}\ \textit{In both cases, the matrix $C$ is the zero matrix and
	\begin{align*}
	\begin{cases}
	\rho \in \{ \pm \text{id} \}, &\text{ if $\rho$ is symplectic}\\
	\rho \in \{\rho_1, \rho_2\}, & \text{ if $\rho$ is anti-symplectic.}
	\end{cases}
	\end{align*}}
	To prove this, we first note that in both cases, $A$ is of the form
	$$ A = \begin{pmatrix}
	a & b\\
	b & c
	\end{pmatrix}. $$
	Since $A^2 = I_2$ and $A\in O(2)$, we have
	\begin{align} \label{relation_2}
	\begin{pmatrix}
	a^2 + b^2 & b(a+c)\\
	b(a+c) & b^2 + c^2
	\end{pmatrix} = \begin{pmatrix}
	1 & 0\\
	0 & 1
	\end{pmatrix},\quad a^2 = c^2,\quad \det(A)=ac-b^2 = \pm 1.
	\end{align}
	If $\rho$ is symplectic, then $\rho(q,p)$ is of the form
	$$ \begin{pmatrix}
	a & b & 0 & 0\\
	b & c & 0 & 0\\
	0 & d & a & b\\
	-d & 0 & b & c
	\end{pmatrix} \cdot \begin{pmatrix}
	q_1\\
	q_2\\
	p_1\\
	p_2
	\end{pmatrix} = \begin{pmatrix}
	aq_1 + bq_2\\
	bq_1 + cq_2\\
	dq_2 + ap_1 + bp_2\\
	-dq_1 + bp_1 + cp_2
	\end{pmatrix}. $$
	Since $AC=-CA$ we have
	\begin{align} \label{relation_3}
	AC = \begin{pmatrix}
	-bd & ad\\
	-cd & bd
	\end{pmatrix} = \begin{pmatrix}
	-bd & -cd\\
	ad & bd
	\end{pmatrix} = -CA,\quad ad=-cd.
	\end{align}
	In the following we show that $AC=0$.\ Then $\det(A)=\pm1$ shows that $C$ needs to be the zero matrix.\ In view of the $\rho$-invariance of $H_2$ we compare
	$$ H_2(q,p) = \frac{1}{2}|p|^2 + p_1q_2 - p_2q_1 - q_1^2 + \frac{1}{2}q_2^2 $$
	with
	\begin{align*}
	H_2\big(\rho(q,p)\big) = &\frac{1}{2}|p|^2 + p_1q_2\big(ad + \det(A)\big) - p_2q_1\big( cd + \det(A) \big) - q_1^2 \left( - \frac{1}{2}d^2 - ad + a^2 - \frac{1}{2}b^2 \right)\\
	& + \frac{1}{2}q_2^2 \left(d^2 + 2cd - 2b^2 + c^2\right) - bdp_1q_1 + bdp_2q_2 + q_1q_2\left(2bd-2ab+bc\right).
	\end{align*}
	Therefore $bd=0$.\ Now looking at the coefficients of $p_1q_2$ and $p_2q_1$, we find $ad=cd$ and by the second equation in (\ref{relation_3}),	$ad=cd=0$.\ Hence $AC=0$ and $C=0$ as well.\ Furthermore
	$$\det(A) = 1,$$
	i.e.\ $A \in SO(2)$, meaning that $A$ is an involutive rotation, hence $A = \pm I_2$ and therefore
	$$ \rho \in \{ \pm \text{id} \}. $$
	Analogously, in the anti-symplectic case,
	$$ \begin{pmatrix}
	a & b & 0 & 0\\
	b & c & 0 & 0\\
	c_1 & c_2 & -a & -b\\
	c_2 & c_3 & -b & -c
	\end{pmatrix} \cdot \begin{pmatrix}
	q_1\\
	q_2\\
	p_1\\
	p_2
	\end{pmatrix} = \begin{pmatrix}
	aq_1 + bq_2\\
	bq_1 + cq_2\\
	c_1q_1 + c_2q_2 - ap_1 - bp_2\\
	c_2q_1 + c_3q_2 - bp_1 - cp_2
	\end{pmatrix}, $$
	where $AC=CA$, yields
	\begin{align} \label{relation_4}
	\begin{pmatrix}
	ac_1 + bc_2 & ac_2 + bc_3\\
	bc_1 + cc_2 & bc_2 + cc_3
	\end{pmatrix} = \begin{pmatrix}
	ac_1 + bc_2 & bc_1 + cc_2\\
	ac_2 + bc_3 & bc_2 + cc_3
	\end{pmatrix},\quad ac_2 + bc_3 = bc_1 + cc_2.
	\end{align}
	Now $H_2\big( \rho(q,p) \big)$ is
	\begin{align*}
	& \frac{1}{2}|p|^2 + p_1q_2\big( -ac_2 - bc_3 - \det(A) \big) - p_2q_1 \big( bc_1 + cc_2 - \det(A) \big)\\
	& -q_1^2 \left( - \frac{1}{2}c_1^2 - \frac{1}{2}c_2^2 - bc_1 + ac_2 + a^2 - \frac{1}{2}b^2 \right) + \frac{1}{2}q_2^2 \left(c_2^2 + c_3^2 + 2cc_2 - 2bc_3 - 2b^2 + c^2\right)\\
	& - p_1q_1 \left(ac_1 + bc_2\right) - p_2q_2\left(bc_2 + cc_3\right) + q_1q_2 \left(c_1c_2 + c_2c_3 + cc_1 - ac_3 - 2ab + bc\right).
	\end{align*}
	Hence $ac_1+bc_2 = bc_2+cc_3=0$ and the coefficients of $p_1q_2$ and $p_2q_1$ imply $-ac_2-bc_3=bc_1+cc_2$.\ The second equation in (\ref{relation_4}) yields $ac_2+bc_3 = bc_1 + cc_2 = 0 $, thus $AC=0$ and $C=0$.\ Moreover by the other coefficients,
	$$ \det(A) = -1,\quad a^2 - \frac{1}{2}b^2 = 1,\quad -2b^2 + c^2 =1, $$
	and together with (\ref{relation_2}),
	$$ b=0,\quad a^2 = c^2 = 1,\quad ac=-1 ,$$
	which correspond to $\rho_1$ and $\rho_2$.
\end{proof}
\noindent
\textbf{Spatial case:}\ The following eight linear symplectic or anti-symplectic involutions leave invariant the spatial Hamiltonian from (\ref{hamiltonian_hill})
\begin{align} \label{hamiltonian_linear_spatial}
H \colon T^* \big( \mathbb{R}^3 \setminus \{ (0,0,0) \} \big) \to \mathbb{R},\quad (q,p) \mapsto \frac{1}{2}|p|^2 - \frac{1}{|q|} + p_1q_2 - p_2q_1 - q_1^2 + \frac{1}{2}q_2^2 + \frac{1}{2}q_3^2.
\end{align}
The reflection across the ecliptic $\{q_3=0\}$ induces the symplectic involution
\begin{align} \label{symplectic_involution}
\sigma \colon T^* \mathbb{R}^3 \to T^*\mathbb{R}^3,\quad (q,p) \mapsto (q_1,q_2,-q_3,p_1,p_2,-p_3).
\end{align}
Its fixed point set
$$\text{Fix}(\sigma) = \{ (q_1,q_2,0,p_1,p_2,0) \}$$
is the planar components, hence the planar problem can be interpreted as $\text{Fix}(\sigma)$ and the $q_1q_2$-plane is invariant under $\sigma$.\

The two commuting anti-symplectic involutions from the planar case are now of the form
\begin{align} \label{involution_2}
&\rho_1 \colon T^* \mathbb{R}^3 \to T^*\mathbb{R}^3,\quad (q,p) \mapsto (q_1,-q_2,q_3,-p_1,p_2,-p_3), \nonumber \\
&\rho_2 \colon T^* \mathbb{R}^3 \to T^*\mathbb{R}^3,\quad (q,p) \mapsto (-q_1,q_2,q_3,p_1,-p_2,-p_3),
\end{align}
where $\rho_1$ corresponds to the reflection at the $q_1q_3$-plane and $\rho_2$ to the reflection at the $q_2q_3$-plane.\ Their product gives the symplectic involution,
\begin{align} \label{involution_3}
\rho_1 \circ \rho_2 = \rho_2 \circ \rho_1 = - \sigma \colon T^* \mathbb{R}^3 \to T^*\mathbb{R}^3,\quad (q,p) \mapsto (-q_1,-q_2,q_3,-p_1,-p_2,p_3),
\end{align}
which corresponds to a rotation around the $q_3$-axis by $\pi$.\ Its fixed point set is
\begin{align} \label{fixed_point_2}
\text{Fix}(-\sigma) = \{ ( 0,0,q_3,0,0,p_3 ) \},
\end{align}
thus the $q_3$-axis is invariant under $-\sigma$.\

Moreover, each of the anti-symplectic involutions $\rho_1$ and $\rho_2$ commutes with $\sigma$ and they yield the two anti-symplectic involutions
\begin{align}
\overline{\rho_1}:= \rho_1 \circ \sigma = \sigma \circ \rho_1 \colon T^* \mathbb{R}^3 \to T^*\mathbb{R}^3,\quad &(q,p) \mapsto (q_1,-q_2,-q_3,-p_1,p_2,p_3)\\
\overline{\rho_2}:= \rho_2 \circ \sigma = \sigma \circ \rho_2 \colon T^* \mathbb{R}^3 \to T^*\mathbb{R}^3,\quad &(q,p) \mapsto (-q_1,q_2,-q_3,p_1,-p_2,p_3).
\end{align}
Note that $\overline{\rho_1}$ corresponds to a rotation around the $q_1$-axis by $\pi$ and $\overline{\rho_2}$ to a rotation around the $q_2$-axis by $\pi$ as well.\

The two symplectic involutions $-\sigma$ and $\sigma$ commute as well and give $-$id.\ Together with id, we have now 8 linear symplectic or anti-symplectic involutions which leave the Hamiltonian (\ref{hamiltonian_linear_spatial}) invariant.\ Four of them are anti-symplectic and four are symplectic.\ They form a group  which we denote by $\text{Sym}^{\text{S}}_{\text{L}}(H)$.\ Their group structure is given by Table \ref{group_structure}.
\begin{table}[H]
	\centering
	\begin{tabular}{c|cccccccc}
		$\circ$ & id & $-$id & $-\sigma$ & $\sigma$ & $\rho_1$ & $\rho_2$ & $\overline{\rho_1}$ & $\overline{\rho_2}$\\
		\hline id & id & $-$id & $-\sigma$ & $\sigma$ & $\rho_1$ & $\rho_2$ & $\overline{\rho_1}$ & $\overline{\rho_2}$ \\
		$-$id & $-$id & id & $\sigma$ & $-\sigma$ & $\overline{\rho_2}$ & $\overline{\rho_1}$ & $\rho_2$ & $\rho_1$ \\
		$-\sigma$ & $-\sigma$ & $\sigma$ & id & $-$id & $\rho_2$ & $\rho_1$ & $\overline{\rho_2}$ & $\overline{\rho_1}$\\
		$\sigma$ & $\sigma$ & $-\sigma$ & $-$id & id & $\overline{\rho_1}$ & $\overline{\rho_2}$ & $\rho_1$ & $\rho_2$\\
		$\rho_1$ & $\rho_1$ & $\overline{\rho_2}$ & $\rho_2$ & $\overline{\rho_1}$ & id & $-\sigma$ & $\sigma$ & $-$id\\
		$\rho_2$ & $\rho_2$ & $\overline{\rho_1}$ & $\rho_1$ & $\overline{\rho_2}$ & $-\sigma$ & id & $-$id & $\sigma$\\
		$\overline{\rho_1}$ & $\overline{\rho_1}$ & $\rho_2$ & $\overline{\rho_2}$ & $\rho_1$ & $\sigma$ & $-$id & id & $-\sigma$\\
		$\overline{\rho_2}$ & $\overline{\rho_2}$ & $\rho_1$ & $\overline{\rho_1}$ & $\rho_2$ & $-$id & $\sigma$ & $-\sigma$ & id
	\end{tabular}
	\caption{Group structure of $\text{Sym}^{\text{S}}_{\text{L}}(H)$}
	\label{group_structure}
\end{table}
\noindent
It is generated by $\{ \rho_1,\rho_2,\sigma \}$ and
$$ \text{Sym}^{\text{S}}_{\text{L}}(H) \cong \mathbb{Z}_2 \times \mathbb{Z}_2 \times \mathbb{Z}_2 . $$
By considering the four symplectic involutions we see that like in the planar case, a Klein four-group arises, namely as a sub-group of $\text{Sym}^{\text{S}}_{\text{L}}(H)$.\ It is generated by $\{\pm \sigma\}$ and we denote it by $\omega\text{-Sym}^{\text{P}}_{\text{L}}(H) \cong \mathbb{Z}_2 \times \mathbb{Z}_2$, hence
$$ \text{Sym}^{\text{P}}_{\text{L}}(H) \cong \omega\text{-Sym}^{\text{S}}_{\text{L}}(H) \subset \text{Sym}^{\text{S}}_{\text{L}}(H). $$
The following proposition shows that these eight linear symmetries are the only ones in the spatial problem.\
\begin{Proposition} \label{prop_spatial}
	The set
	\begin{align} \label{set_2}
	\{ \rho \colon T^* \mathbb{R}^3 \to T^* \mathbb{R}^3 \text{ linear} \mid \rho^2 = \text{ id, } H \circ \rho = H \text{ and } \rho ^* \omega = \pm \omega \}
	\end{align}
	is $\text{Sym}^{\text{S}}_{\text{L}}(H)$.
\end{Proposition}
\begin{proof}
	Let $\rho$ be an element from the set (\ref{set_2}).\ As in the planar case, we prove the proposition in three analogous steps where only the calculation for the third assertion involves more coefficients.\\
	\textbf{Step 1.}\ \textit{The Hamiltonian (\ref{hamiltonian_linear_spatial}) is the sum of
		$$ H_2(q,p) = \frac{1}{2}|p|^2 + p_1q_2 - p_2q_1 - q_1^2 + \frac{1}{2}q_2^2 + \frac{1}{2}q_3^2 \quad \text{ and } \quad H_{-1}(q,p) = - \frac{1}{|q|}, $$
		where $H_2$ is homogeneous of degree 2 and $H_{-1}$ is homogeneous of degree $-1$.\ Hence $H_2 \circ \rho = H_2$ and $H_{-1} \circ \rho = H_{-1}.$}\\
	\textbf{Step 2.}\ \textit{The matrix form of $\rho$ with respect to the splitting $\mathbb{R}^6 = \mathbb{R}^3 \times \mathbb{R}^3$ and to the coordinates $(q_1,q_2,q_3,p_1,p_2,p_3)$ is
		\begin{align*}
		\begin{cases}
		\begin{pmatrix}
		A & 0\\
		C & A
		\end{pmatrix},\quad A \in O(3),\quad A=A^T,\quad C=-C^T,\quad AC=-CA, & \text{ if $\rho$ is symplectic}\\
		\\
		\begin{pmatrix}
		A & 0\\
		C & -A
		\end{pmatrix},\quad A \in O(3),\quad A=A^T,\quad C=C^T,\quad AC=CA, & \text{ if $\rho$ is anti-symplectic.}
		\end{cases}
		\end{align*}}
	\textbf{Step 3.}\ \textit{In both cases, the matrix $C$ is the zero matrix and
		\begin{align*}
		\begin{cases}
		\rho \in \{ \sigma, \overline{\sigma}, \pm \text{id} \}, &\text{ if $\rho$ is symplectic}\\
		\rho \in \{\rho_1, \rho_2\, \overline{\rho_1}, \overline{\rho_2}\}, & \text{ if $\rho$ is anti-symplectic.}
		\end{cases}
		\end{align*}}
	In both cases, $A$ is of the form
	$$ A = \begin{pmatrix}
	a & d & e\\
	d & b & f\\
	e & f & c
	\end{pmatrix}. $$
	Since $A^2 = I_3$ and $A \in O(3)$, we have
	\begin{align} \label{relation_5}
	A^2 = \begin{pmatrix}
	a^2 + d^2 + e^2 & ad+bd+ef & ae+df+ce\\
	ad+bd+ef & d^2 + b^2 + f^2 & de+bf+cf\\
	ae+df+ce & de+bf+cf & e^2 + f^2 + c^2
	\end{pmatrix} = \begin{pmatrix}
	1 & 0 & 0\\
	0 & 1 & 0\\
	0 & 0 & 1
	\end{pmatrix}
	\end{align}
	and
	$$ \det(A)= abc + 2def - af^2 - be^2 - cd^2 = \pm 1. $$
	If $\rho$ is symplectic, then $\rho(q,p)$ is of the form
	$$ \begin{pmatrix}
	a & d & e & 0 & 0 & 0\\
	d & b & f & 0 & 0 & 0\\
	e & f & c & 0 & 0 & 0\\
	0 & c_1 & c_2 & a & d & e\\
	-c_1 & 0 & c_3 & d & b & f\\
	-c_2 & -c_3 & 0 & e & f & c
	\end{pmatrix} \cdot \begin{pmatrix}
	q_1\\
	q_2\\
	q_3\\
	p_1\\
	p_2\\
	p_3
	\end{pmatrix} = \begin{pmatrix}
	aq_1 + dq_2 + eq_3\\
	dq_1 + bq_2 + fq_3\\
	eq_1 + fq_2 + cq_3\\
	c_1q_2 + c_2q_3 + ap_1 + dp_2 + ep_3\\
	-c_1q_1 + c_3q_3 + dp_1 + bp_2 + fp_3\\
	-c_2q_1 - c_3q_2 + ep_1 + fp_2 + cp_3
	\end{pmatrix}. $$
	The equation $AC=-CA$ yields
	\begin{align} \label{relation_6}
	\begin{pmatrix}
	-dc_1 - ec_2 & ac_1 - ec_3 & ac_2 + dc_3\\
	-bc_1 - fc_2 & dc_1 - fc_3 & dc_2 + bc_3\\
	-fc_1 - cc_2 & ec_1 - cc_3 & ec_2 + fc_3
	\end{pmatrix} = \begin{pmatrix}
	-dc_1 - ec_2 & -bc_1 - fc_2 & -fc_1 - cc_2\\
	ac_1 - ec_2 & dc_1 - fc_3 & ec_1 - cc_3\\
	ac_2 + dc_3 & dc_2 + bc_3 & ec_2 + fc_3
	\end{pmatrix},
	\end{align}
	and therefore
	\begin{align} \label{relation_7}
	ac_1 - ec_3 = -bc_1 - fc_2,\quad ac_2 + dc_3 = -fc_1 - cc_2,\quad dc_2 + bc_3 = ec_1 - cc_3.
	\end{align}
	In view of the $\rho$-invariance of $H_2$ we compare
	$$ H_2(q,p) = \frac{1}{2}|p|^2 + p_1q_2 - p_2q_1 - q_1^2 + \frac{1}{2}q_2^2 + \frac{1}{2}q_3^2 $$
	with $H_2\big( \rho(q,p) \big)$, which is
	\begin{align*}
	&\frac{1}{2}|p|^2 + p_1q_2\left(ac_1 - ec_3 + ab - d^2\right) - p_2q_1\left(bc_1 + fc_2 + ab - d^2\right)\\
	& - q_1^2 \left(\frac{1}{2}\left(-c_1^2 -c_2^2 - d^2 - e^2 \right) - ac_1 + a^2 \right)\\
	& + \frac{1}{2}q_2^2 \left(c_1^2 + c_3^2 + 2bc_1 - 2d^2 + b^2 + f^2\right)\\
	& + \frac{1}{2}q_3^2 \left(c_2^2 + c_3^2 + 2fc_2 - 2ec_3 - 2e^2 + f^2 + c^2\right)\\
	& + p_1q_1 \left( -dc_1 - ec_2 \right) + p_2q_2 \left( dc_1 - fc_3 \right) + p_3q_3 \left( ec_2 + fc_3 \right)\\
	& + p_1q_3 \left(ac_2 + dc_3 + af - de\right) + p_3q_1 \left( -fc_1 - cc_2 + de - af \right)\\
	& + p_2q_3 \left( dc_2 + bc_3 + df - be \right) + p_3q_2 \left( ec_1 - cc_3 + be - df \right)\\
	& + q_1q_2 \left( c_2c_3 + 2dc_1 - 2ad + bd + ef \right)\\
	& + q_1q_3 \left( -c_1c_3 + dc_2 + ec_1 - ac_3 - 2ae + df + ce \right)\\
	& + q_2q_3 \left( c_1c_2 + fc_1 + bc_2 - dc_3 - 2de + bf + cf \right).
	\end{align*}
	By the coefficients of $p_i q_i$, for $i=1,2,3$, we immediately have that the diagonal entries of $AC$ in (\ref{relation_6}) are all zero.\ To see that the other entries of $AC$ are also all zero, we set equal the coefficients of $p_1q_2$ with $p_2q_1$, of $p_1q_3$ with $p_3q_1$, of $p_2q_3$ with $p_3q_2$, and use (\ref{relation_7}), which imply
	$$ ac_1 - ec_3 = -bc_1 - fc_2=0,\quad ac_2 + dc_3 = -fc_1 - cc_2=0,\quad dc_2 + bc_3 = ec_1 - cc_3=0. $$
	Hence $AC=0$ and thus $C=0$.\ By the coefficients of $p_1q_2$ and $p_2q_1$,
	$$ ab-d^2 = 1, $$
	which means that $a \neq 0$ and $b \neq 0$.\ In view of $A^2 = I_3$ in (\ref{relation_5}) the two equations $ad+bd+ef=0$ and $ae+df+ce=0$ together with the coefficients of $q_1q_2$ and $q_1q_3$ imply $ad=ae=0$.\ Since $a \neq 0$ we obtain
	$$ d=e=0 .$$
	Furthermore, by the coefficient of $p_1q_3$ we have $af=de=0$, hence $f=0$.\ Together with the coefficients from the second until the fourth lines, we obtain
	$$ ab=1,\quad a^2=b^2=c^2=1 ,$$
	which correspond to $\sigma$, $\overline{\sigma}$ and $\pm$id.\
	
	If $\rho$ is anti-symplectic, then $\rho(q,p)$ is of the form
	$$ \begin{pmatrix}
	a & d & e & 0 & 0 & 0\\
	d & b & f & 0 & 0 & 0\\
	e & f & c & 0 & 0 & 0\\
	c_1 & c_2 & c_3 & -a & -d & -e\\
	c_2 & c_4 & c_5 & -d & -b & -f\\
	c_3 & c_5 & c_6 & -e & -f & -c
	\end{pmatrix} \cdot \begin{pmatrix}
	q_1\\
	q_2\\
	q_3\\
	p_1\\
	p_2\\
	p_3
	\end{pmatrix} = \begin{pmatrix}
	aq_1 + dq_2 + eq_3\\
	dq_1 + bq_2 + fq_3\\
	eq_1 + fq_2 + cq_3\\
	c_1q_1 + c_2q_2 + c_3q_3 - ap_1 - dp_2 - ep_3\\
	c_2q_1 + c_4q_2 + c_5q_3 - dp_1 - bp_2 - fp_3\\
	c_3q_1 + c_5q_2 + c_6q_3 - ep_1 - fp_2 - cp_3
	\end{pmatrix}. $$
	The equation $AC=CA$ yields
	\begin{align*} \footnotesize
	\begin{pmatrix}
	ac_1 + dc_2 + ec_3 & ac_2 + dc_4 + ec_5 & ac_3 + dc_5 + ec_6\\
	dc_1 + bc_2 + fc_3 & dc_2 + bc_4 + fc_5 & dc_3 + bc_5 + fc_6\\
	ec_1 + fc_2 + cc_3 & ec_2 + fc_4 + cc_5 & ec_3 + fc_5 + cc_6
	\end{pmatrix} = \begin{pmatrix}
	ac_1 + dc_2 + ec_3 & dc_1 + bc_2 + fc_3 & ec_1 + fc_2 + cc_3\\
	ac_2 + dc_4 + ec_5 & dc_2 + bc_4 + fc_5 & ec_2 + fc_4 + cc_5\\
	ac_3 + dc_5 + ec_6 & dc_3 + bc_5 + fc_6 & ec_3 + fc_5 + cc_6
	\end{pmatrix}
	\end{align*}
	and thus \footnotesize{
	\begin{align*}
	ac_2 + dc_4 + ec_5 = dc_1 + bc_2 + fc_3,\quad ac_3 + dc_5 + ec_6 = ec_1 + fc_2 + cc_3,\quad dc_3 + bc_5 + fc_6 = ec_2 + fc_4 + cc_5.
	\end{align*} } \normalsize
	Now $H_2\big( \rho(q,p) \big)$ is
	\begin{align*}
	&\frac{1}{2}|p|^2 + p_1q_2\left(-ac_2 - dc_4 - ec_5 + d^2 - ab\right) - p_2q_1\left(dc_1 + bc_2 + fc_3 + d^2 - ab\right)\\
	& - q_1^2 \left(\frac{1}{2}\left(-c_1^2 -c_2^2 - d^2 - e^2 \right) - dc_1 + ac_2 + a^2 \right)\\
	& + \frac{1}{2}q_2^2 \left(c_2^2 + c_4^2 + c_5^2 + 2bc_2 - 2dc_4 - 2d^2 + b^2 + f^2\right)\\
	& + \frac{1}{2}q_3^2 \left(c_3^2 + c_5^2 + c_6^2 + 2fc_3 - 2ec_5 - 2e^2 + f^2 + c^2\right)\\
	& + p_1q_1 \left( -ac_1 - dc_2 - ec_3 \right) + p_2q_2 \left( -dc_2 - bc_4 - fc_5 \right) + p_3q_3 \left( -ec_3 - fc_5 - cc_6 \right)\\
	& + p_1q_3 \left(-ac_3 - dc_5 - ec_6 + de - af\right) + p_3q_1 \left( -ec_1 - fc_2 - cc_3 + af - de \right)\\
	& + p_2q_3 \left( -dc_3 - bc_5 - fc_6 + be - df \right) + p_3q_2 \left( -ec_2 - fc_4 - cc_5 + df - be \right)\\
	& + q_1q_2 \left( c_1c_2 + c_2c_4 + c_3c_5 + bc_1 - ac_4 - 2ad + bd + ef \right)\\
	& + q_1q_3 \left( c_1c_3 + c_2c_5 + c_3c_6 + fc_1 + dc_3 - ec_2 - ac_5 - 2ae + df + ce \right)\\
	& + q_2q_3 \left( c_2c_3 + c_4c_5 + c_5c_6  + fc_2 + bc_3 - ec_4 - dc_5 - 2de + bf + cf \right).
	\end{align*}
	In a similar way to the symplectic case we find that $C=0$.\ In view of $A^2 = I_3$ in (\ref{relation_5}) the three equations $ad+bd+ef=0$, $ae+df+ce=0$ and $de + bf + cf=0$ together with the coefficients of $q_1q_2$, $q_1q_3$ $q_2q_3$ imply $ad=ae=de=0$.\ Since the coefficients of $p_1q_3$ and $p_3q_1$ yield $de=af$, we have
	$$  ad=ae=af=0.$$
	Suppose that $a = 0$, then in view of the coefficients of $q_1^2$ we see that $d^2 + e^2 = -2$ which is a contradiction.\ Hence
	$$a \neq 0,\quad d=e=f=0.$$
	By the first four lines we obtain
	$$ ab = -1,\quad a^2 = b^2 = c^2 = 1,$$
	which correspond to $\rho_1,\rho_2,\overline{\rho_1}$ and $\overline{\rho_2}$.\
\end{proof}

\subsection{Symmetric periodic orbits via shooting}
\label{sec:symmetries}

\textbf{Planar periodic orbits:}\ The fixed point sets of (\ref{involution_1})
\begin{align*}
\text{Fix}(\rho_1) = \{ (q_1,0,0,p_2) \},\quad \text{Fix}(\rho_2) = \{ (0,q_2,p_1,0) \}
\end{align*}
are Lagrangian submanifolds.\ A planar orbit intersects Fix$(\rho_1)$ if it hits the $q_1$-axis perpendicularly, and it intersects Fix$(\rho_2)$ if it hits the $q_2$-axis perpendicularly.\ If an orbit starts at Fix$(\rho_1)$ perpendicularly and hits Fix$(\rho_2)$ perpendicularly in time $t$, then by using the double symmetry $\rho_1,\rho_2$, we obtain a periodic orbit with period $4t = T_q$.\ We refer to this kind of a planar periodic orbit as \textbf{doubly-symmetric}.\ If it starts at Fix($\rho_1$) or Fix($\rho_2$) and hits perpendicularly only the same fixed point set then we call it \textbf{simply-symmetric}.\ Note that a family of simply-symmetric periodic orbits is mapped to another such family by the other symmetry.\

Finding orbits in such a way is known as shooting method.\ In particular, George David Birkhoff (1884--1944) gave an analytical proof for the existence of a planar retrograde periodic orbit which is doubly symmetric (see Figure \ref{1_birkhoff}) for energy level sets below the unique critical value $- \frac{1}{2} 3^{4/3}$ by the so called ``Birkhoff's shooting method", see for instance \cite[pp.\ 140--144]{frauenfelder}.
\begin{figure}[H]
	\centering
	\begin{tikzpicture}[line cap=round,line join=round,>=triangle 45,x=1.0cm,y=1.0cm]
	\clip(-4.5,-2.35) rectangle (4.5,3);
	
	\draw [shift={(0.0,-0.0)},line width=1pt] [decoration={markings, mark=at position 0.82 with {\arrow{<}}}, postaction={decorate}] plot[domain=1.5707963267948966:3.141592653589793,variable=\t]({1.0*2.0*cos(\t r)+-0.0*2.0*sin(\t r)},{0.0*2.0*cos(\t r)+1.0*2.0*sin(\t r)});
	
	\draw [shift={(0.0,-0.0)},dashed,line width=1pt]  plot[domain=-3.141592653589793:1.5707963267948966,variable=\t]({1.0*2.0*cos(\t r)+-0.0*2.0*sin(\t r)},{0.0*2.0*cos(\t r)+1.0*2.0*sin(\t r)});
	
	\draw (-2.15,0.42) node[anchor=north west] {$\boxdot$};
	\draw [->,line width=1pt] (-3.0,0.0) -- (4.0,0.0);
	\draw (3.38,0.6) node[anchor=north west] {$q_1$};
	\draw [->,line width=1pt] (0.0,-2.56) -- (0.0,3);
	\draw (0.16,3.0400000000000005) node[anchor=north west] {$q_2$};
	\draw (-0.44,2.128) node[anchor=north west] {$\boxdot$};
	\draw (0,0) node[anchor=north west] {$e$};
	\begin{scriptsize}
	\draw [fill=black] (-2,0) circle (2pt);
	\draw [fill=black] (0,0) circle (2pt);
	\end{scriptsize}
	\end{tikzpicture}
	\caption{Birkhoff's shooting method for the retrograde periodic orbit}
	\label{1_birkhoff}
\end{figure}
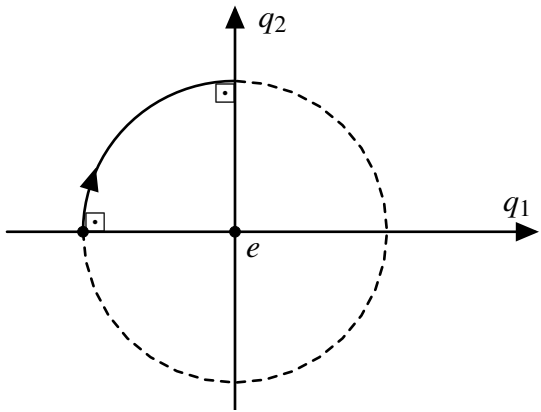
\noindent
Note that in PCR3BP the masses of the sun and the earth are comparable, so its Hamiltonian (\ref{hamiltonian_1}) is only invariant under $\rho_1$.\ In this case Birkhoff used a double shooting argument (see  \cite[pp.\ 144--146]{frauenfelder}).\

Our moon, however, is close to a direct periodic orbit, and for such orbits Birkhoff's analytical proof does not work.\ In 1915 Birkhoff \cite{birkhoff} conjectured that the retrograde one spans a disk-like global surface of section.\ By applying Brouwer's translation theorem to its Poincaré return map one can find a fixed point which should correspond to a direct one.\ This conjecture is still an open question and nowadays modern methods of symplectic geometry are developed to find an answer.\ We refer to \cite{frauenfelder} for a profound discussion.\

Nevertheless, his shooting method can be applied to find and study numerically periodic orbits.\ For direct periodic orbits, we start shooting upwards right from the origin.\

Note that in this paper we work with symplectic $(q,p)$-coordinates, hence the initial conditions for Hill's equation in $(q,p)$-coordinates (\ref{equation_of_motion_q_p}) are determined as follows:\

Since we start perpendicularly at Fix($\rho_1$), we need to give the position and its momentum at time zero, i.e.\ $q_1(0)$ and $p_2(0)$.\ Its velocity $\dot{q}_2(0)$ is determined in the following way.\ Consider the traditional Jacobi integral $\Gamma := -2c$ in $(q,\dot{q})$-coordiantes, where $c$ is the energy from the Hamiltonian (\ref{hamiltonian_2}), meaning that
\begin{align} \label{energy_gamma}
\Gamma = \frac{2}{|q|} + 3q_1^2 - \dot{q}_1^2 - \dot{q}_2^2.
\end{align}
If the starting point $q_1(0)$ and the energy $\Gamma$ are given, then the starting velocity $\dot{q}_2(0)$ is given by the energy condition (\ref{energy_gamma}), i.e.\
\begin{align*}
\big( \dot{q}_2(0) \big)^2 = \frac{2}{|q_1(0)|} + 3\big( q_1(0) \big)^2 - \Gamma
\end{align*}
and therefore in view of the transformation (\ref{equation_of_motion_q_p}), the momentum by
$$p_2(0) = \dot{q}_2(0) + q_1(0).$$
\textbf{Spatial periodic orbits:}\ Different combinations of the fixed point sets of the four linear anti-symplectic symmetries give various doubly and simply symmetric spatial periodic orbits.\ For their initial conditions we consider the Jacobi integral $\Gamma$ in $(q,\dot{q})$-coordiantes in the spatial case which is
\begin{align} \label{energy_gamma_spatial}
\Gamma = \frac{2}{|q|} + 3q_1^2 - q_3^2 - \dot{q}_1^2 - \dot{q}_2^2 - \dot{q}_3^2.
\end{align}
The initial conditions for position and velocity are determined as follows:\
\begin{table}[H] \centering \footnotesize
	\setlength{\extrarowheight}{.5em}
	\begin{tabular}{c|c|c}
		starting perpendicularly & given $\Gamma$ and & by energy condition (\ref{energy_gamma_spatial})\\
		\hline $\text{Fix}(\overline{\rho_1}) = \{ (q_1,0,0,0,p_2,p_3) \}$ & $q_1(0)$, $\dot{q}_2(0)$  & $\big( \dot{q}_3(0) \big)^2 = \frac{2}{|q_1(0)|} + 3\big( q_1(0) \big)^2 - \big( \dot{q}_2(0) \big)^2 - \Gamma$ \\
		$\text{Fix}(\rho_1) = \{ (q_1,0,q_3,0,p_2,0) \}$ & $q_1(0)$, $q_3(0)$  & $\big( \dot{q}_2(0) \big)^2 = 2 / \sqrt{ \big( q_1(0) \big)^2 + \big( q_3(0) \big)^2 } + 3\big( q_1(0) \big)^2 - \big( q_3(0) \big)^2 - \Gamma$ \\
		$\text{Fix}(\overline{\rho_2}) = \{ (0,q_2,0,p_1,0,p_3) \}$ & $q_2(0)$, $\dot{q}_1(0)$  & $\big( \dot{q}_3(0) \big)^2 = \frac{2}{|q_2(0)|} - \big( \dot{q}_1(0) \big)^2 - \Gamma$ \\
		$\text{Fix}(\rho_2) = \{ (0,q_2,q_3,p_1,0,0) \}$ & $q_2(0)$, $q_3(0)$  & $\big( \dot{q}_1(0) \big)^2 = 2 / \sqrt{ \big( q_2(0) \big)^2 + \big( q_3(0) \big)^2 } - \big( q_3(0) \big)^2 - \Gamma$
	\end{tabular}
	\caption{Initial conditions for symmetric spatial periodic orbits}
	\label{initial_conditions_spatial}
\end{table}
\noindent
In view of the transformations (\ref{equation_of_motion_q_p}), the $p$-coordinates are obtained by
$$ p_1(0) = \dot{q}_1(0) - q_2(0),\quad p_2(0) = \dot{q}_2(0) + q_1(0),\quad p_3(0) = \dot{q}_3(0). $$

\subsection{On monodromy \& reduced monodromy for symmetric periodic orbits}
\label{sec:6.5}

In this subsection we apply the discussion from Section \ref{sec:4} to planar and spatial symmetric periodic orbits.\

\subsubsection{For planar ones \& the synodic, anomalistic and draconitic periods}
\label{sec:6.5.1}

Recall from (\ref{symplectic_involution}) and (\ref{involution_2}) the symplectic and anti-symplectic symmetries
$$ \sigma(q,p) = (q_1,q_2,-q_3,p_1,p_2,-p_3),\quad \rho_1 (q,p) = (q_1,-q_2,q_3,-p_1,p_2,-p_3), $$
which commute.\ In the following we focus on planar periodic orbits $q:=(q,p) \in \text{Fix}(\sigma)$ which are symmetric with respect to $\rho_1$.\ Note that the discussion of symmetric planar periodic orbits with respect to $\rho_2$ from (\ref{involution_2}) works analogously.\

Let $T_q$ be the first return time.\ We denote $q_0:=\big( q(0),p(0) \big)$ and after choosing a Lagrangian basis with respect to the symplectic decomposition into planar and spatial components,
$$ T_{q_0} T^*\mathbb{R}^3 = T_{q_0} \text{Fix}(\sigma) \oplus E_{-1}\big(d \sigma (q_0) \big), $$
the monodromy is of the form
$$ d\varphi _H ^{T_q} (q_0) = \begin{pmatrix}
A_p & 0\\
0 & A_s
\end{pmatrix}, $$
where
$$ A_p \colon T_{q_0} \text{Fix}(\sigma) \to T_{q_0} \text{Fix}(\sigma),\quad A_s \colon E_{-1}\big( d\sigma(q_0) \big) \to E_{-1}\big( d\sigma(q_0) \big) $$
and
$$ A_p \in \text{Sp}^{\rho_0}(2),\quad A_s \in \text{Sp}^{\rho_0}(1) .$$
For the planar monodromy $A_p$ we can ignore the spatial components and for $A_s$ the planar ones, such that in view of
$$ \text{Fix}(\rho_1) = \{ (q_1,0,q_3,0,p_2,0) \},\quad E_{-1} \big( d\rho_1 (q_0) \big) = \{ (0,q_2,0,p_1,0,p_3) \} ,$$
Lagrangian basis vectors for $ T_{q_0} \text{Fix}(\sigma)$ and $E_{-1}\big(d \sigma (q_0) \big)$ are given by
$$ \left\{ \begin{pmatrix}
1\\0\\0\\0
\end{pmatrix}, \begin{pmatrix}
0\\0\\0\\1
\end{pmatrix}, \begin{pmatrix}
0\\0\\-1\\0
\end{pmatrix}, \begin{pmatrix}
0\\1\\0\\0
\end{pmatrix} \right\} \text{ and } \left\{ \begin{pmatrix}
1\\0
\end{pmatrix}, \begin{pmatrix}
0\\-1
\end{pmatrix} \right\}. $$
For computing the matrices $A_p$ and $A_s$ we consider the linearized equations in $(q,p)$-coordinates for planar periodic orbits (\ref{linearized_q_p_fix}).\ In particular, for computing $A_s$, which is of the form
$$ A_s =  \begin{pmatrix}
\tilde{a} & \tilde{b}\\
\tilde{c} & \tilde{a}
\end{pmatrix},\quad \tilde{a}^2 - \tilde{b}\tilde{c} = 1, $$
we only need to consider the third line of (\ref{linearized_q_p_fix}).\ This implies the system of linear equations of the form
$$ d\varphi_H^{T_q}\big( (1,0) \big) = \begin{pmatrix}
\tilde{a}\\
\tilde{c}
\end{pmatrix} = \tilde{a} \begin{pmatrix}
1\\
0
\end{pmatrix} - \tilde{c} \begin{pmatrix}
0\\
-1
\end{pmatrix},\quad d\varphi_H^{T_q}\big((0,-1)\big) = \begin{pmatrix}
\tilde{b}\\
\tilde{a}
\end{pmatrix} = \tilde{b} \begin{pmatrix}
1\\
0
\end{pmatrix} - \tilde{a} \begin{pmatrix}
0\\
-1
\end{pmatrix}. $$
For the planar monodromy $A_p$ we consider the first two lines of (\ref{linearized_q_p_fix}) and obtain four vectors by the linearized flow to each of the four Lagrangian basis vectors.\ Since these four vectors are linear combinations of the four Lagrangian basis vectors, the coefficients of these linear combinations give the matrix coefficients of $A_p \in \text{Sp}^{\rho_0}(2)$.\

For the planar reduced monodromy $\overline{A}_p \in \text{Sp}^{\rho_1}(1)$, the restriction to the energy hypersurface $\Sigma$ means to linearize the planar Jacobi integral $\Gamma$ from (\ref{energy_gamma}) in $(q,p)$-coordinates, which is equivalent to
$$ \Delta p_2(0) = \frac{1}{p_2(0) - q_1(0)} \bigg( \frac{- q_1(0)}{|q|^3} + 2q_1(0) + p_2(0) \bigg)\Delta q_1(0). $$
Therefore $\Delta q_1(0)$ and $\Delta p_2(0)$ determine each other such that a first basis vector of
$$  T_{q_0} \text{Fix}(\sigma|_{\Sigma}) / (\text{ker}\omega_{q_0} | _{T_{q_0}\Sigma}) $$
is given by
$$ \tilde{v}:= (\Delta q_1(0),0,0,\Delta p_2(0)). $$
The Hamiltonian vector field $X_H \vert _{\Sigma} (q_0) \in E_{-1}\big(d \sigma (q_0) \big) = \langle (0,0,-1,0),(0,1,0,0) \rangle_{\mathbb{R}} $ is of the form
$$ \big( 0, \Delta q_2(0), \Delta p_1(0), 0 \big) = \big( 0, \dot{q}_2(0), \dot{p}_1(0), 0 \big), $$
which is determined by the inital conditions and Hill's equation in $(q,p)$-coordinates (\ref{equation_of_motion_q_p}).\ For simplicity we choose $\Delta q_1(0)=1$ such that the second basis vector is given by
$$\tilde{w} := (0,0,-1,0). $$
Note that
$$T_{q_0} \text{Fix}(\sigma|_{\Sigma}) = \langle \tilde{v},\tilde{w},X_H \vert _{\Sigma}(q_0) \rangle_{\mathbb{R}}$$
and
$$\omega_{q_0} (\tilde{v},\tilde{w})=1,\quad \omega_{q_0} \big( \tilde{v} , X_H \vert _{\Sigma}(q_0) \big) = 0,\quad \omega_{q_0}\big( \tilde{w} , X_H \vert _{\Sigma}(q_0) \big) = 0.$$
By writing the two vectors
$$ d \varphi_H^{T_q} (\tilde{v}),\quad d \varphi_H^{T_q} (\tilde{w}) $$
as a linear combination of $\tilde{v},\tilde{w}$ and $X_H \vert _{\Sigma} (q_0)$, the coefficients of $\tilde{v}$ and $\tilde{w}$ give the planar reduced monodromy $\overline{A}_p \in \text{Sp}^{\rho_0}(1)$, i.e.\ the reduced monodromy is of the form
$$ \overline{d\varphi _H ^{T_q} | _{\Sigma} (q_0)} = \begin{pmatrix}
\overline{A}_p & 0\\
0 & A_s
\end{pmatrix} = \begin{pmatrix}
a & b & 0 & 0\\
c & a & 0 & 0\\
0 & 0 & \tilde{a} & \tilde{b}\\
0 & 0 & \tilde{c} & \tilde{a}
\end{pmatrix},$$
where $a^2 - bc = 1$ as well.\ Recall from Proposition \ref{prop_signature} that the signatures of $b,c,\tilde{b}$ and $\tilde{c}$ are invariant under the choice of the Lagrangian basis.\ Moreover, for a symmetric planar periodic orbit we separate the elliptic case and each of the four hyerbolic subcases in planar and spatial parts.\

In both elliptic cases, i.e.\ the Floquet multipliers are of the form $e^{\pm \text{i}\theta}$ resp. $e^{\pm \text{i}\vartheta}$, we denote by
$$ \varphi_p := \begin{cases}
\theta, & \text{if } b < 0,\\
2\pi - \theta, & \text{if } b >0,
\end{cases}\quad\quad \varphi_s := \begin{cases}
\vartheta, & \text{if } \tilde{b} < 0,\\
2\pi - \vartheta, & \text{if } \tilde{b} >0,
\end{cases} $$
the \textbf{mean planar angle of rotation} and the \textbf{mean spatial angle of rotation}, respectively.\
In particular, we have
$$ \overline{d \varphi _H^{T_q} | _{\Sigma} (q_0) } = \begin{pmatrix}
\overline{A}_p & 0 \\
0 & A_s
\end{pmatrix} \sim \begin{pmatrix}
\cos \varphi_p & - \sin \varphi_p & 0 & 0\\
\sin \varphi_p & \cos \varphi_p & 0 & 0\\
0 & 0 & \cos \varphi_s & - \sin \varphi_s\\
0 & 0 & \sin \varphi_s & \cos \varphi_s
\end{pmatrix}. $$
With respect to the symplectic splitting the transversal Conley--Zehnder index $\mu_{CZ}$ of $q$ is additive, meaning that
\begin{align*}
\mu_{CZ} = \mu_{CZ}^p + \mu_{CZ}^s,
\end{align*}
where the planar and spatial indices are the Conley--Zehnder indices of any path of symplectic matrices generated by the planar and spatial part of the linearized flow, respectively.\ Furthermore recall from Section \ref{sec:conley_zehnder_index} that the measure of the number of complete rotations of each neighbouring orbit during $T_q$ is given by the Conley--Zehnder index.\

We denote by $\theta(t)$ and $\vartheta(t)$ the respective rotation functions for $t \in [0,T_q]$.\ Note that $\theta(T_q) \equiv \varphi_p \text{ mod } 2 \pi$ and $\vartheta(T_q) \equiv \varphi_s \text{ mod } 2 \pi$.\ Moreover, let $\text{rot}^p(q)$ resp.\ $\text{rot}^s(q)$ be the number of complete rotations of neighbouring orbits during $T_q$.\ As in (\ref{index_1}) and (\ref{index_2}),
$$ \text{rot}^p(q) := \lfloor \theta(T_q)/(2\pi) \rfloor \in \mathbb{Z}, \quad \text{rot}^s(q) := \lfloor \vartheta(T_q)/(2\pi) \rfloor \in \mathbb{Z} $$
and
$$ \mu_{CZ}^p = 2 \lfloor \theta(T_q)/(2\pi) \rfloor + 1,\quad \mu_{CZ}^s = 2 \lfloor \vartheta(T_q)/(2\pi) \rfloor + 1. $$

The \textbf{synodic period} $T_s$ is related to $T_q$ in the following way.\ Recall that $T_s$ is the number of days from full moon to full moon.\ The lunarity $m$ of a moon is defined as the average number of synodic months during a complete rotation of the planet around the sun, which is 365.25 days for the Earth.\ In the Hill lunar problem, since the earth and sun are fixed, we have $m= 2 \pi / T_q$, hence
\begin{align}\label{synodic_period}
T_s = \frac{365.25}{m} = \frac{365.25 \cdot T_q}{2 \pi}.
\end{align}

The \textbf{anomalistic period} $T_a$ is defined as the mean time (in days) for a complete rotation of the planar neighbouring orbits during $T_q$.\ The mean angular velocity $v_a$ of the anomaly is given by $ \theta(T_q) / T_s$ and we obtain $T_a$ by $v_a \cdot T_a = 2 \pi$, i.e.\
\begin{align} \label{anomalistic_period}
T_a = \frac{2 \pi \cdot T_s}{( \mu_{CZ}^p - 1 )\cdot \pi + \varphi_p}.
\end{align}

In a similar way the \textbf{draconitic period} $T_d$ is defined as the mean time (in days) for a complete rotation of the spatial neighbouring orbits during $T_q$, and
\begin{align} \label{draconitic_period}
T_d = \frac{2 \pi \cdot T_s}{( \mu_{CZ}^s - 1 )\cdot \pi + \varphi_s}.
\end{align}
Note that $T_a$ and $T_d$ only exist in elliptic cases.\ In addition, we note that the bounded component of the regularized Hill lunar problem has a contact structure for energies below the critical value, where the Conley--Zehnder indices of closed characteristics are nonnegative (see \cite{lee} for details).\

\subsubsection{For spatial periodic orbits}
\label{sec:6.5.2}

We recall the four linear anti-symplectic symmetries
\begin{align*}
&\overline{\rho_1}(q,p) = (q_1,-q_2,-q_3,-p_1,p_2,p_3),\quad &\rho_1(q,p) = (q_1,-q_2,q_3,-p_1,p_2,-p_3),\\
&\overline{\rho_2}(q,p) = (-q_1,q_2,-q_3,p_1,-p_2,p_3),\quad &\rho_2(q,p) = (-q_1,q_2,q_3,p_1,-p_2,-p_3),
\end{align*}
and consider spatial periodic orbits $q:=(q,p)$ which are symmetric with respect to $\overline{\rho_1}$.\ Note that for any other linear anti-symplectic symmetry the following procedure is analogous.\ Let $q_0 := \big( q(0),p(0) \big)$ and $T_q$ be its first return time.\ In view of
$$ q_0 \in \text{Fix}(\overline{\rho_1}) = \{ (q_1,0,0,0,p_2,p_3) \},\quad E_{-1} \big( d\overline{\rho_1} (q_0) \big) = \{ (0,q_2,q_3,p_1,0,0) \} ,$$
a Lagrangian basis is given by
$$ \left\{ \begin{pmatrix}
	1\\0\\0\\0\\0\\0
	\end{pmatrix}, \begin{pmatrix}
	0\\0\\0\\0\\1\\0
	\end{pmatrix}, \begin{pmatrix}
	0\\0\\0\\0\\0\\1
	\end{pmatrix}, \begin{pmatrix}
	0\\0\\0\\-1\\0\\0
	\end{pmatrix}, \begin{pmatrix}
	0\\1\\0\\0\\0\\0
	\end{pmatrix}, \begin{pmatrix}
	0\\0\\1\\0\\0\\0
	\end{pmatrix} \right\} .$$
By using this basis the monodromy becomes an element of $\text{Sp}^{\rho_0}(3)$.\ By fixing the energy, the linearization of the spatial Jacobi integral $\Gamma$ from (\ref{energy_gamma_spatial}) in $(q,p)$-coordinates is equivalent to
$$ \Delta p_3(0) = \frac{1}{p_3(0)} \Bigg( \bigg( \frac{- q_1(0)}{|q|^3} + 2q_1(0) + p_2(0) \bigg)\Delta q_1(0) + \big(q_1(0) - p_2(0)\big) \Delta p_2(0) \Bigg) , $$
which means that any two of $\Delta p_3(0), \Delta q_1(0)$ and $\Delta p_2(0)$ determine the thirds.\ Hence two basis vectors of the five dimensional vector space $T_{q_0} \Sigma$ are of the form
$$ \big( \Delta q_1(0), 0, 0, 0, \Delta p_2(0), \Delta p_3(0) \big). $$
For simplicity we choose
$$ \tilde{v}_1 = \big( 1, 0, 0, 0, 0, \Delta p_3(0) \big),\quad \tilde{v}_2 = \big( 0, 0, 0, 0, 1, \Delta p_3(0) \big). $$
The Hamiltonian vector field $X_H |_{\Sigma}(q_0) \in E_{-1} \big( d\overline{\rho_1} (q_0) \big)$, which is spanned by the last three Lagrangian basis vectors, is of the form
$$ \big( 0,\Delta q_2(0), \Delta q_3(0), \Delta p_1(0),0,0 \big) =  \big( 0, \dot{q}_2(0), \dot{q}_3(0), \dot{p}_1(0), 0, 0 \big),$$
determined by the inital conditions and Hill's equation in $(q,p)$-coordinates (\ref{equation_of_motion_q_p}).\ By choosing
$$ \tilde{w}_1 = (0,0,0,-1,0,0),\quad \tilde{w}_2 = (0,1,0,0,0,0) $$
we obtain
$$ T_{q_0}\Sigma = \langle \tilde{v}_1,\tilde{v}_2,\tilde{w}_1,\tilde{w}_2,X_H |_{\Sigma}(q_0)\rangle_{\mathbb{R}},\quad \omega_{x_0}(\tilde{v}_i,\tilde{w}_j) = \delta_{ij},\quad i,j=1,2, $$
and that the reduced monodromy is a symplectic matrix from $\text{Sp}^{\rho_0}(2)$, meaning that it is of the form
$$ \overline{d \varphi_H^{T_q} (q_0)} = \begin{pmatrix}
A & B\\
C & A^T
\end{pmatrix},\quad B, C, CA, AB \text{ are symmetric and } A^2 - BC = I_2.$$

\section{Proof of Theorem \ref{theorem_a} and \ref{theorem_b}}
\label{sec:6} 

\subsection{Regularized energy hypersurface}

\textbf{Planar case:}\ To avoid collision orbits, for given energy $c<0$, we regularize the Hamiltonian of the planar Hill lunar problem (\ref{hamiltonian_hill}) by
\begin{align*}
K_c (-p,q) &= \frac{1}{2} \bigg( - \frac{|q|}{2c} \Big( H\big( - \frac{q}{2c}, \sqrt{-2c}p \big) - c \Big) + 1 \bigg)^2 - \frac{1}{2}\\
&= \frac{1}{2} \bigg( \frac{1}{2} \big( 1 + |p|^2 \big) + \frac{p_1 q_2 - p_2 q_1}{(-2c)^{\frac{3}{2}}} + \frac{- q_1^2 + \frac{1}{2}q_2^2}{(-2c)^3} \bigg)^2 |q|^2 - \frac{1}{2} \nonumber
\end{align*}
which satisfies $\Sigma_c:=H^{-1}(c) = K_c^{-1}(0)$.\ Note that the transformation consists of the switch map
\begin{align} \label{switch_map}
(-p,q) \mapsto (q,p),
\end{align}
which is a linear symplectomorphism on $\mathbb{R}^4$, and $ (q,p) \mapsto ( - \frac{q}{2c} , \sqrt{-2c} p )$, which is conformally symplectic with its conformal factor $\frac{1}{\sqrt{-2c}}$.\ In symplectic geometry, we don't want to change the dynamics and indeed this factor only gives rise to a reparametrization of the Hamiltonian flow.\ Moreover, the original angular momentum (\ref{angular_momentum_original}) is invariant under the switch map (\ref{switch_map}), i.e.\
\begin{align} \label{angular_momentum_invariant}
L(q,p) = p_1 q_2 - p_2 q_1 = L(-p,q).
\end{align}
Now the roles of the positions $q$ and momenta $p$ are switched and $-p$ corresponds to the base coordinate and $q$ to the fiber coordinate.\ By thinking of the two-sphere as $S^2 = \mathbb{R}^2 \cup \{ \infty \}$ via stereographic projection from the north pole $N$, we obtain the inclusion
\begin{align} \label{inclusion}
\iota \colon T^* \mathbb{R}^2 \hookrightarrow T^* S^2,
\end{align}
where by adding $N$ the closure of the regularized energy hypersurface $\overline{\iota(\Sigma_c)}$ is a subset of $T^* S^2$.\

Furthermore, for every $c<0$, $K_c$ smoothly extends to a Hamiltonian on $T^* S^2$, and by abuse of notation we denote this canonical smooth extension by the same letter.\ To study the limit case $c \to - \infty$, we replace the energy parameter by $c = \frac{-1}{2r^{2/3}}$, for a homotopy variable $r \in (0,\infty)$.\ Hence we obtain
\begin{align} \label{regularized_hamiltonian}
K_r(-p,q) := \frac{1}{2} \Big( \frac{1}{2} \big( 1 + |p|^2 \big) + (p_1 q_2 - p_2 q_1)r + (- q_1^2 + \frac{1}{2}q_2^2)r^2 \Big)^2 |q|^2 - \frac{1}{2}.
\end{align}
The Hamiltonian (\ref{regularized_hamiltonian}) smoothly extends to $r=0$, where it becomes just the regularized Kepler Hamiltonian
\begin{align*}
K_0 (-p,q) = \frac{1}{2} \Big( \frac{1}{2}(1 + |p|^2) \Big)^2|q|^2 - \frac{1}{2},
\end{align*}
which is the kinetic energy of the ``momentum" $q$ with respect to the round metric on $S^2$, i.e.\ $ g_{ij} = \frac{4 \delta_{ij}}{( 1 + |p|^2 )^2}$.\ In general, the regularization of collision orbits for the Kepler problem in $\mathbb{R}^n$ goes back to Moser \cite{moser}, where the Kepler flow is just the geodesic flow on $S^n$ in the chart given by stereographic projection from $N$.\ To regularize the flow means to add the fiber over $N$.\ Going through $N$ (the point at infinity) corresponds precisely to the collisions, where $p$ explodes.\ In the original picture: By regularizing, the third body moves into the mass at the origin like falling onto some kind of trampoline, i.e.\ at collision it bounces back.\ Hence, these bounce orbits become periodic according to the periodic geodesic flows on $S^2$.\
	
Note that if $K(q,p)$ is the Hamiltonian of the Kepler problem, then on the energy hypersurface $K_0^{-1}(0) = K^{-1}(0)$ the Hamiltonian vector fields are related by
\begin{align} \label{two_hamilt_2}
X_{K_0}|_{K^{-1}(0)}(p,q) = |q| X_K |_{K^{-1}(0)}(-q,p),
\end{align}
see for instance \cite[pp.\ 47--48]{frauenfelder}.\

Now by adding $N$, the closure of the regularized energy hypersurface $\Sigma := K_0^{-1}(0)$, which is the image of $\Sigma$ under the embedding (\ref{inclusion}), i.e.\
$$ \overline{\iota(\Sigma)} \subset T^*S^2 $$
corresponds to the space of great circles on $S^2$ parametrized by arc length, i.e.\ the space of oriented simple closed geodesics on $S^2$.\ Topologically, it is the set of pairs $(-p,q)$ formed by base points $-p$ and unit co-vectors $q$ at $-p$ (see Figure \ref{figure_geodesic}), meaning that it is diffeomorphic to the unit cotangent bundle of $S^2$,
$$ \overline{\iota(\Sigma)} \cong S^*S^2 = \{ (-p,q) : |q|_{-p} = 1 \}, $$
where $|q|_{-p}$ is the length of $q$ with respect to the round co-metric on $S^2$.
\begin{figure}[H]
	\centering
	\begin{tikzpicture}[line cap=round,line join=round,>=triangle 45,x=1cm,y=1cm]
	\clip(-2.2,-2.05) rectangle (2.2,2.05);
	
	\draw [line width=1pt] (0,0) circle (2cm);
	
	\draw [dashed,line width=1pt] (2,0) arc (0:180:2 and 0.5);
	
	\draw [decoration={markings, mark=at position 0.5 with {\arrow{>}}}, postaction={decorate},line width=1pt] (-2,0) arc (180:360:2 and 0.5);
	
	\draw [rotate=60,color=blue,dashed,line width=1pt] (2,0) arc (0:180:2 and 0.5);
	
	\draw [rotate=60,color=blue,decoration={markings, mark=at position 0.8 with {\arrow{>}}}, postaction={decorate},line width=1pt] (-2,0) arc (180:360:2 and 0.5);
	
	\draw [color=blue](-0.5,-1.4) node[anchor=north west] {$-p$};
	\draw [color=blue](0.2,-0.8) node[anchor=north west] {$q$};
	
	\draw [->,color=blue,line width=1pt] (-0.46,-1.46) -- (0.3,-0.85);
	
	\begin{scriptsize}
	\draw [fill=blue] (-0.46,-1.46) circle (2pt);
	\end{scriptsize}
	\end{tikzpicture}
	\caption{The space of parametrized great circles on $S^2$}
	\label{figure_geodesic}
\end{figure}
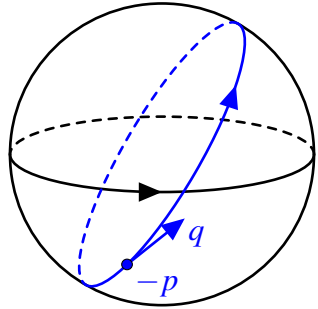
\noindent
By identifying it with the unit tangent bundle, denoted by $SS^2$, it is diffeomorphic to
$$ \overline{\iota(\Sigma)} \cong S^*S^2 \cong SS^2 \cong SO(3) \cong S^3 / \mathbb{Z}_2 \cong \mathbb{R}P^3,$$
see \cite[p.\ 195]{bott_tu}.\\
\\
\textbf{Spatial case:}\ By the same regularization of the Hamiltonian for the spatial Hill lunar problem (\ref{hamiltonian_hill}), the closure of the regularized energy hypersurface is the unit cotangent bundle $S^*S^3$ which is the space of parametrized great circles (parametrized simple closed geodesics) on $S^3$.\ To see that the tangent bundle of $S^3$ is trivial, we identify $\mathbb{R}^4$ with the set of quaternions $\mathbb{H}$ via the canonical bijection
$$ \mathbb{R}^4 \to \mathbb{H},\quad (a,b,c,d) \mapsto a + bi + cj + dk ,$$
with the identities
$$ i^2 = j^2 = k^2 = ikj = -1,\quad ij=k, jk=i, ki=j,\quad ji=-k, kj=-i, ik=-j.  $$
Consider $S^3$ as the unit quaternions, i.e.\
$$ S^3 = \{ a + bi + cj + dk \mid \sqrt{ a^2 + b^2 + c^2 + d^2 } = 1 \}. $$
For $q = a + bi + cj + dk \in S^3$ we have
\begin{align*}
(a,b,c,d) \mapsto (-b,a,-d,c), \quad \quad \quad & q \mapsto iq,\\
(a,b,c,d) \mapsto (-c,d,a,-b), \quad \quad \quad & q \mapsto jq,\\
(a,b,c,d) \mapsto (-d,-c,b,a), \quad \quad \quad & q \mapsto kq.
\end{align*}
It is obvious that the three resulting vectors are orthogonal to each other and to $(a,b,c,d)$, hence they are three linearly independent non vanishing tangent vector fields on $S^3$ and form an orthonormal basis on $T_qS^3$.\ Therefore $TS^3 \cong S^3 \times \mathbb{R}^3 $ and $SS^3 \cong S^3 \times S^2$.\

\subsection{From the Rabinowitz action functional to Morse--Bott}

\subsubsection{Rabinowitz action functional}

\textbf{Planar case:}\ We interpret variationally parametrized periodic orbits of $K_r$ (\ref{regularized_hamiltonian}) as critical points of the Rabinowitz action functional given by
\begin{align} \label{rabinowitz_action_functional}
\mathscr{A}_r := \mathscr{A}^{K_r} \colon \mathcal{L} \times \mathbb{R}_{>0} \to \mathbb{R},\quad (\gamma,\tau ) \mapsto \int_{S^1} \lambda\big( \dot{\gamma}(t) \big) - \tau \int_{S^1} K_r\big(\gamma(t)\big) dt,
\end{align}
where $\mathcal{L} = C^{\infty}(S^1 = \mathbb{R} / \mathbb{Z},T^* S^2)$ is the free loop space of $T^* S^2$, $\lambda$ is the canonical Liouville one-form and where $\tau$ can be regarded as the period of $\gamma$.\ This functional can be thought as the Lagrange multiplier functional of the area functional for the constraint given by the mean value of $K_r$.\

For the critical points of $\mathscr{A}_r$ let $(\gamma, \tau) \in \mathcal{L} \times \mathbb{R}_{>0}$ and $( \hat{\gamma}_1, \hat{\tau}_1 ), (\hat{\gamma}_2, \hat{\tau}_2) \in T_{(\gamma,\tau)} (\mathcal{L} \times \mathbb{R}_{>0})$ two tangent vectors.\ For the gradient of $\mathscr{A}_r$ we denote by $g$ the metric on $\mathcal{L} \times \mathbb{R}_{>0}$ which is defined as the product metric of the $L^2$-metric on $\mathcal{L}$ and the standard metric on $\mathbb{R}_{>0}$.\ For the $L^2$-metric we choose a smooth family $\{J_t\}_{t \in S^1}$ of $\omega$-compatible almost complex structures on $T^* S^2$, meaning that $J_t$ is an automorphism of $T_{\gamma} T^* S^2$ with $J_t^2 = - \text{id}$ and $\omega ( \cdot , J_t \cdot) $ defines a Riemannian metric.\ Then the metric $g$ is given by
$$ g \big( (\hat{\gamma}_1, \hat{\tau}_1 ), (\hat{\gamma}_2, \hat{\tau}_2 ) \big) = \langle \hat{\gamma}_1 , \hat{\gamma}_2 \rangle_{L^2} + \hat{\tau}_1 \hat{\tau}_2 = \int_{S^1} \omega \big( \hat{\gamma}_1 , J_t(\gamma) \hat{\gamma}_2 \big) dt + \hat{\tau}_1 \hat{\tau}_2$$
and with respect to this metric the gradient of $\mathscr{A}_r$ reads
\begin{align} \label{gradient_rabinowitz}
\nabla_g \mathscr{A}_r (\gamma,\tau) = \begin{pmatrix}
- J_t (\gamma) \left( \dot{\gamma}(t) - \tau X_{K_r} \big(\gamma(t)\big) \right)\\
- \int_{S^1} K_r\big(\gamma (t) \big) dt
\end{pmatrix} .
\end{align}
Hence critical points of $\mathscr{A}_r$ consist of pairs $(\gamma,\tau)$ which are solutions of
$$ \begin{cases}
\dot{\gamma}(t) = \tau X_{K_r}\big( \gamma(t) \big)\\
0 = \int_{S^1} K_r\big( \gamma(t) \big) dt
\end{cases} \Leftrightarrow\quad \begin{cases}
\dot{\gamma}(t) = \tau X_{K_r}\big( \gamma(t) \big)\\
0 = K_r\big(\gamma(t)\big),
\end{cases} $$
where the equivalence follows from preservation of energy.\ In other words, the critical points of $\mathscr{A}_r$ are parametrized periodic orbits of $X_{K_r}$ of period $\tau$, i.e.\ of the form $\gamma(t) = \varphi_{K_r}^{\tau t}$, $t \in \mathbb{R}$, on the fixed energy level set $K_r^{-1}(0)$.\

The circle $S^1$ acts on $\mathcal{L}$ by rotating the loop $\gamma$.\ This $S^1$-action extends to an action on $\mathcal{L} \times \mathbb{R}_{>0}$, where $S^1$ acts trivially on $\mathbb{R}_{>0}$, that is
\begin{align} \label{circle_action}
s_* \big( \gamma(t), \tau \big) = \big( \gamma (t + s), \tau \big),\quad t,s \in S^1.
\end{align}
Note that for every $r$ the Rabinowitz action functional $\mathscr{A}_r$ is invariant under (\ref{circle_action}).\ By the Morse Lemma (see for instance \cite[pp.\ 6--8]{milnor}) non-degenerate critical points are isolated.\ Since the critical points of  $\mathscr{A}_r$ come in $S^1$-families, $\mathscr{A}_r$ is never a Morse function, i.e.\ the kernel of its Hessian at a critical point is never just the zero vector space.\ Further, the geodesic flow on $S^2$ is invariant under rotation, thus closed geodesics are not isolated.\

If $\gamma : S^1 \to S^* S^2$ is a parametrized periodic orbit of $X_{K_0}$ corresponding to a simple closed geodesic on $S^2$, then $ (\gamma,2\pi) \in \text{crit}\mathscr{A}_0$.\ We denote the space of all these pairs by
\begin{align} \label{c_critical_set}
C:= \left\{ (\gamma, 2\pi) \mid \dot{\gamma}(t) = 2\pi X_{K_0}\big( \gamma(t) \big), 0 = K_0\big( \gamma(t) \big) \right\} \cong \mathbb{R}P^3 \subset \text{crit}\mathscr{A}_0,
\end{align}
on which the circle $S^1$ also acts by time shift.\ 

We next study the derivative of $\mathscr{A}_r$ with respect to $r$ at $r=0$ at $( \gamma, 2 \pi ) = (-p,q,2\pi) \in C$.\ For this we compute
$$ \frac{\partial K_r}{\partial r} \bigg|_{r=0} (-p,q) = \frac{1}{2} \big(1 + |p|^2\big) (p_1q_2 - p_2q_1)|q|^2 = \sqrt{\big(2 K_0(-p,q) + 1\big)} |q| (p_1 q_2 - p_2 q_1). $$
Recall from (\ref{angular_momentum_invariant}) that $p_1 q_2 - p_2 q_1 = L(q,p) = L(-p,q)$.\ For $( \gamma, 2 \pi ) \in C \subset \text{crit}(\mathscr{A}_0)$ we have that $K_0(\gamma)=0$ and since $L$ is constant along periodic orbits of the Kepler problem we obtain
\begin{align} \label{partial_derivative_functional}
\mathring{\mathscr{A}}_0(\gamma,2\pi) := \frac{\partial \mathcal{A}_r}{\partial r} \bigg|_{r=0} (\gamma,2\pi) = - 2 \pi \int_{S^1} \frac{\partial K_r}{\partial r} \bigg|_{r=0} \big(\gamma(t)\big) &= - 2\pi \int_{S^1} |q|L\big( \gamma(t) \big) dt \nonumber\\
&= - 2\pi L \big( \gamma(0) \big) \int_{S^1} |q|dt.\nonumber\\
&= - 2\pi L (-p,q).
\end{align}
Note that by (\ref{two_hamilt_2}) the last integral term determines the ratio of a Kepler ellipse's period of energy 0 before and after regularization, which is $2\pi$ in this two cases.\ Hence it is 1 and does not depend on the orbit.\ In the \textbf{spatial case}, the analogous procedure gives the same equation and function (\ref{partial_derivative_functional}) of the same form, which we denote by $-L \colon SS^3 \to \mathbb{R}, (-p,q) \mapsto q_1 p_2 - q_2p_1$.\

\subsubsection{Critical points of $-L$ on $SS^2$ and $SS^3$ and their Morse--Bott indices}
\label{sec:explicit_computation}

Recall that if $x$ is a critical point of a Morse--Bott function $f$, then the Morse--Bott index $\text{ind}_f(x)$ of $f$ at $x$ is the number of negative eigenvalues of the Hessian of $f$ at $x$.\\
\\
In view of (\ref{partial_derivative_functional}) we wish to compute the critical points of
\begin{align} \label{function_critical_points}
-L \colon SS^2 \to \mathbb{R},\quad (-p,q) \mapsto q_1p_2 - q_2p_1
\end{align}
and of
\begin{align} \label{function_critical_points_spatial}
-L \colon SS^3 \to \mathbb{R},\quad (-p,q) \mapsto q_1p_2 - q_2p_1.
\end{align}
\begin{lemma1}[Planar case]
	\textit{The critical points of (\ref{function_critical_points}) are exactly two circles over the equator moving in opposite direction (see Figure \ref{figure_geodesic}) and (\ref{function_critical_points}) is Morse--Bott along them.\ Furthermore, one is a maximum and the other one is a minimum, and
	\begin{center}
		\begin{tabular}{c|c|c}
			\textnormal{crit}($-L$) & Morse--Bott index & corresponding periodic orbit in the original picture\\
			\hline maximum & 2 & direct\\
			minimum & 0 & retrograde
		\end{tabular}
	\end{center}}
\end{lemma1}
\begin{proof}
Given the switch (\ref{switch_map}), let $S^2 = \{ (-p_1,-p_2,-p_3) \mid p_1^2 + p_2^2 + p_3^2 = 1 \} \subset \mathbb{R}^3$ be the unit sphere.\ Then view $T^*S^2$ as a subset of $\mathbb{R}^6 = \{ -p_1,-p_2,-p_3,q_1,q_2,q_3 \}$ and the unit tangent bundle $SS^2 \subset TS^2$, which is the closure of the regularized energy hypersurface
$$ \mathbb{R}P^3 \cong S^*S^2 \cong SS^2 = \{ (-p,q) \in \mathbb{R}^3 \times \mathbb{R}^3 \mid \| p \|^2 = \| q \|^2 = 1, \langle - p , q \rangle = 0 \} \subset \mathbb{R}^6. $$
To find the critical points of (\ref{function_critical_points}) we parametrize $S^2$ by
$$ \Phi \colon [0,2\pi) \times (-\frac{\pi}{2},\frac{\pi}{2}) \to \mathbb{R}^3,\quad (\varphi,\theta) \mapsto \begin{pmatrix}
\cos \theta \cos \varphi\\
\cos \theta \sin \varphi\\
\sin \theta
\end{pmatrix}, $$
hence the north and south poles are not parametrized.\ This does not matter since one can check that the same calculation with a chart including the north and south poles shows that there are no critical points of (\ref{function_critical_points}) at the two circles there.\ However, the tangent plane $T_{-p}S^2$ at a point $-p \in S^2$ is spanned by the two orthogonal vectors
$$ \frac{\partial \Phi}{\partial \varphi} = \begin{pmatrix}
- \cos \theta \sin \varphi\\
\cos \theta \cos \varphi\\
0
\end{pmatrix},\quad \frac{\partial \Phi}{\partial \theta} = \begin{pmatrix}
- \sin \theta \cos \varphi\\
- \sin \theta \sin \varphi\\
\cos \theta
\end{pmatrix} $$
with $|| \frac{\partial \Phi}{\partial \varphi} || = \cos \theta$ and $|| \frac{\partial \Phi}{\partial \theta} || = 1$, whence we can take the orthonormal basis $\left\{ (\cos \theta)^{-1} \frac{\partial \Phi}{\partial \varphi}, \frac{\partial \Phi}{\partial \theta} \right\}$.\ Any unit tangent vector $q$ at $-p$ is given by
$$ q = \cos \alpha (\cos \theta)^{-1} \frac{\partial \Phi}{\partial \varphi} + \sin \alpha \frac{\partial \Phi}{\partial \theta} = \cos \alpha \begin{pmatrix}
- \sin \varphi\\
\cos \varphi\\
0
\end{pmatrix} + \sin \alpha \begin{pmatrix}
- \sin \theta \cos \varphi\\
- \sin \theta \sin \varphi\\
\cos \theta
\end{pmatrix} , $$
for $\alpha \in [0,2\pi),$ and we can parametrize $SS^2$ away from the circles at the two poles by $(\varphi,\theta,\alpha)$.\ With this we readily find that (\ref{function_critical_points}) becomes
$$ -L(\varphi,\theta,\alpha) = \cos \theta \cos \alpha. $$
Hence the gradient $\nabla (-L) = ( 0 , - \sin \theta \cos \alpha , - \cos \theta \sin \alpha )^{\mathrm{T}} $ vanishes exactly for $\theta = 0$ and $ \alpha \in \{ 0, \pi \}$, i.e.\ at the circles over the equator.\ For each $\varphi$, the two critical points
\begin{align*}
C_1 := &( \cos \varphi, \sin \varphi, 0, -\sin \varphi, \cos \varphi, 0 ),\quad -L( \varphi,0,0 ) = 1\\
C_2 := &( \cos \varphi, \sin \varphi, 0, \sin \varphi, -\cos \varphi, 0 ),\quad -L( \varphi,0,\pi ) = -1\nonumber
\end{align*}
give the initial conditions at $\varphi_0 := \varphi \in [0,2\pi)$.\ The corresponding trajectories in the loop space $C^{\infty}(S^1,SS^2)$ are, respectively,
\begin{align*}
\gamma_1(t) &= \big( \cos(\varphi_0 + t),\sin (\varphi_0 + t),0,- \sin (\varphi_0 + t), \cos (\varphi_0 + t),0 \big)\\
\gamma_2(t) &= \big( \cos(\varphi_0 - t),\sin (\varphi_0 - t),0, \sin (\varphi_0 - t), - \cos (\varphi_0 - t),0 \big),\nonumber
\end{align*}
for $t \in [0,2\pi)$.\ The Hessian of $-L$ at these two circles $\gamma_1(t)$ and $\gamma_2(t)$ is, respectively:
$$ \text{H}_{-L}(\varphi_0 + t,0,0) = \begin{pmatrix}
0 & 0 & 0\\
0 & -1 & 0\\
0 & 0 & -1
\end{pmatrix},\quad \text{H}_{-L}(\varphi_0 - t,0,\pi) = \begin{pmatrix}
0 & 0 & 0\\
0 & 1 & 0\\
0 & 0 & 1
\end{pmatrix}. $$
Therefore $-L$ is Morse--Bott along these two circles.\ In other words, all critical points are transverse Morse and its Hessian is non-degenerate in the direction normal to $SS^2 \cong \mathbb{R}P^3$.\ The maximum is attained along $\gamma_1(t)$ and the minimum along $\gamma_2(t)$ and their Morse--Bott indices are
\begin{align*}
\text{ind}_{-L}\big( \gamma_1(t) \big) = 2,\quad  \text{ind}_{-L}\big( \gamma_2(t) \big) = 0.
\end{align*}
Note that $\gamma_1(t)$ and $\gamma_2(t)$ are two circles over the equator moving in opposite direction, of Figure \ref{figure_geodesic}.\ To see them in the original picture, we consider the stereographic projection
$$ \sigma_N \colon S^2 \setminus \{N\} \to \mathbb{R}^2,\quad (x_1,x_2,x_3) \mapsto \bigg( \frac{x_1}{1-x_3}, \frac{x_2}{1 - x_3} \bigg), $$
and its cotangent lift
$$ T^* \big(S^2 \setminus \{N\} \big) \to T^*\mathbb{R}^2,\quad (x,y) \mapsto \big( \sigma_N(x), ( d \sigma_N(x)^{\mathrm{T}} )^{-1} (y) \big), $$
where $ (d \sigma_N(x)^{\mathrm{T}} )^{-1} (y) = \big(y_1(1-x_3) + x_1y_3, y_2(1-x_3) + x_2y_3 \big) $.\ Together with the switch (\ref{switch_map}) we obtain that the maximum $\gamma_1$ corresponds to
$$ \big( q_1(t),q_2(t),p_1(t),p_2(t) \big) = \big( -\sin(\varphi_0 + t), \cos (\varphi_0 + t) , - \cos (\varphi_0 + t), - \sin (\varphi_0 + t) \big), $$
which rotates in a direct motion, and the minimum $\gamma_2$ to
$$ \big( q_1(t),q_2(t),p_1(t),p_2(t) \big) = \big( \sin(\varphi_0 - t), - \cos (\varphi_0 - t) , - \cos (\varphi_0 - t), - \sin (\varphi_0 - t) \big), $$
which rotates in a retrograde motion.\ Their angular momentum is $-1$ and $1$, respectively.\
\end{proof}
\begin{lemma1}[Spatial case]
	\textit{The critical points of (\ref{function_critical_points_spatial}) are exactly four circles, and (\ref{function_critical_points_spatial}) is Morse--Bott along them.\ There are one maximum, two saddle points and one minimum.\ The maximum and minimum are two critical points inherited from the planar case, and the two saddle points are two circles moving through the north pole.\ Moreover,
	\begin{center}
		\begin{tabular}{c|c|c|c}
			\textnormal{crit}($-L$) & Morse--Bott index & corresponding periodic orbit in the original picture\\
			\hline maximum & 4 & direct (planar)\\
			saddle point & 2 & collision orbit bouncing back (spatial)\\
			saddle point & 2 & collision orbit bouncing back (spatial)\\
			minimum & 0 & retrograde (planar)
		\end{tabular}
	\end{center}}
\end{lemma1}
\begin{proof}
By using a parametrization of $S^3$ where the north and south poles are not parametrized, one obtains two critical points inherited from the planar problem, namely
\begin{align*}
C_1 := &( \cos \varphi, \sin \varphi, 0, 0, -\sin \varphi, \cos \varphi, 0, 0 ),\quad\quad -L( C_1 ) = 1, &\text{ind}_{-L}(C_1) &= 4,\\
C_2 := &( \cos \varphi, \sin \varphi, 0, 0, \sin \varphi, -\cos \varphi, 0, 0 ),\quad\quad -L( C_2 ) = -1, &\text{ind}_{-L}(C_2) &= 0.\nonumber
\end{align*}
Therefore we parametrize $S^3$ by
$$ \Phi \colon [0,2\pi) \times (-\frac{\pi}{2},\frac{\pi}{2}) \times (-\frac{\pi}{2},\frac{\pi}{2}) \to \mathbb{R}^4,\quad (\varphi,\theta_1,\theta_2) \mapsto \begin{pmatrix}
\sin \theta_2\\
\sin \theta_1 \cos \theta_2\\
\cos \theta_1 \cos \theta_2 \cos \varphi\\
\cos \theta_1 \cos \theta_2 \sin \varphi
\end{pmatrix}. $$
At a point $-p \in S^3$ we choose $ \left\{ (\cos \theta_1)^{-1} (\cos \theta_2)^{-1} \frac{\partial \Phi}{\partial \varphi}, (\cos \theta_2)^{-1} \frac{\partial \Phi}{\partial \theta_1}, \frac{\partial \Phi}{\partial \theta_2} \right\} $ as the orthonormal basis of the tangent plane $T_{-p}S^3$.\ Then every unit tangent vector $q$ at $-p$ is written as
$$ q = \cos \alpha_1 \cos \alpha_2 \begin{pmatrix}
0\\
0\\
- \sin \varphi\\
\cos \varphi
\end{pmatrix} + \cos \alpha_1 \sin \alpha_2 \begin{pmatrix}
0\\
\cos \theta_1 \\
- \sin \theta_1 \cos \varphi\\
- \sin \theta_1 \sin \varphi
\end{pmatrix} + \sin \alpha_1 \begin{pmatrix}
\cos \theta_2\\
-\sin \theta_1 \sin \theta_2\\
- \cos \theta_1 \sin \theta_2 \cos \varphi\\
- \cos \theta_1 \sin \theta_2 \sin \varphi
\end{pmatrix}, $$
for $\alpha_1 \in (-\frac{\pi}{2},\frac{\pi}{2})$, $\alpha_2 \in [0,2\pi)$.\ Moreover, we parametrize $SS^3$ away from circles on $S^3$ with only first and second coordinates by $(\varphi,\theta_1,\theta_2,\alpha_1,\alpha_2)$.\ Then (\ref{function_critical_points_spatial}) becomes
$$ - L (\varphi,\theta_1,\theta_2,\alpha_1,\alpha_2) = \cos \theta_1 \sin \theta_2 \cos \alpha_1 \sin \alpha_2 - \sin \theta_1 \sin \alpha_1. $$
For the critical points, the third component of the gradient
$$ \nabla(-L) = \begin{pmatrix}
0\\
- \sin \theta_1 \sin \theta_2 \cos \alpha_1 \sin \alpha_2 - \cos \theta_1 \sin \alpha_1\\
\cos \theta_1 \cos \theta_2 \cos \alpha_1 \sin \alpha_2\\
- \cos \theta_1 \sin \theta_2 \sin \alpha_1 \sin \alpha_2 - \sin \theta_1 \cos \alpha_1\\
\cos \theta_1 \sin \theta_2 \cos \alpha_1 \cos \alpha_2
\end{pmatrix} $$
implies that $\sin\alpha_2=0$, hence $\alpha_2 \in \{0,\pi\}$.\ Morevoer, by the fifth component we obtain that $\sin \theta_2=0$, i.e.\ $\theta_2=0$.\ The two cases for $\alpha_2$ together with the second and fourth component give the further solutions $\alpha_1=0$ and $\theta_1=0$.\ Therefore in addition there are for every $\varphi$ two critical points
\begin{align*}
C_3 := &(0,0, \cos \varphi, \sin \varphi,0,0, -\sin \varphi, \cos \varphi ),\quad\quad -L( \varphi,0,0,0,0 ) = 0,\\
C_4 := &(0,0, \cos \varphi, \sin \varphi, 0, 0, \sin \varphi, -\cos \varphi ),\quad\quad -L( \varphi,0,0,0,\pi ) = 0,\nonumber
\end{align*}
giving the initial conditions at $\varphi_0 := \varphi \in [0,2\pi)$ for the corresponding trajectories in the loop space $C^{\infty}(S^1,SS^3)$, which are respectively
\begin{align*}
\gamma_3(t) &= \big(0,0, \cos(\varphi + t),\sin (\varphi + t),0,0,- \sin (\varphi + t), \cos (\varphi + t) \big)\\
\gamma_4(t) &= \big(0,0, \cos(\varphi - t),\sin (\varphi - t),0,0, \sin (\varphi - t), - \cos (\varphi - t) \big).\nonumber
\end{align*}
The Hessian at $(\varphi+t,0,0,0,0)$ and $(\varphi-t,0,0,0,\pi)$ is given respectively by
$$ \text{H}_{-L} = \begin{pmatrix}
0 & 0 & 0 & 0 & 0\\
0 & 0 & 0 & -1 & 0\\
0 & 0 & 0 & 0 & 1\\
0 & -1 & 0 & 0 & 0\\
0 & 0 & 1 & 0 & 0
\end{pmatrix},\quad \text{H}_{-L} = \begin{pmatrix}
0 & 0 & 0 & 0 & 0\\
0 & 0 & 0 & -1 & 0\\
0 & 0 & 0 & 0 & -1\\
0 & -1 & 0 & 0 & 0\\
0 & 0 & -1 & 0 & 0
\end{pmatrix}.$$
Hence $-L$ is Morse--Bott along $\gamma_3$ and $\gamma_4$, and in each case there are besides the zero eigenvalue the eigenvalues 1 and $-1$ with double multiplicity, thus $\gamma_3(t)$ and $\gamma_4(t)$ are saddle points with indices
\begin{align*}
\text{ind}_{-L}\big( \gamma_3(t) \big) = \text{ind}_{-L}\big( \gamma_4(t) \big) = 2.
\end{align*}
Moreover, these two circles move trough the north pole for $\varphi_0 \pm t = \frac{\pi}{2}$ in opposite direction, i.e.\ in the original picture, these two great circles correspond to collision orbits moving into the mass at the origin with only the spatial coordinates bouncing back.\ More precisely, analogous to the planar case, by using the stereographic projection from the north pole
$$ \sigma_N \colon S^3 \setminus \{N\} \to \mathbb{R}^3,\quad (x_1,x_2,x_3,x_4) \mapsto \bigg( \frac{x_1}{1-x_4}, \frac{x_2}{1 - x_4}, \frac{x_3}{1 - x_4} \bigg), $$
its cotangent lift
$$ T^* \big(S^3 \setminus \{N\} \big) \to T^*\mathbb{R}^3,\quad (x,y) \mapsto \big( \sigma_N(x), y_1(1-x_4) + x_1y_4, y_2(1-x_4) + x_2y_4, y_3(1-x_4) + x_3y_4 \big) $$
and the switch (\ref{switch_map}), $\gamma_3$ and $\gamma_4$ become
\begin{align*}
\bigg( 0,0,1 - \sin (\varphi_0 + t),0,0,\frac{\cos (\varphi_0 + t)}{\sin (\varphi_0 + t)-1} \bigg), \quad  \bigg( 0,0,\sin (\varphi_0 - t)-1,0,0,\frac{\cos (\varphi_0 - t)}{\sin (\varphi_0 - t)-1} \bigg),
\end{align*}
respectively.\ The first one collides with the origin from above and the second one from below, see Figure \ref{figure_spatial_collision_orbits}.\ Their angular momentum is zero.\
\end{proof}

\subsection{Morse case and bifurcation of family $g$ and $f$ from the geodesic flow}
\label{sec:7.3}

In each case, the $S^1$-action obtained by rotating the loop corresponds to the zero eigenvalue of the Hessian.\ By taking the quotient by this action we obtain the space of oriented unparametrized great circles on $S^2$ (planar case), and the same is true for all $n \geq 2$, see for instance \cite{besse}, \cite{klingenberg}.\ Denote by $\widetilde{S^n_+}$ the space of oriented unparametrized great circles on $S^n$.\ Since an oriented great circle on $S^n$ corresponds to an oriented 2-plane through the origin of $\mathbb{R}^{n+1}$, $\widetilde{S^n_+}$ is diffeomorphic to the oriented Grassmannian $G^+(2,n+1)$ of oriented 2-planes through the origin of $\mathbb{R}^{n+1}$.\ For instance, $\widetilde{S^2_+}$ is clearly $S^2$ by associating to an oriented 2-plane its unit normal vector.\ The space $\widetilde{S^3_+}$, which is important for the spatial problem, is diffeomorphic to $S^2 \times S^2$, see \cite[p.\ 55]{besse}.\ We summarize:
\begin{align} \label{unparametrized_space}
\widetilde{S^n_+} \cong G^+(2,n+1) = \begin{cases}
S^2, & n=2\\
S^2 \times S^2, & n=3
\end{cases}
\end{align}
\textbf{Planar case:}\ By the invariance of the functionals $\mathscr{A}_r$ under the circle action (\ref{circle_action}), they induce action functionals
$$ \overline{\mathscr{A}_r} : ( \mathcal{L} \times \mathbb{R}_{>0} )/S_1 \to \mathbb{R}. $$
Note that the quotient $( \mathcal{L} \times \mathbb{R}_{>0} )/S_1$ is an orbifold.\ Recall the definition (\ref{c_critical_set}) of the critical component $C$.\ Since $(\gamma,2\pi) \in C$ corresponds to a simple closed geodesic on $S^2$, the $S^1$-action at $C \subset \mathcal{L} \times \mathbb{R}_{>0}$ is free, hence the quotient space
$$C/S^1 \subset \text{crit}(\overline{\mathscr{A}_0})$$
is a submanifold of $(\mathcal{L} \times \mathbb{R}_{>0})/S^1$.\ This quotient space is diffeomorphic to $\mathbb{R}P^3/S^1 \cong S^2$ which is the space of oriented unparametrized great circles on $S^2$.\ Hence in view of (\ref{partial_derivative_functional}) the restriction
$$\mathring{\overline{\mathscr{A}_0}} | _{S^2} = \overline{ - 2 \pi L(-p,q) } | _{S^2} \colon S^2 \to \mathbb{R}$$
to the quotient space $C / S^1 \cong S^2$ is a Morse function on $S^2$ having one maximum and one minimum with Morse indices 2 and 0, respectively.\ In particular, $\mathring{\overline{\mathscr{A}_0}} | _{S^2}$ is diffeomorphic to the standard height function on $S^2$.\ Hence two families of periodic orbits bifurcate like the critical points on $S^2$.\ The familly bifurcating from the circular direct one (corresponding to the maximum) is referred to as \textbf{direct periodic orbits}, and the family bifurcating from the circular retrograde orbit (corresponding to the minimum) as \textbf{retrograde periodic orbits}.\ We call them \textbf{family $g$ and $f$}, respectively.\ This bifurcation is generated by a small perturbation of $\mathring{\overline{\mathscr{A}_0}} | _{S^2}$ on $C/S^1 \cong S^2$, as follows from the implicit function theorem.\ More precisely, there exists $\varepsilon > 0$, an open neighborhood $U$ of $S^2$ in $( \mathcal{L} \times \mathbb{R}_{>0} )/S^1$ and a smooth function
\begin{align} \label{small_perturbation}
f : \text{crit} ( \mathring{\overline{\mathscr{A}_0}} | _{S^2} ) \times [0,\varepsilon) \to U
\end{align}
with the following two properties:
\begin{itemize}
	\item[i)] If $\iota : \text{crit} ( \mathring{\overline{\mathscr{A}_0}} | _{S^2} ) \to U $ is the inclusion, then $ f(\cdot,0) = \iota $.
	\item[ii)] For every $\tilde{r} \in (0, \varepsilon)$ the restriction $\mathring{\overline{\mathscr{A}_{\tilde{r}}}} | _U$ is Morse and $ \text{crit}( \mathring{\overline{\mathscr{A}_{\tilde{r}}}} | _U ) = \text{im}( f(\cdot,\tilde{r}) ).$
\end{itemize}
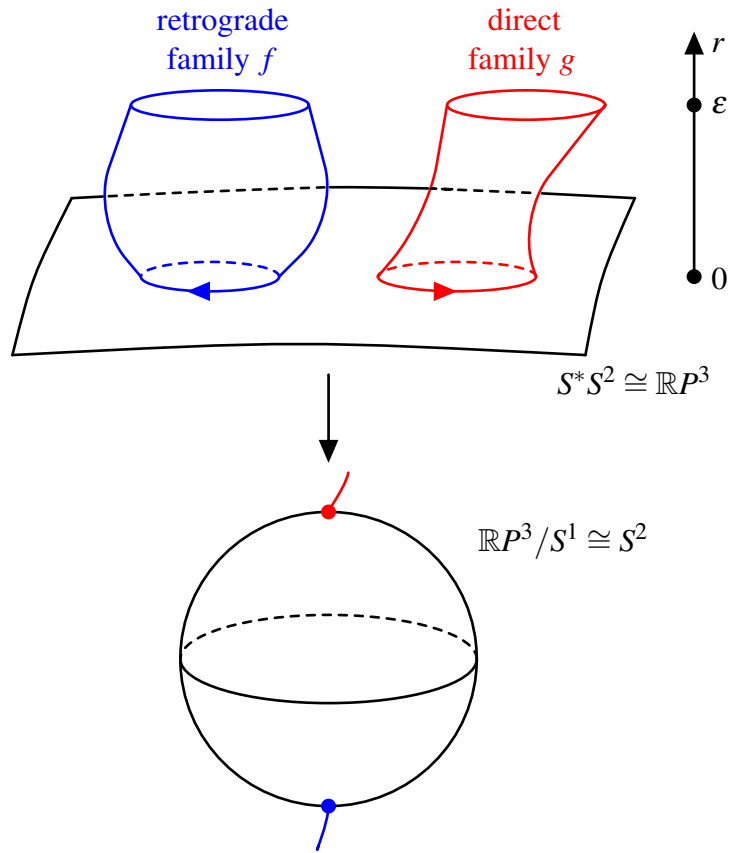
\begin{figure}[H]
	\centering
	\begin{tikzpicture}[line cap=round,line join=round,>=triangle 45,x=1cm,y=1cm,scale=1.3]
	\clip(-6.2,-5.9) rectangle (6.5,2.75);
	
	\draw [dashed,color=blue,line width=1pt] (-0.5,0) arc (0:180:0.7 and 0.15);
	\draw [decoration={markings, mark=at position 0.5 with {\arrow{<}}}, postaction={decorate},color=blue,line width=1pt] (-1.9,0) arc (180:360:0.7 and 0.15);
	
	\draw [color=blue,line width=1pt] (-0.2,1.75) arc (0:360:0.9 and 0.15);
	
	\draw[rounded corners=20pt,color=blue,line width=1pt](-1.9,0)--(-2.4,0.6)--(-2,1.75);
	
	\draw[rounded corners=20pt,color=blue,line width=1pt](-0.5,0)--(0.1,0.6)--(-0.2,1.75);
	
	\draw [dashed,color=red,line width=1pt] (2.1,0) arc (0:180:0.8 and 0.15);
	\draw [decoration={markings, mark=at position 0.5 with {\arrow{>}}}, postaction={decorate},color=red,line width=1pt] (0.5,0) arc (180:360:0.8 and 0.15);
	
	\draw [color=red,line width=1pt] (2.8,1.75) arc (0:360:0.8 and 0.15);
	
	\draw[rounded corners=20pt,color=red,line width=1pt](0.5,0)--(1,0.6)--(1.2,1.75);
	
	\draw[rounded corners=20pt,color=red,line width=1pt](2.1,0)--(1.9,0.6)--(2.8,1.75);
	
	
	\draw [line width=1pt] (-3.2,-0.8) .. controls (-0.3,-0.65) .. (2.6,-0.8);
	
	\draw [line width=1pt] (2.6,-0.8) .. controls (2.75,0) .. (3.1,0.8);
	
	\draw [line width=1pt] (-3.2,-0.8) .. controls (-3,0) .. (-2.6,0.8);
	
	\draw [dash pattern=on 2pt off 2pt,opacity=0] (-2.6,0.8) .. controls (0.25,0.95) ..
	coordinate[pos=0.038] (A)
	coordinate[pos=0.02] (B)
	coordinate[pos=0.046] (C)
	coordinate[pos=0.432] (D)
	coordinate[pos=0.2] (E)
	coordinate[pos=0.445] (F)
	coordinate[pos=0.67] (G)
	coordinate[pos=0.72] (H) 
	coordinate[pos=0.53] (I)
	coordinate[pos=0.87] (J)
	coordinate[pos=0.876] (K)
	coordinate[pos=0.65] (L)
	coordinate[pos=0.94] (M) (3.1,0.8);
	
	
	\draw [line width=1pt] (-2.6,0.8) .. controls (B) .. (A);
	\draw [dashed,line width=1pt] (C) .. controls (E) .. (D);
	\draw [line width=1pt] (F) .. controls (I) .. (G);
	\draw [dashed,line width=1pt] (H) .. controls (L) .. (J);
	\draw [line width=1pt] (K) .. controls (M) .. (3.1,0.8); 
	
	\draw [->,line width=1pt] (3.7,0) -- (3.7,2.5);
	
	\draw (3.7611111111111205,2.55) node[anchor=north west] {$r$};
	\draw (3.7611111111111205,1.95) node[anchor=north west] {$\varepsilon$};
	\draw (3.7611111111111205,0.2) node[anchor=north west] {$0$};
	\draw (2.2,-0.8) node[anchor=north west] {$S^*S^2 \cong \mathbb{R}P^3$};
	
	\draw [->,line width=1pt] (0,-1) -- (0,-1.9);
	\draw [line width=1pt] (0,-3.9) circle (1.5cm);
	
	\draw [dashed,line width=1pt] (1.5,-3.9) arc (0:180:1.5 and 0.45);
	\draw [line width=1pt] (-1.5,-3.9) arc (180:360:1.5 and 0.45);
	
	\draw (1.4,-2.4) node[anchor=north west] {$\mathbb{R}P^3 / S^1 \cong S^2$};
	
	\draw[color=red] (1.5,2.85) node[anchor=north west] {$\text{direct}$};
	
	\draw[color=red] (1.3,2.45) node[anchor=north west] {$\text{family }g$};
	
	\draw[color=blue] (-1.85,2.85) node[anchor=north west] {$\text{retrograde}$};
	
	\draw[color=blue] (-1.75,2.45) node[anchor=north west] {$\text{family }f$};
	
	\begin{scriptsize}
	\draw [fill=black] (3.7,0) circle (2pt);
	\draw [fill=black] (3.7,1.75) circle (2pt);
	\draw [color=red,fill=red] (0,-2.4) circle (2pt);
	\draw [color=blue,fill=blue] (0,-5.4) circle (2pt);
	\draw [color=red,line width=1pt] (0,-2.4) to[out=-130,in=-100] (0.2,-2);
	\draw [color=blue,line width=1pt] (0,-5.4) to[out=-90,in=-110] (-0.1,-5.8);
	\end{scriptsize}
	\end{tikzpicture}
	\caption{Periodic orbits bifurcate like ciritical points of the height function on $S^2$}
	\label{figure_planar_bifurcation}
\end{figure}
\textbf{Spatial case:}\ Recall by (\ref{unparametrized_space}) that the space of oriented unparametrized great circles on $S^3$ is given by $S^2 \times S^2$.\ Thus the restriction
$$\mathring{\overline{\mathscr{A}_0}} | _{S^2 \times S^2} = \overline{ - 2 \pi L (-p,q) } | _{S^2 \times S^2} \colon S^2 \times S^2 \to \mathbb{R}$$
is a Morse function on $S^2 \times S^2$ having one maximum, two saddle points and one minimum with Morse indices 4, 2, 2, and 0, respectively.\ The planar direct orbit, two collision spatial orbits and planar retrograde perdiodic orbit bifurcate now like these critial points on $S^2 \times S^2$, respectively.\ We illustrate the bifurcation scenario in the spatial problem in Figure \ref{figure_spatial_bifurcation}.\
\begin{figure}[H]
	\centering
	\definecolor{grgr}{RGB}{33,189,63}
	\begin{tikzpicture}[line cap=round,line join=round,>=triangle 45,x=1cm,y=1cm,scale=1.3]
	\clip(-3.5,-6.5) rectangle (8.2,2.9);
	
	\draw [dashed,color=blue,line width=1pt] (-0.5,0) arc (0:180:0.7 and 0.15);
	\draw [decoration={markings, mark=at position 0.5 with {\arrow{<}}}, postaction={decorate},color=blue,line width=1pt] (-1.9,0) arc (180:360:0.7 and 0.15);
	
	\draw [color=blue,line width=1pt] (-0.2,1.75) arc (0:360:0.9 and 0.15);
	
	\draw[rounded corners=20pt,color=blue,line width=1pt](-1.9,0)--(-2.4,0.6)--(-2,1.75);
	
	\draw[rounded corners=20pt,color=blue,line width=1pt](-0.5,0)--(0.1,0.6)--(-0.2,1.75);
	
	\draw [dashed,color=red,line width=1pt] (2.1,0) arc (0:180:0.8 and 0.15);
	\draw [decoration={markings, mark=at position 0.5 with {\arrow{>}}}, postaction={decorate},color=red,line width=1pt] (0.5,0) arc (180:360:0.8 and 0.15);
	
	\draw [color=red,line width=1pt] (2.8,1.75) arc (0:360:0.8 and 0.15);
	
	\draw[rounded corners=20pt,color=red,line width=1pt](0.5,0)--(1,0.6)--(1.2,1.75);
	
	\draw[rounded corners=20pt,color=red,line width=1pt](2.1,0)--(1.9,0.6)--(2.8,1.75);
	
	\draw [dashed,color=magenta,line width=1pt] (4.6,0) arc (0:180:0.7 and 0.15);
	\draw [decoration={markings, mark=at position 0.5 with {\arrow{<}}}, postaction={decorate},color=magenta,line width=1pt] (3.2,0) arc (180:360:0.7 and 0.15);
	
	\draw [color=magenta,line width=1pt] (4.5,1.75) arc (0:360:0.6 and 0.15);
	
	\draw[rounded corners=20pt,color=magenta,line width=1pt](3.2,0)--(3.5,0.6)--(3.3,1.75);
	
	\draw[rounded corners=20pt,color=magenta,line width=1pt](4.6,0)--(5.2,0.9)--(4.5,1.75);
	
	\draw [dashed,color=grgr,line width=1pt] (6.9,0) arc (0:180:0.8 and 0.15);
	\draw [decoration={markings, mark=at position 0.5 with {\arrow{>}}}, postaction={decorate},color=grgr,line width=1pt] (5.3,0) arc (180:360:0.8 and 0.15);
	
	\draw [color=grgr,line width=1pt] (6.9,1.75) arc (0:360:0.8 and 0.15);
	
	\draw[rounded corners=20pt,color=grgr,line width=1pt](5.3,0)--(5.6,0.6)--(5.3,1.75);
	
	\draw[rounded corners=20pt,color=grgr,line width=1pt](6.9,0)--(6.6,0.6)--(6.9,1.75);
	
	
	\draw [line width=1pt] (-3.2,-0.8) .. controls (1.9,-0.6) .. (7,-0.8);
	\draw [line width=1pt] (7,-0.8) .. controls (7.15,0) .. (7.5,0.8);
	\draw [line width=1pt] (-3.2,-0.8) .. controls (-3,0) .. (-2.6,0.8);
	\draw [dash pattern=on 2pt off 2pt,opacity=0pt] (-2.6,0.8) .. controls (2.45,1) ..
	coordinate[pos=0.021] (A)
	coordinate[pos=0.01] (B) 
	coordinate[pos=0.025] (C)
	coordinate[pos=0.205] (D)
	coordinate[pos=0.1] (E)
	coordinate[pos=0.21] (F)
	coordinate[pos=0.318] (G)
	coordinate[pos=0.25] (H)
	coordinate[pos=0.326] (I)
	coordinate[pos=0.461] (J)
	coordinate[pos=0.4] (K)
	coordinate[pos=0.469] (L)
	coordinate[pos=0.622] (M)
	coordinate[pos=0.55] (N)
	coordinate[pos=0.629] (O)
	coordinate[pos=0.8] (P)
	coordinate[pos=0.71] (Q)
	coordinate[pos=0.804] (R)
	coordinate[pos=0.844] (S)
	coordinate[pos=0.825] (T)
	coordinate[pos=0.849] (U)
	coordinate[pos=0.942] (V)
	coordinate[pos=0.9] (W)
	coordinate[pos=0.946] (X)
	coordinate[pos=0.97] (Y) (7.5,0.8);

	
	\draw [line width=1pt] (-2.6,0.8) .. controls (B) .. (A);
	\draw [dashed,line width=1pt] (C) .. controls (E) .. (D);
	\draw [line width=1pt] (F) .. controls (H) .. (G);
	\draw [dashed,line width=1pt] (I) .. controls (K) .. (J);
	\draw [line width=1pt] (L) .. controls (N) .. (M);
	\draw [dashed,line width=1pt] (O) .. controls (Q) .. (P);
	\draw [line width=1pt] (R) .. controls (T) .. (S);
	\draw [dashed,line width=1pt] (U) .. controls (W) .. (V);
	\draw [line width=1pt] (X) .. controls (Y) .. (7.5,0.8);
	
	\draw [->,line width=1pt] (7.8,0) -- (7.8,2.5);
	
	\draw (7.8611111111111205,2.55) node[anchor=north west] {$r$};
	\draw (7.8611111111111205,1.95) node[anchor=north west] {$\varepsilon$};
	\draw (7.8611111111111205,0.2) node[anchor=north west] {$0$};
	\draw (5.1,-0.85) node[anchor=north west] {$S^*S^3 \cong S^3 \times S^2$};
	\draw (5.02,-5.3) node[anchor=north west] {$(S^3 \times S^2)/S^1 \cong S^2 \times S^2$};
	
	\draw [->,line width=1pt] (1.85,-1.5) -- (1.85,-2.4);
	
	\draw [line width=1pt] (0,-4.5) circle (1.5cm);
	
	\draw [dashed,line width=1pt] (1.5,-4.5) arc (0:180:1.5 and 0.45);
	\draw [line width=1pt] (-1.5,-4.5) arc (180:360:1.5 and 0.45);
	
	\draw [line width=1pt] (3.8,-4.5) circle (1.5cm);
	
	\draw [dashed,line width=1pt] (5.3,-4.5) arc (0:180:1.5 and 0.45);
	\draw [line width=1pt] (2.3,-4.5) arc (180:360:1.5 and 0.45);
	
	\draw (1.67,-4.34) node[anchor=north west] {$\times$};
	
	\draw[color=red] (1.35,2.5) node[anchor=north west] {$\text{family } g$};
	
	\draw[color=red] (1.1,2.9) node[anchor=north west] {$\text{planar direct}$};
	
	\draw[color=blue] (-1.75,2.5) node[anchor=north west] {$\text{family }f$};
	
	\draw[color=blue] (-2.25,2.9) node[anchor=north west] {$\text{planar retrograde}$};
	
	\draw[color=magenta] (3.25,2.5) node[anchor=north west] {$\text{collision}$};
	
	\draw[color=magenta] (3.37,2.9) node[anchor=north west] {$\text{spatial}$};
	
	\draw[color=grgr] (5.45,2.5) node[anchor=north west] {$\text{collision}$};
	
	\draw[color=grgr] (5.57,2.9) node[anchor=north west] {$\text{spatial}$};
	
	\begin{scriptsize}
	\draw [fill=black] (7.8,0) circle (2pt);
	\draw [fill=black] (7.8,1.75) circle (2pt);
	
	\draw [color=magenta,fill=red] (0,-3) circle (2pt);
	\draw [color=grgr,fill=blue] (0,-6) circle (2pt);
	\draw [color=grgr,fill=red] (3.8,-3) circle (2pt);
	\draw [color=magenta,fill=blue] (3.8,-6) circle (2pt);
	
	
	\draw [color=red,line width=1pt] (0,-3) to[out=-130,in=-100] (0.2,-2.6);
	\draw [color=blue,line width=1pt] (0,-6) to[out=-90,in=-110] (-0.1,-6.4);
	\draw [color=red,line width=1pt] (3.8,-3) to[out=-60,in=-100] (3.6,-2.6);
	\draw [color=blue,line width=1pt] (3.8,-6) to[out=-70,in=-80] (3.9,-6.4);
	\draw [color=grgr,line width=1pt] (3.8,-3) to[out=-130,in=-100] (4,-2.6);
	\draw [color=magenta,line width=1pt] (3.8,-6) to[out=-90,in=-110] (3.7,-6.4);
	\draw [color=magenta,line width=1pt] (0,-3) to[out=-60,in=-100] (-0.2,-2.6);
	\draw [color=grgr,line width=1pt] (0,-6) to[out=-70,in=-80] (0.1,-6.4);
	\end{scriptsize}
	\end{tikzpicture}
	\caption{Bifurcation picture in the spatial case}
	\label{figure_spatial_bifurcation}
\end{figure}
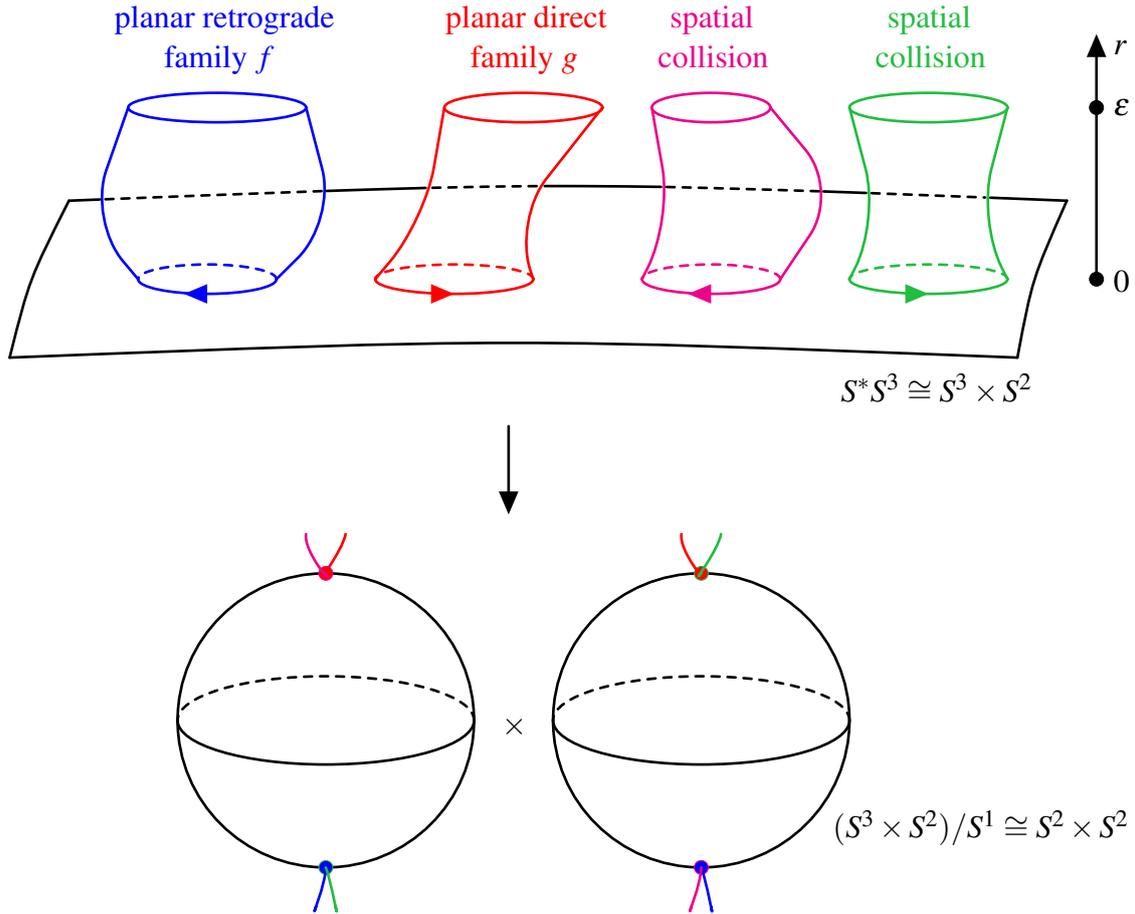

\subsection{Conley--Zehnder indices and the months $T_a$ and $T_d$}

For the planar case, in our transversally non-degenerate setting the transversal Conley--Zehnder index equals the Morse index (see \cite{weber_1}, \cite{weber_2}).\ By the Morse index theorem (see \cite[§15]{milnor}), the Morse--Bott index of every simple closed geodesic on the Morse--Bott component $C \cong \mathbb{R}P^3$ of $\mathcal{A}_0$ is 1, since such a geodesic only goes through one antipodal point.\ After the small perturbation on the quotient space $\mathbb{R}P^3 / S^1 \cong S^2$, the orbits in the families $g$ and $f$ bifurcating from the geodesic flow acquire additionally their respective Morse indices of 2 and 0.\ This proves Theorem \ref{theorem_a}.\

For the spatial case, again by the Morse index theorem, we start with Morse--Bott index 2 for every simple closed geodesic on $SS^3 \cong S^3 \times S^2$.\ After the bifurcation scenario, they obtain additionally their respective Morse indices, which implies Theorem \ref{theorem_b}.\

Now we have
\begin{table}[H]\centering
	\begin{tabular}{c|c|c|c}
		& $\mu_{CZ}$ & $\mu_{CZ}^p$ & $\mu_{CZ}^s$\\
		\hline family $g$ of direct periodic orbits (planar)& 6 & 3 & 3\\
		\hline family $f$ of retrograde periodic orbits (planar) & 2 & 1 & 1
	\end{tabular}
	\caption{Conley--Zehnder indices for very low energies}
	\label{table_indices}
\end{table}
\noindent
For all sufficiently small energies up until the undetermined  $\varepsilon_0 > 0$ given by the implicit function theorem (\ref{small_perturbation}), we obtain for the anomalistic period $T_a$ (\ref{anomalistic_period}) and draconitic period $T_d$ (\ref{draconitic_period}) that
\begin{align} \label{periods_low_energies}
T_a = \begin{cases}
\mathlarger{\frac{2 \pi \cdot T_s}{2 \pi + \varphi_p}}\\[1em]
\mathlarger{\frac{2 \pi \cdot T_s}{ \varphi_p}}
\end{cases} \text{ and }\quad T_d = \begin{cases}
\mathlarger{\frac{2 \pi \cdot T_s}{2 \pi + \varphi_s}} & \text{ for family } g\\[1em]
\mathlarger{\frac{2 \pi \cdot T_s}{ \varphi_s}} & \text{ for family }f.
\end{cases}
\end{align}

\subsection{General Hamiltonians}

\begin{Proposition}
The bifurcation picture of this section holds for all Hamiltonians of the form
\begin{align} \label{hamiltonian_general}
T^*\big( \Omega \setminus \{ (0,0,0) \} \big) \to \mathbb{R},\quad (q,p) \mapsto \frac{1}{2}|p|^2 - \frac{\mu}{|q|} + p_1 q_2 - p_2q_1 + V(q),
\end{align}
where $\Omega \subset \mathbb{R}^3$ is an open subset containing the origin, $V \colon \Omega \to \mathbb{R}$ is a smooth function such that $(0,0,0) \in \text{crit}(V)$ and $\mu > 0$.\
\end{Proposition}
\begin{remark}
These forms of Hamiltonians consist of the rotating Kepler problem plus some velocity independent forces given by $V(q)$, such as the Hill lunar problem (\ref{hamiltonian_hill}).\
\end{remark}
\begin{proof}[Proof of Proposition 6.3]
For the proof one needs to verify the same steps as Frauenfelder and van Koert did in \cite[pp.\ 138--140]{frauenfelder} for the analog statement of the planar case.\ For given $c < 0 $, the regularization of (\ref{hamiltonian_general}) is given by
\begin{align*}
K_c (-p,q) &= \frac{1}{2} \bigg( - \frac{|q|}{2c} \Big( H\big( - \frac{q}{2c}, \sqrt{-2c}p \big) - c - V(0) \Big) + \mu \bigg)^2 - \frac{\mu^2}{2}\\
&= \frac{1}{2} \bigg( \frac{1}{2} \big( 1 + |p|^2 \big) + \frac{p_1 q_2 - p_2 q_1}{(-2c)^{\frac{3}{2}}} - \frac{V(-\frac{q}{2c}) - V(0)}{2c} \bigg)^2 |q|^2 - \frac{\mu^2}{2}.
\end{align*}
As in the spatial Hill lunar problem, we change the energy parameter to $c = \frac{-1}{2r^{2/3}}$ and obtain
\begin{align*}
K_r(-p,q) := \frac{1}{2} \Big( \frac{1}{2} \big( 1 + |p|^2 \big) + (p_1 q_2 - p_2 q_1)r + \big( V(qr^{\frac{2}{3}}) - V(0) \big)r^{\frac{2}{3}} \Big)^2 |q|^2 - \frac{\mu^2}{2}.
\end{align*}
The Hamiltonian
\begin{align*}
K_0 (-p,q) = \frac{1}{2} \Big( \frac{1}{2}(1 + |p|^2) \Big)^2|q|^2 - \frac{\mu^2}{2}
\end{align*}
is independent from the choice of $V$, and its flow on $K_0^{-1}(0)$ is the geodesic flow on $S^3$ up to reparametrization.\ The same calculation as in \cite[pp.\ 138--140]{frauenfelder} shows that $K_r$ is twice continuously differentiable in $r \in [0,\infty)$ and
$$ \frac{\partial K_r}{\partial r} \bigg|_{r=0} (-p,q) = \sqrt{\big(2 K_0(-p,q) + \mu^2\big)} |q| (p_1 q_2 - p_2 q_1) = \sqrt{\big(2 K_0(-p,q) + \mu^2\big)} |q| L(-p,q), $$
which does not depend on $V$ as well.\ Hence, exactly as in the spatial Hill lunar problem, the geodesic flow bifurcates at $r=0$ into the four known periodic orbits.\
\end{proof}
\begin{remark}
	The further assumption that the Hamiltonian (\ref{hamiltonian_general}) is invariant under the symplectic involution $\sigma$ from (\ref{symplectic_involution}) implies the splitting of the Conley--Zehnder indices as in Table \ref{table_indices} and the two periods $T_a$ and $T_d$ in (\ref{periods_low_energies}).\
\end{remark}

\section{Local equivariant Rabinowitz-Floer homology}
\label{sec:local_rabinowitz}

\subsection{Local Floer homology and its Euler characteristic}
\label{sec:7.1}

As in (\ref{rabinowitz_action_functional}), we consider for the Hamiltonian of the spatial Hill lunar problem (\ref{hamiltonian_hill}) the Rabinowitz action functional
\begin{align*}
\mathscr{A}^H \colon \mathcal{L} \times \mathbb{R}_{>0} \to \mathbb{R},\quad (\gamma,\tau ) \mapsto \int_{S^1} \lambda\big( \dot{\gamma}(t) \big) - \tau \int_{S^1} H\big(\gamma(t)\big) dt,
\end{align*}
where $\mathcal{L} = C^{\infty}(S^1,T^* \mathbb{R}^3)$ is the free loop space of $M:=T^* \mathbb{R}^3$.\ The Floer homology for this functional was introduced by Cieliebak--Frauenfelder \cite{cieliebak_frauenfelder}.\ In view of (\ref{gradient_rabinowitz}),
$$ \nabla_g \mathscr{A}^H (\gamma,\tau) = \begin{pmatrix}
- J_t (\gamma) \left( \partial_t \gamma - \tau X_H \big(\gamma(t) \big) \right)\\
- \int_{S^1} H \big( \gamma(t) \big) dt
\end{pmatrix} $$
and hence, the critical points of $\mathscr{A}^H$ are parametrized periodic orbits of $X_H$ of period $\tau$ on the energy hypersurface $H^{-1}(0)$.\ Therefore gradient flow lines are maps $(\gamma,\tau) \in C^{\infty}(\mathbb{R},\mathcal{L} \times \mathbb{R}_{>0})$ satisfying
$$ \partial_s \big( \gamma(s), \tau(s) \big) = \nabla_g \mathscr{A}^H \big( \gamma(s), \tau(s) \big), $$
i.e.\ they are solutions $\big(\gamma,\tau\big) \in C^{\infty} (\mathbb{R} \times S^1,M) \times C^{\infty}(\mathbb{R},\mathbb{R}_{>0})$ of the PDE
$$ \begin{cases}
\partial_s \gamma + J_t (\gamma) \big( \partial_t \gamma - \tau X_H( \gamma ) \big) = 0\\
\partial_s \tau + \int_{S^1} H(\gamma) dt = 0.
\end{cases}$$
We now work locally near a family of non-degenerate periodic orbits.\ Since we consider unparametrized periodic orbits, we need to use the local $S^1$-equivariant Rabinowitz-Floer homology, which is the local $S^1$-equivariant Floer homology of the Rabinowitz functional.\ Since its discussion goes beyond the scope of this article, its details will be discussed in a later work, based on the following articles.\ In the non-equivariant and non-Rabinowitz case, without Lagrangian multipliers, the local Floer homology is described in \cite{ginzburg}, and the equivariant Floer homology and its local version are described in \cite{ginzburg_gurel}.\ In the same way, one can work out the local version of the $S^1$-equivariant Rabinowitz-Floer homology constructed by Frauenfelder--Schlenk \cite{frauenfelder_schlenk}.\ Note that for every $\tilde{k}$-th cover of a periodic orbit there is a local $S^1$-equivariant Rabinowitz Floer homology associated to this $\tilde{k}$-th cover.\ Given a family $\tilde{\gamma}$ of non-degenerate unparametrized periodic orbits, we denote by $RFH^{S^1}_*(\tilde{\gamma})$ its local $S^1$-equivariant Rabinowitz-Floer homology and by $\chi (\tilde{\gamma})$ its Euler characteristic
$$ \chi (\tilde{\gamma}) = \sum_{m \in \mathbb{Z}} (-1)^m \dim RFH^{S^1}_m (\tilde{\gamma}). $$
Since the local homology is invariant, the Euler characteristic is invariant as well.\ The following two examples, which are analogous to the two examples in \cite[p.\ 540]{ginzburg_gurel}, are crucial for our study.\
\begin{example}
	Let $\tilde{\gamma}$ be a family of simple closed non-degenerate periodic orbits.\ Then, $RFH^{S^1}_*(\tilde{\gamma})$ has rank one when $*$ equals the Conley--Zehnder index of $\tilde{\gamma}$ and zero otherwise.\
\end{example}
\begin{example} \label{example_7_2}
	Assume that an iterated planar periodic orbit $x^{\tilde{k}}$ is non-degenerate for all $\tilde{k} \geq 1$.\	Recall that its Conley--Zehnder index is the sum of a planar and spatial index, which we denote by $\mu_{CZ}^p(x^{\tilde{k}})$ and  $\mu_{CZ}^s(x^{\tilde{k}})$.\ Furthermore, each index is given by the index iteration in Table \ref{table_index_iteration}.\ Let  $\mu_{CZ}^p(x)$ and $\mu_{CZ}^s(x)$ be the indices of the underlying simple closed periodic orbit $x$, then if
	$$ \mu_{CZ}^p(x^{\tilde{k}}) \equiv \mu_{CZ}^p(x) \text{ mod } 2 ,\quad \mu_{CZ}^s(x^{\tilde{k}}) \equiv \mu_{CZ}^s(x) \text{ mod } 2, $$
	then $x^{\tilde{k}}$ is called a \textbf{good orbit}.\ Otherwise, $x^{\tilde{k}}$ is called a \textbf{bad orbit}.\ All simple closed periodic orbits are good.\ Furthermore, bad orbits occur as $\tilde{k}$-th cover of negative hyperbolic periodic orbits, where $\tilde{k}$ is even.\ If $\tilde{\gamma}$ is a family of good orbits, then
	$$ RFH^{S^1}_* (\tilde{\gamma} \, ; \, \mathbb{Q}) = \begin{cases}
	\mathbb{Q}, & * = \mu_{CZ}\\
	0, & \text{otherwise}.
	\end{cases} $$
	If $\tilde{\gamma}$ is a family of bad orbits, then
	$$ RFH^{S^1}_* (\tilde{\gamma} \, ; \, \mathbb{Q}) = 0 $$
	in all degrees.\ Hence bad orbits contribute nothing to the local homology and the Euler characteristic.\ Note that for planar families we have two local $S^1$-equivariant Rabinowitz-Floer homologies with its resp.\ Euler characteristics, namely one viewed in the planar problem and one in the spatial problem.\ We denote them by
	$$  pRFH^{S^1}_* (\tilde{\gamma} \, ; \, \mathbb{Q}),\quad \chi_p(\tilde{\gamma}),\qquad \quad sRFH^{S^1}_* (\tilde{\gamma} \, ; \, \mathbb{Q}),\quad \chi_s(\tilde{\gamma}).$$
	They differ only by the index shift given by $\mu_{CZ}^s$.\
\end{example}

\begin{remark} \label{remark_8.1}
	We denote by $\mathfrak{J}$ the space of $\omega$-compatible almost complex structures and by $\mathfrak{M}$ the space of all Riemannian metrics (inner products), i.e.\ positive definite symmetric bilinear forms.\ There is the natural map
	$$ \mathfrak{J} \to \mathfrak{M},\quad J \mapsto g_J := \omega(\cdot, J \cdot). $$
	Starting with a Riemannian metric $g$ there is a construction (see the proof of Proposition 2.50 in \cite[pp.\ 63--65]{mcduff_salamon}) of a map
	$$ \mathfrak{M} \to \mathfrak{J},\quad g \mapsto J_g $$
	such that $J_g$ is an $\omega$-compatible almost complex structure, which depends on $g$, and
	$$ J_{g_J} = J,\quad \sigma^* J_g = J_{\sigma^*g}, $$
	for every symplectomorphism $\sigma$.\ Furthermore, $ \rho^* J_g = - J_{\rho^*g}$, for every anti-symplectic linear map $\rho$.\ Now let $\rho$ be an anti-symplectic involution.\ Since $g$ is not necessarily invariant under $\rho$, we consider its average $\frac{1}{2} (\rho^* g + g)$, which satisfies $ \rho^* \big( \frac{1}{2} (\rho^* g + g) \big) = \frac{1}{2} (\rho^* g + g)$.\ From this $\rho$-invariant Riemannian metric we obtain an $\omega$-compatible almost complex structure $ J_{\frac{1}{2} (\rho^* g + g)} $ which is anti-invariant under $\rho$.\ This gives raise to the symmetry of the gradient flow lines with respect to $\rho$, meaning that we can reflect them with respect to $\rho$.\
\end{remark}

\subsection{The index of planar and spatial families bifurcating from $g$ and $f$}
\label{sec8.2}

Let us consider the two families $g$ and $f$ for very low energies from Section \ref{sec:7.3}.\ In view of their Conley--Zehnder indices from Table \ref{table_indices} we obtain
$$ pRFH^{S^1}_* (g \, ; \, \mathbb{Q}) = \begin{cases}
\mathbb{Q}, & *=3\\
0, &\text{otherwise},
\end{cases}\quad sRFH^{S^1}_* (g \, ; \, \mathbb{Q}) = \begin{cases}
\mathbb{Q}, & *=6\\
0, &\text{otherwise},
\end{cases} $$
and
$$ pRFH^{S^1}_* (f \, ; \, \mathbb{Q}) = \begin{cases}
\mathbb{Q}, & *=1\\
0, &\text{otherwise},
\end{cases}\quad sRFH^{S^1}_* (f \, ; \, \mathbb{Q}) = \begin{cases}
\mathbb{Q}, & *=2\\
0, &\text{otherwise}.
\end{cases} $$
Therefore the resp.\ Euler characteristics are
$$ \chi_p(g) = -1,\quad \chi_s(g) = 1,\quad \chi_p(f) = -1, \quad \chi_s(f) = 1. $$
By Remark \ref{remark_2_7}, if any of $\overline{A}_p$ and $A_s$ or its $\tilde{k}$-th iteration (in the case that the rotation angle is a $\tilde{k}$-th root of unity) moves trough the eigenvalue 1, then the respective index jumps by $\pm 1$ or $\pm 2$.\ Moreover, the crossing of the eigenvalue 1 generates bifurcations of new families of planar or spatial periodic orbits, respectively.\ For details on the existence and properties of such bifurcations we refer to the book of Abraham--Marsden \cite[pp.\ 597--604]{abraham_marsden} and to the articles \cite{deng} and \cite{kim} for a Floer-theoretical approach.\ Since the local Floer homology and its Euler characteristic are not changed under such transitions, together with the signatures, this helps to classify the Conley--Zehnder indices of new bifurcation families and to search for bridges between two families.\ By a \textbf{bridge} we mean a family of periodic orbits with constant Conley--Zehnder index connecting two families of periodic orbits.\ Therefore a bridge is an orbit cylinder between two families.\

\section{Application to symmetric periodic orbits in the spatial Hill lunar problem}
\label{sec:8}

\subsection{Our moon - the companion of the Earth}
\label{sec:our moon}

Hill \cite[p.\ 259]{hill} found that for given energy $\Gamma = 6.5088$ and position $q_1(0) = 0.176097$ one obtains a planar direct periodic orbit (variational orbit), to which our moon is close.\ Our numerical data (see the next subsection) verify that for this energy value the Conley--Zehnder indices $\mu_{CZ}^p=3$ and $\mu_{CZ}^s=3$ do not change, so we use (\ref{periods_low_energies}) to calculate $T_a$ and $T_d$.\

For the given initial data, by our first Python program (see Appendix (\ref{python1})) we compute $\dot{q}_2(0)=2.222972, T_q=0.507959, m=12.369448,$ and in view of (\ref{synodic_period}), the synodic month corresponds to
$$ T_s = 29.528396 .$$

Our second program (see Appendix (\ref{python2})) computes the planar reduced monodromy $\overline{A}_p \in \text{Sp}^{\rho_1}(1)$ and $A_s \in \text{Sp}^{\rho_1}(1)$ as well with their relevant data.\ They are
$$ \overline{A}_p = \begin{pmatrix}
0.900415 & -0.047423\\
3.987879 & 0.900272
\end{pmatrix},\quad A_s = \begin{pmatrix}
0.860448 & -0.035353\\
7.344293 & 0.860422
\end{pmatrix}. $$
Note that $ \det(\overline{A}_p) = 0.999739,\ \text{tr}(\overline{A}_p) = 1.800688, \det(A_s) = 1.000000$ and $\text{tr}(A_s) = 1.720871$.\ Therefore this orbit is planar and spatial elliptic, and the Floquet multipliers are on the unit circle, so of the form $e^{\pm \text{i} \theta}, e^{\pm \text{i} \vartheta}$.\ In particular, in the case of our moon the orbit has to be planar and spatial elliptic since otherwise our moon would fly away.\ Moreover, $ \text{sign}_b(\theta) < 0$ and $\text{sign}_{\tilde{b}}(\vartheta) < 0 $, hence each rotation is by
$$\varphi_p = \theta = 0.450236,\quad \varphi_s = \vartheta = 0.534603,$$
respectively.\ For the anomalistic and draconitic periods we compute
$$T_a = 27.553954, \quad T_d = 27.212712,$$
thus these computed values are a very good approximation to the physically measured data.\

\subsection{Planar direct periodic orbits}
\label{sec:other_direct}

\subsubsection{The family $g$}

Our first plot in Figure \ref{plot_1} is similar to the well-known pictures from Hill \cite[p.\ 261]{hill}, Hénon \cite[p.\ 228]{henon} and Gutzwiller \cite[p.\ 69]{gutzwiller}, \cite[p.\ 621]{gutzwiller_2}.\ It goes until the energy 2.55788, which is the last value with a periodic orbit found by Hill.\ However, Hénon \cite{henon} found further ones for higher energy values.\ These orbits are all doubly-symmetric with respect to $\rho_1$ and $\rho_2$.\
\begin{figure}[H]
	\centering
	\includegraphics[scale=0.65]{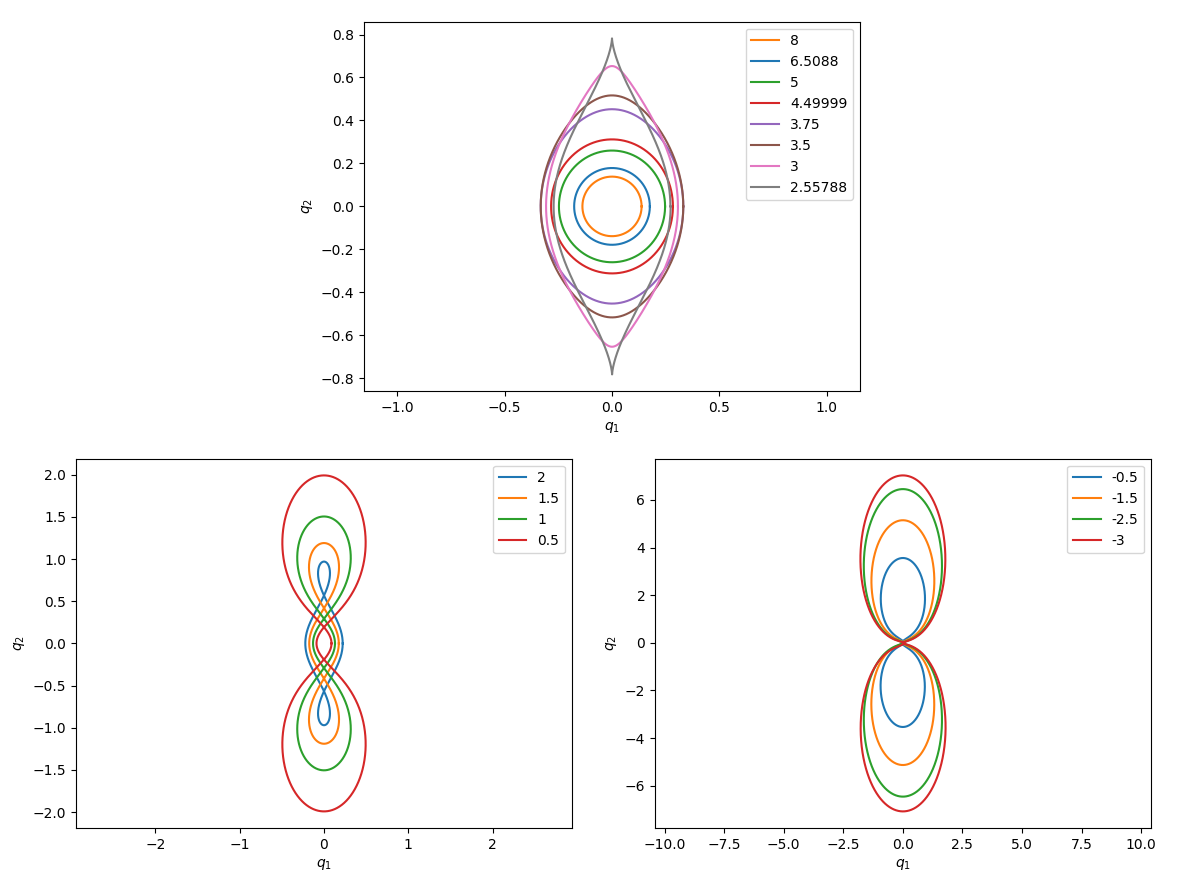}
	\caption{The family $g$}
	\label{plot_1}
\end{figure}
\noindent
We read the following Table \ref{table_3}, and all tables of this type in this section, in the direction of decreasing $\Gamma$.\ Recall that this corresponds to the increase of the energy.\
\begin{table}[H]\scriptsize \centering
	\begin{tabular}{ccccccccccc}
		$\Gamma$ & $q_1(0)$ &  $\dot{q}_2(0)$ & $T_s$ & tr($\overline{A}_p$) & $\text{sign}_{c/b} (\varphi_p / \lambda_p)$ & $T_a$ & tr($A_s$) & $\text{sign}_{\tilde{c}/\tilde{b}} (\varphi_s / \lambda_s)$ & $T_d$ & $\mu_{CZ}^p /\mu_{CZ}^s / \mu_{CZ}$\\
		\hline8 & 0.13772 & 2.56 & 19.78 & 1.89 & $(+/-)$ $\varphi_p =  0.31$ & 18.82 & 1.87 & $(+/-)$ $\varphi_s = 0.35$ & 18.73 & 3 / 3 / 6\\
		6.5088 & 0.176097 & 2.22 & 29.52 & 1.80 & $(+/-)$ $\varphi_p =  0.44$ & 27.55 & 1.72 & $(+/-)$ $\varphi_s = 0.53$ & 27.21 & 3 / 3 / 6\\
		5 & 0.247 & 1.81 & 53.64 & 1.65 & $(+/-)$ $\varphi_p =  0.59$ & 49.01 & 1.07 & $(+/-)$ $\varphi_s = 1.01$ & 46.27 & 3 / 3 / 6\\
		4.49999 & 0.283500 & 1.67 & 71.25 & 1.99 & $(+/-)$ $\varphi_p = 0.03$ & 70.91 & 0.43 & $(+/-)$ $\varphi_s = 1.34$ & 58.66 & 3 / 3 / 6\\
		4.278924 & 0.301158 & 1.62 & 82.44 & 2.71 & $(-/-)$ $\lambda_p = 2.26$ & & $\approx 0$ & $(+/-)$ $\varphi_s = 1.57$ & 65.95 & 2 / 3 / 5\\
		3.876616 & 0.327645 & 1.59 & 109.5 & 8.64 & $(-/-)$ $\lambda_p = 8.53$ & & $-$1.0 & $(+/-)$ $\varphi_s = 2.09$ & 82.15 & 2 / 3 / 5\\
		3.75 & 0.33178 & 1.61 & 119.3 & 14.1 & $(-/-)$ $\lambda_p = 14.04$ & & $-$1.3 & $(+/-)$ $\varphi_s = 2.27$ & 87.63 & 2 / 3 / 5\\
		3.5 & 0.331730 & 1.69 & 139.9 & 37.5 & $(-/-)$ $\lambda_p = 37.48$ & & $-$1.7 & $(+/-)$ $\varphi_s = 2.62$ & 98.05 & 2 / 3 / 5\\
		3.057471 & 0.310843 & 1.91 & 171.2 & 154.3 & $(-/-)$ $\lambda_p = 154.3$ & & $-$2 & $(+/-)$ $\varphi_s = 3.14$ & 114.1 & 2 / 3 / 5\\
		3 & 0.306900 & 1.94 & 175.1 & 178.9 & $(-/-)$ $\lambda_p = 178.9$ & & $-$1.9 & $(-/+)$ $\varphi_s = 3.20$ & 116.0 & 2 / 3 / 5\\
		2.55788 & 0.271795 & 2.24 & 204.8 & 453.9 & $(-/-)$ $\lambda_p = 453.9$ & & $-$1.7 & $(-/+)$ $\varphi_s = 3.64$ & 129.7 & 2 / 3 / 5\\
		2.073537 & 0.228450 & 2.61 & 238.9 & 932.6 & $(-/-)$ $\lambda_p = 932.6$ & & $-$0.9 & $(-/+)$ $\varphi_s = 4.18$ & 143.3 & 2 / 3 / 5\\
		2 & 0.221683 & 2.67 & 244.5 & 1034 & $(-/-)$ $\lambda_p = 1034$ & & $-$0.8 & $(-/+)$ $\varphi_s = 4.29$ & 145.2 & 2 / 3 / 5\\
		1.746370 & 0.198221 & 2.90 & 264.9 & 1356 & $(-/-)$ $\lambda_p = 1356$ & & $\approx 0$ & $(-/+)$ $\varphi_s = 4.71$ & 151.4 & 2 / 3 / 5\\
		1.5 & 0.175446 & 3.16 & 287.1 & 1746 & $(-/-)$ $\lambda_p = 1746$ & & 1.22 & $(-/+)$ $\varphi_s = 5.37$ & 154.8 & 2 / 3 / 5\\
		1.383094 & 0.164715 & 3.35 & 298.6 & 1916 & $(-/-)$ $\lambda_p = 1916$ & & 2.00 & $(+/+)$ $\lambda_s = 1.01$ & & 2 / 4 / 6\\
		1 & 0.130319 & 3.79 & 341.7 & 2601 & $(-/-)$ $\lambda_p = 2601$ & & 5.86 & $(+/+)$ $\lambda_s = 5.69$ & & 2 / 4 / 6\\
		0.5 & 0.089019 & 4.68 & 412.1 & 4913 & $(-/-)$ $\lambda_p = 4913$ & & 13.4 & $(+/+)$ $\lambda_s = 13.3$ & & 2 / 4 / 6\\
		$-$0.5 & 0.033920 & 7.71 & 550.5 & 9830 & $(-/-)$ $\lambda_p = 9830$ & & 26.5 & $(+/+)$ $\lambda_s = 26.4$ & & 2 / 4 / 6\\
		$-$1.5 & 0.013560 & 12.2 & 620.5 & 27860 & $(-/-)$ $\lambda_p = 27860$ & & 16.2 & $(+/+)$ $\lambda_s = 16.2$ & & 2 / 4 / 6\\
		$-$2.5 & 0.006313 & 17.8 & 653.8 & 63560 & $(-/-)$ $\lambda_p = 63560$ & & 6.07 & $(+/+)$ $\lambda_s = 5.90$ & & 2 / 4 / 6\\
		$-$3 & 0.004523 & 21.1 & 665.7 & 90300 & $(-/-)$ $\lambda_p = 90300$ & & 3.69 & $(+/+)$ $\lambda_s = 3.39$ & & 2 / 4 / 6
	\end{tabular}
	\caption{The family $g$}
	\label{table_3}
\end{table}
\noindent
Table \ref{table_3} shows that the maximum of $q_1(0)$ is reached at about $\Gamma=3.75$ and the orbits come closer and closer to the earth.\ According to \cite[pp.\ 230--234]{henon}, based on numerical results and not on analytical arguments, the distance $q_1(0)$ converges to $0$ if the energy $\Gamma$ goes to $-\infty$.\ Hence in the limit there is a collision.\ In that case the period is $4 \pi$, thus the synodic month $T_s$ takes 730.5 days.\ In addition, by \cite[p.\ 319]{henon_2}, tr$(A_s)$ converges to 2 from above.\ Moreover, the speed of the orbit increases if it is closer to the earth.\

Very shortly above $\Gamma = 4.49999$ the planar index $\mu_{CZ}^p$ jumps from 3 to 2 since the rotation by $\varphi_p$ goes to zero at that point.\ From then on the orbits are planar positive hyperbolic type I.\ Slightly before this transition the anomalistic period $T_a$ is almost the synodic period $T_s$.\ Furthermore, at this transition a new family of planar periodic orbits bifurcates (see the family $g'$ in the next subsection).\ This bifurcation arises below the critical value $3^{4/3}$ and above the energy value for our moon.\

Furthermore, these orbits are all spatial elliptic until shortly before the energy value $\Gamma = 1.383094$ where they become spatial positive hyperbolic type II and $\varphi_s$ goes through $2 \pi$.\ Thus $\mu_{CZ}^s$ jumps from 3 to 4, and shortly before this change, the draconitic period $T_d$ is almost half of the synodic period $T_s$.\ At this transition a new family of spatial periodic orbits bifurcates from the planar one which we discuss in the Subsection \ref{sec:9.4.0}.\

We note that all initial data are from \cite{henon}, \cite{henon_2} and \cite{hill}, except the ones for the energy values 4.278924, 3.876616, 2.073537 and 1.746370, where $\varphi_s$ is a 3th resp.\ 4th root of unity, which are from \cite{kalantonis}.\

\subsubsection{The family $g'$}

Hénon \cite{henon} found starting from the energy value $\Gamma=4.49999$ (see Table \ref{table_3}), where the planar index $\mu_{CZ}^p$ of $g$ jumps from 3 to 2, a new family of planar direct periodic orbits bifurcating from $g$.\ At this bifurcation the double-symmetry breaks, meaning that the orbits are only symmetric with respect to $\rho_1$ (see Figure \ref{figure_retrograde_plot_1}).\ By using $\rho_2$, i.e.\ the reflection on the $q_2$-axis, these orbits appear twice.\ The data of the family $g'$ are collected in the Table \ref{table_6}.\ The family $g'$ and its symmetric family start being planar as well as spatial elliptic and they start with the planar index 3.\ Let us verify that this is in accordance with the Euler characteristics before and after bifurcation.\ The local planar Floer homology before this transition is
$$ pRFH^{S^1}_* (g \, ; \, \mathbb{Q}) = \begin{cases}
\mathbb{Q}, & *=3\\
0, &\text{otherwise},
\end{cases} $$
and by the index shift by $\mu_{CZ}^s = 3$, it is in the spatial problem
$$ sRFH^{S^1}_* (g \, ; \, \mathbb{Q}) = \begin{cases}
\mathbb{Q}, & *=6\\
0, &\text{otherwise}.
\end{cases} $$
Therefore the Euler characteristics are in the planar problem
$$ \chi_p(g) = (-1)^3 = -1,\quad \text{resp.}\quad \chi_p(g) = (-1)^2 + 2 \cdot (-1)^3 = -1 ,$$
and in the spatial problem
$$ \chi_s(g) = (-1)^6 = 1,\quad \text{resp.}\quad \chi_s(g) = (-1)^5 + 2 \cdot (-1)^6 = 1 .$$
\begin{figure}[H]
	\centering
	\includegraphics[scale=0.65]{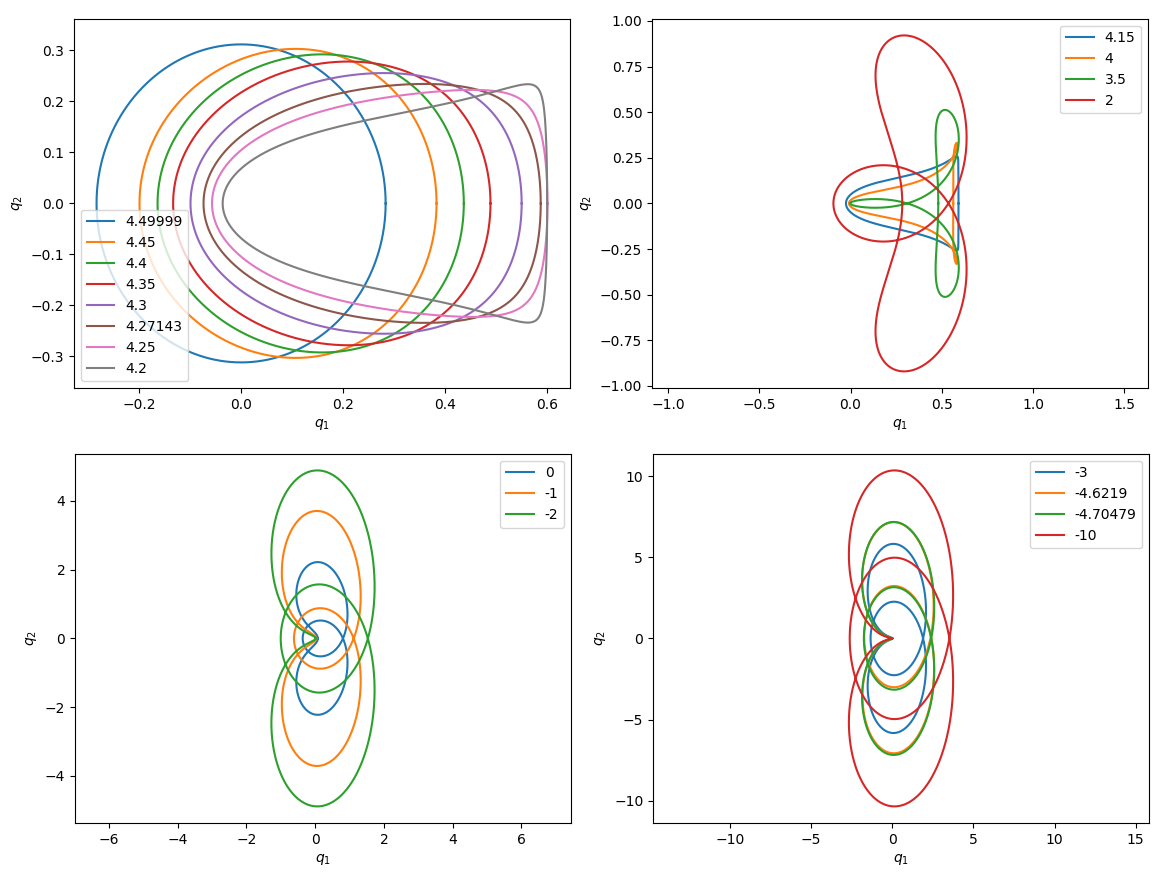}
	\caption{The family $g'$}
	\label{figure_retrograde_plot_1}
\end{figure}
\noindent

\begin{table}[H]\scriptsize \centering
	\begin{tabular}{ccccccccccc}
		$\Gamma$ & $q_1(0)$ &  $\dot{q}_2(0)$ & $T_s$ & tr($\overline{A}_p$) & $\text{sign}_{c/b} (\varphi_p / \lambda_p$) & $T_a$ & tr($A_s$) & $\text{sign}_{\tilde{c}/\tilde{b}} (\varphi_s / \lambda_s)$ & $T_d$ & $\mu_{CZ}^p / \mu_{CZ}^s / \mu_{CZ}$ \\
		\hline4.49999 & 0.283500 & 1.67 & 71.25 & 1.99 & $(+/-)$ $\varphi_p = 0.03$ & 70.92 & 0.43 & $(+/-)$ $\varphi_s = 1.34$ & 58.66 & 3 / 3 / 6 \\
		4.45 & 0.383360 & 1.09 & 76.34 & 1.77 & $(+/-)$ $\varphi_p = 0.47$ & 70.96 & 0.11 & $(+/-)$ $\varphi_s = 1.51$ & 61.52 & 3 / 3 / 6 \\
		4.435711 & 0.399433 & 1.02 & 78.07 & 1.69 & $(+/-)$ $\varphi_p = 0.55$ & 71.72 & $\approx 0$ & $(+/-)$ $\varphi_s = 1.57$ & 62.45 & 3 / 3 / 6\\
		4.4 & 0.436840 & 0.86 & 83.11 & 1.45 & $(+/-)$ $\varphi_p = 0.75$ & 74.18 & $-$0.3 & $(+/-)$ $\varphi_s = 1.73$ & 65.11 & 3 / 3 / 6\\
		4.35 & 0.489180 & 0.67 & 93.16 & 0.95 & $(+/-)$ $\varphi_p = 1.07$ & 79.57 & $-$0.9 & $(+/-)$ $\varphi_s = 2.07$ & 70.02 & 3 / 3 / 6\\
		4.347942 & 0.491443 & 0.66 & 93.70 & 0.93 & $(+/-)$ $\varphi_p = 1.08$ & 79.88 & $-$1.0 & $(+/-)$ $\varphi_s = 2.09$ & 70.27 & 3 / 3 / 6\\
		4.3 & 0.550290 & 0.49 & 112.1 & 0.04 & $(+/-)$ $\varphi_p = 1.54$ & 90.01 & $-$1.8 & $(+/-)$ $\varphi_s = 2.73$ & 78.18 & 3 / 3 / 6\\
		4.285183 & 0.570854 & 0.44 & 122.2 & $-$0.5 & $(+/-)$ $\varphi_p = 1.86$ & 94.28 & $-$1.9 & $(+/-)$ $\varphi_s = 3.13$ & 81.54 & 3 / 3 / 6\\
		4.282893 & 0.573907 & 0.43 & 124.0 & $-$0.7 & $(+/-)$ $\varphi_p = 1.94$ & 94.73 & $-$2.0 & $(+/+)$ $\lambda_s = -1.06$ & & 3 / 3 / 6\\
		4.280603 & 0.576960 & 0.42 & 125.9 & $-$0.9 & $(+/-)$ $\varphi_p = 2.04$ & 95.05 & $-$1.9 & $(-/+)$ $\varphi_s = 3.14$ & 83.95 & 3 / 3 / 6\\
		4.27143 & 0.587690 & 0.41 & 133.8 & $-$1.9 & $(+/-)$ $\varphi_p = 3.11$ & 89.47 & $-$1.8 & $(-/+)$ $\varphi_s = 3.47$ & 86.15 & 3 / 3 / 6\\
		4.25 & 0.600900 & 0.40 & 150.0 & $-$8.1 & $(+/+)$ $\lambda_p = -8.0$ & & $-$1.2 & $(-/+)$ $\varphi_s = 4.04$ & 91.25 & 3 / 3 / 6\\
		4.242877 & 0.602371 & 0.40 & 154.0 & $-$11.2 & $(+/+)$ $\lambda_p = -$11.1 & & $-$0.9 & $(-/+)$ $\varphi_s = 4.18$ & 92.43 & 3 / 3 / 6\\
		4.200105 & 0.600418 & 0.46 & 169.0 & $-$35.6 & $(+/+)$ $\lambda_p = -$35.6 & & $\approx 0$ & $(-/+)$ $\varphi_s = 4.71$ & 96.58 & 3 / 3 / 6\\
		4.2 & 0.600400 & 0.46 & 169.0 & $-$35.6 & $(+/+)$ $\lambda_p = -$35.6 & & $\approx 0$ & $(-/+)$ $\varphi_s = 4.71$ & 96.58 & 3 / 3 / 6\\
		4.15 & 0.59171 & 0.52 & 177.8 & $-$67.2 & $(+/+)$ $\lambda_p = -$67.2 & & 0.6 & $(-/+)$ $\varphi_s = 5.07$ & 98.86 & 3 / 3 / 6\\
		4 & 0.562913 & 0.70 & 190.1 & $-150$ & $(+/+)$ $\lambda_p = -150$ & & 1.3 & $(-/+)$ $\varphi_s = 5.43$ & 102.0 & 3 / 3 / 6\\
		3.5 & 0.480802 & 1.16 & 207.9 & $-305$ & $(+/+)$ $\lambda_p = -304$ & & 1.9 & $(-/+)$ $\varphi_s = 6.02$ & 106.1 & 3 / 3 / 6\\
		3.390159 & 0.464697 & 1.24 & 211.0 & $-$313 & $(+/+)$ $\lambda_p = -$313 & & 2.0 & $(+/+)$ $\lambda_s = 1.00$ & & 3 / 4 / 7\\
		2 & 0.283653 & 2.30 & 262.1 & $-421$ & $(+/+)$ $\lambda_p = -421$ & & 2.6 & $(+/+)$ $\lambda_s = 2.16$ & & 3 / 4 / 7\\
		1.5 & 0.224220 & 2.75 & 293.2 & $-439$ & $(+/+)$ $\lambda_p = -439$ & & 2.7 & $(+/+)$ $\lambda_s = 2.36$ & & 3 / 4 / 7\\
		1 & 0.167780 & 3.31 & 337.7 & $-465$ & $(+/+)$ $\lambda_p = -465$ & & 2.7 & $(+/+)$ $\lambda_s = 2.35$ & & 3 / 4 / 7\\
		0.5 & 0.116370 & 4.08 & 401.6 & $-511$ & $(+/+)$ $\lambda_p = -511$ & & 2.1 & $(+/+)$ $\lambda_s = 1.32$ & & 3 / 4 / 7\\
		0.477157 & 0.114196 & 4.13 & 405.1 & $-$514 & $(+/+)$ $\lambda_p = -$514 & & 2.0 & $(+/+)$ $\lambda_s = 1.03$ & & 3 / 4 / 7\\
		0.063099 & 0.078843 & 5.03 & 474.4 & $-$596 & $(+/+)$ $\lambda_p = -$596 & & $\approx 0$ & $(+/-)$ $\varphi_s = 1.57$ & 210.8 & 3 / 5 / 8\\
		0 & 0.074220 & 5.19 & 485.8 & $-614$ & $(+/+)$ $\lambda_p = -614$ & & $-$0.4 & $(+/-)$ $\varphi_s = 1.78$ & 212.6 & 3 / 5 / 8\\
		$-$0.081977 & 0.068564 & 5.40 & 500.2 & $-$639 & $(+/+)$ $\lambda_p = -$639 & & $-$0.9 & $(+/-)$ $\varphi_s = 2.09$ & 214.4 & 3 / 5 / 8\\	
		$-$0.219528 & 0.059949 & 5.79 & 524.6 & $-$686 & $(+/+)$ $\lambda_p = -$686 & & $-$1.9 & $(+/-)$ $\varphi_s = 3.13$ & 209.9 & 3 / 5 / 8\\
		$-$1 & 0.029281 & 8.32 & 643.0 & $-880$ & $(+/+)$ $\lambda_p = -880$ & & $-$6.6 & $(-/-)$ $\lambda_s = -6.5$ &  & 3 / 5 / 8\\
		$-$2 & 0.014641 & 11.7 & 743.8 & $-732$ & $(+/+)$ $\lambda_p = -732$ & & $-$11.8 & $(-/-)$ $\lambda_s = -11.7$ & & 3 / 5 / 8\\
		$-$3 & 0.008613 & 15.3 & 796.6 & $-499$ & $(+/+)$ $\lambda_p = -499$ & & $-$14.9 & $(-/-)$ $\lambda_s = -14.8$ & & 3 / 5 / 8\\
		$-$4.69219 & 0.004302 & 21.6 & 835.1 & $-2$ & $(+/+)$ $\lambda_p = -1.0$ & & $-$18.3 & $(-/-)$ $\lambda_s = -18.2$ & & 3 / 5 / 8\\
		$-$4.69849 & 0.004292 & 21.6 & 836.2 & 0.07 & $(-/+)$ $\varphi_p = 4.75$ & 476.2 & $-$18.9 & $(-/-)$ $\lambda_s = -18.8$ & & 3 / 5 / 8\\
		$-$4.70479 & 0.004283 & 21.7 & 836.3 & $2$ & $(-/-)$ $\lambda_p = 1.0$ & & $-$19.5 & $(-/-)$ $\lambda_s = -19.4$ & & 4 / 5 / 9\\
		$-$5 & 0.003870 & 22.8 & 840.4 & 96.8 & $(-/-)$ $\lambda_p = 96.8$ & & $-$20.3 & $(-/-)$ $\lambda_s = -20.2$ & & 4 / 5 / 9\\
		$-$6 & 0.002837 & 26.6 & 850.9 & 432 & $(-/-)$ $\lambda_p = 432$ & & $-$22.9 & $(-/-)$ $\lambda_s = -22.8$ & & 4 / 5 / 9\\
		$-$10 & 0.001162 & 41.6 & 870.8 & 1989 & $(-/-)$ $\lambda_p = 1989$ & & $-$59.5 & $(-/-)$ $\lambda_s = -59.5$ & & 4 / 5 / 9\\
		$-$20 & 0.000363 & 74.2 & 876.6 & 7190 & $(-/-)$ $\lambda_p = 7190$ & & $-$3727 & $(-/-)$ $\lambda_s = -3727$ & & 4 / 5 / 9
	\end{tabular}
	\caption{The family $g'$}
	\label{table_6}
\end{table}
\noindent
Moreover, we see the same behaviour as for the family $g$ of the distances $q_1(0)$, namely there is a collision in the limit.\ Note that we have found the initial conditions for the $\Gamma$ value $-4.69849$ by ourselves.\ Since at this $\Gamma$ we have $\varphi_p=4.75$, we can imply that the planar index jumps from 3 to 4  in which we can see that $\varphi_p$ goes through $2 \pi$ and hence $\mu_{CZ}^p$ jumps from 3 to 4.\ The initial data for the energy values 4.435711, 4.347942, 4.242877, 4.200105, 0.063099 and $-0.081977$, where $\varphi_s$ is a 3rd resp.\ 4th root of unity, are from \cite{kalantonis} and the others from \cite{henon} and \cite{henon_2}.\

\subsection{Planar retrograde periodic orbits}
\label{sec:retrograde}

\subsubsection{The family $f$}
Some of its orbits are plotted in Figure \ref{fig_fam_f} and its data are given in Table \ref{table_9}.\ We observe that the retrograde periodic orbits are all planar and spatial elliptic and are at larger and larger distance from the earth.\ Therefore the index of the simple closed orbits does not change.\
\begin{figure}[H]
	\centering
	\includegraphics[scale=0.65]{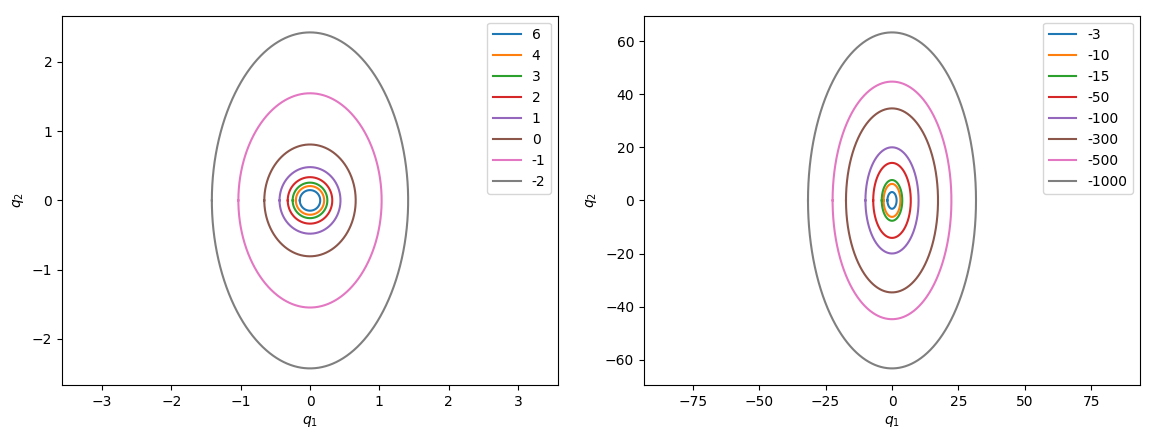}
	\caption{The family $f$}
	\label{fig_fam_f}
\end{figure}
\begin{table}[H]\scriptsize \centering
	\begin{tabular}{ccccccccccc}
		$\Gamma$ & $q_1(0)$ &  $\dot{q}_2(0)$ & $T_s$ & tr($\overline{A}_p$) & $\text{sign}_{c/b} (\varphi_p)$ & $T_a$ & tr($A_s$) & $\text{sign}_{\tilde{c}/\tilde{b}} (\varphi_s)$ & $T_d$ & $\mu_{CZ}^p / \mu_{CZ}^s / \mu_{CZ}$\\
		\hline6 & $-$0.147790 & 2.75 & 19.72 & 1.88 & $(-/+)$ $\varphi_p = 5.93$ & 20.87 & 1.89 & $(-/+)$ $\varphi_s = 5.97$ & 20.80 & 1 / 1 / 2\\
		4 & $-$0.204210 & 2.43 & 31.19 & 1.70 & $(-/+)$ $\varphi_p = 5.73$ & 34.19 & 1.75 & $(-/+)$ $\varphi_s = 5.78$ & 33.89 & 1 / 1 / 2\\
		3 & $-$0.250710 & 2.27 & 41.49 & 1.48 & $(-/+)$ $\varphi_p = 5.55$ & 47.40 & 1.59 & $(-/+)$ $\varphi_s = 5.63$ & 46.24 & 1 / 1 / 2\\
		2 & $-$0.321630 & 2.12 & 58.28 & 1.05 & $(-/+)$ $\varphi_p = 5.27$ & 69.48 & 1.30 & $(-/+)$ $\varphi_s = 5.42$ & 67.52 & 1 / 1 / 2\\
		1.359293 & $-$0.389537 & 2.05 & 75.25 & 0.57 & $(-/+)$ $\varphi_p = 5.00$ & 94.5 & 0.99 & $(-/+)$ $\varphi_s = 5.23$ & 90.35 & 1 / 1 / 2\\
		1 & $-$0.439910 & 2.03 & 88.25 & 0.20 & $(-/+)$ $\varphi_p = 4.81$ & 115.1 & 0.78 & $(-/+)$ $\varphi_s = 5.11$ & 108.4 & 1 / 1 / 2\\
		0.755141 & $-$0.481217 & 2.02 & 99.10 & $-$0.08 & $(-/+)$ $\varphi_p = 4.66$ & 133.4 & 0.61 & $(-/+)$ $\varphi_s = 5.02$ & 123.8 & 1 / 1 / 2\\
		0.015388 & $-$0.655072 & 2.07 & 145.5 & $-$0.99 & $(-/+)$ $\varphi_p = 4.18$ & 218.3 & 0.18 & $(-/+)$ $\varphi_s = 4.80$ & 190.3 & 1 / 1 / 2\\
		0 & $-$0.659660 & 2.08 & 146.7 & $-$1.01 & $(-/+)$ $\varphi_p = 4.18$ & 220.6 & 0.18 & $(-/+)$ $\varphi_s = 4.80$ & 192.0 & 1 / 1 / 2\\
		$-$0.2154 & $-$0.72779 & 2.13 & 165.0 & $-$1.24 & $(-/+)$ $\varphi_p = 4.03$ & 256.8 & 0.14 & $(-/+)$ $\varphi_s = 4.78$ & 216.7 & 1 / 1 / 2\\
		$-$0.2538 & $-$0.74179 & 2.14 & 168.3 & $-$1.23 & $(-/+)$ $\varphi_p = 4.04$ & 261.4 & 0.14 & $(-/+)$ $\varphi_s = 4.78$ & 221.1 & 1 / 1 / 2\\
		$-$0.2681 & $-$0.74679 & 2.14 & 169.6 & $-$1.24 & $(-/+)$ $\varphi_p = 4.04$ & 263.7 & 0.14 & $(-/+)$ $\varphi_s = 4.78$ & 222.8 & 1 / 1 / 2\\
		$-$0.2847 & $-$0.75279 & 2.15 & 171.1 & $-$1.24 & $(-/+)$ $\varphi_p = 4.03$ & 266.2 & 0.14 & $(-/+)$ $\varphi_s = 4.78$ & 224.7 & 1 / 1 / 2\\
		$-$0.3152 & $-$0.76259 & 2.16 & 174.0 & $-$1.30 & $(-/+)$ $\varphi_p = 4.00$ & 273.1 & 0.15 & $(-/+)$ $\varphi_s = 4.78$ & 228.3 & 1 / 1 / 2\\
		$-$0.5269 & $-$0.84189 & 2.24 & 196.9 & $-$1.37 & $(-/+)$ $\varphi_p = 3.95$ & 308.1 & 0.22 & $(-/+)$ $\varphi_s = 4.82$ & 252.4 & 1 / 1 / 2\\
		$-$1 & $-$1.034000 & 2.47 & 237.1 & $-$1.29 & $(-/+)$ $\varphi_p = 4.01$ & 372.4 & 0.62 & $(-/+)$ $\varphi_s = 5.02$ & 296.2 & 1 / 1 / 2\\
		$-$1.411618 & $-$1.199879 & 2.71 & 267.3 & $-$0.99 & $(-/+)$ $\varphi_p = 4.18$ & 401.1 & 1.02 & $(-/+)$ $\varphi_s = 5.25$ & 319.7 & 1 / 1 / 2\\
		$-$2 & $-$1.416810 & 3.07 & 296.5 & $-$0.46 & $(-/+)$ $\varphi_p = 4.47$ & 416.8 & 1.44 & $(-/+)$ $\varphi_s = 5.52$ & 337.6 & 1 / 1 / 2\\
		$-$3 & $-$1.731950 & 3.62 & 323.2 & 0.28 & $(-/+)$ $\varphi_p = 4.85$ & 418.3 & 1.76 & $(-/+)$ $\varphi_s = 5.79$ & 350.3 & 1 / 1 / 2\\
		$-$10 & $-$3.162278 & 6.37 & 357.4 & 1.64 & $(-/+)$ $\varphi_p = 5.67$ & 395.8 & 1.99 & $(-/+)$ $\varphi_s = 6.18$ & 362.8 & 1 / 1 / 2\\
		$-$15 & $-$3.872983 & 7.77 & 360.9 & 1.80 & $(-/+)$ $\varphi_p = 5.83$ & 388.8 & 1.99 & $(-/+)$ $\varphi_s = 6.23$ & 363.9 & 1 / 1 / 2\\
		$-$50 & $-$7.071067 & 14.1 & 364.5 & 1.96 & $(-/+)$ $\varphi_p = 6.09$ & 375.4 & 1.99 & $(-/+)$ $\varphi_s = 6.27$ & 365.0 & 1 / 1 / 2\\
		$-$100 & $-$10 & 20.0 & 364.9 & 1.98 & $(-/+)$ $\varphi_p = 6.17$ & 371.4 & 1.99 & $(-/+)$ $\varphi_s = 6.28$ & 365.1 & 1 / 1 / 2\\
		$-$300 & $-$17.320508 & 34.6 & 365.2 & 1.99 & $(-/+)$ $\varphi_p = 6.23$ & 368.0 & 1.99 & $(-/+)$ $\varphi_s = 6.28$ & 365.2 & 1 / 1 / 2\\
		$-$500 & $-$22.360679 & 44.7 & 365.2 & 1.99 & $(-/+)$ $\varphi_p = 6.25$ & 367.1 & 1.99 & $(-/+)$ $\varphi_s = 6.28$ & 365.2 & 1 / 1 / 2\\
		$-$1000 & $-$31.622776 & 63.2 & 365.2 & 1.99 & $(-/+)$ $\varphi_p = 6.26$ & 366.3 & 1.99 & $(-/+)$ $\varphi_s = 6.28$ & 365.2 & 1 / 1 / 2
	\end{tabular}
	\caption{The family $f$}
	\label{table_9}
\end{table}
\noindent
Moreover, the planar rotation angle $\varphi_p$ as well as the spatial one $\varphi_s$ decrease and then increase.\ The orbits for the $\Gamma$ values from $-0.2154$ to $-0.5269$ we have found ourselves.\ They show that $\varphi_s$ never becomes a 4th root of unity, hence the smallest one is a 5th one.\ Both rotation angles approach $2\pi$.\ In other words, all three periods approach 365.25 days, which is the period of the earth around the sun.\ The initial data for the $\Gamma$ values 1.359293 and 0.755141, where $\varphi_s$ is a 6th resp. 5th root of unity, are from the data provided on request by the author of \cite{kalantonis}.\ All the others are from \cite{henon_0} and \cite{henon}.\

\textbf{The limit case.}\ Like Hénon \cite[p.\ 227]{henon} for $\Gamma \to -\infty$ in the planar problem, we can neglect the gravitational force of the earth for a first approximation, since $q_1(0)$ increases for higher energies.\ Then the spatial Hill equation in $(q,\dot{q})$-coordiantes (see (\ref{equation_of_motion_0})) reduces to
\begin{equation} \label{limit_hill}
\left\{  \begin{array}{l}
\mathlarger{\ddot{q}_1 = 2 \dot{q}_2 + 3q_1} \\
\mathlarger{\ddot{q}_2 = -2\dot{q}_1} \\
\mathlarger{\ddot{q}_3 = -q_3.}
\end{array}  \right.
\end{equation}
A planar solution of (\ref{limit_hill}) is given by
\begin{align} \label{limit_solution}
q_1(t) = c \cos t, \quad q_2(t)= - 2c \sin t,
\end{align}
where $c \in \mathbb{R}$ is a constant.\ This limit solution in $q$-variables is an ellipse in a retrograde motion with the earth at the origin as its center, semi-major axis $c$ and semi-minor axis $-2c$.\ The energy condition (\ref{energy_gamma}) implies
$$\Gamma = 3q_1^2(t) - \dot{q}_1^2(t) - \dot{q}_2^2(t) = -c^2,$$
hence the relation between the initial condition for the position and the energy is
\begin{align} \label{energy_limit}
q_1(0) = c = - \sqrt{- \Gamma}.
\end{align}
Note that for Table \ref{table_9}, from $\Gamma = -15$ on, we use the formula (\ref{energy_limit}) for $q_1(0)$.\
\begin{Proposition} \label{proposition_retrograde}
	\textit{In the limit case for $\Gamma \to -\infty$, the planar retrograde periodic orbit converges to a planar retrograde periodic orbit with synodic period $T_s$ of 365.25 days and which is planar and spatial degenerate.\ In particular, \textnormal{tr}$(\overline{A}_p)$ and \textnormal{tr}$(A_s)$ converge to 2 from below and $\varphi_p$ and $\varphi_s$ to $2 \pi$.\ In other words, each of $T_a$ and $T_d$ converges to 365.25 days during $T_s$.\ }
\end{Proposition}
\begin{proof}
To obtain the Hamiltonian in the limit, for a constant $\gamma < 0$ we zoom out by the coordinate transformation
$$ \phi_\gamma \colon T^* \mathbb{R}^3 \to T^* \mathbb{R}^3,\quad (q,p) \mapsto \big( \sqrt{(-2 \gamma)} q , \sqrt{(-2 \gamma)} p \big), $$
which is conformally symplectic, i.e.\ $\phi_\gamma ^* \omega = -2 \gamma \omega$.\ We introduce the family of Hamiltonians
$$ H_\gamma \colon T^* \big( \mathbb{R} \setminus \{(0,0,0)\} \big) \to \mathbb{R},\quad (q,p) \mapsto - \frac{1}{2 \gamma} (H \circ \phi_\gamma) (q,p), $$
and compute
$$H_\gamma (q,p) = \frac{1}{2} \big( (p_1 + q_2)^2 + (p_2 - q_1)^2 + p_3^2 \big) - \frac{3}{2}q_1^2 + \frac{1}{2}q_3^2 - \frac{1}{|q|2 \gamma \sqrt{(-2 \gamma)}}. $$
\\
For $\gamma \to -\infty$, $H_\gamma$ converges uniformly in the $C^{\infty}$-topology on each compact subset to the Hamiltonian
$$ \widetilde{H} \colon T^* \mathbb{R}^3 \to \mathbb{R},\quad (q,p) \mapsto \frac{1}{2} \big( (p_1 + q_2)^2 + (p_2 - q_1)^2 + p_3^2 \big) - \frac{3}{2}q_1^2 + \frac{1}{2}q_3^2. $$
Note that this limit Hamiltonian does not contain the gravitational force of the earth, in contrast to the spatial Hill lunar problem (\ref{hamiltonian_2}).\ We restrict to the planar case and consider the energy hypersurface $\Sigma := \widetilde{H}^{-1}(\frac{1}{2})$.\ Hence by the relation (\ref{energy_limit}) and the limit solution (\ref{limit_solution}) we obtain
$$ c = -1,\quad q_1(t) = - \cos t,\quad q_2(t) = 2 \sin t, $$
with the first return time $T_q = 2 \pi$.\ Thus the synodic month is 365.25 days.\ Since the gravitational force disappears, the linearized equation in $(q,\dot{q})$-coordiantes reduces to
\begin{equation*}
\left\{  \begin{array}{l}
\mathlarger{\Delta \ddot{q}_1 = 2 \Delta \dot{q}_2 + 3 \Delta q_1  } \\
\mathlarger{ \Delta \ddot{q}_2 = -2 \Delta \dot{q}_1 } \\
\mathlarger{ \Delta \ddot{q}_3 = - \Delta q_3.}
\end{array}  \right.
\end{equation*}
By a short calculation, planar linearized solutions in $(q,p)$-coordinates are given by
\begin{equation*}
\left\{  \begin{array}{l}
\mathlarger{ \Delta q_1(t) = c_1 \cos t + c_2 \sin t + 2c_3  } \\
\mathlarger{ \Delta q_2(t) = -2c_1 \sin t + 2 c_2 \cos t - 3 c_3 t + c_4 } \\
\mathlarger{ \Delta p_1(t) = c_1 \sin t - c_2 \cos t + 3 c_3 t - c_4 } \\
\mathlarger{ \Delta p_2(t) = - c_1 \cos t - c_2 \sin t - c_3, }
\end{array}  \right.
\end{equation*}
where $c_1$, $c_2$, $c_3$ and $c_4$ are constants.\ Note that these solutions are not periodic along the flow if $c_3 \neq 0$.\ The basis vectors of the tangent space at $q_0$ are of the form
$$ \big( \Delta q_1(0), \Delta q_2(0), \Delta p_1(0), \Delta p_2(0) \big) = ( c_1 + 2c_3, 2c_2 + c_4, - c_2 - c_4, - c_1 - c_3 ). $$
The energy condition (see Subsection \ref{sec:6.5.1}) implies
$$ \Delta q_1(0) = - \Delta p_2(0),\quad c_3= 0.$$
Therefore the linearized solutions on the energy hypersurface are $2 \pi$ periodic, which means that
$$\overline{A}_p = \begin{pmatrix}
1 & 0\\
0 & 1
\end{pmatrix},$$
where the eigenvalue 1 has algebraic and geometric multiplicity 2.\

A real fundamental system for the spatial equation $\Delta \ddot{q}_3 + \Delta q_3 = 0$ is given by $\{ \cos t, \sin t \}$, hence the $2 \pi$-periodic solutions in $(q,p)$-coordiantes are of the form
\begin{equation*}
\left\{  \begin{array}{l}
\mathlarger{ \Delta q_3(t) = c_1 \cos t + c_2 \sin t } \\
\mathlarger{ \Delta p_3(t) = -c_1 \sin t + c_2 \cos t,}
\end{array}  \right.
\end{equation*}
where $c_1$ and $c_2$ are constants.\ In particular, we have\\
$$A_s = \begin{pmatrix}
1 & 0\\
0 & 1
\end{pmatrix}, $$
where the eigenvalue $1$ has algebraic and geometric multiplicity $2$ as well.\ Therefore the limit orbit is planar and spatial degenerate, and the planar and spatial neighbouring orbits make exactly one complete rotation during $T_q$, which corresponds to 365.25 days.
\end{proof}

\subsubsection{Family $g_3$:\ Planar bifurcation from a 3rd cover of $f$ to a 3rd cover of $f$}

At the $\Gamma$ values 0.015388 and $-1.411618$ for the family $f$, where $\varphi_p$ is a 3th root of unity (see Table \ref{table_9}), the planar index of each 3rd cover jumps from 5 to 3 resp.\ from 3 to 5.\ Before and after each transition of each 3rd cover of these two retrograde orbits, new families of planar retrograde periodic orbits bifurcate.\ These families were discovered by Hénon \cite{henon_0}, \cite{henon_1} in which he named them family $g_3$.\

The orbits in these families do not lose the symmetry from $f$, i.e.\ they are also doubly-symmetric with respect to $\rho_1$ and $\rho_2$ (see Figure \ref{family_g_3}).\ Their data are collected in Table \ref{table_10}.\
\begin{figure}[H]
	\centering
	\includegraphics[scale=0.42]{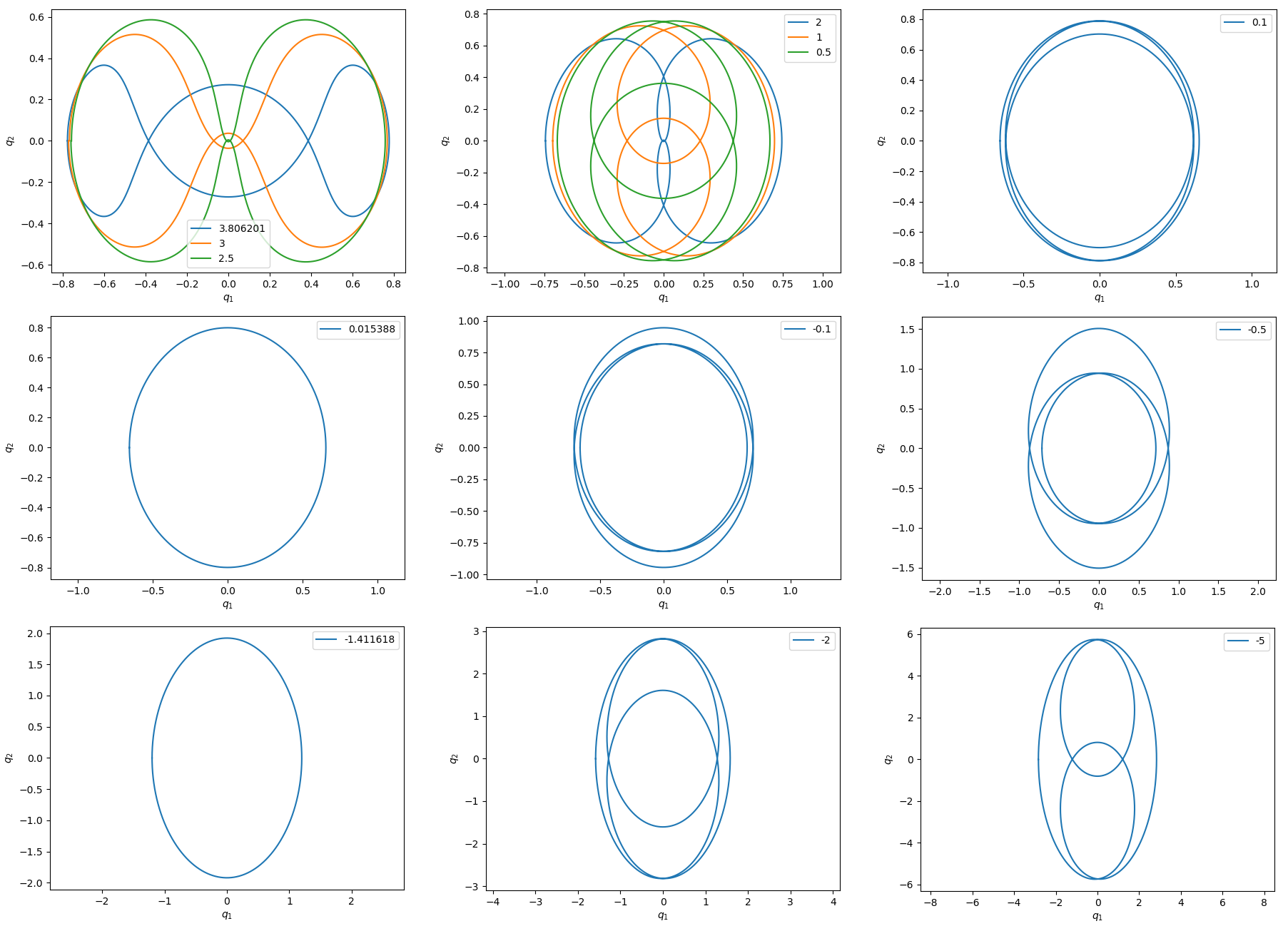}
	\caption{The family $g_3$}
	\label{family_g_3}
\end{figure}
\noindent
Recall that the planar index of the 3rd cover of the $f$-orbit at 0.015388 jumps from 5 to 3 and the one of the $f$-orbit at $-1.411618$ from 3 to 5.\ At each transition each new family is planar positive hyperbolic, hence at each transition each new family has $\mu_{CZ}^p = 4$, and thereby the resp.\ local Floer homology as well as its Euler characteristic are zero.\

Furthermore, the family starting at 0.015388 and ending at $-1.411618$ has constant Conley--Zehnder index, thus it forms a planar bridge between these 3rd covers.\ Note that this bridge is planar positive hyperbolic and spatial elliptic.\

Moreover, if the energy decreases, then the family $g_3$ ends at a degenerate periodic orbit which is of birth-death type.\ In other words, the family $g_3$ starts from one branch of a degenerate planar periodic orbit which is of birth-death type.\ Recall from the introduction that a periodic orbit of birth-death type is a degenerate orbit from which two families bifurcate with an index difference of 1 and into the same energy direction.\ Its local Floer homology and its Euler characteristic are therefore zero.\

\begin{table}[H]\scriptsize \centering
	\begin{tabular}{ccccccccccc}
		$\Gamma$ & $q_1(0)$ &  $\dot{q}_2(0)$ & $T_s$ & tr($\overline{A}_p$) & $\text{sign}_{c/b} (\varphi_p / \lambda_p)$ & $T_a$ & tr($A_s$) & $\text{sign}_{\tilde{c}/\tilde{b}} (\varphi_s / \lambda_s)$ & $T_d$ & $\mu_{CZ}^p / \mu_{CZ}^s / \mu_{CZ}$\\
		\hline & & & & & & & & &  & birth-death\\
		3.806201 & $-$0.77896 & 0.76 & 386.1 & 2.01 & $(-/-)$ $\lambda_p = 1.01$ &  & 2.10 & $(+/+)$ $\lambda_s = 1.37$ &  & 4 / 6 / 10\\
		3.2 & $-$0.77176 & 1.08 & 279.2 & 5921 & $(-/-)$ $\lambda_p = 5921$ &  & $-$1.7 & $(-/+)$ $\varphi_s = 3.62$ & 108.36 & 4 / 5 / 9\\
		3 & $-$0.76939 & 1.17 & 271.6 & 3665 & $(-/-)$ $\lambda_p = 3665$ &  & $-$1.9 & $(-/+)$ $\varphi_s = 3.45$ & 106.5 & 4 / 5 / 9\\
		2.5 & $-$0.75961 & 1.36 & 262.7 & 1649 & $(-/-)$ $\lambda_p = 1649$ &  & $-$1.9 & $(-/+)$ $\varphi_s = 3.20$ & 104.6 & 4 / 5 / 9\\
		2 & $-$0.74453 & 1.53 & 263.5 & 367 & $(-/-)$ $\lambda_p = 367$ &  & $-1.9$ & $(+/-)$ $\varphi_s = 3.03$ & 106.1 & 4 / 5 / 9\\
		1 & $-$0.69838 & 1.82 & 298.0 & 26.4 & $(-/-)$ $\lambda_p = 26.3$ &  & $-$1.7 & $(+/-)$ $\varphi_s = 2.62$ & 123.2 & 4 / 5 / 9\\
		0.5 & $-$0.66969 & 1.95 & 346.9 & 5.40 & $(-/-)$ $\lambda_p = 5.21$ &  & $-$1.2 & $(+/-)$ $\varphi_s = 2.26$ & 147.0 & 4 / 5 / 9\\
		0.1 & $-$0.65482 & 2.05 & 417.3 & 2.07 & $(-/-)$ $\lambda_p = 1.31$ &  & $-$0.6 & $(+/-)$ $\varphi_s = 1.90$ & 181.2 & 4 / 5 / 9\\
		0.015388 & $-$0.655072 & 2.07 & 436.5 & 2.00 & $\lambda_p^3 = 1.00$ &  & $-$0.5 & $(+/-)$ $\varphi_s = 1.84$ & 190.3 & 5$\shortto$3 / 5 / 10$\shortto$8\\
		$-$0.1 & $-$0.65866 & 2.10 & 465.3 & 2.12 & $(+/+)$ $\lambda_p = 1.41$ &  & $-$0.4 & $(+/-)$ $\varphi_s = 1.81$ & 203.3 & 4 / 5 / 9\\
		$-$0.5 & $-$0.7161 & 2.19 & 575.5 & 3.50 & $(+/+)$ $\lambda_p = 3.19$ &  & $-$0.8 & $(+/-)$ $\varphi_s = 2.03$ & 247.6 & 4 / 5 / 9\\
		$-$1 & $-$0.92602 & 2.39 & 709.6 & 2.92 & $(+/+)$ $\lambda_p = 2.53$ &  & $-$1.7 & $(+/-)$ $\varphi_s = 2.58$ & 294.2 & 4 / 5 / 9\\
		$-$1.3 & $-$1.1221 & 2.61 & 779.7 & 2.1 & $(+/+)$ $\lambda_p = 1.32$ &  & $-$1.9 & $(+/-)$ $\varphi_s = 3.01$ & 314.3 & 4 / 5 / 9\\
		$-$1.411618 & $-$1.199879 & 2.71 & 801.9 & 2.00 & $\lambda_p^3 = 1.00$ &  & $-$1.9 & $(-/+)$ $\varphi_s = 3.19$ & 319.7 & 3$\shortto$5 / 5 / 8$\shortto$10\\
		$-$1.6 & $-$1.32953 & 2.89 & 834.0 & 2.24 & $(-/-)$ $\lambda_p = 1.63$ &  & $-$1.8 & $(-/+)$ $\varphi_s = 3.48$ & 326.5 & 4 / 5 / 9\\
		$-$2 & $-$1.58227 & 3.28 & 882.4 & 4.63 & $(-/-)$ $\lambda_p = 4.41$ &  & $-$1.3 & $(-/+)$ $\varphi_s = 4.00$ & 334.6 & 4 / 5 / 9\\
		$-$3 & $-$2.0889 & 4.11 & 941.8 & 27.1 & $(-/-)$ $\lambda_p = 27.0$ &  & 0.25 & $(-/+)$ $\varphi_s = 4.84$ & 339.9 & 4 / 5 / 9\\
		$-$4 & $-$2.49041 & 4.83 & 967.9 & 84.7 & $(-/-)$ $\lambda_p = 84.7$ &  & 1.37 & $(-/+)$ $\varphi_s = 5.46$ & 337.2 & 4 / 5 / 9\\
		$-$5 & $-$2.49041 & 5.45 & 982.2 & 188 & $(-/-)$ $\lambda_p = 188$ &  & 2.25 & $(-/-)$ $\lambda_s = 1.64$ &  & 4 / 6 / 10\\
		$-$9 & $-$3.90443 & 7.43 & 1005 & 1223 & $(-/-)$ $\lambda_p = 1223$ &  & 5.40 & $(-/-)$ $\lambda_s = 5.21$ &  & 4 / 6 / 10
	\end{tabular}
	\caption{The family $g_3$}
	\label{table_10}
\end{table}

\subsection{Spatial periodic orbits bifurcating from planar ones}

\subsubsection{From the spatial index jump of $g$}
\label{sec:9.4.0}

From the family $g$ at the $\Gamma$ value 1.383094 (see Table \ref{table_3}), where $\mu_{CZ}^s$ jumps from 3 to 4, a new family of spatial periodic orbits bifurcates, which was found by Batkhin--Batkhina \cite{batkhin} (family $g_{2v}$).\ Some of its orbits are plotted in Figure \ref{figure_spatial_g}.\ They are doubly-symmetric with respect to
$$ \rho_1(q,p) = \{ (q_1,-q_2,q_3,-p_1,p_2,-p_3) \},\quad \rho_2(q,p) = \{ (-q_1,q_2,q_3,p_1,-p_2,-p_3) \}, $$
hence they start perpendicularly and hit perpendicularly
$$ \text{Fix}(\rho_1) = \{ (q_1,0,q_3,0,p_2,0) \},\quad \text{Fix}(\rho_2) = \{ (0,q_2,q_3,p_1,0,0) \}. $$
Recall in view of the symmetry $\sigma(q,p)=(q_1,q_2,-q_3,p_1,p_2,-p_3)$ that $\rho_1 \circ \rho_2 = \rho_2 \circ \rho_1 = -\sigma$, and note that the orbits are invariant under $-\sigma$, but by using $\sigma$, these spatial orbits give raise to yet another family of spatial orbits.\ Since $\sigma \circ \rho_i = \rho_i \circ \sigma = \overline{\rho_i}$, for $i \in \{1,2\}$, the orbits of the symmetric family are doubly-symmetric with respect to $\overline{\rho_1}$ and $\overline{\rho_2}$ (see the last row in Figure \ref{figure_spatial_g}).\ The local spatial Floer homology before this transition is
$$ sRFH^{S^1}_* (g \, ; \, \mathbb{Q}) = \begin{cases}
\mathbb{Q}, & *=5\\
0, &\text{otherwise},
\end{cases} $$
and after bifurcation the index of $g$ is 6.\ In view of the data given in Table \ref{table_10_1} (its initial data was provided on personal request by the first author of \cite{batkhin}), after this transition the index of the family $g_{2v}$ and its symmetric family equals 5.\ We check that the Euler characteristics are
$$ \chi_s(g) = (-1)^5 = -1,\quad \text{resp.}\quad \chi_s(g) = 2 \cdot (-1)^5 + (-1)^6 = -1 .$$
\begin{figure}[H]
	\centering
	\includegraphics[scale=0.59]{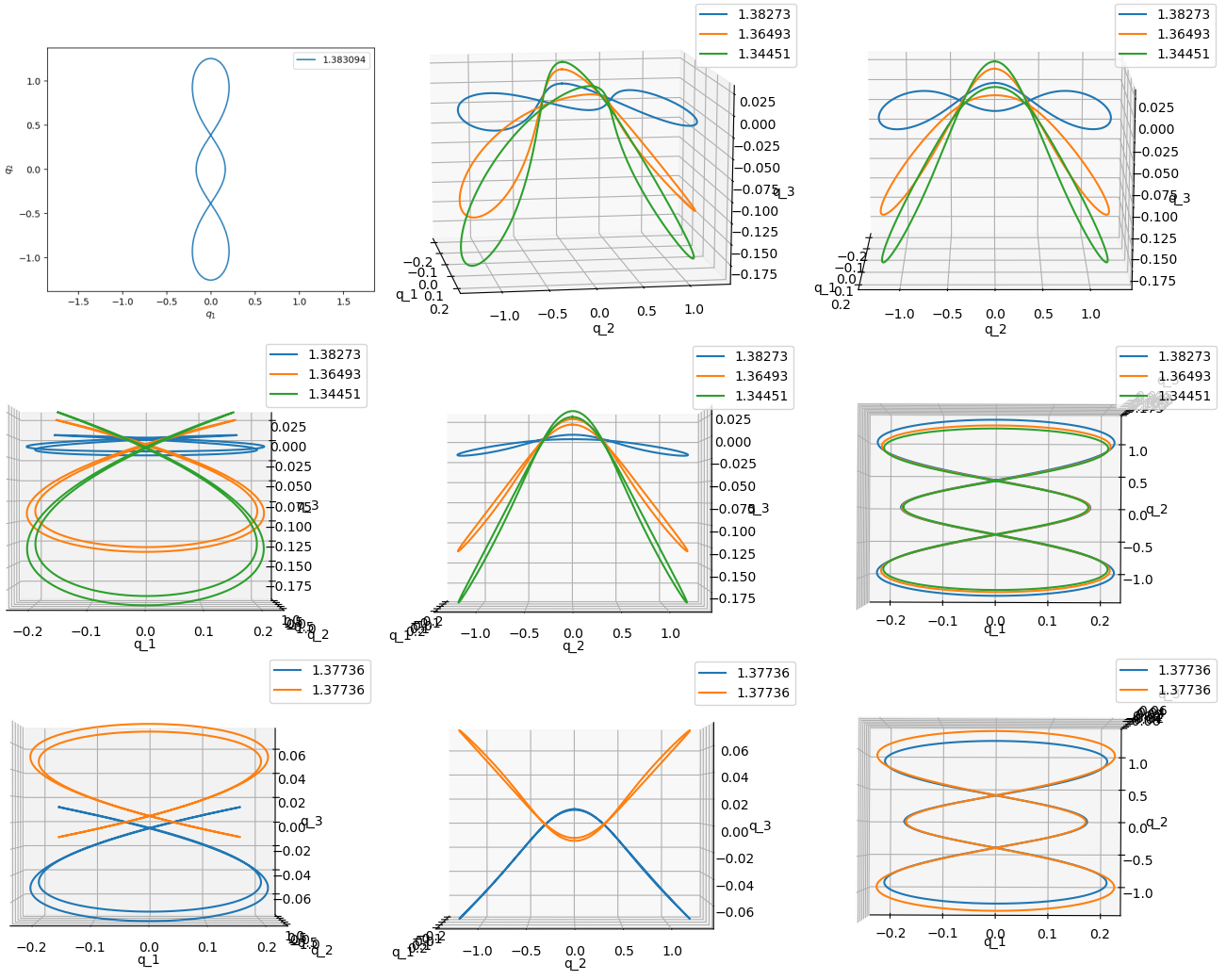}
	\caption{From the spatial index jump of $g$ (family $g_{2v}$)}
	\label{figure_spatial_g}
\end{figure}
\begin{table}[H]\scriptsize \centering
	\begin{tabular}{ccccccc}
		$\Gamma$ & $q_1(0)$ & $q_3(0)$ & $\dot{q}_2(0)$ & $T_q / 2$ & $\text{sign}_{C/B}$ and Floquet multipliers & $\mu_{CZ}$\\
		\hline 1.383094 & $-$0.164715 & 0 & $-$3.2924 & 2.56855 & $(-/-)$ $\lambda = 1916.4$ \& $\varphi_s =0$ & $2+(3 \shortto 4) = 5 \shortto 6$\\
		1.38273 & $-$0.164661 & 0.002986 & $-$3.2928 & 2.56857 & $(-/-)$ $\lambda = 1881.2$ \& $(-/+)$ $\varphi = 6.180$ & 5\\
		1.37736 & $-$0.16386 & 0.011810 & $-$3.2980 & 2.5678 & $(-/-)$ $\lambda = 1906.7$ \& $(-/+)$ $\varphi = 5.993$ & 5\\
		1.36493 & $-$0.162007 & 0.020887 & $-$3.3101 & 2.566 & $(-/-)$ $\lambda = 2013.8$ \& $(-/+)$ $\varphi = 5.791$ & 5\\
		1.34451 & $-$0.158974 & 0.030092 & $-$3.3304 & 2.563 & $(-/-)$ $\lambda = 1853.9$ \& $(-/+)$ $\varphi = 5.531$ & 5\\
		1.30865 & $-$0.153669 & 0.040951 & $-$3.3669 & 2.55756 & $(-/-)$ $\lambda = 1816.1$ \& $(-/+)$ $\varphi = 5.241$ & 5
	\end{tabular}
	\caption{From the spatial index jump of $g$ (family $g_{2v}$)}
	\label{table_10_1}
\end{table} 

\subsubsection{The 2nd cover of $g$ and the 2nd cover of $g'$}
\label{sec:9.4.1}

From the double cover of the orbits of the families $g$ and $g'$ at the $\Gamma$ values 3.057471 resp.\ 4.285183 each spatial index jumps by +2 (index jump from 9 to 11) resp.\ +1 (index jump from 10 to 11) (see Table \ref{table_3} resp.\ \ref{table_6}).\ Note that in view of Example \ref{example_7_2}, in the latter there exist bad orbits, which are the double covers of the $g'$-orbits after the index jump.\ More precisely, the underlying simple closed orbits of $g'$ are planar elliptic, and the spatial behaviour changes from elliptic to negative hyperbolic with the resp.\ indices $ \mu_{CZ}^p(g') = \mu_{CZ}^s(g') = 3$.\ The indices of the double cover before and after this transition are
$$ \mu_{CZ}^p(g'^2) = \mu_{CZ}^s(g'^2) = 5,\quad \text{resp.}\quad \mu_{CZ}^p(g'^2) = 5,\quad \mu_{CZ}^s(g'^2) = 6. $$
Recall that the bad orbits do not contribute to the local Floer homology groups nor to the Euler characteristics.\

At these transitions new families of spatial symmetric periodic orbits bifurcate.\ In particular, two families bifurcate from the double cover of $g$ and one family from the double cover of $g'$.\ All these families bifurcate after the index jump.\ Using the same notation as in Figure \ref{overview_conclusion} of the introduction, the bifurcation graph for this scenario is shown in Figure \ref{overview_double_cover}.\ Note that we draw the families of bad orbits by dashed black edges, thus the dashed black edges do not contribute to the local Floer homology groups nor to the Euler characteristics.\

\begin{figure}[H]
	\centering
	\definecolor{grgr}{RGB}{33,189,63}
	\begin{tikzpicture}[line cap=round,line join=round,>=triangle 45,x=1.0cm,y=1.0cm]
	\clip(-7.5,-1) rectangle (6,8);
	
	\draw [->, line width=1pt] (-3-3,-0.5) -- (-3-3,7.5);
	\draw (-3-3,7.5) node[anchor=south] {$\Gamma$};
	\draw (-4-3,7.5) node[anchor=south] {$-\infty$};
	\draw (-4-3,-0.5) node[anchor=north] {$+\infty$};
	\draw[fill] (-3-3,0) circle (1.5pt);
	\draw (-3-3,0) node[anchor=east] {4.285};
	\draw[fill] (-3-3,3) circle (1.5pt);
	\draw (-3-3,3) node[anchor=east] {3.057};
	\draw[fill] (-3-3,4) circle (1.5pt);
	\draw (-3-3,4) node[anchor=east] {3.013};
	\draw[fill] (-3-3,5) circle (1.5pt);
	\draw (-3-3,5) node[anchor=east] {2.952};
	\draw[fill] (-3-3,6.5) circle (1.5pt);
	\draw (-3-3,6.5) node[anchor=east] {2.484};

	\draw [line width=1pt,color=blue] (0,3) .. controls (0.5,4.5) and (1.5,5.5) .. (2,6.5);
	
	\draw [dashed,line width=1pt,color=blue] (0,3) .. controls (-0.5,4.5) and (-1.5,5.5) .. (-2,6.5);
	
	\draw [color=blue] (0.4,3.15) node[anchor=south] {$9$};
	\draw [color=blue] (-0.4,3.15) node[anchor=south] {$9$};
	
	\draw [color=blue] (0.75,3.7) node[anchor=south] {$10$};
	\draw [color=blue] (-0.75,3.7) node[anchor=south] {$10$};
	
	\draw [color=blue] (1.3,4.7) node[anchor=south] {$9$};
	\draw [color=blue] (-1.3,4.7) node[anchor=south] {$9$};
	
	\draw [line width=1pt,color=blue] (2,6.5) .. controls (2.5,4) and (2.5,3) .. (4,0);
	
	\draw [dashed,line width=1pt,color=blue] (-2,6.5) .. controls (-2.5,4) and (-2.5,3) .. (-4,0);
	
	\draw [color=blue] (3.4,0.3) node[anchor=south] {$10$};
	\draw [color=blue] (-3.4,0.3) node[anchor=south] {$10$};
	
	\draw [line width=1pt,color=grgr] (0,3) .. controls (0.2,3.1) and (0.3,3.2) .. (1.2,3.5);
	
	\draw [dashed,line width=1pt,color=grgr] (0,3) .. controls (-0.2,3.1) and (-0.3,3.2) .. (-1.2,3.5);
	
	\draw [color=grgr] (1.2,3.5) node[anchor=north] {$10$};
	\draw [color=grgr] (-1.2,3.5) node[anchor=north] {$10$};
	
	\draw [line width=1pt] (0,2.5) -- (0,3.5);
	
	\draw (0,2.5) node[anchor=north] {$9$};
	\draw (0,3.5) node[anchor=south] {$11$};
	
	\draw[fill] (0,3) circle (2pt);
	\draw (0.3,3) node[anchor=west] {$g^2$};
	
	\draw[fill] (0.33,3.78) circle (2pt);
	\draw[fill] (-0.33,3.78) circle (2pt);
	
	\draw[fill] (0.72,4.5) circle (2pt);
	\draw[fill] (-0.72,4.5) circle (2pt);
	
	\draw [line width=1pt] (4,-0.5) -- (4,0);
	\draw [dashed,line width=1pt] (4,0) -- (4,0.5);
	
	\draw (4,-0.5) node[anchor=north] {$10$};
	\draw (4,0.5) node[anchor=south] {$11$};
	
	\draw [line width=1pt] (-4,-0.5) -- (-4,0);
	\draw [dashed,line width=1pt] (-4,0) -- (-4,0.5);
	
	\draw (-4,-0.5) node[anchor=north] {$10$};
	\draw (-4,0.5) node[anchor=south] {$11$};
	
	\draw[fill] (4,0) circle (2pt);
	\draw (4.1,0) node[anchor=west] {$g'^2$};
	
	\draw[fill] (-4,0) circle (2pt);
	\draw (-4.1,0) node[anchor=east] {$g'^2$};
	
	\draw[fill] (-2,6.5) circle (2pt);
	\draw (-2,6.5) node[anchor=south] {b-d};
	
	\draw[fill] (2,6.5) circle (2pt);
	\draw (2,6.5) node[anchor=south] {b-d};
	
	\end{tikzpicture}
	\caption{The bifurcation graph between the 2nd cover of $g'$ and the 2nd cover of $g$ with the families \textcolor{blue}{$g1v$} and \textcolor{grgr}{$g_{1v}^{YOZ}$}}
	\label{overview_double_cover}
\end{figure}
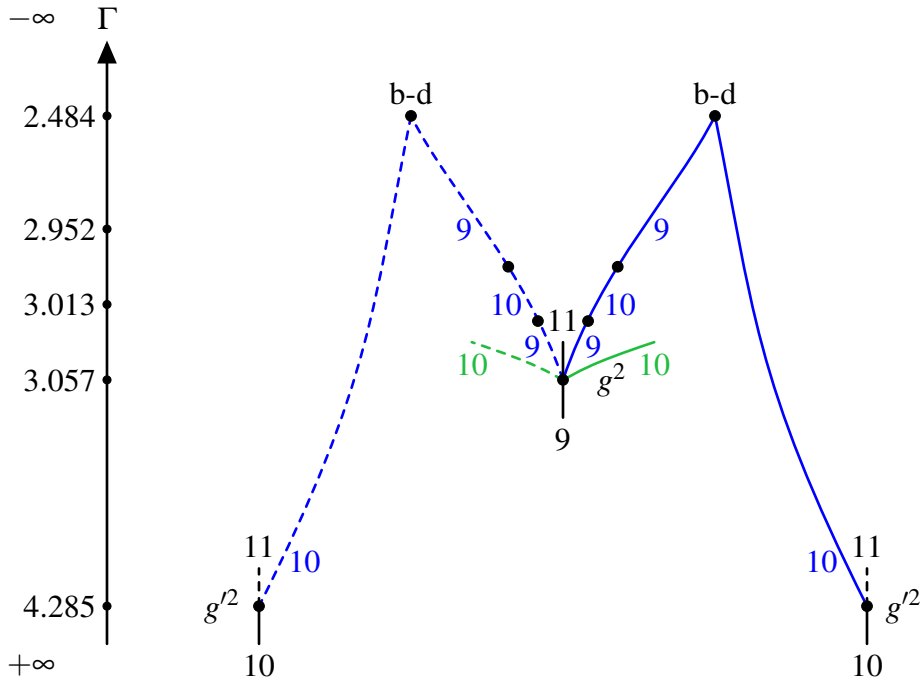
\noindent
The family \textcolor{blue}{$g1v$} bifurcates from the double cover of $g$.\ It was found by Michalodimitrakis \cite{michalodimitrakis} where the initial conditions for our data given in Table \ref{table_11} are from.\ Some of its orbits are plotted in Figure \ref{figure_double_cover_g}).\ If the energy increases, then the orbits end at a degenerate periodic orbit of birth-death type.\ The other branch bifurcation of this birth-death type periodic orbit is the branch bifurcation from the double cover of $g'$.\ All the orbits of the family \textcolor{blue}{$g1v$} are doubly-symmetric with respect to $\overline{\rho_1}$ and $\rho_1$, therefore they start perpendicularly and hit perpendicularly
$$ \text{Fix}(\overline{\rho_1}) = \{ (q_1,0,0,0,p_2,p_3) \},\quad \text{Fix}(\rho_1) = \{ (q_1,0,q_3,0,p_2,0) \}. $$
Moreover, they are invariant under $\sigma$, but the symmetry $-\sigma$ yields the symmetrical family (dashed in the bifurcation graph), whose orbits are also symmetric with respect to $\rho_2$ and $\overline{\rho_2}$ (see the last row in Figure \ref{figure_double_cover_g}).\

The other family \textcolor{grgr}{$g_{1v}^{YOZ}$} bifurcating from the double cover of $g$ was discovered by Batkhin--Batkhina \cite{batkhin}.\ The orbits are doubly-symmetric with respect to $\overline{\rho_2}$ and $\rho_2$.\ Note that we have not studied them further, but since at the value $\Gamma = 3.057$ the index of the double cover of $g$ jumps from 9 to 11 and the family \textcolor{blue}{$g1v$} and its symmetric family start with index 9, the family \textcolor{grgr}{$g_{1v}^{YOZ}$} and its symmetric family have to start with index 10.\

The Euler characteristics are
$$ \chi_s(g^2) = (-1)^9 = -1,\quad \text{resp.}\quad \chi_s(g^2) = 2\cdot(-1)^9 + 2\cdot(-1)^{10} + (-1)^{11} = -1 $$
and
$$ \chi_s(g'^2) = (-1)^{10} = 1,\quad \text{resp.}\quad \chi_s(g'^2) = (-1)^{10} + 0 \cdot(-1)^{11} = 1.$$

\begin{figure}[H]
	\centering
	\includegraphics[scale=0.59]{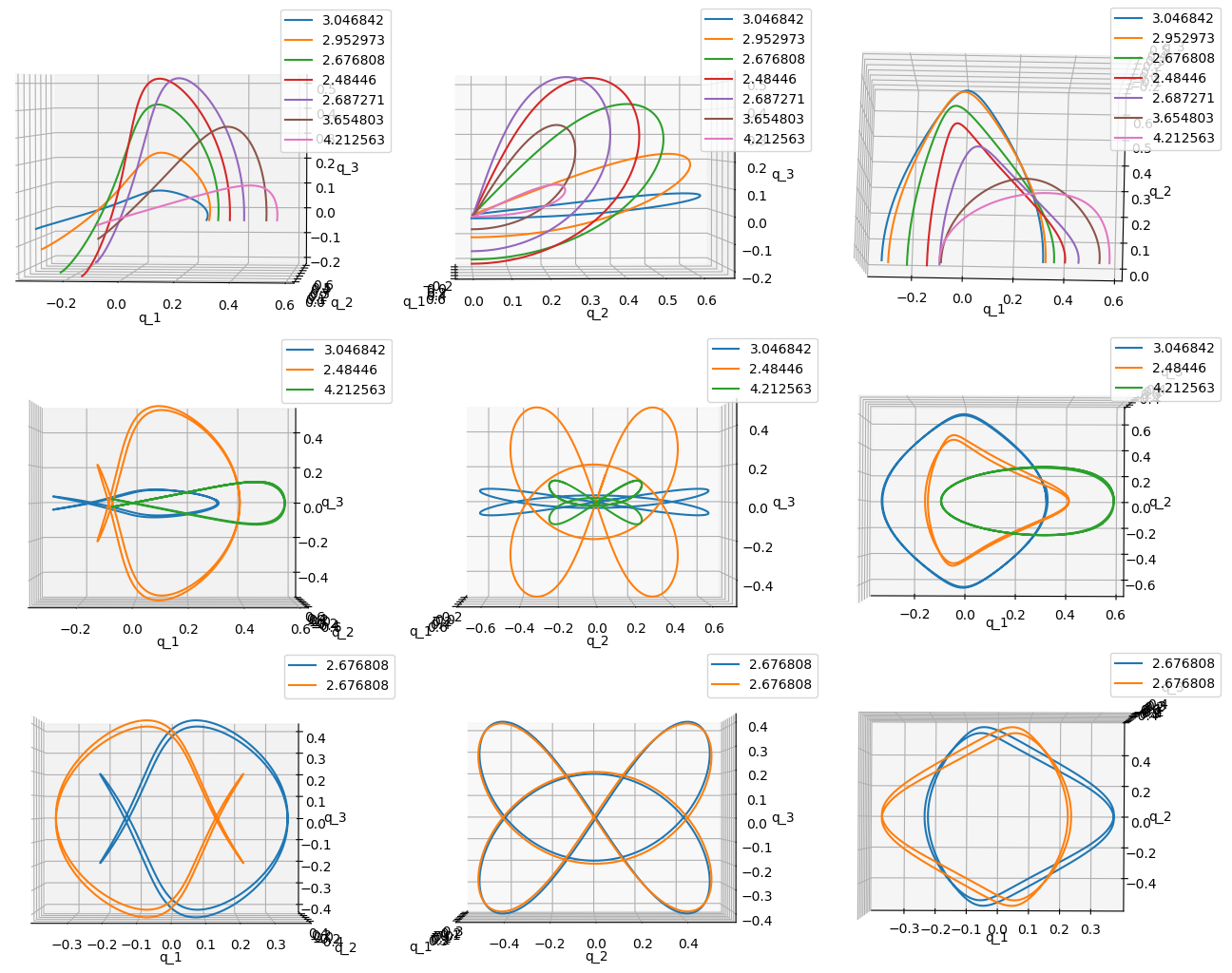}
	\caption{From the 2nd cover of $g$ to the 2nd cover of $g'$ (family \textcolor{blue}{$g1v$})}
	\label{figure_double_cover_g}
\end{figure}
\begin{table}[H]\scriptsize \centering
	\begin{tabular}{ccccccc}
		$\Gamma$ & $q_1(0)$ &  $\dot{q}_2(0)$ & $\dot{q}_3(0)$ & $T_q / 4$ & $\text{sign}_{C/B}$ and Floquet multipliers & $\mu_{CZ}$\\
		\hline 3.057471 & 0.310843 & 1.9148 & 0 & 1.4727 & $(-/-)$ $\lambda_p^2 = 24713.62$ \& $\varphi_s^2 =0$ & $4+(5 \shortto 7) = 9 \shortto 11$\\
		3.046842 & 0.311905 & 1.88871 & 0.299 & 1.4733 & $(-/-)$ $\lambda = 22211$ \& $(-/+)$ $\varphi = 5.847$ & 9\\
		3.013637 & 0.31526 & 1.807899 & 0.600 & 1.4751 & $(-/-)$ $\lambda_1 = 20289$ \& $(-/-)$ $\lambda_2 = 1.250$ & 10\\
		2.952973 & 0.321529 & 1.663566 & 0.899 & 1.4783 & $(-/-)$ $\lambda = 15610$ \& $(-/+)$ $\varphi = 6.117$ & 9\\
		2.852855 & 0.332366 & 1.43388 & 1.199 & 1.4833 & $(-/-)$ $\lambda = 10175$ \& $(-/+)$ $\varphi = 6.013$ & 9\\
		2.676808 & 0.353928 & 1.048739 & 1.500 & 1.4895 & $(-/-)$ $\lambda = 3938.3$ \& $(-/+)$ $\varphi = 5.772$ & 9\\
		2.48446 & 0.396 & 0.561861 & 1.649 & 1.4675 & $(-/-)$ $\lambda = 300.99$ \& $(-/+)$ $\varphi = 5.661$ & 9\\
		& & & & & & birth-death\\
		2.687271 & 0.44745 & 0.364873 & 1.500 & 1.3488 & $-8.925 \pm 3.401\text{i}$ \& $-0.097 \pm 0.037\text{i}$ & 10\\
		3.202673 & 0.49409 & 0.370874 & 1.199 & 1.2151 & $-3.433 \pm 6.225\text{i}$ \& $-0.067 \pm 0.123\text{i}$ & 10\\
		3.654803 & 0.52756 & 0.401491 & 0.899 & 1.1343 & $-4.565 \pm 2.778\text{i}$ \& $-0.159 \pm 0.097\text{i}$ & 10\\
		3.998524 & 0.551501 & 0.424737 & 0.600 & 1.0857 & $-2.900 \pm 0.096\text{i}$ \& $-0.344 \pm 0.011\text{i}$ & 10\\
		4.212563 & 0.565996 & 0.438274 & 0.300 & 1.0596 & $(-/+)$ $\varphi_1 = 3.507$ \& $(-/+)$  $\varphi_2 = 4.757$ & 10\\
		4.285183 & 0.570854 & 0.4426 & 0 & 1.0513 & $(-/+)$ $\varphi_p^2 = 3.72$ \& $\varphi_s^2 =0$ & $5+(5 \shortto 6) = 10 \shortto 11$
	\end{tabular}
	\caption{From the 2nd cover of $g$ to the 2nd cover of $g'$ (family \textcolor{blue}{$g1v$})}
	\label{table_11}
\end{table}

\subsubsection{The 3rd cover of $g$, the 5th cover of $f$ and the 3rd cover of $g'$}

In this subsection we collect the underlying data for the bifurcation graph in Figure \ref{overview_conclusion} and its discussion from the introduction with the families \textcolor{blue}{$f_g^{(2,3)}$}, \textcolor{grgr}{$f_g^{(2cut,3)}$}, \textcolor{red}{$f_{g'}^{(2,3)}$} and \textcolor{magenta}{$f_{g'}^{(2cut,3)}$}.\ All the inital data are provided on personal request from the author of \cite{kalantonis}.\

The orbits of the family \textcolor{blue}{$f_g^{(2,3)}$} are doubly-symmetric with respect to $\overline{\rho_1}$ and $\overline{\rho_2}$.\ Some of its orbits are plotted in Figure \ref{figure_spatial_g_f}.\ Its symmetric family is obtained by using $\sigma$ (see the last row in Figure \ref{figure_spatial_g_f}).\ The data for the orbits are given in Table \ref{table_12}.\

\begin{figure}[H]
	\centering
	\includegraphics[scale=0.5]{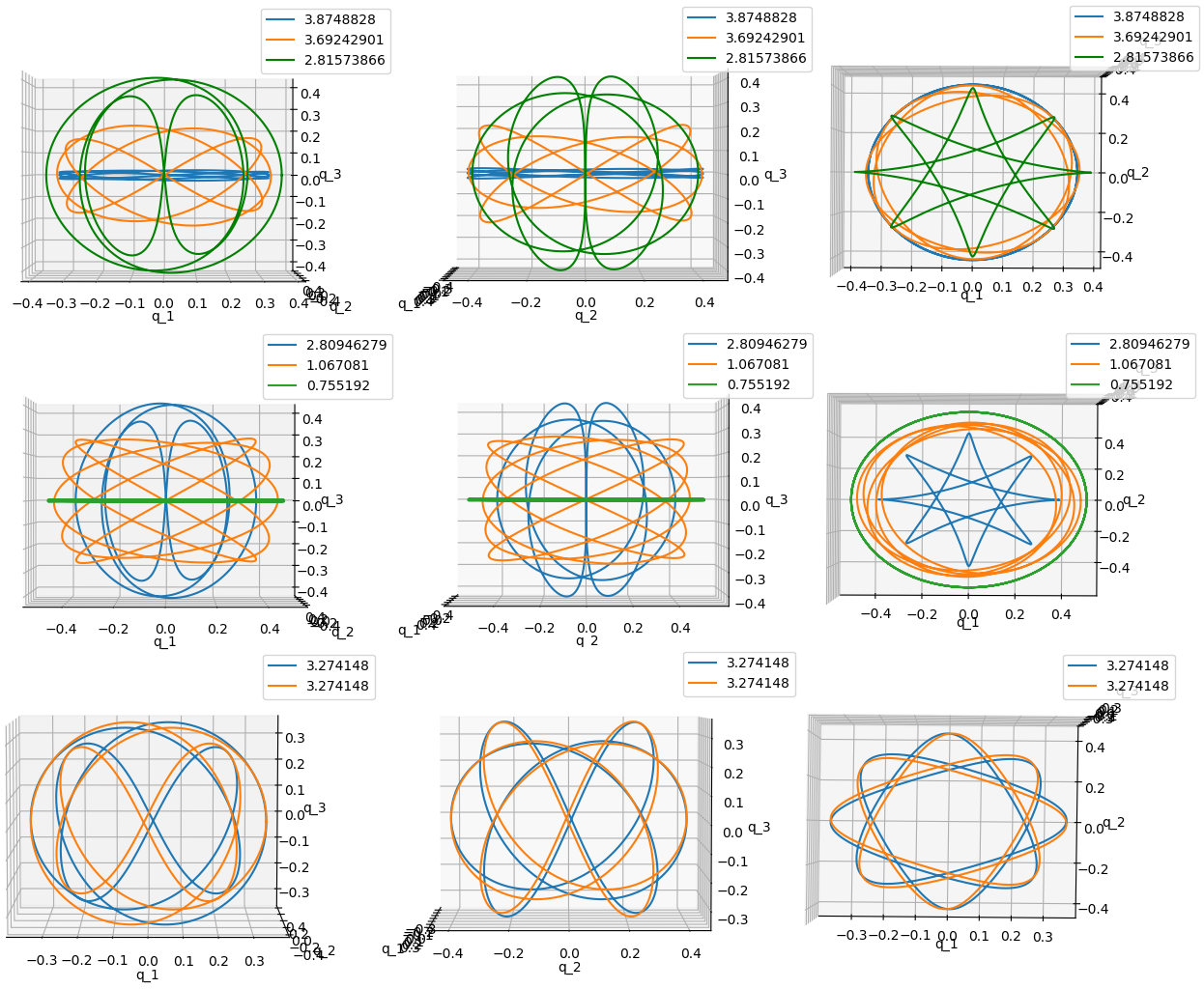}
	\caption{From the 3rd cover of $g$ to the 5th cover of $f$ (family \textcolor{blue}{$f_g^{(2,3)}$})}
	\label{figure_spatial_g_f}
\end{figure}

\begin{table}[H]\tiny \centering
	\begin{tabular}{ccccccc}
		$\Gamma$ & $q_1(0)$ &  $\dot{q}_2(0)$ & $\dot{q}_3(0)$ & $T_q / 2$ & $\text{sign}_{C/B}$ and Floquet multipliers & $\mu_{CZ}$\\
		\hline 3.876616 & 0.327645 & 1.596748 & 0 & 2.82652 & $(-/-)$ $\lambda_p^3 = 619.96$ \& $\varphi_s^3 =0$ & $6+(7 \shortto 9) = 13 \shortto 15$\\
		3.876404 & 0.327652 & 1.596390 & 0.034 & 2.82665 & $(-/-)$ $\lambda_1 = 620.39$ \& $(-/-)$ $\lambda_2 = 1.003$ & 14\\
		3.874882 & 0.327708 & 1.593827 & 0.100 & 2.82688 & $(-/-)$ $\lambda_1 = 615.97$ \& $(-/-)$ $\lambda_2 = 1.010$ & 14\\
		3.866582 & 0.328011 & 1.579855 & 0.239 & 2.82811 & $(-/-)$ $\lambda_1 = 592.21$ \& $(-/-)$ $\lambda_2 = 1.008$ & 14\\
		3.692429 & 0.334427 & 1.292631 & 0.975 & 2.85507 & $(-/-)$ $\lambda_1 = 238.02$ \& $(-/-)$ $\lambda_2 = 1.023$ & 14\\
		3.274148 & 0.350335 & 0.647245 & 1.544 & 2.92956 & $(-/-)$ $\lambda_1 = 1.228$ \& $(-/-)$ $\lambda_2 = 1.053$ & 14\\
		3.274118 & 0.350344 & 0.647179 & 1.544 & 2.92982 & $(-/-)$ $\lambda_1 = 1.231$ \& $\lambda_2 = 1$ & $14 \shortto 15$\\
		3.272832 & 0.350387 & 0.645310 & 1.545 & 2.92982 & $(-/-)$ $\lambda = 1.229$ \& $(+/-)$ $\varphi = 0.334$ & 15\\
		3.239877 & 0.351673 & 0.597055 & 1.569 & 2.93637 & $(-/-)$ $\lambda = 1.270$ \& $(+/-)$ $\varphi = 1.750$ & 15\\
		3.171140 & 0.354373 & 0.497587 & 1.612 & 2.95039 & $(-/-)$ $\lambda = 1.402$ \& $(-/+)$ $\varphi = 3.154$ & 15\\
		3.088817 & 0.357637 & 0.380540 & 1.656 & 2.96784 & $(-/-)$ $\lambda = 1.852$ \& $(-/+)$ $\varphi = 4.538$ & 15\\
		2.815738 & 0.368722 & 0.008233 & 1.736 & 3.03160 & $(-/-)$ $\lambda = 17.24$ \& $(-/+)$ $\varphi = 5.876$ & 15\\
		2.809462 & 0.368982 & $-$0.000038 & 1.737 & 3.03319 & $(-/-)$ $\lambda = 17.64$ \& $(-/+)$ $\varphi = 5.879$ & 15\\
		2.659074 & 0.375277 & $-$0.194503 & 1.747 & 3.07285 & $(-/-)$ $\lambda = 25.89$ \& $(-/+)$ $\varphi = 5.922$ & 15\\
		1.783658 & 0.415479 & $-$1.185880 & 1.463 & 3.39795 & $(-/-)$ $\lambda = 20.95$ \& $(-/+)$ $\varphi = 6.051$ & 15\\
		1.067081 & 0.456943 & $-$1.812982 & 0.805 & 3.89795 & $(-/-)$ $\lambda = 2.986$ \& $(-/+)$ $\varphi = 6.213$ & 15\\
		0.939632 & 0.466113 & $-$1.904426 & 0.613 & 4.03095 & $(-/-)$ $\lambda = 1.035$ \& $(-/+)$ $\varphi = 5.563$ & 15\\
		0.813156 & 0.476186 & $-$1.988018 & 0.338 & 4.18381 & $(-/-)$ $\lambda = 1.015$ \& $(-/+)$ $\varphi = 4.764$ & 15\\
		0.755192 & 0.481212 & $-$2.023751 & 0.010 & 4.26211 & $(-/-)$ $\lambda = 1.008$ \& $(-/+)$ $\varphi = 4.503$ & 15\\
		0.755141 & 0.481217 & $-$2.02378 & 0 & 4.26215 & $\varphi_s^5 = 0$ \& $(-/+)$ $\varphi_p^5 = 4.503$ & $7+(9 \shortto 7) = 16 \shortto 14$
	\end{tabular}
	\caption{From the 3rd cover of $g$ to the 5th cover of $f$ (family \textcolor{blue}{$f_g^{(2,3)}$})}
	\label{table_12}
\end{table}

The orbits of the family \textcolor{red}{$f_{g'}^{(2,3)}$} are simply-symmetric with respect to $\overline{\rho_1}$.\ Some of its orbits are plotted in Figure \ref{figure_spatial_g'_g_1}.\ The 3rd row in Figure \ref{figure_spatial_g'_g_1} shows the symmetric orbit obtained by using $\overline{\rho_2}$, which bifurcates from the planar orbit which is symmetric to $g'$.\ The symmetry $\sigma$ yields the symmetric family bifurcation from the same planar orbit of $g'$ (see the last row in Figure \ref{figure_spatial_g'_g_1}).\ The data for the orbits are given in Table \ref{table_14}.\

\begin{figure}[H]
	\centering
	\includegraphics[scale=0.48]{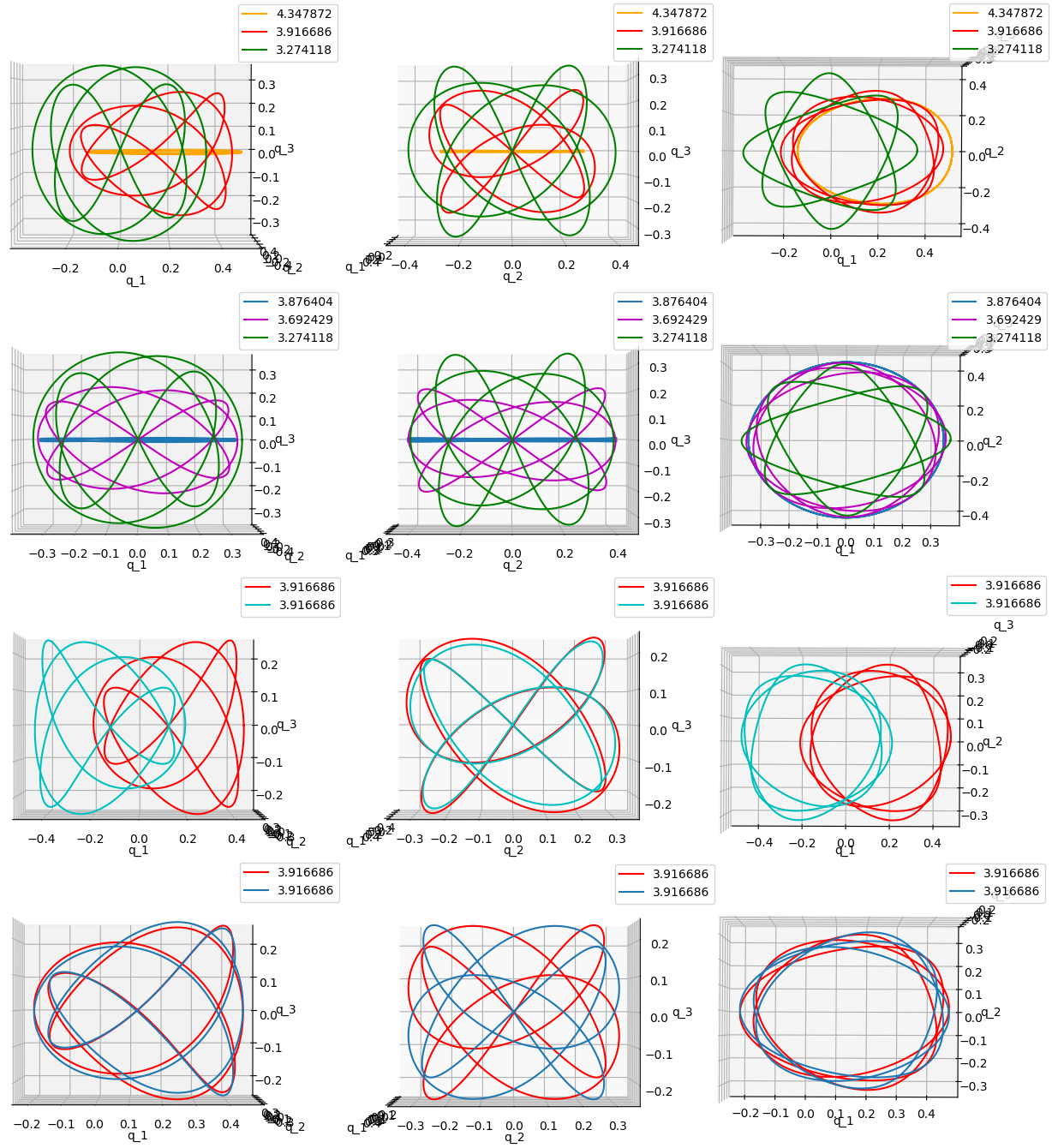}
	\caption{From the 3rd cover of $g'$ to the index jump $(14 \shortto 15)$ in Table \ref{table_12} (family \textcolor{red}{$f_{g'}^{(2,3)}$})}
	\label{figure_spatial_g'_g_1}
\end{figure}

\begin{table}[H]\tiny \centering
	\begin{tabular}{ccccccc}
		$\Gamma$ & $q_1(0)$ &  $\dot{q}_2(0)$ & $\dot{q}_3(0)$ & $T_q / 2$ & $\text{sign}_{C/B}$ and Floquet multipliers & $\mu_{CZ}$\\
		\hline 4.347942 & 0.491443 & 0.668022 & 0 & 2.41792 & $(-/+)$ $\varphi_p^3 = 3.260$ \& $\varphi_s^3 =0$ & $7+(7 \shortto 9) = 14 \shortto 16$\\
		4.347872 & 0.491437 & 0.668012 & 0.010 & 2.41797 & $(-/+)$ $\varphi = 3.261$ \& $(-/-)$ $\lambda = 1.014$ & 15\\
		4.332819 & 0.490248 & 0.666260 & 0.154 & 2.42320 & $(-/+)$ $\varphi = 3.238$ \& $(-/-)$ $\lambda = 1.006$ & 15\\
		4.308139 & 0.488286 & 0.663388 & 0.251 & 2.43186 & $(-/-)$ $\lambda_1 = -1.008$ \& $(-/-)$ $\lambda_2 = 1.012$ & 15\\
		3.916686 & 0.454659 & 0.618385 & 0.848 & 2.58670 & $(-/-)$ $\lambda_2 = -1.591$ \& $(-/-)$ $\lambda_2 = 1.390$ & 15\\
		3.720467 & 0.435295 & 0.597320 & 1.042 & 2.67843 & $(+/-)$ $\varphi = 3.126$ \& $(-/-)$ $\lambda = 1.705$ & 15\\
		3.419999 & 0.397616 & 0.577871 & 1.322 & 2.84056 & $(+/-)$ $\varphi = 1.622$ \& $(-/-)$ $\lambda = 2.049$ & 15\\
		3.274124 & 0.350644 & 0.646271 & 1.542 & 2.92956 & $(+/-)$ $\varphi = 0.030$ \& $(-/-)$ $\lambda = 1.231$ & 15\\
		3.274118 & 0.350344 & 0.647179 & 1.544 & 2.92982 & $\lambda_1 = 1$ \& $(-/-)$ $\lambda_2 = 1.231$ & $14 \shortto 15$\\
		3.274123 & 0.350042 & 0.648097 & 1.545 & 2.92956 & $(+/-)$ $\varphi = 0.029$ \& $(-/-)$ $\lambda = 1.230$ & 15\\
		3.690961 & 0.238777 & 1.376348 & 1.720 & 2.69315 & $(+/-)$ $\varphi = 2.849$ \& $(-/-)$ $\lambda = 1.755$ & 15\\
		3.917774 & 0.199644 & 1.850406 & 1.672 & 2.58622 & $(-/-)$ $\lambda_1 = -1.592$ \& $(-/-)$ $\lambda_2 = 1.389$ & 15\\
		4.308371 & 0.137956 & 3.113046 & 0.744 & 2.43178 & $(+/-)$ $\varphi = 3.145$ \& $(-/-)$ $\lambda = 1.016$ & 15\\
		4.347515 & 0.132083 & 3.292440 & 0.080 & 2.41809 & $(+/-)$ $\varphi = 3.265$ \& $(-/-)$ $\lambda = 1.053$ & 15\\
		4.347942 & 0.132020 & 3.294475 & 0 & 2.41792 & $(-/+)$ $\varphi_p^3 = 3.260$ \& $\varphi_s^3 =0$ & $7+(7 \shortto 9) = 14 \shortto 16$
	\end{tabular}
	\caption{From the 3rd cover of $g'$ to the index jump $(14 \shortto 15)$ in Table \ref{table_12} (family \textcolor{red}{$f_{g'}^{(2,3)}$})}
	\label{table_14}
\end{table}

The orbits of the family \textcolor{grgr}{$f_g^{(2cut,3)}$} are doubly-symmetric with respect to $\rho_1$ and $\rho_2$.\ Some of these orbits are plotted in Figure \ref{figure_spatial_g'_g_2}.\ Its symmetric family is obtained by using $\sigma$ (see the last row in Figure \ref{figure_spatial_g'_g_2}).\ The data for the orbits are given in Table \ref{table_13}.\

\begin{figure}[H]
	\centering
	\includegraphics[scale=0.59]{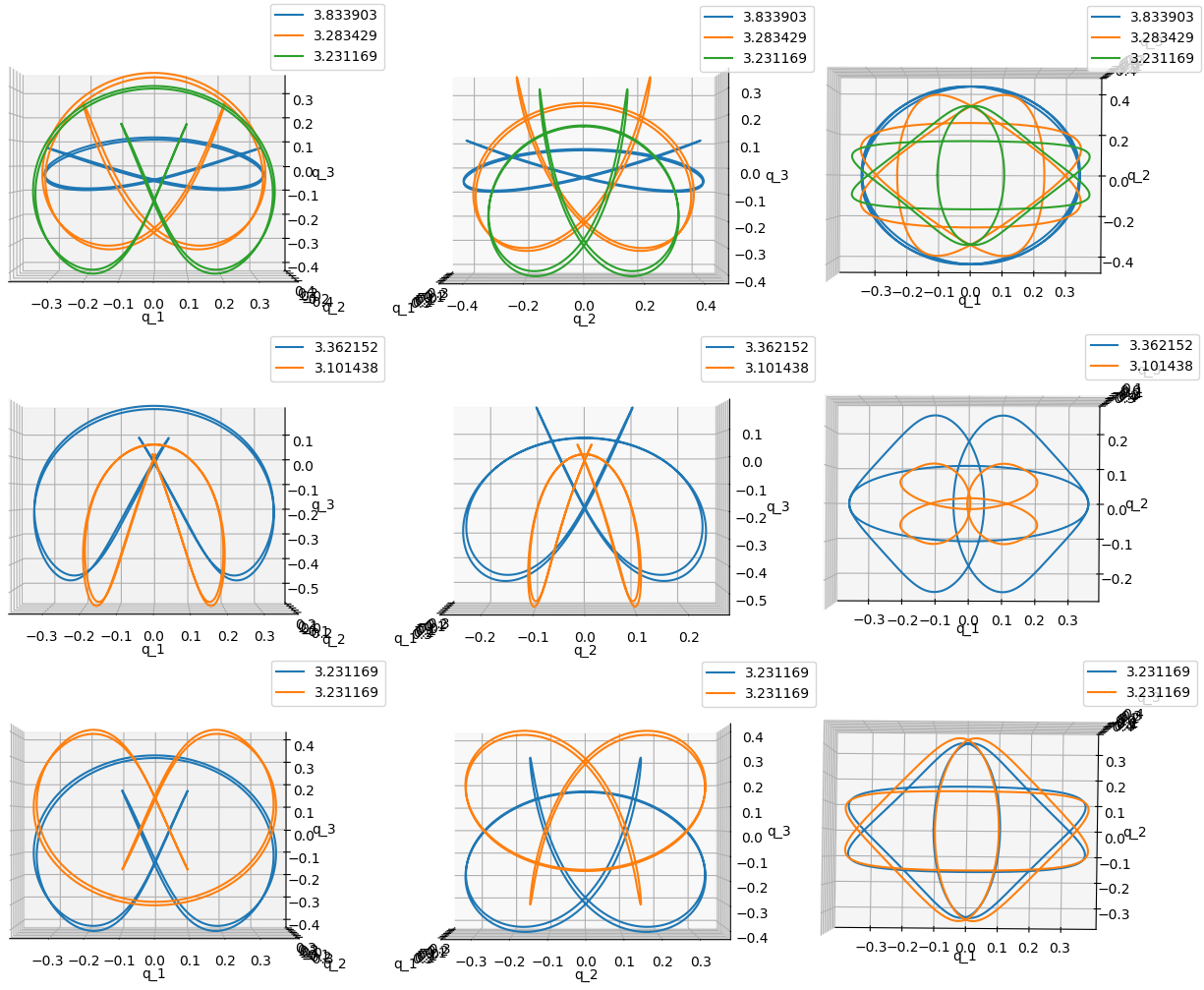}
	\caption{From the 3rd cover of $g$ to collision (family \textcolor{grgr}{$f_g^{(2cut,3)}$})}
	\label{figure_spatial_g'_g_2}
\end{figure}

\begin{table}[H]\tiny \centering
	\begin{tabular}{ccccccc}
		$\Gamma$ & $q_1(0)$ & $q_3(0)$ & $\dot{q}_2(0)$ & $T_q / 2$ & $\text{sign}_{C/B}$ and Floquet multipliers & $\mu_{CZ}$\\
		\hline 3.876616 & $-$0.327645 & 0 & $-$1.596 & 2.82652 & $(-/-)$ $\lambda_p^3 = 619.96$ \& $\varphi_s^3 =0$ & $6+(7 \shortto 9) = 13 \shortto 15$\\
		3.833903 & $-$0.320106 & 0.072500 & $-$1.600 & 2.83301 & $(-/-)$ $\lambda = 505.63$ \& $(-/+)$ $\varphi = 6.278$ & 13\\
		3.283429 & $-$0.212345 & 0.250499 & $-$1.696 & 2.92604 & $(-/-)$ $\lambda = 1.725$ \& $(-/+)$ $\varphi = 6.048$ & 13\\
		3.280180 & $-$0.211571 & 0.250950 & $-$1.698 & 2.92662 & $\lambda = 1$ \& $(-/+)$ $\varphi = 6.081$ & $13 \shortto 14$\\
		3.279799 & $-$0.211477 & 0.250999 & $-$1.698 & 2.92669 & $(+/-)$ $\varphi_1 = 0.186$ \& $(-/+)$ $\varphi_2 = 6.080$ & 14\\
		3.189269 & $-$0.186389 & 0.258719 & $-$1.766 & 2.94124 & $(+/-)$ $\varphi_1 = 3.129$ \& $(-/+)$ $\varphi_2 = 5.968$ & 14\\
		3.136701 & $-$0.150862 & 0.240737 & $-$1.978 & 2.92889 & $(-/+)$ $\varphi_1 = 6.277$ \& $(-/+)$ $\varphi_2 = 4.689$ & 14\\
		3.136701 & $-$0.150811 & 0.24068 & $-$1.978 & 2.92882 & $\lambda = 1$ \& $(-/+)$ $\varphi = 4.691$ & birth-death\\
		3.231169 & $-$0.097655 & 0.166387 & $-$2.671 & 2.77297 & $(+/+)$ $\lambda = 3.712$ \& $(-/+)$ $\varphi = 4.346$ & 15\\
		3.362152 & $-$0.042376 & 0.077137 & $-$4.400 & 2.43258 & $\lambda = 1$ \& $(-/+)$ $\varphi = 3.498$ & birth-death\\
		3.362152 & $-$0.042344 & 0.077087 & $-$4.401 & 2.43234 & $(-/+)$ $\varphi_1 = 6.249$ \& $(-/+)$ $\varphi_2 = 3.497$ & 14\\
		3.329430 & $-$0.024918 & 0.050537 & $-$5.671 & 2.29164 & $(-/+)$ $\varphi_1 = 5.662$ \& $(-/+)$ $\varphi_2 = 3.143$ & 14\\
		3.101438 & $-$0.004205 & 0.014387 & $-$11.41 & 2.06718 & $(-/+)$ $\varphi_1 = 6.012$ \& $(-/+)$ $\varphi_2 = 2.426$ & 14
	\end{tabular}
	\caption{From the 3rd cover of $g$ to collision (family \textcolor{grgr}{$f_g^{(2cut,3)}$})}
	\label{table_13}
\end{table}

The orbits of the family \textcolor{magenta}{$f_{g'}^{(2cut,3)}$} are simply-symmetric with respect to $\rho_1$.\ Some of these orbits are plotted in Figure \ref{figure_spatial_g'_g_3}.\ The 3rd row in Figure \ref{figure_spatial_g'_g_3} shows the symmetric orbits obtained by using $\rho_2$, which bifurcate from the planar orbit which is symmetric to $g'$.\ The symmetry $\sigma$ yields the symmetric family bifurcation from the same planar orbit of $g'$ (see the last row in Figure \ref{figure_spatial_g'_g_3}).\ The data for the orbits are given in Table \ref{table_15}.\

\begin{figure}[H]
	\centering
	\includegraphics[scale=0.54]{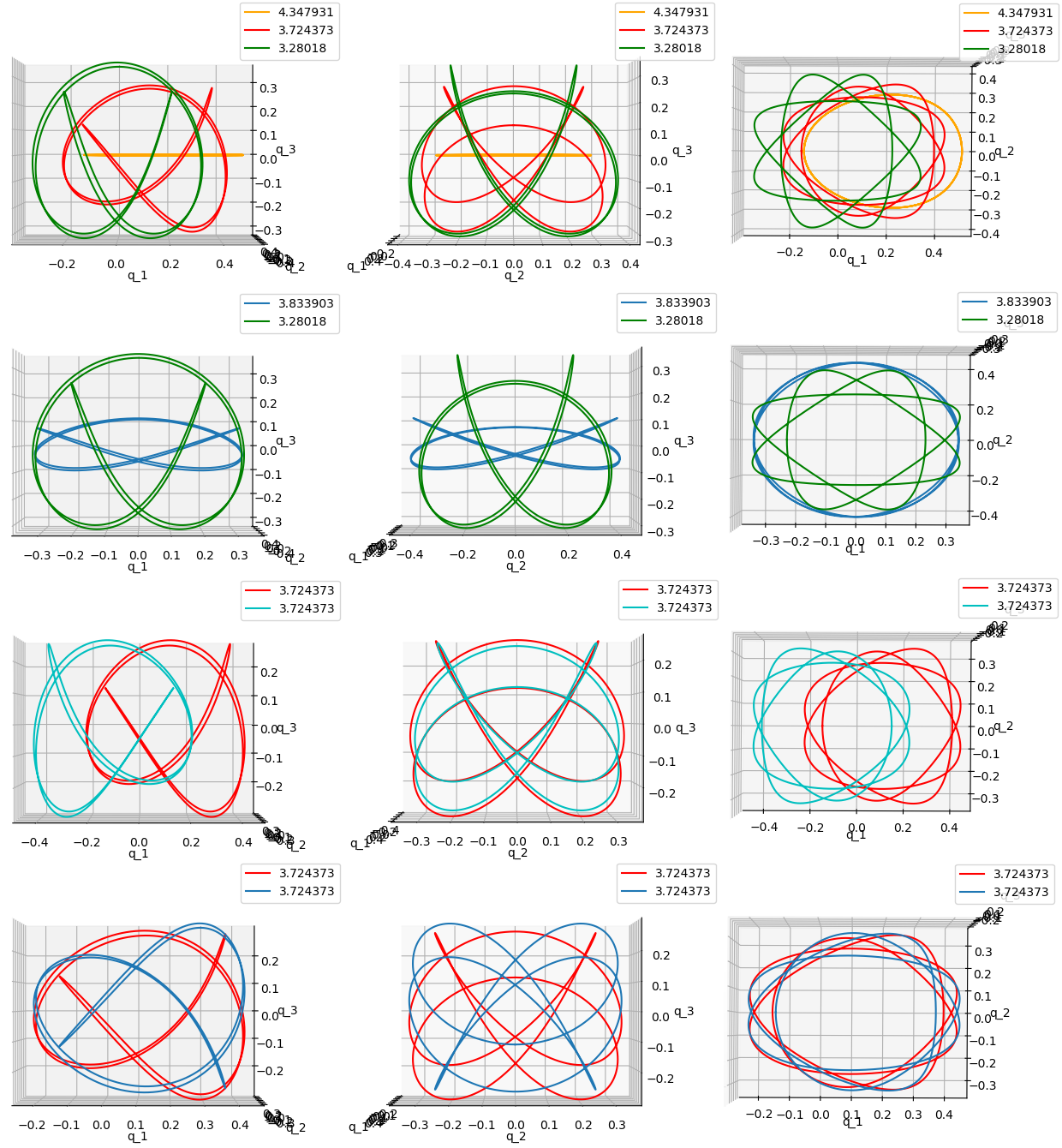}
	\caption{From the 3rd cover of $g'$ to the index jump $(13 \shortto 14)$ in Table \ref{table_13} (family \textcolor{magenta}{$f_{g'}^{(2cut,3)}$})}
	\label{figure_spatial_g'_g_3}
\end{figure}

\begin{table}[H]\tiny \centering
	\begin{tabular}{ccccccc}
		$\Gamma$ & $q_1(0)$ & $q_3(0)$ & $\dot{q}_2(0)$ & $T_q / 2$ & $\text{sign}_{C/B}$ and Floquet multipliers & $\mu_{CZ}$\\
		\hline 4.347942 & $-$0.132020 & 0 & $-$3.294 & 2.41792 & $(-/+)$ $\varphi_p^3 = 3.260$ \& $\varphi_s^3 =0$ & $7+(7 \shortto 9) = 14 \shortto 16$\\
		4.347931 & $-$0.132019 & 0.000400 & $-$3.294 & 2.41795 & $(-/+)$ $\varphi_1 = 3.261$ \& $(-/+)$ $\varphi_2 = 6.275$ & 14\\
		4.307008 & $-$0.131939 & 0.024999 & $-$3.261 & 2.43226 & $(-/-)$ $\lambda = -1.020$ \& $(-/+)$ $\varphi = 5.858$ & 14\\
		3.724373 & $-$0.138067 & 0.120000 & $-$2.692 & 2.67594 & $(+/-)$ $\varphi_1 = 3.051$ \& $(-/+)$ $\varphi_2 = 5.741$ & 14\\
		3.381250 & $-$0.165321 & 0.192269 & $-$2.133 & 2.86377 & $(+/-)$ $\varphi_1 = 1.013$ \& $(-/+)$ $\varphi_2 = 5.311$ & 14\\
		3.381167 & $-$0.165336 & 0.192296 & $-$2.133 & 2.86382 & $0.551 \pm 0.816 \text{i}$ \& $0.541 \pm 0.858 \text{i}$ & 14\\
		3.329148 & $-$0.177357 & 0.211699 & $-$1.990 & 2.89575 & $0.973 \pm 0.919 \text{i}$ \& $0.543 \pm 0.511 \text{i}$ & 14\\
		3.280274 & $-$0.209888 & 0.249451 & $-$1.710 & 2.92656 & $1.039 \pm 0.161 \text{i}$ \& $0.939 \pm 0.145 \text{i}$ & 14\\
		3.280237 & $-$0.210257 & 0.249782 & $-$1.707 & 2.92659 & $(+/-)$ $\varphi_1 = 0.131$ \& $(-/+)$ $\varphi_2 = 6.127$ & 14\\
		3.280180 & $-$0.211571 & 0.250950 & $-$1.698 & 2.92662 & $\lambda = 1$ \& $(-/+)$ $\varphi = 6.081$ & $13 \shortto 14$\\
		3.280236 & $-$0.212874 & 0.252084 & $-$1.689 & 2.92659 & $(+/-)$ $\varphi_1 = 0.122$ \& $(-/+)$ $\varphi_2 = 6.121$ & 14\\
		3.280241 & $-$0.212932 & 0.252134 & $-$1.688 & 2.92658 & $1.002 \pm 0.146 \text{i}$ \& $0.976 \pm 0.142 \text{i}$ & 14\\
		3.357252 & $-$0.266603 & 0.281661 & $-$1.390 & 2.87838 & $0.773 \pm 0.967 \text{i}$ \& $0.504 \pm 0.630 \text{i}$ & 14\\
		3.380289 & $-$0.275442 & 0.283788 & $-$1.350 & 2.86435 & $0.571 \pm 0.871 \text{i}$ \& $0.527 \pm 0.801 \text{i}$ & 14\\
		3.381154 & $-$0.275758 & 0.283852 & $-$1.349 & 2.86383 & $(+/-)$ $\varphi_1 = 1.006$ \& $(-/+)$ $\varphi_2 = 5.304$ & 14\\
		3.454933 & $-$0.300084 & 0.286253 & $-$1.247 & 2.82013 & $(+/-)$ $\varphi_1 = 1.663$ \& $(-/+)$ $\varphi_2 = 5.497$ & 14\\
		3.727521 & $-$0.370162 & 0.266224 & $-$0.999 & 2.67439 & $(+/+)$ $\lambda = -1.025$ \& $(-/+)$ $\varphi_2 = 5.744$ & 14\\
		4.347932 & $-$0.491441 & 0.001213 & $-$0.668 & 2.41795 & $(-/+)$ $\varphi_1 = 3.261$ \& $(-/+)$ $\varphi_2 = 6.279$ & 14\\
		4.347942 & $-$0.491443 & 0 & $-$0.668 & 2.41792 & $(-/+)$ $\varphi_p^3 = 3.260$ \& $\varphi_s^3 =0$ & $7+(7 \shortto 9) = 14 \shortto 16$
	\end{tabular}
	\caption{From the 3rd cover of $g'$ to the index jump $(13 \shortto 14)$ in Table \ref{table_13} (family \textcolor{magenta}{$f_{g'}^{(2cut,3)}$})}
	\label{table_15}
\end{table}

\subsubsection{The 4th cover of $g$, the 6th cover of $f$ and the 4th cover of $g'$}

At the value $\Gamma = 4.435711$ the spatial index of the 4th cover of the orbit of the family $g'$ jumps by +2 and its index jumps from 18 to 20 (see Table \ref{table_6}).\ Moreover, at the value $\Gamma = 4.278924$ the spatial index of the 4th cover of $g$ jumps from 17 to 19 (see Table \ref{table_3}).\ The bifurcation graph for this case is illustrated in Figure \ref{overview_fourth_cover}.\

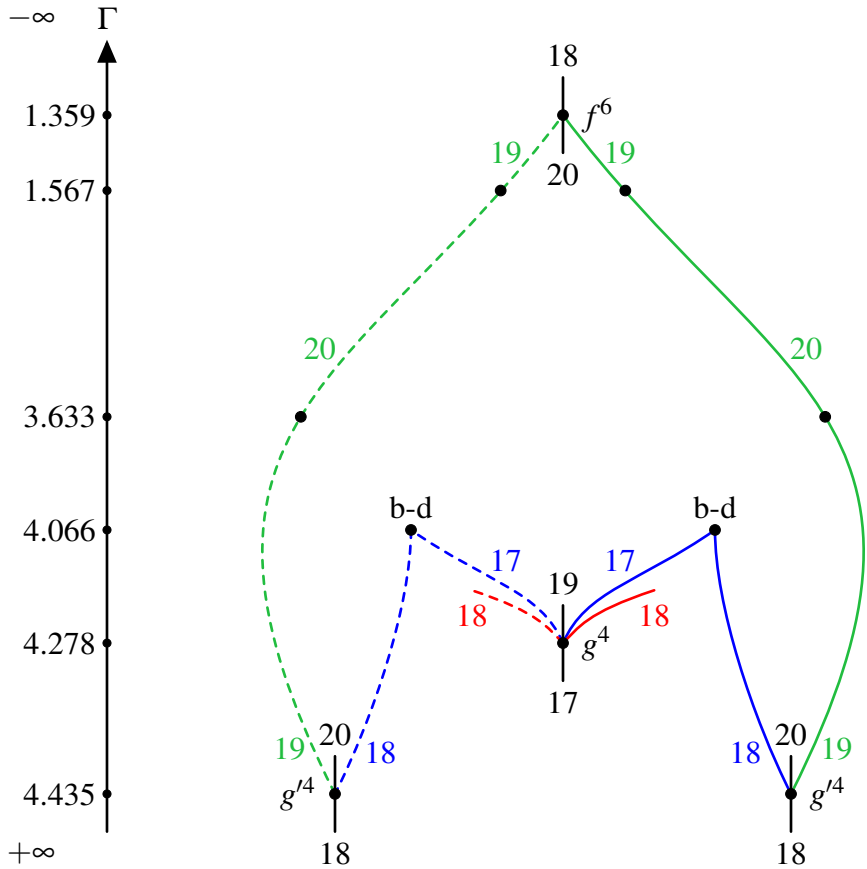
\begin{figure}[H]
	\centering
	\definecolor{grgr}{RGB}{33,189,63}
	\begin{tikzpicture}[line cap=round,line join=round,>=triangle 45,x=1.0cm,y=1.0cm]
	\clip(-7.5,-1) rectangle (6,10.5);
	
	\draw [->, line width=1pt] (-3-3,-0.5) -- (-3-3,10);
	\draw (-3-3,10) node[anchor=south] {$\Gamma$};
	\draw (-4-3,10) node[anchor=south] {$-\infty$};
	\draw (-4-3,-0.5) node[anchor=north] {$+\infty$};
	\draw[fill] (-3-3,0) circle (1.5pt);
	\draw (-3-3,0) node[anchor=east] {4.435};
	\draw[fill] (-3-3,2) circle (1.5pt);
	\draw (-3-3,2) node[anchor=east] {4.278};
	\draw[fill] (-3-3,3.5) circle (1.5pt);
	\draw (-3-3,3.5) node[anchor=east] {4.066};
	\draw[fill] (-3-3,5) circle (1.5pt);
	\draw (-3-3,5) node[anchor=east] {3.633};
	\draw[fill] (-3-3,8) circle (1.5pt);
	\draw (-3-3,8) node[anchor=east] {1.567};
	\draw[fill] (-3-3,9) circle (1.5pt);
	\draw (-3-3,9) node[anchor=east] {1.359};

	\draw [line width=1pt,color=blue] (0,2) .. controls (0.3,2.7) and (1,2.8) .. (2,3.5);
	
	\draw [dashed,line width=1pt,color=blue] (0,2) .. controls (-0.3,2.7) and (-1,2.8) .. (-2,3.5);
	
	\draw [color=blue] (0.75,2.8) node[anchor=south] {$17$};
	\draw [color=blue] (-0.75,2.8) node[anchor=south] {$17$};
	
	\draw [line width=1pt,color=red] (0,2) .. controls (0.2,2.2) and (0.3,2.4) .. (1.2,2.7);
	
	\draw [dashed,line width=1pt,color=red] (0,2) .. controls (-0.2,2.2) and (-0.3,2.4) .. (-1.2,2.7);
	
	\draw [color=red] (1.2,2.65) node[anchor=north] {$18$};
	\draw [color=red] (-1.2,2.65) node[anchor=north] {$18$};

	\draw [line width=1pt,color=blue] (2,3.5) .. controls (2,2.5) and (2.5,1) .. (3,0);
	
	\draw [dashed,line width=1pt,color=blue] (-2,3.5) .. controls (-2,2.5) and (-2.5,1) .. (-3,0);
	
	\draw [color=blue] (2.4,0.25) node[anchor=south] {$18$};
	\draw [color=blue] (-2.4,0.25) node[anchor=south] {$18$};

	\draw [line width=1pt,color=grgr] (3,0) .. controls (5.5,5) and (2.6,5.5) .. (0,9);
	
	\draw [dashed,line width=1pt,color=grgr] (-3,0) .. controls (-5.5,5) and (-2.6,5.5) .. (0,9);
	
	\draw [color=grgr] (3.6,0.3) node[anchor=south] {$19$};
	\draw [color=grgr] (-3.6,0.3) node[anchor=south] {$19$};
	
	\draw [color=grgr] (3.2,6.2) node[anchor=north] {$20$};
	\draw [color=grgr] (-3.2,6.2) node[anchor=north] {$20$};
	
	\draw [color=grgr] (0.75,8.8) node[anchor=north] {$19$};
	\draw [color=grgr] (-0.75,8.8) node[anchor=north] {$19$};
	
	\draw [line width=1pt] (3,-0.5) -- (3,0.5);
	\draw (3,-0.5) node[anchor=north] {$18$};
	\draw (3,0.5) node[anchor=south] {$20$};
	
	\draw[fill] (3,0) circle (2pt);
	\draw (3.1,0) node[anchor=west] {$g'^4$};
	
	\draw [line width=1pt] (-3,-0.5) -- (-3,0.5);
	\draw (-3,-0.5) node[anchor=north] {$18$};
	\draw (-3,0.5) node[anchor=south] {$20$};
	
	\draw[fill] (-3,0) circle (2pt);
	\draw (-3.1,0) node[anchor=east] {$g'^4$};
	
	\draw [line width=1pt] (0,1.5) -- (0,2.5);
	\draw (0,1.5) node[anchor=north] {$17$};
	\draw (0,2.5) node[anchor=south] {$19$};
	
	\draw[fill] (0,2) circle (2pt);
	\draw (0.1,2) node[anchor=west] {$g^4$};
	
	\draw [line width=1pt] (0,8.5) -- (0,9.5);
	\draw (0,8.5) node[anchor=north] {$20$};
	\draw (0,9.5) node[anchor=south] {$18$};
	
	\draw[fill] (0,9) circle (2pt);
	\draw (0.1,9) node[anchor=west] {$f^6$};

	\draw[fill] (-2,3.5) circle (2pt);
	\draw (-2,3.5) node[anchor=south] {b-d};
	
	\draw[fill] (2,3.5) circle (2pt);
	\draw (2,3.5) node[anchor=south] {b-d};
	
	\draw[fill] (0.82,8) circle (2pt);
	\draw[fill] (-0.82,8) circle (2pt);
	\draw[fill] (3.45,5) circle (2pt);
	\draw[fill] (-3.45,5) circle (2pt);

	\end{tikzpicture}
	\caption{The bifurcation graph between the 4th cover of $g'$, the 4th cover of $g$ and the 6th cover of $f$ with the families \textcolor{blue}{$f_{g'}^{(1,4)}$}, \textcolor{grgr}{$f_{g'}^{(1cut,4)}$} and \textcolor{red}{$f_{g}^{(1,4)}$}}
	\label{overview_fourth_cover}
\end{figure}
\noindent
The three families \textcolor{blue}{$f_{g'}^{(1,4)}$}, \textcolor{grgr}{$f_{g'}^{(1cut,4)}$} and \textcolor{red}{$f_{g}^{(1,4)}$} were found by Kalantonis \cite{kalantonis}, who provided on personal request the initial data.\ All orbits are doubly-symmetric with respect to $\overline{\rho_1}$ and $\rho_1$, hence they are invariant under $\sigma$.\ Each symmetric family (dashed) is obtained by using the symmetry $-\sigma$.\ Note that the orbits of each symmetric family are doubly-symmetric with respect to $\overline{\rho_2}$ and $\rho_2$.\

The family \textcolor{blue}{$f_{g'}^{(1,4)}$} consists of two branches bifurcating respectively from the 4th cover of $g'$ at the value $\Gamma = 4.435$ and from the 4th cover of $g$ at the value $\Gamma = 4.278$.\ The two branches meet at the value $\Gamma = 4.066$ at a degenerate orbit of birth-death type.\ Some orbits of the family \textcolor{blue}{$f_{g'}^{(1,4)}$} are plotted in Figure \ref{figure_spatial_g'_f_1}, where the last row shows a symmetric orbit.\ The data for the orbits are collected in Table \ref{table_17}.\

The family \textcolor{grgr}{$f_{g'}^{(1cut,4)}$} bifurcating from the 4th cover of $g'$ at the value $\Gamma = 4.435$ ends planar at the 6th cover of the retrograde orbit at the value $\Gamma = 1.359$.\ Note that inbetween there are two index jumps.\ In view of Table \ref{table_9}, the index of $f^6$ at the value $\Gamma = 1.359$ jumps from 20 to 18.\ At this transition, the Euler characteristics show that there are still undiscovered families branching out from $f^6$.\ Some of the orbits of the family \textcolor{grgr}{$f_{g'}^{(1cut,4)}$} are plotted in Figure \ref{figure_spatial_g'_f}, where the last row shows a symmetric orbit, and the data are collected in Table \ref{table_16}.\

We have not studied the family \textcolor{red}{$f_{g}^{(1,4)}$} further, but since at the value $\Gamma = 4.278$ the index of the 4th cover of $g$ jumps from 17 to 19 and the family \textcolor{blue}{$f_{g'}^{(1,4)}$} and its symmetric family start with index 17, the family \textcolor{red}{$f_{g}^{(1,4)}$} and its symmetric family have to start with index 18.\

\begin{figure}[H]
	\centering
	\includegraphics[scale=0.56]{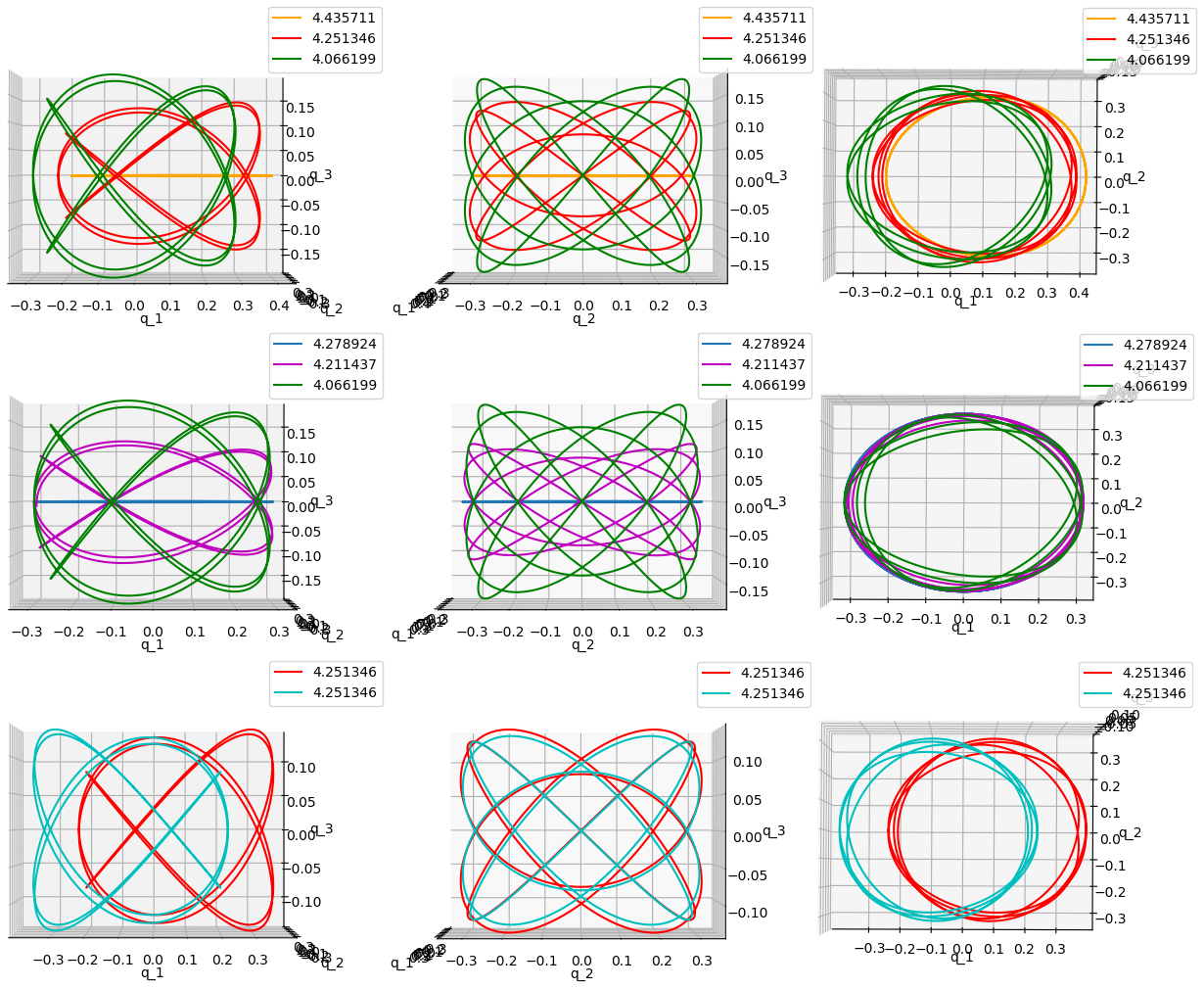}
	\caption{From the 4th cover of $g'$ to the 4th cover of $g$ (family \textcolor{blue}{$f_{g'}^{(1,4)}$})}
	\label{figure_spatial_g'_f_1}
\end{figure}

\begin{table}[H]\scriptsize \centering
	\begin{tabular}{ccccccc}
		$\Gamma$ & $q_1(0)$ &  $\dot{q}_2(0)$ & $\dot{q}_3(0)$ & $T_q / 4$ & $\text{sign}_{C/B}$ and Floquet multipliers & $\mu_{CZ}$\\
		\hline 4.435711 & $-$0.188043 & $-$2.511218 & 0 & 1.34304 & $(+/-)$ $\varphi_p^4 = 2.224$ \& $\varphi_s^4 =0$ & $9+(9 \shortto 11) = 18 \shortto 20$\\
		4.435711 & $-$0.188043 & $-$2.511218 & $-$0.0004 & 1.34305 & $(+/-)$ $\varphi_1 = 2.232$ \& $(-/+)$ $\varphi_2 = 6.261$ & 18\\
		4.384342 & $-$0.198011 & $-$2.339590 & $-$0.6000 & 1.35458 & $(+/-)$ $\varphi_1 = 2.079$ \& $(-/+)$ $\varphi_2 = 6.210$ & 18\\
		4.251346 & $-$0.226833 & $-$1.933856 & $-$0.9900 & 1.38612 & $(+/-)$ $\varphi_1 = 1.616$ \& $(-/+)$ $\varphi_2 = 6.196$ & 18\\
		4.068814 & $-$0.294711 & $-$1.340590 & $-$1.0866 & 1.43386 & $(+/-)$ $\varphi_1 = 0.143$ \& $(-/+)$ $\varphi_2 = 6.195$ & 18\\
		4.068801 & $-$0.294728 & $-$1.340490 & $-$1.0866 & 1.43386 & $1.000 \pm 0.087\text{i}$ \& $0.986 \pm 0.141 \text{i}$ & 18\\
		4.067084 & $-$0.297534 & $-$1.324490 & $-$1.0798 & 1.43439 & $1.094 \pm 0.018\text{i}$ \& $0.913 \pm 0.016 \text{i}$ & 18\\
		4.066595 & $-$0.298720 & $-$1.317990 & $-$1.0766 & 1.43456 & $(-/-)$ $\lambda_1 = 1.119$ \& $(+/+)$ $\lambda_2 = 1.064$ & 18\\
		4.066199 & $-$0.300790 & $-$1.307190 & $-$1.0703 & 1.43477 & $(-/-)$ $\lambda_1 = 1.321$ \& $(+/+)$ $\lambda_2 = 1.001$ & 18\\
		4.066199 & $-$0.300810 & $-$1.307090 & $-$1.0702 & 1.43477 & $(-/-)$ $\lambda_1 = 1.322$ \& $\lambda_2 = 1$ & birth-death\\
		4.066199 & $-$0.300851 & $-$1.306890 & $-$1.0701 & 1.43478 & $(-/-)$ $\lambda = 1.325$ \& $\varphi = 6.276$ & 17\\
		4.083662 & $-$0.304840 & $-$1.312317 & $-$1.0167 & 1.43352 & $(-/-)$ $\lambda = 2.779$ \& $\varphi = 6.275$ & 17\\
		4.211437 & $-$0.302530 & $-$1.514041 & $-$0.6178 & 1.42338 & $(-/-)$ $\lambda = 15.30$ \& $\varphi = 6.275$ & 17\\
		4.278924 & $-$0.301158 & $-$1.623019 & $-$0.0003 & 1.41824 & $(-/-)$ $\lambda = 26.32$ \& $\varphi = 6.276$ & 17\\
		4.278924 & $-$0.301158 & $-$1.623018 & 0 & 1.41824 & $\lambda_p^4 = 26.26$ \& $\varphi_s^4 = 0$ & $8+(9 \shortto 11) = 17 \shortto 19$
	\end{tabular}
	\caption{From the 4th cover of $g'$ to the 4th cover of $g$ (family \textcolor{blue}{$f_{g'}^{(1,4)}$})}
	\label{table_17}
\end{table}

\begin{figure}[H]
	\centering
	\includegraphics[scale=0.59]{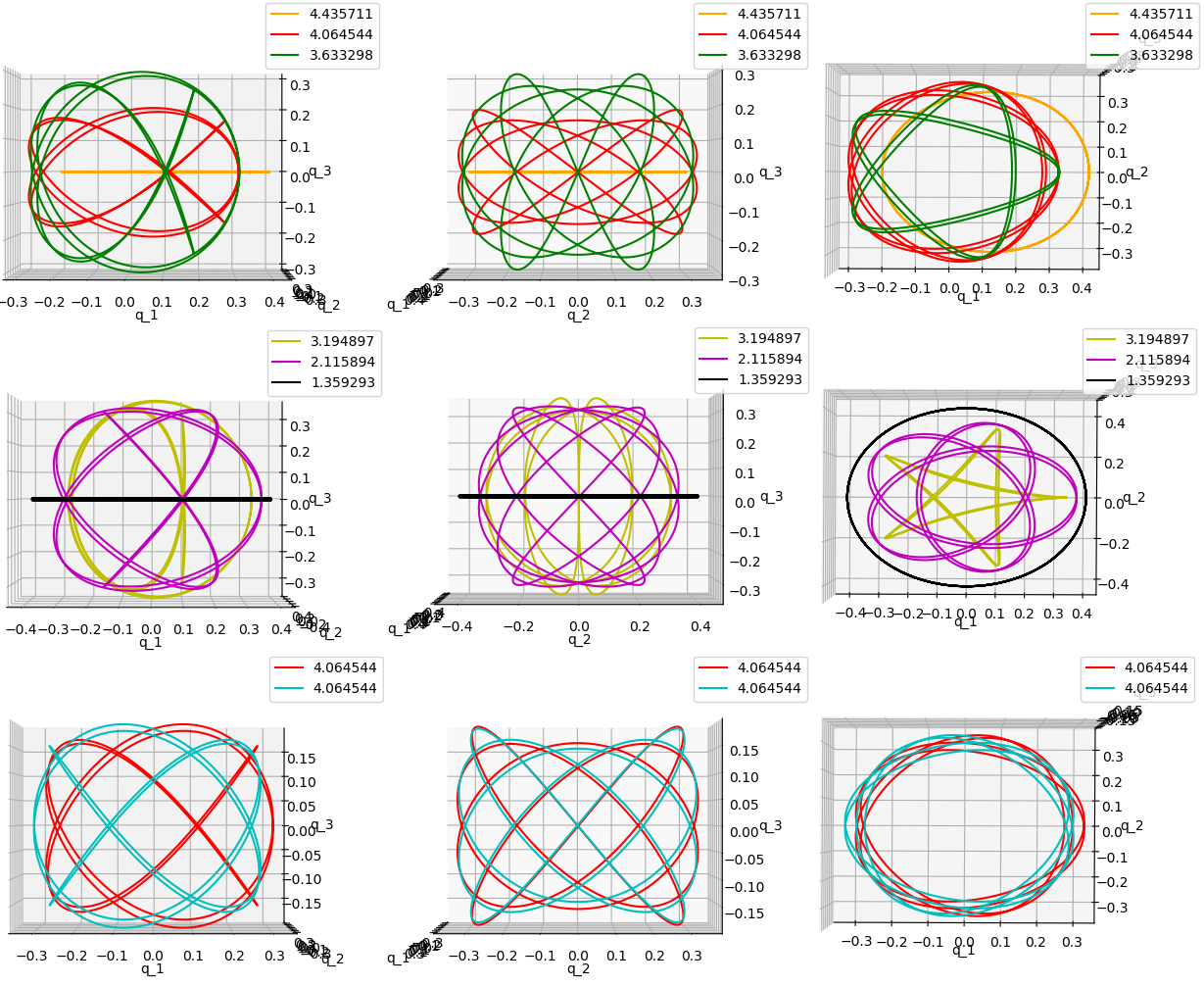}
	\caption{From the 4th cover of $g'$ to the 6th cover of $f$ (family \textcolor{grgr}{$f_{g'}^{(1cut,4)}$})}
	\label{figure_spatial_g'_f}
\end{figure}

\begin{table}[H]\scriptsize \centering
	\begin{tabular}{ccccccc}
		$\Gamma$ & $q_1(0)$ &  $\dot{q}_2(0)$ & $\dot{q}_3(0)$ & $T_q / 4$ & $\text{sign}_{C/B}$ and Floquet multipliers & $\mu_{CZ}$\\
		\hline 4.435711 & 0.399433 & 1.024708 & 0 & 1.34304 & $(+/-)$ $\varphi_p^4 = 2.224$ \& $\varphi_s^4 =0$ & $9+(9 \shortto 11) = 18 \shortto 20$\\
		4.435711 & 0.399433 & 1.024704 & 0.001 & 1.34305 & $(+/-)$ $\varphi = 2.225$ \& $(-/-)$ $\lambda = 1.009$ & 19\\
		4.358078 & 0.388745 & 1.036921 & 0.405 & 1.36065 & $(+/-)$ $\varphi = 1.998$ \& $(-/-)$ $\lambda = 1.045$ & 19\\
		4.064544 & 0.313078 & 1.240223 & 1.039 & 1.43467 & $(+/-)$ $\varphi = 0.298$ \& $(-/-)$ $\lambda = 1.035$ & 19\\
		3.888069 & 0.309540 & 1.005463 & 1.359 & 1.45012 & $(+/-)$ $\varphi = 3.139$ \& $(-/-)$ $\lambda = 1.006$ & 19\\
		3.721469 & 0.313247 & 0.753339 & 1.546 & 1.46539 & $(-/+)$ $\varphi = 4.711$ \& $(-/-)$ $\lambda = 1.007$ & 19\\
		3.634638 & 0.315238 & 0.624474 & 1.617 & 1.47381 & $(-/+)$ $\varphi = 6.123$ \& $(-/-)$ $\lambda = 1.010$ & 19\\
		3.633298 & 0.315269 & 0.622499 & 1.618 & 1.47394 & $(-/-)$ $\lambda_1 = 1.088$ \& $(-/-)$ $\lambda_2 = 1.010$ & 20\\
		3.194897 & 0.325968 & $-$0.000610 & 1.805 & 1.52212 & $(-/-)$ $\lambda_1 = 11.37$ \& $(-/-)$ $\lambda_2 = 1.009$ & 20\\
		3.039053 & 0.330063 & $-$0.210854 & 1.817 & 1.54192 & $(-/-)$ $\lambda_1 = 13.01$ \& $(-/-)$ $\lambda_2 = 1.009$ & 20\\
		2.115894 & 0.358524 & $-$1.326885 & 1.444 & 1.70332 & $(-/-)$ $\lambda_1 = 7.647$ \& $(-/-)$ $\lambda_2 = 1.006$ & 20\\
		1.567391 & 0.380122 & $-$1.874481 & 0.783 & 1.86052 & $(-/-)$ $\lambda_1 = 1.026$ \& $(-/-)$ $\lambda_2 = 1.006$ & 20\\
		1.567110 & 0.380135 & $-$1.874738 & 0.782 & 1.86062 & $(-/-)$ $\lambda = 1.007$ \& $(-/+)$ $\varphi = 6.240$ & 19\\
		1.507647 & 0.382751 & $-$1.928321 & 0.662 & 1.88232 & $(-/-)$ $\lambda = 1.007$ \& $(-/+)$ $\varphi = 5.551$ & 19\\
		1.411644 & 0.387099 & $-$2.012299 & 0.393 & 1.91992 & $(-/-)$ $\lambda = 1.007$ \& $(-/+)$ $\varphi = 5.085$ & 19\\
		1.359329 & 0.389535 & $-$2.056719 & 0.010 & 1.94188 & $(-/-)$ $\lambda = 1.014$ \& $(-/+)$ $\varphi = 4.888$ & 19\\
		1.359293 & 0.389537 & $-$2.05674 & 0 & 1.94179 & $\varphi_s^6 = 0$ \& $(-/+)$ $\varphi_p^6 = 4.886$ & $9+(11 \shortto 9)= 20 \shortto 18$
	\end{tabular}
	\caption{From the 4th cover of $g'$ to the 6th cover of $f$ (family \textcolor{grgr}{$f_{g'}^{(1cut,4)}$})}
	\label{table_16}
\end{table}

\begin{appendix}
	\setcounter{secnumdepth}{0}
	\section{Appendix - Python Codes}
	\label{sec:python_codes}
	
	In this appendix we collect our four Python codes, which we use to compute the data given in Section \ref{sec:8}.\ Our numerical method is the fourth-order Runge-Kutta method, which can be found for instance in \cite[pp.\ 51--53]{sewell}.\
	
	The first and second code are used for planar symmetric periodic orbits.\ From given initial data, the first code gives the starting velocity $\dot{q}_2(0)$, the first return time $T_q$ and the synodic period $T_s$.\ It also plots planar symmetric periodic orbits in $q_1q_2$-coordinates starting at $\text{Fix}(\rho_1) = \{ (q_1,0,0,p_2) \}$.\ The second code follows the description of Subsection \ref{sec:6.5.1}, thus it linearizes along a planar symmetric periodic orbit and calculates the planar reduced monodromy $\overline{A}_p \in \text{Sp}^{\rho_1}(1)$ and $A_s \in \text{Sp}^{\rho_1}(1)$.\ Moreover, it gives the determinant, the trace, the elliptic or hyperbolic behaviour and the Floquet multipliers.\ In elliptic cases, it computes $T_a$ and $T_d$ as well.\ Note that this code is written for the case that $\mu_{CZ}^p=\mu_{CZ}^s=3$ and for the variational orbit of our moon from Section \ref{sec:our moon}.\ 
	
	The third code plots spatial symmetric orbits in $q_1q_2q_3$-coordinates starting at $\text{Fix}(\overline{\rho_1}) = \{ (q_1,0,0,0,p_2,p_3) \}$.\ The fourth code, as described in Subsection \ref{sec:6.5.2}, linearizes along a spatial symmetric periodic orbit starting at $\text{Fix}(\overline{\rho_1})$ and computes the spatial reduced monodromy, which is a matrix in $\text{Sp}^{\rho_1}(2)$, its Floquet multipliers and signatures $\text{sign}_B(\lambda)$ and $\text{sign}_C(\lambda)$.\ Both codes are written for the orbits of Subsection \ref{sec:9.4.1}.\
	
	For any other symmetric periodic orbit one only needs to adjust the respective code with its initial data.\
	\begin{align} \label{python1}
	\text{	\textbf{1st Python code:}}
	\end{align}
	
	\tiny{
\begin{lstlisting}[language=Python,breaklines=true]
import math
import numpy
from matplotlib import pyplot as plt
def planar_orbit(x,g): #x=q1(0), g is energy Gamma
    y = 0 #y=q2(0)
    u = 0 #u=p1(0)=xdot(0)-y
    v_q2 = math.sqrt(3*(x**2) + (2/(abs(x))) - g) #v_q2(0)
    v = v_q2 + x #v=p2(0)=ydot(0)+x
    h = 0.000005
    t = 0
    print('q2dot(0) = ',v_q2)
    print('p2(0) = ',v)
    xArray = []
    yArray = []
    while (t <= 20): #fourth order runge-kutta
        x1 = u + y #q1dot=p1+q2
        y1 = v - x #q2dot=p2-q1
        u1 = v - x + 3*x - x/((math.sqrt(x**2 + y**2))**3) #p1dot
        v1 = - u - y - y/((math.sqrt(x**2 + y**2))**3) #p2dot
        fx2 = x + h/2*x1
        fy2 = y + h/2*y1
        fu2 = u + h/2*u1
        fv2 = v + h/2*v1     
        x2 = fu2 + fy2
        y2 = fv2 - fx2
        u2 = fv2 - fx2 + 3*fx2 - fx2/((math.sqrt(fx2**2 + fy2**2))**3)
        v2 = - fu2 - fy2 - fy2/((math.sqrt(fx2**2 + fy2**2))**3)
        fx3 = x + h/2*x2
        fy3 = y + h/2*y2
        fu3 = u + h/2*u2
        fv3 = v + h/2*v2
        x3 = fu3 + fy3
        y3 = fv3 - fx3
        u3 = fv3 - fx3 + 3*fx3 - fx3/((math.sqrt(fx3**2 + fy3**2))**3)
        v3 = - fu3 - fy3 - fy3/((math.sqrt(fx3**2 + fy3**2))**3)
        fx4 = x + h*x3
        fy4 = y + h*y3
        fu4 = u + h*u3
        fv4 = v + h*v3    
        x4 = fu4 + fy4
        y4 = fv4 - fx4
        u4 = fv4 - fx4 + 3*fx4 - fx4/((math.sqrt(fx4**2 + fy4**2))**3)
        v4 = - fu4 - fy4 - fy4/((math.sqrt(fx4**2 + fy4**2))**3)
        x = x + h/6*(x1 + 2*x2 + 2*x3 + x4)
        y = y + h/6*(y1 + 2*y2 + 2*y3 + y4)
        u = u + h/6*(u1 + 2*u2 + 2*u3 + u4)
        v = v + h/6*(v1 + 2*v2 + 2*v3 + v4)   
        t = t + h  
        xArray.append(x)
        yArray.append(y)
        if (x > 0) and (-0.0001 <= y <= -0.0000000001):
            break        
    print('T_q = ', t,'m = ', 2*(math.pi)/t, 'T_s = ', 365.25/(2*(math.pi)/t))
    plt.figure('planar symmetric orbits') #same window
    plt.plot(xArray, yArray, label=g)
    plt.axis('equal') #same axis, distance
    plt.legend()   
    plt.xlabel('$q_1$')
    plt.ylabel('$q_2$')
x = 0.176097
g = 6.5088
planar_orbit(x,g)
plt.show() #show plot
\end{lstlisting}}
	
	\normalsize{
		\begin{gather} \label{python2}
		\text{	\textbf{2nd Python code:}}
		\end{gather}}
	
	\tiny{
\begin{lstlisting}[language=Python,breaklines=true]
import math
import numpy
import cmath
from matplotlib import pyplot as plt
T_q = 0.507959999999094 #T_q from the 1st code
T_s = 365.25/(2*(math.pi)/T_q) #T_s synodic period
def deltap2(x,g,a):#a=deltaq1(0), b=deltaq2(0), c=deltap1dot(0), d=deltap2dot(0)
    v_q2 = math.sqrt(3*(x**2) + (2/(abs(x))) - g)
    v = v_q2 + x #p2(0)
    d = ( -x/((abs(x))**3) + 2*x + v)*a/(v-x)
    print('deltap2(0) = ', d)
    return(d)
def planar_linearization(x,g,a,b,c,d): #planar linearization along a planar periodic orbit
    v_q2 = math.sqrt(3*(x**2) + (2/(abs(x))) - g)
    v = v_q2 + x #p2(0)
    y = 0
    u = 0
    h = 0.000005
    t = 0
    while (t <= T_q):
        x1 = u + y #q1dot=p1+q2
        y1 = v - x #q2dot=p2-q1
        u1 = v - x + 3*x - x/((math.sqrt(x**2 + y**2))**3) #p1dot
        v1 = - u - y - y/((math.sqrt(x**2 + y**2))**3) #p2dot
        fx2 = x + h/2*x1
        fy2 = y + h/2*y1
        fu2 = u + h/2*u1
        fv2 = v + h/2*v1     
        x2 = fu2 + fy2
        y2 = fv2 - fx2
        u2 = fv2 - fx2 + 3*fx2 - fx2/((math.sqrt(fx2**2 + fy2**2))**3)
        v2 = - fu2 - fy2 - fy2/((math.sqrt(fx2**2 + fy2**2))**3)
        fx3 = x + h/2*x2
        fy3 = y + h/2*y2
        fu3 = u + h/2*u2
        fv3 = v + h/2*v2
        x3 = fu3 + fy3
        y3 = fv3 - fx3
        u3 = fv3 - fx3 + 3*fx3 - fx3/((math.sqrt(fx3**2 + fy3**2))**3)
        v3 = - fu3 - fy3 - fy3/((math.sqrt(fx3**2 + fy3**2))**3)
        fx4 = x + h*x3
        fy4 = y + h*y3
        fu4 = u + h*u3
        fv4 = v + h*v3    
        x4 = fu4 + fy4
        y4 = fv4 - fx4
        u4 = fv4 - fx4 + 3*fx4 - fx4/((math.sqrt(fx4**2 + fy4**2))**3)
        v4 = - fu4 - fy4 - fy4/((math.sqrt(fx4**2 + fy4**2))**3)
        x = x + h/6*(x1 + 2*x2 + 2*x3 + x4)
        y = y + h/6*(y1 + 2*y2 + 2*y3 + y4)
        u = u + h/6*(u1 + 2*u2 + 2*u3 + u4)
        v = v + h/6*(v1 + 2*v2 + 2*v3 + v4)
        xx1 = c + b #linearization
        yy1 = d - a
        uu1 = d - a + ( ( 2*(x**2) - y**2 )/( (math.sqrt(x**2 + y**2))**5 ) + 3 )*a + ( ( 3*x*y )/( (math.sqrt(x**2 + y**2))**5 ) )*b
        vv1 = - c - b + ( ( 3*x*y )/( (math.sqrt(x**2 + y**2))**5 ) )*a + ( ( 2*(y**2) - x**2 )/( (math.sqrt(x**2 + y**2))**5 ))*b
        fa2 = a + h/2*xx1
        fb2 = b + h/2*yy1
        fc2 = c + h/2*uu1
        fd2 = d + h/2*vv1
        xx2 = fc2 + fb2
        yy2 = fd2 - fa2
        uu2 = fd2 - fa2 + ( ( 2*(x**2) - y**2 )/( (math.sqrt(x**2 + y**2))**5 ) + 3 )*fa2 + ( ( 3*x*y )/( (math.sqrt(x**2 + y**2))**5 ) )*fb2
        vv2 = - fc2 - fb2 + ( ( 3*x*y )/( (math.sqrt(x**2 + y**2))**5 ) )*fa2 + ( ( 2*(y**2) - x**2 )/( (math.sqrt(x**2 + y**2))**5 ))*fb2
        fa3 = a + h/2*xx2
        fb3 = b + h/2*yy2
        fc3 = c + h/2*uu2
        fd3 = d + h/2*vv2
        xx3 = fc3 + fb3
        yy3 = fd3 - fa3
        uu3 = fd3 - fa3 + ( ( 2*(x**2) - y**2 )/( (math.sqrt(x**2 + y**2))**5 ) + 3 )*fa3 + ( ( 3*x*y )/( (math.sqrt(x**2 + y**2))**5 ) )*fb3
        vv3 = - fc3 - fb3 + ( ( 3*x*y )/( (math.sqrt(x**2 + y**2))**5 ) )*fa3 + ( ( 2*(y**2) - x**2 )/( (math.sqrt(x**2 + y**2))**5 ))*fb3
        fa4 = a + h*xx3
        fb4 = b + h*yy3
        fc4 = c + h*uu3
        fd4 = d + h*vv3
        xx4 = fc4 + fb4
        yy4 = fd4 - fa4
        uu4 = fd4 - fa4 + ( ( 2*(x**2) - y**2 )/( (math.sqrt(x**2 + y**2))**5 ) + 3 )*fa4 + ( ( 3*x*y )/( (math.sqrt(x**2 + y**2))**5 ) )*fb4
        vv4 = - fc4 - fb4 + ( ( 3*x*y )/( (math.sqrt(x**2 + y**2))**5 ) )*fa4 + ( ( 2*(y**2) - x**2 )/( (math.sqrt(x**2 + y**2))**5 ))*fb4      
        a = a + h/6*(xx1 + 2*xx2 + 2*xx3 + xx4)
        b = b + h/6*(yy1 + 2*yy2 + 2*yy3 + yy4)
        c = c + h/6*(uu1 + 2*uu2 + 2*uu3 + uu4)
        d = d + h/6*(vv1 + 2*vv2 + 2*vv3 + vv4)        
        t = t + h
    print('a = ', a, 'b = ', b, 'c = ', c, 'd = ', d)
    return(a,b,c,d)
def hvf(x,g): #hamiltonian vector field
    v_q2 = math.sqrt(3*(x**2) + (2/(abs(x))) - g)
    p1dot = v_q2 + 3*x - x/((abs(x))**3) #p1dot(0)
    print('a = ', 0 , 'b = ', v_q2, 'c = ', p1dot, 'd = ', 0)
    return (v_q2,p1dot)
x = 0.176097 #basis vectors
g = 6.5088
a = 1
b = 0
c = 0
d = deltap2(x,g,a)
a_1, a_2, a_3, a_4 = planar_linearization(x,g,a,b,c,d)
x = 0.176097
g = 6.5088
a = 0
b = 0
c = - 1
d = 0
b_1, b_2, b_3, b_4 = planar_linearization(x,g,a,b,c,d)
x = 0.176097 #hamiltonian vector field
g = 6.5088
c_1, c_2 = hvf(x,g)
a11 = a_1 #matrix coefficients of \overline{A}_p
a21 = - a_3 + (a_2/c_1)*c_2
a12 = b_1
a22 = - b_3 + (b_2/c_1)*c_2
print('a11 = ',a11, 'a12 = ',a12, '\na21 = ',a21,'a22 = ',a22)
det = a11*a22 - a12*a21
print ('det = ',det)
tr = a11 + a22
print ('tr = a11 + a22 = ',tr)
if abs(tr) < 2:
    x1 = 0.5*tr
    y = 0.5*math.sqrt(4-(tr)**2)
    print('elliptic')
    if a12 < 0:
        cn = complex(x1,y)
        th = cmath.phase(cn)
        print('theta = ', th)
        print('T_a = ',(2*math.pi*T_s)/(2*math.pi + th))
    if a12 > 0:
        cn = complex(x1,-y)
        th = cmath.phase(cn) + 2*math.pi
        print('theta = ', th)
        print('T_a = ',(2*math.pi*T_s)/(2*math.pi + th))
if abs(tr) > 2:
    print('hyperbolic')
    print ('real ev1 = ',0.5*(tr + math.sqrt(abs(tr)**2 - 4)))
    print ('real ev2 = ',0.5*(tr - math.sqrt(abs(tr)**2 - 4)))
    print('1/ev1 = ',1/(0.5*(tr + math.sqrt(abs(tr)**2 - 4))))
def spatial_linearization(x,g,a,b): #spatial linearization along a planar periodic orbit
    v = math.sqrt(3*(x**2) + (2/(abs(x))) - g)
    y = 0 #a=deltaq3(0), b=deltaq3dot(0)=deltap3(0)
    u = 0
    h = 0.000005
    t = 0
    while (t <= T_q):
        m1 = u
        n1 = v
        k1 = 2*v + 3*x - (x/(math.sqrt(x**2 + y**2))**3)
        l1 = -2*u - (y/(math.sqrt(x**2 + y**2))**3)
        fx2 = x + h/2*m1
        fy2 = y + h/2*n1
        m2 = u + h/2*k1
        n2 = v + h/2*l1
        k2 = 2*n2 + 3*fx2 - (fx2/(math.sqrt((fx2)**2 + (fy2)**2))**3)
        l2 = -2*m2 - (fy2/(math.sqrt((fx2)**2 + (fy2)**2))**3)
        fx3 = x + h/2*m2
        fy3 = y + h/2*n2
        m3 = u + h/2*k2
        n3 = v + h/2*l2
        k3 = 2*n3 + 3*fx3 - (fx3/(math.sqrt((fx3)**2 + (fy3)**2))**3)
        l3 = -2*m3 - (fy3/(math.sqrt((fx3)**2 + (fy3)**2))**3)
        fx4 = x + h*m2
        fy4 = y + h*n2
        m4 = u + h*k3
        n4 = v + h*l3
        k4 = 2*n4 + 3*fx4 - (fx4/(math.sqrt((fx4)**2 + (fy4)**2))**3)
        l4 = -2*m4 - (fy4/(math.sqrt((fx4)**2 + (fy4)**2))**3)
        x = x + (h/6)*(m1 + 2*m2 + 2*m3 + m4)
        u = u + (h/6)*(k1 + 2*k2 + 2*k3 + k4)
        y = y + (h/6)*(n1 + 2*n2 + 2*n3 + n4)
        v = v + (h/6)*(l1 + 2*l2 + 2*l3 + l4)
        mm1 = b #linearization
        kk1 = (-1 - 1/((math.sqrt(x**2 + y**2))**3))*a
        fa2 = a + h/2*mm1
        mm2 = b + h/2*kk1
        kk2 = (-1 - 1/((math.sqrt(x**2 + y**2))**3))*fa2
        fa3 = a + h/2*mm2
        mm3 = b + h/2*kk2
        kk3 = (-1 - 1/((math.sqrt(x**2 + y**2))**3))*fa3
        fa4 = a + h*mm3
        mm4 = b + h*kk3
        kk4 = (-1 - 1/((math.sqrt(x**2 + y**2))**3))*fa4
        a = a + (h/6)*(mm1 + 2*mm2 + 2*mm3 + mm4)
        b = b + (h/6)*(kk1 + 2*kk2 + 2*kk3 + kk4)
        t = t + h
    print('a = ',a, 'b = ',b)
    return(a,b)
x = 0.176097 #two Lagrangian basis vectors
g = 6.5088
a = 1
b = 0
e_1, e_2 = spatial_linearization(x,g,a,b)
x = 0.176097
g = 6.5088
a = 0
b = -1
f_1, f_2 = spatial_linearization(x,g,a,b)
a11 = e_1 #matrix coefficients of A_s
a21 = - e_2
a12 = f_1
a22 = - f_2
print('a11 = ',a11, 'a12 = ',a12, '\na21 = ',a21,'a22 = ',a22)
det = a11*a22 - a12*a21
print ('det = ',det)
tr = a11 + a22
print ('tr = a11 + a22 = ',tr)
if abs(tr) < 2:
    x1 = 0.5*tr
    y = 0.5*math.sqrt(4-(tr)**2)
    print('elliptic')
    if a12 < 0:
        cn = complex(x1,y)
        th = cmath.phase(cn)
        print('vartheta = ', th)
        print('T_d = ',(2*math.pi*T_s)/(2*math.pi + th))
    if a12 > 0:
        cn = complex(x1,-y)
        th = cmath.phase(cn) + 2*math.pi
        print('vartheta = ', th)
        print('T_d = ',(2*math.pi*T_s)/(2*math.pi + th))
if abs(tr) > 2:
    print('hyperbolic')
    print ('real ev1 = ',0.5*(tr + math.sqrt(abs(tr)**2 - 4)))
    print ('real ev2 = ',0.5*(tr - math.sqrt(abs(tr)**2 - 4)))
    print('1/ev1 = ',1/(0.5*(tr + math.sqrt(abs(tr)**2 - 4))))
\end{lstlisting}}
	
	\normalsize{
		\begin{gather*} \label{python3}
		\text{	\textbf{3rd Python code:}}
		\end{gather*}}
	
	\tiny{
\begin{lstlisting}[language=Python,breaklines=true]
import math
import numpy as np
from matplotlib import pyplot as plt
from mpl_toolkits import mplot3d
def spatial_orbit(x,v_q2,g): #x=q1(0)
    y = 0 #y=q2(0)
    z = 0 #z=q3(0)
    u = 0 #u=p1()
    v = v_q2 + x #v=p2(0)=ydot+x
    w = math.sqrt(3*(x**2) + (2/(abs(x))) - (v_q2)**2 - g) #w=p3(0)=q3dot(0)
    h = 0.000005
    t = 0
    print('p3(0)=q3dot(0) = ',w)
    xArray = []
    yArray = []
    zArray = []
    while (t <= 1.0596*4): #fourth order runge kutta method
        x1 = u + y #q1dot=p1+q2
        y1 = v - x #q2dot=p2-q1
        z1 = w #q3dot=w=p3
        u1 = v - x + 3*x - x/((math.sqrt(x**2 + y**2 + z**2))**3) #p1dot
        v1 = - u - y - y/((math.sqrt(x**2 + y**2 + z**2))**3) #p2dot
        w1 = (-z)/((math.sqrt(x**2 + y**2 + z**2))**3) - z #p3dot
        fx2 = x + h/2*x1
        fy2 = y + h/2*y1
        fz2 = z + h/2*z1
        fu2 = u + h/2*u1
        fv2 = v + h/2*v1
        fw2 = w + h/2*w1
        x2 = fu2 + fy2
        y2 = fv2 - fx2
        z2 = fw2
        u2 = fv2 - fx2 + 3*fx2 - fx2/((math.sqrt(fx2**2 + fy2**2 + fz2**2))**3)
        v2 = - fu2 - fy2 - fy2/((math.sqrt(fx2**2 + fy2**2 + fz2**2))**3)
        w2 = (-fz2)/((math.sqrt(fx2**2 + fy2**2 + fz2**2))**3) - fz2
        fx3 = x + h/2*x2
        fy3 = y + h/2*y2
        fz3 = z + h/2*z2
        fu3 = u + h/2*u2
        fv3 = v + h/2*v2
        fw3 = w + h/2*w2
        x3 = fu3 + fy3
        y3 = fv3 - fx3
        z3 = fw3
        u3 = fv3 - fx3 + 3*fx3 - fx3/((math.sqrt(fx3**2 + fy3**2 + fz3**2))**3)
        v3 = - fu3 - fy3 - fy3/((math.sqrt(fx3**2 + fy3**2 + fz3**2))**3)
        w3 = (-fz3)/((math.sqrt(fx3**2 + fy3**2 + fz3**2))**3) - fz3
        fx4 = x + h*x3
        fy4 = y + h*y3
        fz4 = z + h*z3
        fu4 = u + h*u3
        fv4 = v + h*v3
        fw4 = w + h*w3  
        x4 = fu4 + fy4
        y4 = fv4 - fx4
        z4 = fw4
        u4 = fv4 - fx4 + 3*fx4 - fx4/((math.sqrt(fx4**2 + fy4**2 + fz4**2))**3)
        v4 = - fu4 - fy4 - fy4/((math.sqrt(fx4**2 + fy4**2 + fz4**2))**3)
        w4 = (-fz4)/((math.sqrt(fx4**2 + fy4**2 + fz4**2))**3) - fz4
        x = x + h/6*(x1 + 2*x2 + 2*x3 + x4)
        y = y + h/6*(y1 + 2*y2 + 2*y3 + y4)
        z = z + h/6*(z1 + 2*z2 + 2*z3 + z4)
        u = u + h/6*(u1 + 2*u2 + 2*u3 + u4)
        v = v + h/6*(v1 + 2*v2 + 2*v3 + v4)
        w = w + h/6*(w1 + 2*w2 + 2*w3 + w4)
        t = t + h
        xArray.append(x)
        yArray.append(y)
        zArray.append(z)
    fig = plt.figure('spatial symmetric orbits')
    ax = fig.gca(projection ='3d') #same window
    ax.plot3D(xArray, yArray, zArray,label=g)
    ax.legend()
    ax.set_xlabel('q_1')
    ax.set_ylabel('q_2')
    ax.set_zlabel('q_3')
x = 0.565996
v_q2 = 0.438274
g = 4.212563
spatial_orbit(x,v_q2,g)
plt.show() #show plot
\end{lstlisting}}
	
	\normalsize{
		\begin{gather*} \label{python4}
		\text{	\textbf{4th Python code:}}
		\end{gather*}}
	
	\tiny{
\begin{lstlisting}[language=Python,breaklines=true]
import math
import numpy as np
from matplotlib import pyplot as plt
T_q = 1.0596*4
def spatial_linearization(x,v_q2,g,a,b,c,d,e,f):
    y = 0#a=deltaq1(0), b=deltaq2(0), c=deltaq3(0), d=deltap1(0), e=deltap2(0), f=deltap3(0)
    z = 0
    u = 0
    v = v_q2 + x
    w = math.sqrt(2/abs(x) + 3*(x**2) - (v_q2)**2 - g)
    print('dotq3(0) = ',w)
    h = 0.000005
    t = 0
    while (t <= T_q): #fourth order runge kutta method
        x1 = u + y
        y1 = v - x
        z1 = w
        u1 = v - x + 3*x - x/((math.sqrt(x**2 + y**2 + z**2))**3)
        v1 = - u - y - y/((math.sqrt(x**2 + y**2 + z**2))**3)
        w1 = (-z)/((math.sqrt(x**2 + y**2 + z**2))**3) - z
        fx2 = x + h/2*x1
        fy2 = y + h/2*y1
        fz2 = z + h/2*z1
        fu2 = u + h/2*u1
        fv2 = v + h/2*v1
        fw2 = w + h/2*w1      
        x2 = fu2 + fy2
        y2 = fv2 - fx2
        z2 = fw2
        u2 = fv2 - fx2 + 3*fx2 - fx2/((math.sqrt(fx2**2 + fy2**2 + fz2**2))**3)
        v2 = - fu2 - fy2 - fy2/((math.sqrt(fx2**2 + fy2**2 + fz2**2))**3)
        w2 = (-fz2)/((math.sqrt(fx2**2 + fy2**2 + fz2**2))**3) - fz2
        fx3 = x + h/2*x2
        fy3 = y + h/2*y2
        fz3 = z + h/2*z2
        fu3 = u + h/2*u2
        fv3 = v + h/2*v2
        fw3 = w + h/2*w2
        x3 = fu3 + fy3
        y3 = fv3 - fx3
        z3 = fw3
        u3 = fv3 - fx3 + 3*fx3 - fx3/((math.sqrt(fx3**2 + fy3**2 + fz3**2))**3)
        v3 = - fu3 - fy3 - fy3/((math.sqrt(fx3**2 + fy3**2 + fz3**2))**3)
        w3 = (-fz3)/((math.sqrt(fx3**2 + fy3**2 + fz3**2))**3) - fz3
        fx4 = x + h*x3
        fy4 = y + h*y3
        fz4 = z + h*z3
        fu4 = u + h*u3
        fv4 = v + h*v3
        fw4 = w + h*w3      
        x4 = fu4 + fy4
        y4 = fv4 - fx4
        z4 = fw4
        u4 = fv4 - fx4 + 3*fx4 - fx4/((math.sqrt(fx4**2 + fy4**2 + fz4**2))**3)
        v4 = - fu4 - fy4 - fy4/((math.sqrt(fx4**2 + fy4**2 + fz4**2))**3)
        w4 = (-fz4)/((math.sqrt(fx4**2 + fy4**2 + fz4**2))**3) - fz4
        x = x + h/6*(x1 + 2*x2 + 2*x3 + x4)
        y = y + h/6*(y1 + 2*y2 + 2*y3 + y4)
        z = z + h/6*(z1 + 2*z2 + 2*z3 + z4)
        u = u + h/6*(u1 + 2*u2 + 2*u3 + u4)
        v = v + h/6*(v1 + 2*v2 + 2*v3 + v4)
        w = w + h/6*(w1 + 2*w2 + 2*w3 + w4)
        xx1 = d + b #linearization
        yy1 = e - a
        zz1 = f
        uu1 = e - a + ( ( 2*(x**2) - y**2 - z**2 )/( (math.sqrt(x**2 + y**2 + z**2))**5 ) + 3 )*a + ( ( 3*x*y )/( (math.sqrt(x**2 + y**2 + z**2))**5 ) )*b + ( ( 3*x*z )/( (math.sqrt(x**2 + y**2 + z**2))**5 ) )*c
        vv1 = - d - b + ( ( 3*x*y )/( (math.sqrt(x**2 + y**2 + z**2))**5 ) )*a + ( ( 2*(y**2) - x**2 - z**2 )/( (math.sqrt(x**2 + y**2 + z**2))**5 ))*b + ( ( 3*y*z )/( (math.sqrt(x**2 + y**2 + z**2))**5 ))*c
        ww1 = ( ( 3*x*z )/( (math.sqrt(x**2 + y**2 + z**2))**5 ) )*a + ( ( 3*y*z )/( (math.sqrt(x**2 + y**2 + z**2))**5 ) )*b + ( ( 2*(z**2) - x**2 - y**2 )/( (math.sqrt(x**2 + y**2 + z**2))**5 ) - 1 )*c    
        fa2 = a + h/2*xx1
        fb2 = b + h/2*yy1
        fc2 = c + h/2*zz1
        fd2 = d + h/2*uu1
        fe2 = e + h/2*vv1
        ff2 = f + h/2*ww1
        xx2 = fd2 + fb2
        yy2 = fe2 - fa2
        zz2 = ff2
        uu2 = fe2 - fa2 + ( ( 2*(x**2) - y**2 - z**2 )/( (math.sqrt(x**2 + y**2 + z**2))**5 ) + 3 )*fa2 + ( ( 3*x*y )/( (math.sqrt(x**2 + y**2 + z**2))**5 ) )*fb2 + ( ( 3*x*z )/( (math.sqrt(x**2 + y**2 + z**2))**5 ) )*fc2
        vv2 = - fd2 - fb2 + ( ( 3*x*y )/( (math.sqrt(x**2 + y**2 + z**2))**5 ) )*fa2 + ( ( 2*(y**2) - x**2 - z**2 )/( (math.sqrt(x**2 + y**2 + z**2))**5 ))*fb2 + ( ( 3*y*z )/( (math.sqrt(x**2 + y**2 + z**2))**5 ))*fc2
        ww2 = ( ( 3*x*z )/( (math.sqrt(x**2 + y**2 + z**2))**5 ) )*fa2 + ( ( 3*y*z )/( (math.sqrt(x**2 + y**2 + z**2))**5 ) )*fb2 + ( ( 2*(z**2) - x**2 - y**2 )/( (math.sqrt(x**2 + y**2 + z**2))**5 ) - 1 )*fc2        
        fa3 = a + h/2*xx2
        fb3 = b + h/2*yy2
        fc3 = c + h/2*zz2
        fd3 = d + h/2*uu2
        fe3 = e + h/2*vv2
        ff3 = f + h/2*ww2
        xx3 = fd3 + fb3
        yy3 = fe3 - fa3
        zz3 = ff3
        uu3 = fe3 - fa3 + ( ( 2*(x**2) - y**2 - z**2 )/( (math.sqrt(x**2 + y**2 + z**2))**5 ) + 3 )*fa3 + ( ( 3*x*y )/( (math.sqrt(x**2 + y**2 + z**2))**5 ) )*fb3 + ( ( 3*x*z )/( (math.sqrt(x**2 + y**2 + z**2))**5 ) )*fc3
        vv3 = - fd3 - fb3 + ( ( 3*x*y )/( (math.sqrt(x**2 + y**2 + z**2))**5 ) )*fa3 + ( ( 2*(y**2) - x**2 - z**2 )/( (math.sqrt(x**2 + y**2 + z**2))**5 ))*fb3 + ( ( 3*y*z )/( (math.sqrt(x**2 + y**2 + z**2))**5 ))*fc3
        ww3 = ( ( 3*x*z )/( (math.sqrt(x**2 + y**2 + z**2))**5 ) )*fa3 + ( ( 3*y*z )/( (math.sqrt(x**2 + y**2 + z**2))**5 ) )*fb3 + ( ( 2*(z**2) - x**2 - y**2 )/( (math.sqrt(x**2 + y**2 + z**2))**5 ) - 1 )*fc3      
        fa4 = a + h*xx3
        fb4 = b + h*yy3
        fc4 = c + h*zz3
        fd4 = d + h*uu3
        fe4 = e + h*vv3
        ff4 = f + h*ww3
        xx4 = fd4 + fb4
        yy4 = fe4 - fa4
        zz4 = ff4
        uu4 = fe4 - fa4 + ( ( 2*(x**2) - y**2 - z**2 )/( (math.sqrt(x**2 + y**2 + z**2))**5 ) + 3 )*fa4 + ( ( 3*x*y )/( (math.sqrt(x**2 + y**2 + z**2))**5 ) )*fb4 + ( ( 3*x*z )/( (math.sqrt(x**2 + y**2 + z**2))**5 ) )*fc4
        vv4 = - fd4 - fb4 + ( ( 3*x*y )/( (math.sqrt(x**2 + y**2 + z**2))**5 ) )*fa4 + ( ( 2*(y**2) - x**2 - z**2 )/( (math.sqrt(x**2 + y**2 + z**2))**5 ))*fb4 + ( ( 3*y*z )/( (math.sqrt(x**2 + y**2 + z**2))**5 ))*fc4
        ww4 = ( ( 3*x*z )/( (math.sqrt(x**2 + y**2 + z**2))**5 ) )*fa4 + ( ( 3*y*z )/( (math.sqrt(x**2 + y**2 + z**2))**5 ) )*fb4 + ( ( 2*(z**2) - x**2 - y**2 )/( (math.sqrt(x**2 + y**2 + z**2))**5 ) - 1 )*fc4       
        a = a + h/6*(xx1 + 2*xx2 + 2*xx3 + xx4)
        b = b + h/6*(yy1 + 2*yy2 + 2*yy3 + yy4)
        c = c + h/6*(zz1 + 2*zz2 + 2*zz3 + zz4)
        d = d + h/6*(uu1 + 2*uu2 + 2*uu3 + uu4)
        e = e + h/6*(vv1 + 2*vv2 + 2*vv3 + vv4)
        f = f + h/6*(ww1 + 2*ww2 + 2*ww3 + ww4)        
        t = t + h
    print('a = ', a, 'b = ', b, 'c = ', c, 'd = ', d, 'e = ', e,'f = ',f)
    return(a,b,c,d,e,f)
def hvf(x,v_q2,g): #hamiltonian vector field
    w = math.sqrt(2/abs(x) + 3*(x**2) - (v_q2)**2 - g)
    p1dot = v_q2 + 3*x - x/(abs(x)**3) #p1dot(0)
    print('a = ', 0 , 'b = ', v_q2, 'c = ', w, 'd = ', p1dot, 'e = ',0,'f = ',0)
    return (v_q2,w,p1dot)
x = 0.565996
v_q2 = 0.438274
g = 4.212563
a = 1
b = 0
c = 0
d = 0
e = 0
f = (-x/(abs(x)**3) + 3*x + v_q2)/math.sqrt(2/abs(x) + 3*(x**2) - (v_q2)**2 - g)
a_1, a_2, a_3, a_4, a_5, a_6 = spatial_linearization(x,v_q2,g,a,b,c,d,e,f)
x = 0.565996
v_q2 = 0.438274
g = 4.212563
a = 0
b = 0
c = 0
d = 0
e = 1
f = (-v_q2)/math.sqrt(2/abs(x) + 3*(x**2) - (v_q2)**2 - g)
b_1, b_2, b_3, b_4, b_5, b_6 = spatial_linearization(x,v_q2,g,a,b,c,d,e,f)
x = 0.565996
v_q2 = 0.438274
g = 4.212563
a = 0
b = 0
c = 0
d = -1
e = 0
f = 0
c_1, c_2, c_3, c_4, c_5, c_6 = spatial_linearization(x,v_q2,g,a,b,c,d,e,f)
x = 0.565996
v_q2 = 0.438274
g = 4.212563
a = 0
b = 1
c = 0
d = 0
e = 0
f = 0
d_1, d_2, d_3, d_4, d_5, d_6 = spatial_linearization(x,v_q2,g,a,b,c,d,e,f)
x = 0.565996
v_q2 = 0.438274
g = 4.212563
e_1, e_2, e_3 = hvf(x,v_q2,g)
a11 = a_1 #matrix coefficients
a12 = b_1
a13 = c_1
a14 = d_1
b11 = a_5
b12 = b_5
b13 = c_5
b14 = d_5
c11 = (a_3/e_2)*e_3 - a_4
c12 = (b_3/e_2)*e_3 - b_4
c13 = (c_3/e_2)*e_3 - c_4
c14 = (d_3/e_2)*e_3 - d_4
d11 = a_2 - (a_3/e_2)*e_1
d12 = b_2 - (b_3/e_2)*e_1
d13 = c_2 - (c_3/e_2)*e_1
d14 = d_2 - (d_3/e_2)*e_1
reduced_monodromy = [ [a11,a12,a13,a14] , [b11,b12,b13,b14], [c11,c12,c13,c14], [d11,d12,d13,d14] ]
print('spatial reduced monodromy = ')
print(a11,a12,a13,a14) #1st row
print(b11,b12,b13,b14) #2nd row
print(c11,c12,c13,c14) #3rd row
print(d11,d12,d13,d14) #4th row
mA = [ [a11,a12], [b11,b12] ]
mB = [ [a13,a14], [b13,b14] ]
mC = [ [c11,c12], [d11,d12] ]
mD = [ [c13,c14], [d13,d14] ]
mAt = [ [a11,b11], [a12,b12] ]
mBt = [ [a13,b13], [a14,b14] ]
mCt = [ [c11,d11], [c12,d12] ]
mDt = [ [c13,d13], [c14,d14] ]
print('det = ', np.linalg.det(reduced_monodromy))
print('AtC = ', np.matmul(mAt,mC))
print('BtD = ', np.matmul(mBt,mD))
print('AtD - CtB = ', np.matmul(mAt,mD) - np.matmul(mCt,mB))
ew, ev = np.linalg.eig(reduced_monodromy) #eigenvalues
print('ew = ',ew)
ewAt, evAt = np.linalg.eig( mAt ) #eigenvectors of A^T
print('ewAt = ', ewAt)
v1,v2 = evAt[0]
v3,v4 = evAt[1]
evAt1 = [ [v1,v3] ]
evAt2 = [ [v2,v4] ]
print('evAt1 = ', evAt1)
print('evAt2 = ', evAt2)
evAt1t = [ [v1],[v3] ]
evAt2t = [ [v2],[v4] ]
print('vAt_1tBvAt_1 = ', np.matmul(np.matmul(evAt1,mB),evAt1t)) #sign_B
print('vAt_2tBvAt_2 = ', np.matmul(np.matmul(evAt2,mB),evAt2t))
ewA, evA = np.linalg.eig( mA ) #eigenvectors of A
print('ewA = ', ewA)
w1,w2 = evA[0]
w3,w4 = evA[1]
evA1 = [ [w1,w3] ]
evA2 = [ [w2,w4] ]
print('evA1 = ', evA1)
print('evA2 = ', evA2)
evA1t = [ [w1],[w3] ]
evA2t = [ [w2],[w4] ]
print('vA_1tCvA_1 = ', np.matmul(np.matmul(evA1,mC),evA1t)) #sign_C
print('vA_2tCvA_2 = ', np.matmul(np.matmul(evA2,mC),evA2t))
\end{lstlisting}}
	
\end{appendix}

(C. Aydin) \textsc{Institut de mathématiques, Université de Neuchâtel, Rue Emile-Argand 11, CH-2000 Neuchâtel - Switzerland}\newline
\textit{E-mail address:} \textbf{cengiz.aydin@unine.ch}

\end{document}